\documentclass[10pt]{article} 
\usepackage[margin=1in]{geometry}
 
\usepackage{listings}
\usepackage{geometry}

\usepackage{tikz}
\usetikzlibrary{backgrounds}
\usepackage{amsmath,amsthm,amsfonts,amssymb,amscd}
\usepackage[colorlinks=true, pdfstartview=FitV, linkcolor=blue,citecolor=blue, urlcolor=blue]{hyperref}

\usepackage{mathtools}
\usepackage{mathabx}
\usepackage{float}
\usepackage{makeidx}
\usepackage{stmaryrd}
\usepackage{mathrsfs}
\usepackage{emptypage}
\usepackage{xfrac}
\usepackage{bbm}

\usepackage[doi=false,isbn=false,url=false,eprint=false,maxbibnames=99]{biblatex} %Imports biblatex package
\usepackage{graphicx}

\renewcommand{\epsilon}{\varepsilon}

\newcommand{\p}{\ensuremath{\partial}}

\newcommand{\mc}{\ensuremath{\mathcal}}

\definecolor{labelkey}{rgb}{0,0,1}

\def\les{\lesssim}

\def\eps{\varepsilon}
\renewcommand*{\div}{\ensuremath{\mathrm{div\,}}}

\newcommand{\ga}{\gamma}
\newcommand{\al}{\alpha}

\newcommand{\mt}{\widetilde}

\newcommand*{\supp}{\ensuremath{\mathrm{supp\,}}}

\renewcommand*{\tilde}{\widetilde}
\renewcommand*{\hat}{\widehat}
\renewcommand*{\bar}{\overline}

\newcommand{\T}{{\mathbb T}}

\newtheorem{theorem}{Theorem}[section]
\newtheorem{lemma}[theorem]{Lemma}
\newtheorem{prop}[theorem]{Proposition}
\newtheorem{proposition}[theorem]{Proposition}

\theoremstyle{definition}
\newtheorem{definition}[theorem]{Definition}

\newtheorem{remark}[theorem]{Remark}

\numberwithin{equation}{section}

\def\p{\partial}

\def\f1r{{\frac{1}{r}}  }
\def\p{\partial}

\def\f1r{{\frac{1}{r}}  }

\allowdisplaybreaks[4]

\title{Non-radial implosion for compressible Euler and Navier-Stokes in $\T^3$ and $\mathbb{R}^3$}
 \author{G. Cao-Labora, J. G\'omez-Serrano, J. Shi, G. Staffilani}

\date{} %
\begin{document}

\maketitle
 \begin{abstract}
 In this paper we construct smooth, non-radial solutions of the compressible Euler and Navier-Stokes equation that develop an imploding finite time singularity. Our construction is motivated by the works \cite{Merle-Raphael-Rodnianski-Szeftel:implosion-i,Merle-Raphael-Rodnianski-Szeftel:implosion-ii,Buckmaster-CaoLabora-GomezSerrano:implosion-compressible}, but is flexible enough to handle both periodic and non-radial initial data. 
 \end{abstract}
\tableofcontents

\section{Introduction}

In this paper we are concerned with the compressible Euler and Navier-Stokes equations, corresponding to solutions of

\begin{align} \begin{split} \label{eq:CNS}
\rho \p_t u + \rho u \cdot \nabla u &= - \frac{1}{\gamma} \nabla (\rho^{\gamma}) + \nu \Delta u   \\
\p_t \rho + \rm{ div } (\rho u) &= 0.
\end{split} \end{align}
where $\gamma > 1$ is the adiabatic constant and $\nu = 0$ for Euler, $\nu =1$ for Navier-Stokes. We will consider two domains: $y\in \mathbb{T}^{3}$ (torus) or $y\in \mathbb{R}^{3}$ (whole space).
 
Specifically, we will be concerned with the construction of \textit{imploding singularities}, that is solutions of \eqref{eq:CNS} that develop a finite time singularity where both $\| u \|_{L^\infty},\| \rho \|_{L^\infty}$ become infinite in finite time, as opposed to a \textit{shock}, where $u,\rho$ stay bounded but their gradients blow-up.

\subsection{Historical Background}

Generically, finite time singularities for the compressible Euler and Navier-Stokes equations are in the form of shocks. By using Riemann invariants, Lax \cite{Lax:singularities-nonlinear-hyperbolic-pde} showed that in 1D shocks may develop. His results were later generalized and improved by John \cite{John:singularities-1d-wave} and Liu \cite{Liu:singularities-nlw-quasilinear-hyperbolic-pde}.

Later, Sideris \cite{Sideris:singularities-3d-compressible} employed a virial type argument to prove singularity formation in 2D and 3D. Yin \cite{Yin:formation-shock-waves-3d-compressible-euler} proved shock formation and development in 3D under spherical symmetry. In a series of landmark monographs
\cite{Christodoulou:shock-development-problem,Christodoulou:formation-shocks-3d-fluids-book,Christodoulou-Lisibach:shock-development-spherical-symmetry,Christodoulou-Miao:compressible-flow-euler} Christodoulou, as well as Christodoulou--Miao and Christodoulou--Lisibach investigated the shock formation and the shock development problem without radial symmetry, and developed a new framework to study such problems. 
In recent years, there has been an emergence of new and exciting results. Buckmaster, Drivas, Neal, Rickard, Shkoller, and Vicol \cite{Buckmaster-Shkoller-Vicol:formation-shocks-2d-isentropic-compressible-euler,Buckmaster-Shkoller-Vicol:point-shocks-3d-compressible-euler,Buckmaster-Shkoller-Vicol:shock-formation-vorticity-creation-3d-euler,Buckmaster-Drivas-Shkoller-Vicol:simultaneous-development-shocks-cusps-2d-euler,Neal-Rickard-Shkoller-Vicol:stable-shock,Neal-Shkoller-Vicol:characteristics-shock-2d-euler-symmetry-entropy} developed a new framework and studied the shock formation problem in 2D and 3D, even after the first blow-up time or in the context of low regularity, as well as the classification of pre-shocks and the development problem. Luk--Speck \cite{Luk-Speck:shocks-2d-compressible-euler-vorticity,Luk-Speck:stability-shock-3d-euler-vorticity-entropy} proved the first result of shock formation involving non-zero vorticity and variable entropy, within the framework developed by Christodoulou. Abbrescia--Speck \cite{Abbrescia-Speck:singular-boundary-3d-compressible-euler} studied the maximal development problem. An--Chen--Yin \cite{An-Chen-Yin:ill-posedness-2d-mhd,An-Chen-Yin:ill-posedness-3d-mhd} proved ill-posedness for compressible MHD and compressible Euler in 2D and 3D. See also the work of Su \cite{Su:shock-formation-2d-euler-self-similar}. For more references and developments in this direction we refer the reader to the ICM survey of Buckmaster--Drivas--Shkoller--Vicol \cite{Buckmaster-Drivas-Shkoller-Vicol:icm-survey}.

Nonetheless, shock formation is not the only possible singularity that can occur, and self-similar solutions exist. Guderley \cite{Guderley:singularities-radial} was the first to construct radial, self-similar imploding singularities, though non-smooth.
Guderley's construction was later improved and generalized by Jenssen--Tsikkou 
\cite{Jenssen-Tsikkou:amplitude-blowup-radial-isentropic-euler,Jenssen-Tsikkou:radially-symmetric-non-isentropic-euler}. See also the review \cite{MeyerTerVehn-Schalk:selfsimilar-compression-waves} by Meyer-ter-Vehn--Schalk and the work of Sedov \cite{Sedov:book-similarity}.

In the case of compressible Navier-Stokes, Germain--Iwabuchi \cite{Germain-Iwabuchi:self-similar-compressible-ns} constructed forward smooth, self-similar solutions, as well as non-singular ones due to cavitation. In \cite{Germain-Iwabuchi-Leger:backward-self-similar-compressible-ns}, Germain--Iwabuchi--L\'eger studied conditions on the (non)-existence of backward self-similar solutions and settled the existence in the case with degenerate density-dependent viscosity in \cite{Germain-Iwabuchi-Leger:self-similar-degenerate-compressible-ns}. We also mention here the work of Guo--Jiang \cite{Guo-Jiang:self-similar-isothermal-compressible-ns} in the isothermal case and of Li--Chen--Xie \cite{Li-Chen-Xie:self-similar-compressible-1d} in the density-dependent viscosity case and the work by Lazarus 
\cite{Lazarus:selfsimilar-shocks-cavities,Lazarus:selfsimilar-shocks-cavities-erratum} on the converging shock and collapsing cavity problems.
For more examples of self-similar solutions in the context of fluid mechanics we refer the reader to the review paper by Eggers--Fontelos \cite{Eggers-Fontelos:self-similarity}. Xin \cite{Xin:blowup-compressible-navier-stokes-compact-density}, and Rozanova \cite{Rozanova:blow-up-decreasing-solutions-compressible-ns} proved blow-up for compressible Navier-Stokes in the case of compactly supported and rapidly decaying density respectively.

In a series of breakthrough papers \cite{Merle-Raphael-Rodnianski-Szeftel:implosion-i,Merle-Raphael-Rodnianski-Szeftel:implosion-ii,Merle-Raphael-Rodnianski-Szeftel:implosion-nls}, Merle--Rapha\"el--Rodnianski--Szeftel constructed first smooth, radially imploding solutions to the compressible Euler equation and later built upon them to construct imploding singularities for compressible Navier-Stokes (with decaying density at infinity) and the energy supercritical defocusing nonlinear Schrödinger equation. In their construction, they established a sequence of quantized self-similar scalings accumulating to a critical value, leading to a sequence of imploding self-similar profiles. This was done for almost every value of $\gamma$. In \cite{Buckmaster-CaoLabora-GomezSerrano:implosion-compressible} (see also the review paper \cite{Buckmaster-CaoLabora-GomezSerrano:implosion-compressible-review}), the first two authors together with Buckmaster improved the result to cover all cases of $\gamma$. Moreover, they showed the existence of non-decaying imploding singularities for Navier-Stokes. For a very careful and detailed numerical study of smooth imploding solutions as well as their stability we point out to the work of Biasi \cite{Biasi:self-similar-compressible-euler}.

\subsection{Main result} \label{sec:12}
In \cite{Merle-Raphael-Rodnianski-Szeftel:implosion-ii}, Merle--Rapha\"el--Rodnianski--Szeftel wrote:

\begin{itemize}
\item \textit{
One could, in principle, be
able to reconnect the profile to one with constant density for large x and rapidly
decaying velocity, instead. This should lead to a singularity formation result for
Navier-Stokes for solutions with constant density at infinity. Even more generally,
the analysis should be amenable to other boundary conditions and domains, e.g. Navier-Stokes and Euler equations on a torus. %An example of such adaptation in the context of a nonlinear heat equation and a domain with Dirichlet boundary condition is provided by \cite{Collot:nonradial-blowup-semilinear-heat}.
}

\item ...[The main theorem in \cite{Merle-Raphael-Rodnianski-Szeftel:implosion-ii}] \textit{is proved for spherically symmetric
initial data. The symmetry is used in a very soft way, and we expect that the blow
up ... is stable modulo finitely many instabilities for non symmetric
perturbations.}

\end{itemize}

Our main Theorems answer these two questions on the positive. Broadly speaking, we are able to construct finite time imploding singularities for both Euler and Navier-Stokes with either: 
\begin{enumerate}
    \item smooth periodic initial data (\textbf{Theorem \ref{th:periodic}}).
    \item smooth, non-radially symmetric initial data that does not vanish at infinity (\textbf{Theorem \ref{th:euclidean}}). 
\end{enumerate}

The main difficulty in order to tackle this problem is that, contrary to \cite{Merle-Raphael-Rodnianski-Szeftel:implosion-ii} or \cite{Buckmaster-CaoLabora-GomezSerrano:implosion-compressible}, we need to perform a linear stability analysis that allows for \emph{non-radially symmetric} perturbations. There are numerous examples where the upgrade from the radial to the nonradial stability analysis has required a considerable amount of new ideas: for example the energy critical defocusing quintic NLS \cite{Bourgain:gwp-defocusing-critical-nls-radial,Grillakis:nls,Colliander-Keel-Staffilani-Takaoka-Tao:global-wellposedness-energy-critical-nls-R3} or the energy critical wave equation \cite{Duyckaerts-Kenig-Merle:universality-blowup-energy-critical-wave,Duyckaerts-Kenig-Merle:universality-blowup-energy-critical-wave-nonradial}. An important difference with respect to those results is that the compressible Euler profiles are only linearly stable up to potentially finitely many unstable modes. We expect that both the radial and non-radial settings have instabilities and the non-radial one may have additional unstable modes, but we do not investigate this further. We will divide the linear stability in three parts: the dissipativity of the operator (Section \ref{sec:dissipativity}), the maximality (Section \ref{sec:maximality}) and the smoothness of the unstable modes (Section \ref{sec:smoothness}). All three of them pose obstructions in the non-radial setting and are affected substantially by including non-radial perturbations. 

The dissipativity is based on very delicate properties of the profiles (\eqref{eq:radial_repulsivity} and \eqref{eq:angular_repulsivity}). While \eqref{eq:radial_repulsivity} is a central part of \cite{Merle-Raphael-Rodnianski-Szeftel:implosion-i} and \cite{Buckmaster-CaoLabora-GomezSerrano:implosion-compressible}, angular perturbations also require the new property \eqref{eq:angular_repulsivity}. %We show it in Lemmas \ref{lemma:angular_repulsivity} and \ref{lemma:auxiliary}. 
The angular perturbations also require to study the high derivative analysis with respect to $\nabla^{m}$, instead of $\Delta^{m/2}$, since %the gap $\tilde \eta$ in inequality \eqref{eq:angular_repulsivity} is too small to accommodate for 
one can not afford the embedding constants that would arise if one tries to control angular derivatives by $\Delta^{m/2}$. Studying the high derivatives of the perturbation using $\nabla^m$ complicates the analysis (since one needs to consider an $m$-tensor instead of a scalar quantity) but allows to solve that problem.

The maximality of the operator %(Sections \ref{sec:maximality} and \ref{sec:smoothness}) 
is completely different from \cite{Buckmaster-CaoLabora-GomezSerrano:implosion-compressible,Merle-Raphael-Rodnianski-Szeftel:implosion-ii} because it becomes the existence problem of a singular PDE in three dimensions, which is much more complicated than the radial case (where one just needs to solve a singular ODE). Instead of solving the PDE directly, we use a Galerkin approximation together with uniform bounds on the real part of the spectrum of our approximations, which allow us to deduce the maximality of the limit operator.

For the smoothness of unstable modes, we decompose the system via vector spherical harmonics, which decouple the PDE system among those modes. To the best of our knowledge, this is the first time that vector spherical harmonics are used to study the stability problem near a radial solution for fluids PDE, and we believe this will be a very useful tool for non-radial behaviour around radially symmetric profiles. That way, we obtain a singular ODE for every mode. This singular ODE is still more complicated than in the radial setting, since the terms with angular derivatives cannot be diagonalized. This implies that we obtain some coefficients in the ODE that are not uniformly bounded over all modes (since they correspond to angular derivatives). Therefore, the smoothness of each mode does not imply the smoothness of the sum (unlike in the radial case, where there is only one mode), and we need uniform bounds across all the modes.  

Another obstruction for the smoothness of unstable modes consists of the singular behaviour of the system of ODEs at $R=1$. While this is present in the radial case, the interaction of this issue with the angular derivative terms mentioned above makes it impossible to implement  the same treatment. To resolve this obstacle  we use an abstract argument exploiting the fact that the space of unstable modes is of finite dimension, and therefore the smooth unstable modes will form the same set as the $H^m$ unstable modes for some $m$ sufficiently large.

The non-linear stability argument also poses some obstructions with respect to the radial setting. The main difference is that we cannot diagonalize the system with Riemann invariants, so we need to treat some of the highest derivative terms as errors. In order to do that, we need to propagate new estimates in our bootstrap argument.

The fact that we are now able to construct finite time imploding singularities for both Euler and Navier-Stokes equations in a periodic setting may provide a new tool to address blow up or norm inflation phenomena, and more generally, questions related to wave turbulence theory. In fact since in \cite{Merle-Raphael-Rodnianski-Szeftel:implosion-nls} the authors were able to transfer blow up results from a fluid equation to blow up results for NLS, it is conceivable to think that one may be able to transfer the construction of solutions of fluid equations with growing Sobolev norms (e.g. \cite{Kiselev-Nazarov:simple-energy-pump-sqg,He-Kiselev:small-scale-creation-sqg,Zlatos:exponential-growth-vorticity-gradient-euler-torus,Kiselev-Yao:small-scale-ipm,Kiselev-Sverak:double-exponential-euler-boundary}) to solutions of dispersive equations with growing  Sobolev norms. Let us recall that Bourgain \cite{Bourgain:growth-sobolev-norms-schrodinger-quasiperiodic-potential, Bourgain:growth-sobolev-norms-hamiltonian-pde} was the first to relate the growth of Sobolev norms for solutions to  periodic NLS equations to the notion of energy transfer, a fundamental question in wave turbulence.

\subsection{Setup of the problem}

First of all, let us define $\alpha = \frac{\gamma - 1}{2}$ and the rescaled sound speed $\sigma = \frac{1}{\alpha} \rho^{\alpha}$. We have that
\begin{align*}
\rho \p_t u + \rho u \nabla u &= - \frac{\rho}{\gamma - 1} \nabla  (\rho^{2\alpha})  + \nu \Delta u  = -\frac{\rho}{\gamma - 1} (\alpha^2 \nabla (\sigma^2)) + \nu \Delta u \\
&= -\rho \frac{2 \alpha^2}{\gamma - 1} \sigma \nabla \sigma  + \nu \Delta u = -\alpha \rho \sigma \nabla \sigma  + \nu \Delta u  \\
\p_t \sigma &= \rho^{\alpha-1} \p_t \rho = -\rho^{\alpha - 1} \div \left( \rho u \right) 
= -\rho^{\alpha} \div (u) - \rho^{\alpha - 1} \nabla \rho \cdot u = -\alpha \sigma \div (u) - \nabla \sigma \cdot{u},
\end{align*}
so the equations for $u, \sigma$ read:
\begin{align} \label{eq:US1} \begin{split}
\p_t u &=  -u  \nabla u + \alpha \sigma \nabla \sigma + \nu \frac{\Delta u}{\sigma^{1/\alpha}}, \\
\p_t \sigma &= -\alpha \sigma \div(u) - \nabla \sigma\cdot u.
\end{split}\end{align}

Now we perform the self-similar change of variables (depending on the self-similar parameter $r>1$):
\begin{align} \begin{split} \label{eq:SS_coordinates1}
 u(x, t) &= \frac{(T-t)^{\frac{1}{r} - 1} }{r}  U \left(  \frac{x}{(T-t)^{\frac{1}{r}}} , - \frac{\log (T-t)}{r}\right) \\
 \sigma(x, t) &= \frac{(T-t)^{\frac{1}{r} - 1} }{r}  S \left(  \frac{x}{(T-t)^{\frac{1}{r}}} , - \frac{\log (T-t)}{r}\right)
\end{split} \end{align}
and define the new self-similar space and time coordinates to be
\begin{equation} \label{eq:SS_coordinates2}
s = - \frac{\log (T-t)}{r}, \qquad\mbox{ and } \qquad y = \frac{x}{(T-t)^{\frac{1}{r}}} = e^s x.
\end{equation}
We also define
\begin{equation*}
 s_0 = - \frac{\log (T)}{r},
\end{equation*}
corresponding to $t = 0$.

In self-similar coordinates, the equations read:
\begin{align} \label{eq:US2} \begin{split}
\p_s U &= -(r-1)U - (y + U) \cdot \nabla U - \alpha S \nabla S + \nu C_{\rm{dis}} e^{-\delta_{\rm{dis}}s }\frac{\Delta U}{S^{1/\alpha}}, \\
\p_s S &= -(r-1)S - (y + U) \cdot \nabla S - \alpha S \div (U),
\end{split} \end{align}
with $y\in e^s\mathbb{T}_{L}^3=\{(y_1,y_2,y_3) \; \mbox{s.t.} \; -e^sL \leq y_i \leq e^s L\}$ in the case of the torus, and $y\in \mathbb R^3$ in the Euclidean case. We have denoted
\begin{equation*}
C_{\rm{dis}} = \frac{r^{1 + \frac{1}{\alpha} }}{\alpha^{ \frac{1}{\alpha} }}, \qquad \mbox{ and } \qquad \delta_{\rm{dis}} = \frac{r-1}{\alpha} + r - 2.
\end{equation*}
Moreover, we will restrict ourselves to the range of parameters
\begin{equation} \label{eq:range_r} 
1 < r < \begin{cases}
1 + \frac{2}{\left( 1 + \sqrt{ \frac{2}{\gamma - 1} } \right)^2}, \qquad &\mbox{ for } 1 < \gamma < \frac{5}{3}, \\
\frac{3\gamma - 1}{2 + \sqrt{3} (\gamma - 1)}, \qquad &\mbox{ for } \gamma \geq \frac{5}{3}.
\end{cases}
\end{equation}
Note that both expressions agree for $\gamma = 5/3$. In particular, this implies the bounds
\begin{equation} \label{eq:rough_range_r} 
1 < r < \sqrt 3, \qquad \mbox{ and } \qquad r   <  2 - \frac{1}{\gamma},
\end{equation}
for all $1 < \gamma < \infty$. 

In the case of Navier-Stokes ($\nu = 1$), we need to impose an extra condition on our range of parameters, namely:
\begin{equation} \label{eq:condition_for_treating_dissipation}
\delta_{\rm{dis}} = \frac{r-1}{\alpha} + r - 2 > 0.
\end{equation}

From \cite{Buckmaster-CaoLabora-GomezSerrano:implosion-compressible,Merle-Raphael-Rodnianski-Szeftel:implosion-i}, we know that there exist radially symmetric profiles $(\bar U, \bar S)$ that solve \eqref{eq:US2} for $\nu = 0$, with $r$ in the range \eqref{eq:range_r}. That is
\begin{equation}  \label{eq:ss_profiles}
(r-1) \bar U + (y + \bar U) \cdot \nabla \bar U + \alpha \bar S \nabla \bar S = 0, \qquad \mbox{ and } \qquad 
(r-1) \bar S + (y + \bar U) \cdot \nabla \bar S + \alpha \bar S \div (\bar U) = 0.
\end{equation}

More concretely, in the range \eqref{eq:range_r},  \eqref{eq:condition_for_treating_dissipation}, we know the existence of profiles solving \eqref{eq:ss_profiles} for almost every $\gamma < 1+2/\sqrt 3$, due to \cite{Merle-Raphael-Rodnianski-Szeftel:implosion-i}. The set of possible $\gamma$ is not explicit, although $\gamma = 5/3$ (monoatomic gases) is explicitly excluded. We know the existence of profiles for $\gamma = 7/5$ due to \cite{Buckmaster-CaoLabora-GomezSerrano:implosion-compressible}. 

If we drop condition \eqref{eq:condition_for_treating_dissipation} (which is only needed for the stability of Navier-Stokes, not the Euler one), \cite{Buckmaster-CaoLabora-GomezSerrano:implosion-compressible} gives the existence of profiles solving \eqref{eq:ss_profiles} in the regime \eqref{eq:range_r} for all $\gamma > 1$, including $\gamma = 5/3$.

Moreover, all the profiles discussed above satisfy
\begin{align} 
 \bar{S} &>0, \label{eq:profiles_positive} \\
 |\nabla^j \bar U| +  |\nabla^j \bar S | &\les \langle R \rangle^{-(r-1)-j} , \quad \forall j \geq 0, \quad \mbox{ and } \quad \bar S \gtrsim \langle R \rangle^{-r+1}. \label{eq:profiles_decay} \\
1 + \p_R \bar U_R - \alpha |\p_R \bar S | &> \tilde{\eta}, \label{eq:radial_repulsivity}\\
1 + \frac{\bar U_R}{R} - \alpha |\p_R \bar S | &> \tilde{\eta},\label{eq:angular_repulsivity} \end{align}
for some $\tilde \eta > 0$. Here $R=|y|$ is the radial variable, $j$ is a natural number.

The decay property \eqref{eq:profiles_decay} follows from a standard analysis of a focus point in the phase portrait of the ODE that $(\bar U, \bar S)$ satisfy. This is shown in \cite[Lemma A.39]{Buckmaster-CaoLabora-GomezSerrano:implosion-compressible} and an analogous proof works for the profiles in \cite{Merle-Raphael-Rodnianski-Szeftel:implosion-i} (the $j=0$ case is also contained in \cite[Theorem 2.3 - Item 3]{Merle-Raphael-Rodnianski-Szeftel:implosion-i}). 

The radial repulsivity property \eqref{eq:radial_repulsivity} is essential for performing the stability in the radially symmetric case and it corresponds to \cite[Theorem 2.3 - Item 5; Lemma 2.4]{Merle-Raphael-Rodnianski-Szeftel:implosion-i} and to \cite[Lemma A.36]{Buckmaster-CaoLabora-GomezSerrano:implosion-compressible} for their respective profiles.

The property \eqref{eq:angular_repulsivity} is a new extra repulsivity property that is needed when there is angular dependence. This property does not appear in the previous literature since it is not needed in the radially symmetric case. We show \eqref{eq:angular_repulsivity} in Lemma \ref{lemma:angular_repulsivity} in the Appendix.

We will study equation \eqref{eq:US2} in two settings: the periodic setting and the whole space setting. All of our discussion will focus on the periodic setting and can be adapted to the whole space easily. We set $\mathbb{T}^3$ to be the 3-dimensional torus of period $2$, and let $\mathbb T^3_L$, be the 3-dimensional torus or period $2L$.

\begin{remark}
While we have decided to focus on the physically relevant case of 3 dimensions, our results are expected to hold in any dimension, provided the existence of base objects (i.e. globally self-similar profiles that are solutions to the corresponding ODE similar to  \eqref{eq:ss_profiles}).
\end{remark}

\begin{theorem} \label{th:periodic}  Let $\nu = 1$. Let $\bar U, \bar S$ be self-similar profiles solving \eqref{eq:ss_profiles} and satisfying \eqref{eq:profiles_positive}--\eqref{eq:angular_repulsivity}, for some $r$ in the ranges \eqref{eq:range_r}, \eqref{eq:condition_for_treating_dissipation}. Let $T > 0$ sufficiently small, and $L > 0$ sufficiently large. 

Then, there exists $C^\infty$ initial data $(u_0,\rho_0)$, with $\rho_0 > 0$, for which equation \eqref{eq:CNS} on $\mathbb T^3_L$ blows up at time $T$ in a self-similar manner. More concretely, for any fixed $y\in \mathbb R^3$, we have:
\begin{align*} 
\lim_{t\rightarrow T^-}r (T-t)^{1-\frac{1}{r}} u\left((T-t)^{\frac1r}y,t\right) &= \bar U (|y|), \\
\lim_{t\rightarrow T^-} \left( \alpha^{-1 } r (T-t)^{1-\frac{1}{r}} \right)^{1/\alpha} \rho\left((T-t)^{\frac1r}y,t\right) &= \bar S (|y|)^{1/\alpha} \,.
\end{align*}
Moreover, there exists a finite codimension set of initial data satisfying the above conclusions (see Remark \ref{rem:codimension} for more details).
\end{theorem}

\begin{theorem} \label{th:euclidean}  Let $\nu = 1$. Let $\bar U, \bar S$ be self-similar profiles solving \eqref{eq:ss_profiles} and satisfying \eqref{eq:profiles_positive}--\eqref{eq:angular_repulsivity}, for some $r$ in the ranges \eqref{eq:range_r}, \eqref{eq:condition_for_treating_dissipation}. Let $T > 0$ sufficiently small, and $c > 0$ sufficiently small. 

Then, there exists $C^\infty$ non-radially symmetric initial data $(u_0,\rho_0)$, with $\rho_0 > c$, for which equation \eqref{eq:CNS} on $\mathbb R^3$ blows up at time $T$ in self-similar manner. More concretely, for any fixed $y \in \mathbb R^3$, we have:
\begin{align*}  
\lim_{t\rightarrow T^-}r(T-t)^{1-\frac{1}{r}} u\left((T-t)^{\frac1r}y,t\right) &= \bar U (|y|), \\
\lim_{t\rightarrow T^-} \left( \alpha^{-1} r (T-t)^{1-\frac{1}{r}} \right)^{1/\alpha} \rho\left((T-t)^{\frac1r}y,t\right) &= \bar S (|y|)^{1/\alpha} \,.
\end{align*}
Moreover, there exists a finite codimension set of initial data satisfying the above conclusions (see Remark \ref{rem:codimension} for more details).
\end{theorem}
\begin{proof}
The proof is analogous as in the torus case except two differences. First, the boundary terms that appear due to the integration by parts in the energy estimates are 0. This will not affect the proof because in the torus case, the boundary terms always cancel with each other. Second, we can use different interpolation Lemmas \ref{lemma:GN_general}, \ref{lemma:GN_generalnoweightwholespace} in place of Lemmas \ref{lemma:GN_generalnoweighttorus}, \ref{lemma:GN_generaltorus}.  
\end{proof}
\begin{remark} Equation \eqref{eq:CNS} has the following scaling invariance. If $u, \rho$ are solutions to \eqref{eq:CNS}, then
\begin{equation*} 
u_\lambda (x, t) = \frac{1}{\lambda^{\alpha} } u \left( \frac{x}{\lambda^{\alpha+1}}, \frac{t}{\lambda^{2\alpha+1}} \right)
\qquad \mbox{ and } \qquad
\rho_\lambda (x, t) = \frac{1}{\lambda} \rho \left( \frac{x}{\lambda^{\alpha+1}}, \frac{t}{ \lambda^{2\alpha+1} }  \right) 
\end{equation*}
are also solutions. In the periodic setting, if $u, \rho$ are defined on $\mathbb T^3_L$, we obtain that $u_\lambda, \rho_\lambda$ are defined on $\mathbb T^3_{L \cdot \lambda^{\alpha+1} }$.

As a consequence, we can generalize Theorem \ref{th:periodic} for any torus size $2L'$ (not necessarily large) just by taking $(u_\lambda, \rho_\lambda)$ for $\lambda = \left( \frac{L'}{L} \right)^{\frac{1}{1+\alpha}}$, where $u,\rho$ are the solutions from Theorem \ref{th:periodic}.
\end{remark}

\begin{remark} Both results hold also for the Euler case, $\nu = 0$. Moreover, in such a case, condition \eqref{eq:condition_for_treating_dissipation} is not needed. In fact, condition \eqref{eq:condition_for_treating_dissipation} guarantees that $\delta_{\mathrm{dis}} > 0$, hence the term $ \nu C_{\rm{dis}} e^{-\delta_{\rm{dis}}s }\frac{\Delta U}{S^{1/\alpha}}$ in \eqref{eq:US2} has an exponential decay, which is crucial in the proof of stability. If $\nu=0$, this term vanishes and thus the condition is not needed. The proof is analogous to the Navier-Stokes one, except that one does not need to bound the dissipation term when doing energy estimates. Note that the Euler result with periodic boundary conditions and \textit{radially symmetric perturbation} could be deduced from the non-periodic one (\cite{Merle-Raphael-Rodnianski-Szeftel:implosion-nls, Buckmaster-CaoLabora-GomezSerrano:implosion-compressible}) via finite speed of propagation. We have bounded speed of propagation even though the velocity blows up at time $T$, since $\| u \|_{L^\infty (\mathbb R^3 \setminus B(1))}$ remains uniformly bounded up to $T$. Our setting, however, allows us to deduce singularity formation for compressible Euler in the periodic setting for \textbf{any} \textit{small perturbation} (in a finite codimension subspace in which we show stability, which will be defined later).
\end{remark}

\begin{remark}

Given the local behaviour of the singularity, one can construct solutions that blow-up at $m$ points by gluing $m$ copies of the blow-up profile from Theorem \ref{th:euclidean} in an $m$-fold configuration, far apart from each other so that they can be treated as small perturbations of each of the individual singularities (this is in the spirit of \cite{Duyckaerts-Kenig-Merle:universality-blowup-energy-critical-wave-nonradial,Kenig-Merle:gwp-scattering-blowup-critical-focusing-nlw,Krieger-Schlag-Tataru:slow-blowup-critical-focusing-wave,Cortazar-delPino-Musso:infinite-time-bubbling-critical-nlh,Daskalopoulos-delPino-Sesum:type-ii-ancient-compact-solutions-yamabe}). See for example \cite{Musso-Pacard-Wei:stationary-solutions-euler,Fan:loglog-blow-up-m-points} for related works in the context of NLS.

\end{remark}

\subsection{Structure of the proof of Theorems \ref{th:periodic} and \ref{th:euclidean}}

In this subsection we explain how to construct finite energy solutions for the Navier-Stokes equation out of the globally self-similar solutions (which are stationary solutions of the self-similar equations \eqref{eq:ss_profiles}) of the Euler equation, both in the periodic and whole space case. From now on, we will only focus on the former case, though our estimates can be made uniform in $L$ (assumed sufficiently large).

We first perform a truncation to the aforementioned stationary solution such that the truncated solution $(\mathfrak X({e^{-s}}y)\bar{U}, \mathfrak X({e^{-s}}y)\bar{S})$ is periodic on $e^s \mathbb T_L^3$. Here $\mathfrak X(x)$ is a smooth cut off function with 
\[
\supp \mathfrak  X(x)\subset B(0,1), \ \ \mathfrak X(x)=1 \ \text{on}\ B(0,\frac{1}{2}), \ |\nabla { \mathfrak X}|\leq 10. 
\]

Moreover, we define $\hat{X}(y, s) = \mathfrak X(e^{-s}y)$. We further add a periodic perturbation $(\tilde{U},\tilde{S})$ to this truncated profile:

\begin{equation} \label{eq:periodic_perturbation}
 U = \tilde{U} + \mathfrak X({e^{-s}}y)\bar U=\tilde{U}+\hat{X}  \bar{U}, \qquad \mbox{ and } \qquad S = \tilde{S}  + \mathfrak X({e^{-s}}y) \bar S=\tilde{S} +\hat{X} \bar{S}.
\end{equation}
The equation for the perturbation $(\tilde{U},\tilde{S})$ reads:
\begin{align}\label{nsequation1}
&\partial_{s}\tilde{U}=\underbrace{-(r-1)\tilde{U}-(y+\hat{X}\bar{U})\cdot{\nabla\tilde{U}}-\al(\hat{X}\bar{S})\cdot{\nabla\tilde{S}}-\tilde{U}\cdot\nabla(\hat{X}\bar{U}) -\alpha \tilde{S} \nabla (\hat{X}\bar{S})}_{\mathcal L_{u}^{e}}\underbrace{-\tilde{U}\cdot \nabla\tilde{U}-\tilde{S}\cdot\nabla \tilde{S}}_{\mathcal N_{u}} \\\nonumber
&\underbrace{-\partial_{s}(\hat{X}\bar{U})-(r-1)(\hat{X}\bar{U})-(\hat{X}\bar{U}+y)\cdot\nabla(\hat{X}\bar{U})-\alpha(\hat{X}\bar{S})\cdot\nabla(\hat{X}\bar{S})}_{\mathcal E_u}
+ \underbrace{ \nu C_{\rm{dis}} e^{-\delta_{\rm{dis}}s }\frac{\Delta U}{S^{1/\alpha}} }_{\mc F_{dis}}
\end{align}
\begin{align} \begin{split} \label{nsequation2}
&\partial_{s}\tilde{S}=\underbrace{-(r-1)\tilde{S}-(y+\hat{X}\bar{U})\cdot{\nabla\tilde{S}}-\al(\hat{X}\bar{S})\div(\tilde{U})-\tilde{U}\cdot\nabla(\hat{X}\bar{S})-\alpha \tilde{S} \div(\hat{X}\bar{U})}_{\mathcal L_{s}^{e}}\underbrace{-\tilde{S}\cdot \nabla\tilde{U}-\alpha\tilde{S}\div(\tilde{U})}_{\mathcal N_{s}}\\
&\underbrace{-\partial_{s}(\hat{X}\bar{S})-(r-1)(\hat{X}\bar{S})-(\hat{X}\bar{U}+y)\cdot\nabla(\hat{X}\bar{S})-\alpha(\hat{X}\bar{S})\div(\hat{X}\bar{U})}_{\mathcal E_s}.
\end{split} \end{align}
Moreover, since $(\bar{U},\bar{S})$ is a stationary profile, we can simplify $\mathcal E_u$, $\mathcal E_s$ and have

\begin{align}
\mathcal E_u &= -(r-1)(\hat{X}\bar{U})-\hat{X}(\bar{U}+y)\cdot\nabla(\bar{U})-\alpha\hat{X}(\bar{S})\cdot\nabla(\bar{S}) -(\hat{X}^2-\hat{X})\bar{U}\cdot\nabla\bar{U} \notag \\
&\qquad-\hat{X}\bar{U}\cdot\nabla (\hat{X})\bar{U}-
\alpha(\hat{X}^2-\hat{X})\bar{S}\cdot\nabla\bar{S}-\alpha\hat{X}\bar{S}\cdot \nabla(\hat{X})\bar{S}-\partial_{s}(\hat{X})\bar{U}-(y\cdot \nabla\hat{X})\bar{U} \notag \\
&=-(\hat{X}^2-\hat{X})\bar{U}\cdot\nabla\bar{U}-\hat{X}\bar{U}\cdot\nabla (\hat{X})\bar{U}-
\alpha(\hat{X}^2-\hat{X})\bar{S}\cdot\nabla\bar{S}-\alpha\hat{X}\bar{S}\cdot \nabla(\hat{X})\bar{S}, \label{Eudefinition} \\
\mathcal E_s &= -(r-1)(\hat{X}\bar{S})-\hat{X}(\bar{U}+y)\cdot\nabla(\bar{S})-\alpha\hat{X}(\bar{S})\div{\bar{U}}-(\hat{X}^2-\hat{X})\bar{U}\cdot\nabla\bar{S}\notag \\
&\qquad-\hat{X}\bar{U}\cdot\nabla (\hat{X})\bar{S}-\alpha
(\hat{X}^2-\hat{X})\bar{S}\div(\bar{U})-\alpha\hat{X}\bar{U}\cdot \nabla(\hat{X})\bar{S}-\partial_{s}(\hat{X})\bar{S}-(y\cdot \nabla\hat{X})\bar{S} \notag \\
&=-(\hat{X}^2-\hat{X})\bar{U}\cdot\nabla\bar{S}-(\alpha+1)\hat{X}\bar{U}\cdot\nabla (\hat{X})\bar{S}-\alpha
(\hat{X}^2-\hat{X})\bar{S}\div(\bar{U}),\label{Esdefinition}
\end{align}

where we have used that $\partial_s(\hat{X}) = -y \cdot \nabla \hat{X}$. 

Let us define $C_0$ sufficiently large as in Lemma \ref{C0choice}. Let $\chi_1: \mathbb R \rightarrow [0, 1]$ be a smooth symmetric cut-off function supported on $\left[ -\frac32 C_0, \frac32 C_0 \right]$ and satisfying $\chi_1 (x) = 1$ for $|x| \leq C_0$. Let $\chi_2 : \mathbb R \rightarrow [0, 1]$ be a smooth symmetric cut-off function supported on $\left[ -\frac52 C_0, \frac52 C_0 \right]$ and satisfying $\chi_2 (x) = 1$ for $|x| \leq 2 C_0$, $\chi_2(x)>0 $ in $\left(-\frac52 C_0, \frac52 C_0 \right)$ . 
\begin{figure}[h]
\centering
\includegraphics[width=0.6\textwidth]{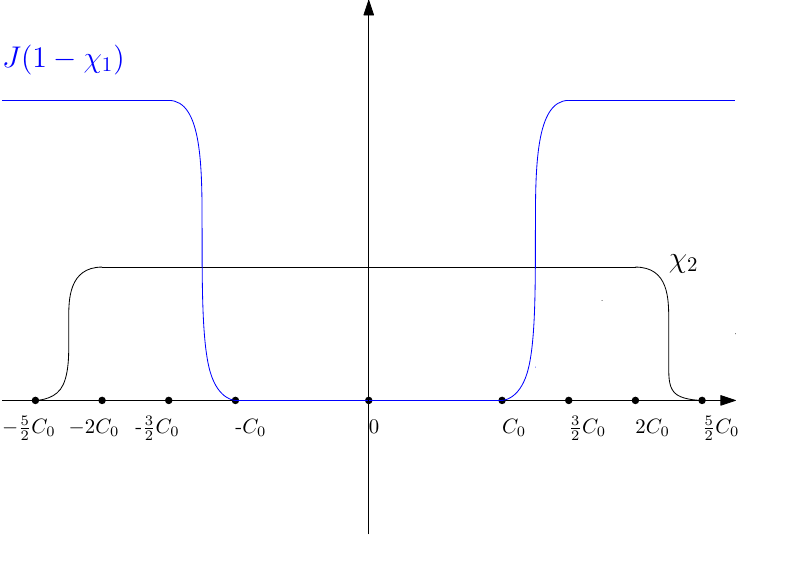}
\caption{The graph of $J(1-\chi_1)$ and $\chi_2$.}
\label{fig:cut off}
\end{figure}

Using those cut-offs, we define the cut-off linear operator as:
\begin{equation} \label{eq:cutoffL}
\mc L_u = \chi_2 \mc L_u^e - J(1-\chi_1), \quad \mc L_s = \chi_2 \mc L_s^e - J(1-\chi_1),
\end{equation}
where $J$ is a sufficiently large constant that  will be fixed later.

We will study the cut-off linearized operator $(\mathcal L_u, \mathcal L_s)$ in the space $X$, which we define as follows. First, we consider $\tilde X$ to be the space $H^{m}_u(Q , \mathbb{R}^3) \times  H^{m}_s (Q, \mathbb{R})$  where the $u, s$ subscripts make reference to $\tilde U$ (vector field) and $\tilde S$ (scalar field). $Q$ is a cube centered at the origin of length $4C_0$ and we take periodic boundary conditions. We select $m$ to be a sufficiently large integer. We define $X$ to be the subspace of $\tilde X$ formed by $(U, S)\in \tilde X$ such that both $U$ and $S$ are compactly supported on $B(0, 3C_0)$. One can thus think of $X$ as a subspace of $H^{m}_u (B(0, 3C_0), \mathbb R^3) \times H^{m}_s (B(0, 3C_0), \mathbb R)$ with the appropriate vanishing conditions at the boundary. Whenever we take Fourier series of $(U, S) \in X$ it should be understood that we do Fourier series on the cube $Q$.

We put in our space $X$ the standard inner product
\begin{equation*}
\langle (f_u, f_s), (g_u, g_s) \rangle_X = \int_{B(0, 3C_0)} \left(  \nabla^m f_u \cdot \nabla^m g_u + \nabla^m f_s \cdot \nabla^m g_s + f_u \cdot g_u + f_s g_s \right)
\end{equation*}
and we obtain the inherited norm
\begin{equation}\label{spacexnorm}
\| (f_u, f_s) \|_X^2 = \int_{B(0, 3C_0)} \left( |\nabla^m f_u |^2 + (\nabla^m f_s )^2 + |f_u|^2 + f_s^2 \right),
\end{equation}
where $\nabla^m f_s$ is an $m$-tensor indexed by $[3]^m = \{ 1, 2, 3 \}^m$  containing all possible combinations of directional derivatives. In the case of $\nabla^m f_u$, we have an $(m+1)$-tensor since we have an extra index indicating the component of the vector $f_u$.

We now introduce the definition of a maximally dissipative operator.

\begin{definition}\label{def:dissimax}
A dissipative operator is a linear operator $A$ defined on a linear subspace $D(A)$ of Hilbert space X such that for all $x \in D(A)$
\begin{align*}
\Re \langle Ax,x\rangle \leq 0.
\end{align*}
 A dissipative operator is called maximally dissipative if it is dissipative and for all $\lambda > 0 $ the operator $\lambda I- A$ is surjective, meaning that the image when applied to the domain $D$ is the whole space $X$.
 \end{definition}

We will show that $\mc L = (\mc L_u, \mc L_s)$ is maximally dissipative modulo a finite-rank perturbation and that its eigenfunctions are smooth. 

\begin{proposition} \label{prop:maxdissmooth} For $\delta_g$ sufficiently small, $J, m$ sufficiently large, the following holds. The Hilbert space $X$ can be decomposed as $V_{\rm{sta}} \oplus V_{\rm{uns}}$, where both $V_{\rm{sta}}, V_{\rm{uns}}$ are invariant subspaces of $\mc L$. $V_{\rm{uns}}$ is finite dimensional and formed by smooth functions. Moreover, there exists a metric  $B$ of $V_{uns}$ such that the decomposition satisfies:
\begin{align} \begin{split} \label{eq:decomposition_condition}
\Re\langle \mc L v,  v \rangle_B &\geq \frac{-6}{10}\delta_g \| v \|_B^2,  \qquad  \forall v \in V_{\rm{uns}}, \\
\left\| e^{t\mc L} v \right\|_X &\leq e^{-t\delta_g/2} \| v \|_X, \qquad \forall v \in V_{\rm{sta}},
\end{split} \end{align}
where we use $\Re$ to denote the real part.

\end{proposition}

Section \ref{sec:linear} will consist of a proof of Proposition \ref{prop:maxdissmooth}. For the main difficulties we refer to the discussion about the linear stability in Section \ref{sec:12}. Let us first sketch an outline of the proof.

We first define the finite codimension subspace $X_K \subset X$, as the elements of $X$ for which the Fourier coefficients $\mathcal{F}( U_i)(k) = \mathcal{F} (S) (k) = 0$ for all $|k| < K$. The first step of the proof will consist in showing that $\mc L+1$ is dissipative for elements of $X_K$ (cf. Definition \ref{def:dissimax}), concretely:
\begin{equation}
\langle \mc L (U, S), (U, S) \rangle_X \leq -\| (U, S) \|_X^2, \qquad \forall (U, S) \in X_K.
\end{equation}

The second step will consist of a Galerkin type argument showing that $\mc L - \lambda$ is surjective on $X$ for sufficiently large $\lambda$. That will allow us to conclude that $\mc L+1$ is a finite-dimensional perturbation of a maximally dissipative operator on $X$. In particular, this gives a very good understanding of the spectra of $\mc L$ and $\mc L^\ast$ for eigenvalues with real part bigger than $-1$ and allows us to define a finite dimensional space $V_{\rm{uns}}$ and a finite codimensional space $V_{\rm{sta}}$ satisfying \eqref{eq:decomposition_condition} as in Proposition \ref{prop:maxdismod}.

The third and last step will be to conclude that the elements of $V_{\rm{uns}}$ are smooth. In order to do so, we will consider $V_{\rm{uns}}(m)$ depending on $m$ and show that all the spaces $V_{\rm{uns}}(m)$, restricted to $B(0, C_0)$, coincide for all sufficiently large $m$. The main step will be to show that a function defined satisfying the eigenfunction equation on $B(0, C_0)$ can be extended uniquely to an eigenfunction on $B(0, 3C_0)$ (Lemma \ref{lemma:eigenextension}). In order to do so, we will decouple the equation using vector spherical harmonics, and obtain uniform bounds over all modes by controlling a carefully chosen form of energy (Lemma \ref{eigenfunctionbound}).

Once Proposition \ref{prop:maxdissmooth} is proven, we will conclude Theorem \ref{th:periodic} by establishing nonlinear stability. We will use three different ingredients. First, we need to treat the equation outside the cut-offs $\chi_1$ and $\chi_2$ (since the linear analysis only covers the region where those cut-offs are $1$). To do so, we will carefully design our bootstrap assumptions so that we can ensure that the perturbations are being damped in this region along trajectories $y = y_0 e^s$ (which correspond to stationary points after undoing the self-similar change of variables). Secondly, we will handle the dissipation, which is not included in the equation for the profiles, and therefore it cannot be treated as a perturbation. However, condition \eqref{eq:condition_for_treating_dissipation} ensures that the dissipative term will decay exponentially in self-similar variables. In order to bound it, we will use weighted energy estimates at a higher derivative level than the linear stability (regularity higher than $m$), together with lower bounds for $S$ to rule out possible vacuum. Finally, we need to control the (finitely many) unstable modes arising from the linear stability (that is, the modes given by $V_{\mathrm{uns}}$ in Proposition \ref{prop:maxdissmooth}). This can be done by restricting the initial data to a finite codimension set (as it is done in Theorem \ref{th:periodic} and Theorem \ref{th:euclidean}) and the way to formalize it is via a topological argument that exploits the unstable structure via Brouwer's fixed point theorem (see for example \cite{Cote-Martel-Merle:construction-multisoliton-supercritical-gkdv-nls}).

\subsection{Organization of the paper}

The paper is organized as follows: Section \ref{sec:linear} is devoted to the study of the linearized operator around $\bar U, \bar S$, and it consists of the proof of Proposition \ref{prop:maxdissmooth}. It is organized in three parts: Section \ref{sec:dissipativity} shows that the linearized operator is dissipative, Section \ref{sec:maximality} shows that it is maximal, and Section \ref{sec:smoothness} shows that the unstable modes are smooth. Section \ref{sec:nonlinear} upgrades the linear estimates to nonlinear, bootstrapping the bounds of the perturbation and using a topological argument to conclude that there exists a set of initial conditions of finite codimension that lead to a finite time imploding singularity. Finally, the Appendix contains a proof of property \eqref{eq:angular_repulsivity} for the globally self-similar solution, as well as some other properties of the solution and some functional analytic results.

\subsection{Notation}

Throughout the paper we will use the following convention:

We will use $C_x$ (resp. $C$) to denote any constant that may depend on $x$ (resp. independent of all the other parameters). These constants may change from line to line but they are uniformly bounded by a universal constant dependent on $x$ (independent of all other parameters). Similarly, we will denote by 
$X\lesssim Y$ and by 
$X\lesssim_x Y$ whenever $X\leq C Y$ and $X\leq C_x Y$ for some $C$, $C_x$ respectively.

 We use $\nabla^{l}f$ for the $l$-th order tensor containing all $l$ order derivatives of $f$. In the case where $f$ is a vector (typically $f = U$) then $\nabla^l U$ is a $(l+1)$-th tensor. We also denote by $| \nabla^k f|$ its $\ell^2$ norm. Note that this norm controls all the derivatives \textit{accounting for reorderings of indices}. For example, we have
 \begin{equation*}
|\nabla^j f|^2 = \sum_{\beta \in [3]^j} |\p_\beta f |^2 \geq \binom{j}{j'} |\p_1^{j'} \p_2^{j-j'} f|^2 .
 \end{equation*}
 Thus, we see that $|\nabla^j f|$ has a stronger control of the mixed derivatives by a combinatorial number, just because any mixed derivative appears multiple times in $\nabla^j f$ as the indices are reordered.

%$f^\ast$:$f^\ast$ denotes the complex conjugate of $f$

For $\beta=(\beta_1,\beta_2...,\beta_K)$, we let  $\partial_{\beta}=\partial_{\beta_1,\beta_2,\beta_3...\beta_K}$, and $\partial_{\beta^{(j)}}=\partial_{\beta_1,\beta_2,...\beta_{j-1},\beta_{j+1}....\beta_{K}}$ (that is, $\beta_j$ denotes the $j$-th component of $\beta$ and $\beta^{(j)}$ denotes the subindex obtained by erasing the $j$-th component in $\beta$).

We will use $y$ as our self-similar variable and we will denote $R = |y|$ for the radius (given that $r$ is reserved, since $1/r$ is the self-similar exponent of our ansatz). We will denote with $\hat{R}$ the unitary vector field in the radial direction, that is, $\hat R = \frac{y}{|y|} = \frac{y}{R}$. 

In Subsection \ref{sec:smoothness} we will also use the decomposition into spherical harmonics. For any $n = (n_1, n_2) \in \mathbb Z^2$ such that $|n_2| \leq n_1$ there exists a normalized eigenfunction of the Laplacian on the sphere: $e_n(\theta, \psi)$ (where $\theta$ and $\psi$ are the azimuthal and polar angles of spherical coordinates). It satisfies that $\Delta e_n = -n_1 (n_1+1)e_n$ and the system $\{ e_n \}$ is orthonormal. In particular, any function $f(y)$ can be decomposed into its spherical modes as $$f(y) = \sum_{|n_2|\leq n_1} f_n(R) e_n(\theta, \psi), $$ where the functions $f_n$ only depend on the radius. We will normally omit the condition $|n_2| \leq n_1$ in the summation subindex and whenever we are summing over modes we will just indicate it as $\sum_n$. We will also use vector spherical harmonics \cite{Hill:spherical-harmonics} that allow to decompose any vector field $V(y)$ as $$V(y) = \sum_n \left( V_{R, n}(R) e_n \hat R + V_{\Psi, n} (R) R\nabla e_n + V_{\Phi, n}(R) y \wedge \nabla e_n \right). $$
Here, we note that the modes $V_{R, n}(R)$, $V_{\Psi, n}(R)$ and $V_{\Phi, n}(R)$ are radially symmetric functions.

\subsection{Acknowledgements}
GCL and GS have been supported by NSF under grant DMS-2052651 and DMS-2306378. GCL, JGS and JS have been partially supported by the MICINN (Spain) research grant number PID2021–125021NA–I00.
JGS has been partially supported by NSF under Grants DMS-2245017, DMS-2247537 and DMS-2434314, by the AGAUR project 2021-SGR-0087 (Catalunya). 
JS has been partially supported by an AMS-Simons Travel Grant. GS has been supported by the Simons Foundation through the Simons Collaboration on Wave Turbulence. JGS is also thankful for the hospitality of the MIT Department of Mathematics, where parts of this paper were done.
We thank Tristan Buckmaster for useful conversations.

\section{Study of the linearized operator}
\label{sec:linear}

\subsection{Dissipativity} \label{sec:dissipativity}
In this subsection we will show that $\mc L+1$ is dissipative (cf. Definition \ref{def:dissimax}) on $X_K$, a finite codimension subspace of $X$ where the Fourier modes smaller than $K$ vanish, that is $\mathcal F( U_i)(k) = 0$, $\mathcal F (S)(k) = 0$ for $|k| < K$. 

Note that for any $f \in X_K$ we have the inequality $\| \nabla f \|_{L^2} \geq K \|f\|_{L^2}$.\footnote{With a slight abuse of notation, we denote $\| f \|_{L^2}^2 = \| f_u \|_{L^2}^2 + \| f_s \|_{L^2}^2$} For $f, g \in X_K$, we have
\begin{equation*}
\left| \langle f, g \rangle_X - \langle \nabla^m f, \nabla^m g\rangle_{L^2} \right|  \leq \| f \|_{L^2} \| g \|_{L^2} \leq K^{-2m} \|f\|_X \|g\|_X
\end{equation*}

Therefore, it suffices to show that
\begin{equation}
\int_{B(3C_0)} \nabla^m \mc L_u \cdot \nabla^m \tilde U + \nabla^m \mc L_s \cdot \nabla^m \tilde S \leq - \| (\tilde U, \tilde S )\|_X^2,
\end{equation}
where we denote by $B(3C_0)$ the ball of radius $3C_0$ (for the rest of this section, the center of the ball is always assumed to be zero and we just indicate its radius).

 Now, let us divide the operators $\mc L_u, \mc L_s$ as:
\begin{align} \label{Lequation}
\mc L_u &=  \underbrace{-\chi_{2} (y + \bar U) \cdot \nabla \tilde U -\chi_{2} \alpha \bar S \nabla \tilde S }_{\mc M_u}\underbrace{- J(1-\chi_1)\tilde{U}}_{\mc D_u}
 \underbrace{ -\chi_{2}(r-1) \tilde U  - \chi_{2}\tilde U \cdot \nabla \bar U  - \chi_{2}\alpha \tilde S \nabla \bar S }_{\mc L_u -\mc D_u- \mc M_u}, \\
\mc L_s &= \underbrace{  - \chi_{2}(y + \bar U) \cdot \nabla \tilde S - \alpha\chi_{2} \bar S \text{div} ( \tilde U ) }_{\mc M_s}\underbrace{- J(1-\chi_1)\tilde{S}}_{\mc D_s}
\underbrace{  -(r-1)\chi_{2}\tilde S - \chi_{2}\tilde U \cdot \nabla \bar S  - \chi_{2}\alpha \tilde S \text{div} (\bar U) }_{\mc L_s -\mc D_s- \mc M_s}.
\end{align}
Notice that $\hat{X}=1$ whenever $\chi_1,\chi_2\neq 0$, so the cut-off $\hat X$ does not appear in the expressions of $\mc L_u$, $\mc L_s$.

Now we start to show the dissipativity. 
\begin{lemma}\label{lemma:goodtermestimate} Taking $K$ sufficiently large in terms of $J$, and $J$ large enough depending on $m$, there is a universal constant $C$ (independent of $m$) such that
\begin{align}
&\int_{B(3C_0)} | \nabla^m (\mc L_u - \mc M_u-\mc D_u) \cdot \nabla^m \tilde U | \leq C \| (\tilde U, \tilde S) \|_X^2,\\
%&\leq C (\|\nabla^{m}\tilde{U}\|_{L^2({B(\chi_2)})}^2+\|\nabla^{m}\tilde{S}\|_{L^2({B(\chi_2)})}^2+\|\nabla^{m}\tilde{U}\|_{L^2({B^c(\chi_2)})}^2+\|\nabla^{m}\tilde{S}\|_{L^2({B^c(\chi_2)})}^2
&\int_{B(3C_0)} | \nabla^m (\mc L_s - \mc M_s-\mc D_s) \cdot \nabla^m \tilde S | \leq C \| (\tilde U, \tilde S) \|_X^2,
%&\leq C (\|\nabla^{m}\tilde{U}\|_{L^2({B(\chi_2)})}^2+\|\nabla^{m}\tilde{S}\|_{L^2({B(\chi_2)})}^2+\|\nabla^{m}\tilde{U}\|_{L^2({B^c(\chi_2)})}^2+\|\nabla^{m}\tilde{S}\|_{L^2({B^c(\chi_2)})}^2
\end{align}
and
\begin{align}\label{Duestimate}
\int_{B(3C_0)} \nabla^m (\mc D_u) \cdot \nabla^m \tilde U  &\leq -J\int_{B(3C_0)}(1-\chi_1)|\nabla^m \tilde U|^2 + C \| (\tilde U, \tilde S) \|_X^2,\\
%&\leq -(J-C)(\|\nabla^{m}\tilde{U}\|_{L^2({B^c(\chi_2)})}^2+\|\nabla^{m}\tilde{S}\|_{L^2({B^c(\chi_2)})}^2)+C(\|\nabla^{m}\tilde{U}\|_{L^2({B(\chi_2)})}^2+\|\nabla^{m}\tilde{S}\|_{L^2({B(\chi_2)})}^2)
\label{Dsestimate}
\int_{B(3C_0)} \nabla^m (\mc D_s) \cdot \nabla^m \tilde S &\leq -J\int_{B(3C_0)}(1-\chi_1)|\nabla^m \tilde S|^2 + C \| (\tilde U, \tilde S) \|_X^2,
%&\leq -(J-C)(\|\nabla^{m}\tilde{U}\|_{L^2({B^c(\chi_2)})}^2+\|\nabla^{m}\tilde{S}\|_{L^2({B^c(\chi_2)})}^2)+C(\|\nabla^{m}\tilde{U}\|_{L^2({B(\chi_2)})}^2+\|\nabla^{m}\tilde{S}\|_{L^2({B(\chi_2)})}^2)
\end{align}
\end{lemma}
\begin{proof} Note that in the terms $\mc L_\circ - \mc M_\circ - \mc D_\circ$ or in $\mc D_\circ$ (for $\circ \in \{ u, s \}$) there are no derivatives in the perturbation. We distribute all the derivatives of $\nabla^m$ among those terms and consider two cases: either all $m$ derivatives hit the perturbation or not. If we are in the second case, we get a factor $\frac{1}{K}$ because $\| \tilde U \|_{H^{m-1}} \leq \frac{1}{K} \| \tilde U \|_{H^{m}}$. Since $K$ is sufficiently large with respect to $m$, this is enough to absorb any other high derivative hitting the profile and the cut-off, like $\| \nabla^m \bar U \|_{L^\infty}$, $\| \nabla^m \chi_1 \|_{L^\infty}$,$\| \nabla^m \chi_2 \|_{L^\infty}$. Since there are $O\left( 3^m\right)$ such terms, the total contribution of those terms can be bounded by $ \| (\tilde U, \tilde S ) \|_X^2$ taking $K$ sufficiently large.

The remaining case is when all derivatives hit the perturbation. That case for $\mc D_u$ and $\mc D_s$ directly leads to the first terms in the right-hand sides of \eqref{Duestimate} and \eqref{Dsestimate}, respectively. This concludes the proof of \eqref{Duestimate}--\eqref{Dsestimate}.

We just need to treat the case where all derivatives fall in the perturbation for $\mc L_u - \mc M_u - \mc D_u$ or $\mc L_s - \mc M_s - \mc D_s$. In that case, we obtain terms like 
\begin{align*}
\sum_{\beta \in [3]^K}\int_{B(3C_0)} (\chi_2(\p_\beta \tilde  U) \cdot \nabla \bar U) \cdot \p_\beta \tilde U &=
\sum_{j, k} \sum_{\beta \in [3]^K}\int_{B(3C_0)} (\chi_2(\p_\beta \tilde  U_j) \cdot \p_j \bar U_k) \cdot \p_\beta \tilde U_k \\
&\les \sum_{\beta \in [3]^K} \int_{B(3C_0)} \chi_2 |\p_\beta \tilde U|^2 \les \| (\tilde U, \tilde S) \|_X^2,
\end{align*}
given that $|\chi_2| \leq 1$ and $\| \bar U \|_{W^{1, \infty}} \les 1$. Other possibilities are very similar to this term but interchanging $\tilde U$ with $\tilde S$ or $\bar U$ with $\bar S$.
\end{proof}

Now, let us focus on the bounds of the main contribution, coming from $\mc M_u$ and $\mc M_s$. We have that
\begin{equation} 
\mc M_{u, i} = -\sum_j \chi_2(y_j + \bar U_j) \p_j \tilde U_i - \alpha \chi_2\bar S \p_{i} \tilde S \quad \mbox{ and } \quad
\mc M_s = -\sum_j \chi_2 (y_j + \bar U_j) \p_j \tilde S - \alpha \chi_2\bar S \text{div} (\tilde U).
\end{equation}

We will need to compute the integral $$ \int_{B(3C_0)} \nabla^m \mc M_u \cdot \nabla^m \tilde U + \nabla^m \mc M_s \nabla^m \tilde S. $$ We distribute the derivatives of $\nabla^m \mc M_u$ and $\nabla^m \mc M_s$ and identify the terms with $m$ or $m+1$ derivatives in the perturbation (note that $\mc M_u$ and $\mc M_s$ already have one derivative on the perturbations $\tilde U$, $\tilde S$). 

The terms of $\nabla^m \mc  M_u$ and $\nabla^m \mc M_s$ with $m+1$ or $m$ derivatives on the perturbation are:
\begin{align*}
\mc J_{1, i, \beta}' &= -\sum_{j \in [3]} \chi_2 (y_j + \bar U_j) \p_{\beta} \p_j \tilde U_i, \qquad 
&\mc J_{1, i, \beta} &= -  \sum_{\ell = 1}^m \sum_{j\in [3]} \chi_2 (\p_{\beta_\ell} \bar U_j) \cdot \p_{\beta^{(\ell)}} \p_j \tilde U_i, \\
\mc J_{2, i, \beta}' &= -\alpha \chi_2\bar S \p_{\beta} \p_{i} \tilde S, \qquad
&\mc J_{2, i, \beta} &=  -\sum_{\ell = 1}^m \alpha  \chi_2 (\p_{\beta_\ell} \bar S) \cdot \p_{\beta^{(\ell)}} \p_{i} \tilde S, \\
\mc J_{3, \beta}' &= -\sum_{j \in [3]} \chi_2(y_j + \bar U_j) \p_{\beta} \p_j \tilde S, \qquad
& \mc J_{3, \beta} &=  - \sum_{\ell = 1}^m \sum_{j \in [3]} \chi_2 (\p_{\beta_\ell} \bar U_j ) \cdot \p_{\beta^{(\ell)}} \p_j \tilde S, \\
\mc J_{4, \beta}' &=   -\alpha \chi_2\bar S \p_{\beta} \text{div} (\tilde U), \qquad
& \mc J_{4, \beta} &=   - \sum_{\ell = 1}^m \alpha  \chi_2 (\p_{\beta_\ell} \bar S) \cdot \p_{\beta^{(\ell)}} \text{div} (\tilde U),
\end{align*}
where $i \in [3]$ and $\beta \in [3]^m$, together with
\begin{align}\begin{split} \label{eq:chiterms}
 \mc J_{\chi_2,u,i, \beta} &=    -\sum_{\ell=1}^m \p_{\beta_\ell}(\chi_2)\cdot \sum_{j\in [3]} \left( \p_{\beta^{(\ell)}}  \p_j \tilde U_i \cdot (y_j + \bar U_j)-  \alpha \p_{\beta^{(\ell)}} \p_{j} \tilde S  \cdot \bar S\right),  \\
\mc J_{\chi_2,s, \beta} &=- \sum_{\ell = 1}^m \p_{\beta_\ell}(\chi_2)\cdot \sum_{j\in [3]} \left( \p_{\beta^{(\ell)}} \p_j \tilde S \cdot (y_j + \bar U_j)  - \alpha \p_{\beta^{(\ell)}} \div (\tilde U) \bar S \right),
\end{split} \end{align}
and
\begin{align} \begin{split}\label{eq:dominantterms}
\mc J_{1, \beta, i}^\ast &= - \sum_{\ell = 1}^m \chi_2 \sum_{j\in [3]} \p_{\beta_\ell}(  y_j) \cdot  \p_{\beta^{(\ell)}} \p_j \tilde U_i = - m \chi_2  \p_{\beta} \tilde U_i,\\
\mc J_{3, \beta}^\ast &= - \sum_{\ell = 1}^m \chi_2 \sum_{j\in [3]} \p_{\beta_\ell} ( y_j) \cdot \p_{\beta^{(\ell)}} \p_j  \tilde S 
= -m \chi_2 \p_{\beta} \tilde S.
\end{split}\end{align}

The terms $\mc J_\ell'$ correspond to terms with $m+1$ derivatives on the perturbation (that is, all $m$ derivatives from $\nabla^m$ fall on the perturbation, which already has a derivative). In the rest of the terms, only $m-1$ derivatives from $\nabla^m$ hit the perturbation, so we need to distribute that extra derivative. The terms $\mc J_k$ correspond to the cases where that derivative falls on the profile $\bar U$ or $\bar S$, the terms $\mc J_{\chi_2, k}$ correspond to the cases where that derivative falls on the cut-off $\chi_2$, and the terms $\mc J_\ell^\ast$ correspond to the cases where that derivative falls on $y$. 

First of all, let us treat the rest of the cases, where less than $m-1$ derivatives from $\nabla^m$ fall on the perturbation, that is, the terms that we have not written above.
\begin{lemma}[Terms from $\mc M$ with less than $m-1$ derivatives on the perturbation] \label{lemma:Mloworder} For $K$ sufficiently large depending on $J$ and $m$, we have:
\begin{align} \begin{split} \label{eq:Mloworder}
\sum_{\beta \in [3]^m} \sum_{i \in [3]} \int_{B(3C_0)} \left( \p_{\beta} \mc M_{u, i} - (\mc J_{1, i, \beta}^\ast + \mc J_{1, i, \beta}' + \mc J_{1, i, \beta} + \mc J_{2, i, \beta}' + \mc J_{2, i, \beta} + \mc J_{\chi_2, u, i, \beta} ) \right) \p_{\beta} \tilde U_i &\leq \| (\tilde U, \tilde S) \|_X^2, \\
 \sum_{\beta \in [3]^m}\int_{B(3C_0)} \left( \p_{\beta} \mc M_{s, i} - (\mc J_{3, \beta}^\ast + \mc J_{3, \beta}' + \mc J_{3, \beta} + \mc J_{4, \beta}' + \mc J_{4, \beta} + \mc J_{\chi_2, s, \beta} ) \right) \p_{\beta} \tilde S &\leq \| (\tilde U, \tilde S) \|_X^2.
\end{split}\end{align}
\end{lemma}
\begin{proof} By definition of the $\mc J$ terms, any term obtained by distributing derivatives in $\p_{\beta} \mc M_{u, i} - (\mc J_{1, i, \beta}^\ast + \mc J_{1, i, \beta}' + \mc J_{1, i, \beta} + \mc J_{2, i, \beta}' + \mc J_{2, i, \beta} + \mc J_{\chi_2, u, i, \beta} ) $ has at most $m-1$ derivatives on the perturbation. Therefore, any such term is bounded in $L^2$ by:
\begin{align*}
&C  \left( \| \chi_2 \|_{W^{m, \infty}} + \| J(1-\chi_1) \|_{W^{m, \infty}} \right) \left(  \| \bar U \|_{W^{m, \infty}} +  \| \bar S \|_{W^{m, \infty}} \right) \left( \| \tilde U \|_{H^{m-1}(B(3C_0))} +  \| \tilde S \|_{H^{m-1}(B(3C_0))} \right) \\
&\qquad \les_{J, m} \frac{\| (\tilde U, \tilde S) \|_X}{K}
\end{align*}
where the $\les_{J, m}$ sign indicates that the implied constant is allowed to depend on $J$ and $m$. Given that there are $\les 3^m$ possible terms, we obtain the desired statement provided $K$ is sufficiently large in terms of $m$ and $J$.
\end{proof}

Now, we just need to bound the $\mc J$ terms. The dominant terms will be the terms $\mc J_\ell^\ast$ and most of the work will consist in showing that they dominate the $\mc J_\ell$ terms. Thus, we start with the $\mc J_{\chi_2}$ terms and the $\mc J_{\ell}'$ terms, which will be easier to bound.

\begin{lemma}[Terms with a derivative falling on the cut-off] \label{lemma:cutoffterms}
We have that
\begin{align} \begin{split} \label{eq:cutoffterms}
\sum_{\beta \in [3]^m} &\left( \sum_{i=1}^3 \int_{B(3C_0)} \mc J_{\chi_2, u, i, \beta} \p_{\beta} \tilde U_i
+ \int_{B(3C_0)} \mc J_{\chi_2, s, \beta} \p_{\beta} \tilde S \right) \\
&\les  m \sum_{\beta \in [3]^m }\int_{B(3C_0)} (1-\chi_1) \left( |\p_{\beta} \tilde U|^2 + (\p_{\beta} \tilde S)^2 \right) .
\end{split} \end{align}
\end{lemma}
\begin{proof} Let $\circ \in \{ u, s\}$, we have that
\begin{align*}
|\mc J_{\chi_2, \circ, \beta}| &\les \sum_{\ell = 1}^m | \p_{\beta_\ell} \chi_2| ( 1 + |\bar U| + |\bar S| )  \left( | \nabla \p_{\beta^{(\ell)}} \tilde U| + | \nabla \p_{\beta^{(\ell)}} \tilde S |\right)  \\
&\les |\nabla \chi_2| \sum_{\ell = 1}^m (|\nabla \p_{\beta^{(\ell)}} \tilde U| + | \nabla \p_{\beta^{(\ell)} } \tilde S | ).
\end{align*}
Thus, the left hand side of \eqref{eq:cutoffterms} is bounded (up to an absolute constant) by
\begin{equation*}
\sum_{\beta \in [3]^m} \int_{B(3C_0)} |\nabla \chi_2| \left( \sum_{\ell = 1}^m (|\nabla \p_{\beta^{(\ell)}} \tilde U| + | \nabla \p_{\beta^{(\ell)} } \tilde S | ) \right) \left( |\p_{\beta} \tilde U| + | \p_{\beta} \tilde S | \right)
\end{equation*}

Now, we use pointwise Cauchy-Schwarz, that is: $$|\p_{j} \p_{\beta^{(\ell)}} \tilde U| | \p_{\beta} \tilde S| \leq \frac12 |\p_j \p_{\beta^{(\ell)}} \tilde U |^2 + \frac12 |\p_{\beta} \tilde S|^2$$ (and similarly for other terms arising from the product). Thus, we see that the left hand side of \eqref{eq:cutoffterms} is bounded (up to an absolute constant) by
\begin{equation} \label{eq:narita1}
 m \sum_{\beta \in [3]^m} \int_{B(3C_0)} |\nabla \chi_2| \left( \left| \p_{\beta} \tilde U \right|^2 + \left| \p_{\beta} \tilde S \right|^2 \right).
\end{equation}

We conclude the Lemma by noting that $|\nabla \chi_2| \les 1-\chi_1$, since $\chi_1 = 0$ in the annulus where $\nabla \chi_2$ is supported.

\end{proof}

\begin{lemma}[Energy estimates for terms with $m+1$ derivatives]\label{lemma:highderterms} We have that
\begin{align*}
&\sum_{\beta \in [3]^m}\int_{B(3C_0)} \left( \sum_i (\mc J_{1, i, \beta}' + \mc J_{2, i, \beta}') \p_{\beta} \tilde U_i + (\mc J_{3, \beta}' + \mc J_{4, \beta}') \p_{\beta} \tilde S \right) dy \les \| (\tilde U, \tilde S) \|_X^2.
\end{align*}
\end{lemma}
\begin{proof} The bound for $\mc J_{1,i}'$ follows from integration by parts:
\begin{equation*}
\int_{B(3C_0)} \mc J_{1, i, \beta}' \p_{\beta} \tilde U_i dy = - \int_{B(3C_0)} \chi_2(y+\bar U) \nabla (\p_{\beta} \tilde U_i ) \p_{\beta} \tilde U_i dy = \frac{1}{2} \int_{B(3C_0)} \text{div}(\chi_2(y + \bar U)) (\p_{\beta} \tilde U_i )^2 dy.
\end{equation*}
Note that there are no boundary terms since $\chi_2$ is supported on $B(5C_0/2)$. Similarly, for $\mc J_3'$ we have:
\begin{equation*}
\int_{B(3C_0)} \mc J_{3, \beta}' \p_{\beta} \tilde S dy = -\int_{B(3C_0)} \chi_2(y + \bar U) \cdot \nabla (\p_{\beta}  \tilde S ) \p_{\beta} \tilde S dy = \frac{1}{2} \int_{B(3C_0)} \text{div} (\chi_2(y + \bar U)) (\p_{\beta} \tilde S )^2 dy.
\end{equation*}

With respect to the terms $\mc J_2$ and $\mc J_4$, we have
\begin{align*}
\int_{B(3C_0)} \sum_i (\mc J_{2, i, \beta}' \p_{\beta} \tilde U_i) + \mc J_{4, \beta}' \p_{\beta} \tilde S 
&= -\alpha \int_{B(3C_0)} \chi_2\bar S \left( \sum_i \p_i \p_{\beta}  \tilde S  \p_{\beta} \tilde U_i + \text{div} (\p_{\beta} \tilde U ) \p_{\beta}  \tilde S \right) dy\\
&= -\alpha \int_{B(3C_0)} \chi_2\bar S  \div \left( \p_{\beta}  \tilde S  \p_{\beta} \tilde U_i \right)  dy\\
&= \alpha \int_{B(3C_0)} \nabla(\chi_2 \bar S) \cdot \p_{\beta} \tilde U \p_{\beta} \tilde S dy.
\end{align*}
Using Cauchy-Schwarz, we conclude the statement.
\end{proof}

Combining all the previous Lemmas (Lemmas \ref{lemma:goodtermestimate}, \ref{lemma:Mloworder}, \ref{lemma:cutoffterms}, \ref{lemma:highderterms}) and obtain that for some sufficiently large constant $C$:
\begin{align} \begin{split}
\int_{B(3C_0)} &\left( \nabla^m \mc L_u \cdot \nabla^m \tilde U + \nabla^m \mc L_s \cdot \nabla^m \tilde S \right) \leq C \| (\tilde U, \tilde S )\|_X^2 + (- J + C m) \int_{B(3C_0)} (1-\chi_1) \left( |\nabla^m \tilde U|^2 + (\nabla^m \tilde S)^2 \right)  \\
&\quad 
+ \sum_{\beta\in [3]^m} \int_{B(3C_0)} \left( \mc J_{1, \beta}^\ast \p_{\beta} \tilde U + \mc J_{3, \beta}^\ast \p_{\beta} \tilde S \right) + \sum_{\beta\in [3]^m}  \int_{B(3C_0)} \left( \sum_i (\mc J_{1, i, \beta} + \mc J_{2, i, \beta}) \p_{\beta} \tilde U_i + ( \mc J_{3, \beta} + \mc J_{4, \beta} ) \p_{\beta} \tilde S \right) \\
& \leq C \| (\tilde U, \tilde S )\|_X^2 - \int_{B(3C_0)} \left( \frac{J(1-\chi_1)}{2} + m \chi_2 \right)  \left( |\nabla^m \tilde U|^2 + (\nabla^m \tilde S)^2 \right)  \\
&\quad + \sum_{\beta \in [3]^m }\int_{B(3C_0)} \left( \sum_i (\mc J_{1, i, \beta} + \mc J_{2, i, \beta}) \p_{\beta} \tilde U_i + ( \mc J_{3, \beta} + \mc J_{4, \beta} ) \p_{\beta} \tilde S \right) \label{eq:alcorcon}
\end{split} \end{align}
where in the second inequality we used that $J$ is sufficiently large depending on $m$ and the explicit expressions for $\mc J_1^\ast$, $\mc J_3^\ast$.
Now, we proceed to study the terms arising from $\mc{J}_{1, i}$, $\mc{J}_{2, i}$, $\mc{J}_{3}$, $\mc{J}_{4}$.

\subsubsection{Term $\mathcal{J}_{1, i}$}

We have that 
\begin{align} \label{eq:ponteceso}
\sum_{i=1}^3 \sum_{\beta \in [3]^m} \int_{B(3C_0)} \mc J_{1, i, \beta} \p_{\beta} \tilde U_i 
&= - \sum_{\ell = 1}^m \sum_{i \in [3]} \sum_{j \in [3]} \sum_{\beta \in [3]^m}
\int_{B(3C_0)} \chi_2 (\p_{\beta_\ell} \bar U_j) \cdot \p_j \p_{\beta^{(\ell)}} \tilde U_i \p_{\beta_\ell} \p_{\beta^{(\ell)}} \tilde U_i \\
&= - \sum_{\ell = 1}^m \sum_{i \in [3]} \sum_{j \in [3]} \sum_{\beta^{(\ell)} \in [3]^{m-1}} \sum_{\beta_\ell \in [3]}
\int_{B(3C_0)} \chi_2 (\p_{\beta_\ell} \bar U_j) \cdot \p_j \p_{\beta^{(\ell)}} \tilde U_i \p_{\beta_\ell} \p_{\beta^{(\ell)}} \tilde U_i. \nonumber
\end{align}

Using the radial symmetry of $\bar U$ (that is $\bar U = \bar U_R \frac{y}{R}$), we have
\begin{equation} \label{eq:pontecesures}
    \p_{\beta_\ell} \bar U_j = \p_{\beta_\ell} \left( \frac{y_j}{R} \bar U_R \right) = \left( \p_R \bar U_R - \frac{\bar U_R}{R} \right) \frac{y_j y_{\beta_\ell}}{R^2} + \delta_{\beta_\ell, j} \frac{\bar U_R}{R},
\end{equation}

Using \eqref{eq:pontecesures} into \eqref{eq:ponteceso}, and noting that $\sum_j \frac{y_j}{R}\p_j = \p_R$ we have that
\begin{align}
\sum_{i\in [3]} \sum_{\beta \in [3]^m} \int_{B(3C_0)} \mc J_{1, i, \beta} \p_{\beta} \tilde U_i 
&= - \sum_{\ell = 1}^m \sum_{i \in [3]} \sum_{\beta^{(\ell)} \in [3]^{m-1} }
\int_{B(3C_0)} \chi_2\left( \p_R \bar U_R - \frac{\bar U_R}{R} \right)  \cdot | \p_R \p_{\beta^{(\ell)}} \tilde U_i |^2 \notag  \\
&\qquad - \sum_{\ell = 1}^m \sum_{i \in [3]} \sum_{\beta \in [3]^{m} }
\int_{B(3C_0)} \chi_2 \frac{\bar U_R}{R}   | \p_{\beta} \tilde U_i |^2  \notag \\
&= - m \sum_{i \in [3]} \sum_{\tilde \beta \in [3]^{m-1} }
\int_{B(3C_0)} \chi_2\left( \p_R \bar U_R - \frac{\bar U_R}{R} \right)  \cdot | \p_R \p_{\tilde \beta} \tilde U_i |^2 \notag  \\
&\qquad - m \sum_{i \in [3]} \sum_{\tilde \beta \in [3]^{m-1} }
\int_{B(3C_0)} \chi_2 \frac{\bar U_R}{R}   \left( | \p_R \p_{\tilde \beta} \tilde U_i |^2 + | \nabla_{\theta} \p_{\tilde \beta} \tilde U_i |^2 \right) \notag, 
\end{align}
where we decomposed $|\nabla \p_{\tilde \beta} \tilde U_i|^2 = |\p_R\p_{\tilde \beta} \tilde U_i|^2 + |\nabla_{\theta} \p_{\tilde \beta} \tilde U_i|^2$, where $\nabla_{\theta}$ corresponds to all the angular components of the gradient in spherical coordinates. 

Finally, we obtain 
\begin{equation} \label{eq:J1}
\sum_{i\in [3]} \sum_{\beta \in [3]^m} \int_{B(3C_0)} \mc J_{1, i, \beta} \p_{\beta} \tilde U_i = -m \sum_{i=1}^3 \int_{B(3C_0)} \chi_2 \left( \p_R \bar U_R \left| \p_R \nabla^{m-1} \tilde U_i \right|^2 + \frac{\bar U_R}{R} \left| \nabla_\theta \nabla^{m-1} \tilde U_i \right| ^2 \right).
\end{equation}

\subsubsection{Term $\mc{J}_3$}

Similarly as in the case of $\mc J_1$, we compute:
\begin{align*}
\sum_{\beta \in [3]^m }\int_{B(3C_0)} \mc J_{3, \beta} \p_{\beta} \tilde S
&= - \sum_{\ell = 1}^m \sum_{\beta \in [3]^m } \sum_{j \in [3]}\int_{B(3C_0)} \chi_2 (\p_{\beta_\ell} \bar U_j ) \p_{\beta^{(\ell)}} \p_j \tilde S \p_{\beta} \tilde S \\ 
&= - \sum_{\ell = 1}^m \sum_{\beta^{(\ell)} \in [3]^{m-1} } \sum_{j \in [3]} \sum_{\beta_\ell \in [3]}\int_{B(3C_0)} \chi_2 (\p_{\beta_\ell} \bar U_j )  \p_j \p_{\beta^{(\ell)}}  \tilde S \p_{\beta_\ell} \p_{\beta^{(\ell)}} \tilde S. 
\end{align*}
Using again \eqref{eq:pontecesures}, we have that 
\begin{align}
\sum_{\beta \in [3]^m }\int_{B(3C_0)} \mc J_{3, \beta} \p_{\beta} \tilde S
&= - \sum_{\ell = 1}^m \sum_{\beta^{(\ell)} \in [3]^{m-1} } \int_{B(3C_0)} \chi_2 \left( \p_R \bar U_R - \frac{\bar U_R}{R}\right) |\p_R \p_{\beta^{(\ell)}}  \tilde S|^2 \notag \\
&\qquad  - \sum_{\ell = 1}^m \sum_{\beta \in [3]^{m} } \int_{B(3C_0)} \chi_2 \frac{\bar U_R}{R} | \p_{\beta}  \tilde S|^2 \notag \\
&= - m \sum_{\tilde \beta \in [3]^{m-1} } \int_{B(3C_0)} \chi_2 \left( \p_R \bar U_R - \frac{\bar U_R}{R}\right) |\p_R \p_{\tilde \beta}  \tilde S|^2 \notag \\
&\qquad  - m \sum_{\tilde \beta \in [3]^{m-1} } \int_{B(3C_0)} \chi_2 \frac{\bar U_R}{R} \left( \left| \p_R \p_{\tilde \beta}  \tilde S \right|^2 + \left| \nabla_\theta \p_{\tilde \beta}  \tilde S \right|^2 \right) \notag,
\end{align}

\begin{equation} \label{eq:J3}
\sum_{\beta \in [3]^m }\int_{B(3C_0)} \mc J_{3, \beta} \p_{\beta} \tilde S
\leq -m  \int_{B(3C_0)} \chi_2 
\left(
\p_R \bar U_R \left| \p_R \nabla^{m-1}\tilde S \right|^2 + 
\frac{\bar U_R}{R} \left| \nabla_\theta \nabla^{m-1}\tilde S \right|^2.
\right)  
\end{equation}

\subsubsection{Terms $\mc J_{2, i}$ }
%\begin{align*}
%&\mc J_{2, i, \beta} &=  -\sum_{\ell = 1}^m \alpha  \chi_2 (\p_{\beta_\ell} \bar S) \cdot \p_{\beta^{(\ell)}} \p_{i} \tilde S \\
%
%\end{align*}
With respect to $\mc J_2$, we have
\begin{align*}
\sum_{\beta\in [3]^m} \sum_{i=1}^3 \int_{B(3C_0)} \mc J_{2, i, \beta} \p_{\beta} \tilde U_i &= - \sum_{\ell = 1}^m \sum_{i \in [3]}  \sum_{\beta \in [3]^m} \int_{B(3C_0)} \chi_2 \alpha \p_{\beta_\ell} \bar S \p_i \p_{\beta^{(\ell)}} \tilde S \p_{\beta} \tilde U_i \\ 
&= - \sum_{\ell = 1}^m \sum_{i \in [3]}  \sum_{\beta^{(\ell)} \in [3]^{m-1} } \int_{B(3C_0)} \chi_2 \alpha \p_R \bar S  \p_i \p_{\beta^{(\ell)}} \tilde S \sum_{\beta_{\ell} \in [3] }
\frac{y_{\beta_\ell}}{R}  \p_{\beta_\ell} \p_{\beta^{(\ell)}} \tilde U_i \\
&= - \sum_{\ell = 1}^m \sum_{i \in [3]}  \sum_{\beta^{(\ell)} \in [3]^{m-1} } \int_{B(3C_0)} \chi_2 \alpha \p_R \bar S  \p_i \p_{\beta^{(\ell)}} \tilde S 
\p_R  \p_{\beta^{(\ell)}} \tilde U_i \\
&\leq \sum_{\ell = 1}^m   \sum_{\beta^{(\ell)} \in [3]^{m-1} } \int_{B(3C_0)} \chi_2 \alpha |\p_R \bar S| \sum_{i \in [3]} \left( \frac{|\p_i \p_{\beta^{(\ell)}} \tilde S|^2}{2} +
\frac{|\p_R  \p_{\beta^{(\ell)}} \tilde U_i|^2}{2} \right) \\ 
&\leq \sum_{\ell = 1}^m   \int_{B(3C_0)} \chi_2 \alpha |\p_R \bar S| \sum_{\beta \in [3]^{m} }\left( \frac{| \p_{\beta} \tilde S|^2}{2} +
 \sum_{i \in [3]}  \frac{|  \p_{\beta} \tilde U_i|^2}{2} \right).
\end{align*}
We obtain that
\begin{equation} \label{eq:J2}
\sum_{\beta\in [3]^m} \sum_{i=1}^3 \int_{B(3C_0)} \mc J_{2, i, \beta} \p_{\beta} \tilde U_i \leq \frac{m \alpha}{2} \int_{B(3C_0)} \chi_2 |\p_R \bar S| \left( |\nabla^m \tilde S|^2 + | \nabla^m \tilde U |^2 \right).
\end{equation}

\subsubsection{Term $\mc J_{4}$}
%& \mc J_{4, \beta} &=   - \sum_{\ell = 1}^m \alpha  \chi_2 (\p_{\beta_\ell} \bar S) \cdot \p_{\beta^{(\ell)}} \text{div} (\tilde U)

We have that
\begin{align*}
\sum_{\beta \in [3]^m} \int_{B(3C_0)} \mc J_{4, \beta} \p_\beta \tilde S 
&= - \sum_{\ell = 1}^m \sum_{i\in [3]} \sum_{\beta \in [3]^m} \int_{B(3C_0)} \chi_2 \alpha \p_{\beta_\ell} \bar S \p_i \p_{\beta^{(\ell)}} \tilde U_i \p_{\beta} \tilde S \\
&= - \sum_{\ell = 1}^m \sum_{i\in [3]} \sum_{\beta^{(\ell)} \in [3]^{m-1}} \int_{B(3C_0)} \chi_2 \alpha \p_R \bar S \p_i \p_{\beta^{(\ell)}} \tilde U_i \sum_{\beta_\ell \in [3]}\frac{y_{\beta_\ell}}{R}\p_{\beta_\ell} \p_{\beta^{(\ell)}} \tilde S \\
&= -  \sum_{\ell = 1}^m \sum_{i\in [3]} \sum_{\beta^{(\ell)} \in [3]^{m-1}} \int_{B(3C_0)} \chi_2 \alpha \p_R \bar S \p_i \p_{\beta^{(\ell)}} \tilde U_i \p_{R} \p_{\beta^{(\ell)}} \tilde S \\
&\leq  \sum_{\ell = 1}^m  \sum_{\beta^{(\ell)} \in [3]^{m-1}} \int_{B(3C_0)} \chi_2 \alpha | \p_R \bar S | \sum_{i\in [3]} \left( \frac{| \p_i \p_{\beta^{(\ell)}} \tilde U_i |^2}{2} + \frac{ | \p_{R} \p_{\beta^{(\ell)}} \tilde S |^2}{2}\right) \\
&\leq \sum_{\ell = 1}^m  \int_{B(3C_0)} \chi_2 \alpha | \p_R \bar S | \sum_{\beta \in [3]^{m}}  \left(  \sum_{i\in [3]} \frac{| \p_{\beta} \tilde U_i |^2}{2} + \frac{ |  \p_{\beta} \tilde S |^2}{2}\right).
\end{align*}
Thus, we obtain that
\begin{equation} \label{eq:J4}
\sum_{\beta \in [3]^m} \int_{B(3C_0)} \mc J_{4, \beta} \p_\beta \tilde S  
\leq \frac{m\alpha}{2} \int_{B(3C_0)} \chi_2 |\p_R \bar S| \left( |\nabla^m \tilde U|^2 + |\nabla^m \tilde S|^2 \right).
\end{equation}

\subsubsection{Conclusion}
%alcorcon
%\int_{B(3C_0)} &\left( \nabla^m \mc L_u \cdot \nabla^m \tilde U + \nabla^m \mc L_s \nabla^m \tilde S \right) 
%& \leq C \| (\tilde U, \tilde S )\|_X^2 - \int_{B(3C_0)} \left( \frac{J(1-\chi_1)}{2} + 2m \chi_2 \right)  \left( |\nabla^m \tilde U|^2 + (\nabla^m \tilde S)^2 \right)  \\
%&\quad + \int_{B(3C_0)} \left( \sum_i (\mc J_{1, i} + \mc J_{2, i}) \nabla^m \tilde U_i + ( \mc J_3 + \mc J_4 ) \nabla^m \tilde S \right)

Recall the radial and angular repulsivity properties of the profile from \eqref{eq:radial_repulsivity}--\eqref{eq:angular_repulsivity}:
\begin{equation*}
     - \p_R \bar U_R +  \alpha | \p_R \bar S | < 1-\tilde{\eta} \qquad \mbox{ and } \qquad - \frac{\bar U_R}{R} + \alpha |\p_R \bar S | < 1-\tilde{\eta},
\end{equation*}
for some $\tilde{\eta} > 0$. Combining \eqref{eq:J1}, \eqref{eq:J3}, \eqref{eq:J2}, \eqref{eq:J4}, and using the repulsivity properties of the profile, we obtain
\begin{align*}
&\int_{B(3C_0)} \left( \sum_{i=1}^3 (\mc J_{1, i} + \mc J_{2, i}) \nabla^m \tilde U_i + ( \mc J_3 + \mc J_4 ) \nabla^m \tilde S \right) \leq
m (1-\tilde{\eta}) \left( \| \nabla^m \tilde U \|_{L^2}^2 + \| \nabla^m \tilde S \|_{L^2}^2 \right),
\end{align*}
for some $\tilde{\eta} > 0$ fixed.\footnote{$\tilde{\eta}$ depends only on the profile from \cite{Buckmaster-CaoLabora-GomezSerrano:implosion-compressible}, that is, it depends on $r$ and $\gamma$. We let $r$, $\gamma$ (and therefore our profile) to be fixed throughout all this work.} Now, we plug this back into \eqref{eq:alcorcon}, and get
\begin{align*}
\int_{B(3C_0)} \left( \nabla^m \mc L_u \cdot \nabla^m \tilde U + \nabla^m \mc L_s \nabla^m \tilde S \right) 
 &\leq C \| (\tilde U, \tilde S )\|_X^2 - \int_{B(3C_0)} \left( \frac{J(1-\chi_1)}{2} + m \chi_2 \right)  \left( |\nabla^m \tilde U|^2 + (\nabla^m \tilde S)^2 \right)  \\
&\quad + m (1-\tilde{\eta}) \left( \| \nabla^m \tilde U \|_{L^2}^2 + \| \nabla^m \tilde S \|_{L^2}^2 \right).
\end{align*}
Since $J$ is sufficiently large depending on $m$, we can assume $J \geq 2m$. Noting that $\chi_1$ is $0$ whenever $\chi_2 < 1$ (that is, $\chi_1 = 0$ for $|y| \geq 2C_0$), we see that $\left( \frac{J(1-\chi_1)}{2} + m \chi_2 \right) \geq m$. Therefore
\begin{align*}
\int_{B(3C_0)} \left( \nabla^m \mc L_u \cdot \nabla^m \tilde U + \nabla^m \mc L_s \nabla^m \tilde S \right) 
 &\leq C \| (\tilde U, \tilde S )\|_X^2 -  m \tilde{\eta} \int_{B(3C_0)}  \left( |\nabla^m \tilde U|^2 + (\nabla^m \tilde S)^2 \right)  \\
 & \leq \left( C- \frac{m\tilde{\eta}}{2} \right) \| (\tilde U, \tilde S )\|_X^2.
\end{align*}
Here we used that for $(\tilde U, \tilde S) \in X_K$ we have that $\|\nabla^m \tilde U \|_{L^2}^2 + \| \nabla^m \tilde S \|_{L^2}$ is greater than $\frac12 \| (\tilde U, \tilde S)\|_X^2$.

Taking $m$ sufficiently large so that $\frac{m\tilde{\eta}}{2} > C+1$, we conclude the desired inequality: 
\begin{equation} \label{eq:diss_final}
    \int \nabla^m \mc L_u \cdot \nabla^m \tilde U + \nabla^m \mc L_s \cdot \nabla^m \tilde S \leq - \|(\tilde U, \tilde S ) \|_X^2, \qquad \forall (\tilde U, \tilde S) \in X_K.
\end{equation}

\subsection{Maximality} \label{sec:maximality}
In this section we show that $\mc L:X \to X$ satisfies that $\mc L - \lambda$ is surjective for some sufficiently large $\lambda$. As a consequence, we will obtain that $\mc L$ is a maximally dissipative operator modulo compact perturbations, and this will give a good description of the spectra of $\mc L$ and $\mc L^\ast$ that will allow us to define the spaces of stability and instability of $\mc L$.

 Now we show an equivalent definition of maximality:
 \begin{lemma}\label{maximalitylemma}
 For a dissipative operator $A$ defined on a linear subspace $D(A)$ of a Hilbert space X, if there exists $\lambda_0>0$ such that $\lambda_0I-A$ is surjective, then for all $\lambda>0$, $\lambda I-A$ is surjective.
 \end{lemma}
 \begin{proof}
Since $\lambda_0I -A$ is surjective,  for any $g\in X$, there exists $f\in X$ such that $(\lambda_0I-A)g=f$. From the dissipativity, we also have
\begin{align*}
\langle g,f\rangle_X=\langle (\lambda_0I -A)f,f\rangle _{X}\geq \lambda_0 \langle f,f\rangle_{X}.
\end{align*}
Therefore, $\lambda_0I-A$ is invertible and 
\[
\|(\lambda_0I-A)^{-1}\|\lesssim \frac{1}{\lambda_0}.
\]
Then the convergence radius of $(\lambda I-\lambda_0 I-A)^{-1}$ is  at least $\lambda_0$. Hence for all $\lambda \in (0,2\lambda_0)$, $\lambda I-A$ is surjective. By induction on $\lambda$, we have the result.
 \end{proof}

We now introduce a energy estimate Lemma in $X$. 
\begin{lemma} \label{lemma:rough_energy_bound}
There exists  some constant $C_{J, m}$ (depending on $m$ and $J$), such that
\begin{equation*}
\langle \mc L(\tilde{U},\tilde{S}), (\tilde{U},\tilde{S}) \rangle_X \leq C_{J, m} \| (\tilde{U},\tilde{S}) \|_X^2, \qquad \forall (\tilde{U},\tilde{S}) \in X.
\end{equation*}
\end{lemma}
\begin{proof}
Since $C_m$ is allowed to depend on $m$, the proof is straightforward and follows from integration by parts. From \eqref{Lequation}, we have
\[
\nabla^{m}\mc L_u =  -\chi_{2} ((y + \bar U) \cdot \nabla ) (\nabla^{m} \tilde{U}) -\chi_{2} \alpha \bar S  (\nabla^{m+1}\tilde S) + \mathcal T_{u},
\]
\[
\nabla^{m}\mc L_s =  -\chi_{2} ((y + \bar U) \cdot \nabla ) (\nabla^{m} \tilde S) -\chi_{2} \alpha \bar S \nabla^{m}\div (\tilde {U}) + \mathcal T_{s},
\]
where the terms $\mc T_{u}$, $\mc T_{s}$ have at most $m$ derivatives on $U$ or $S$. Thus, they satisfy

\[
\| \mathcal{T}_{\circ}\|_{L^2} \lesssim_{J, m} \| \nabla^m \tilde{U} \|_{L^2} + \| \nabla^m \tilde{S} \|_{L^2} + \| \tilde{U} \|_{L^2} + \| \tilde{S} \|_{L^2} \les \| (\tilde{U},\tilde{S}) \|_X, \qquad \mbox{ for } \circ \in \{ u, s \}.
\]
Using those bounds and integrating by parts the top order terms, we get
\begin{align*}
&\langle \nabla^{m} \mc L_{u}, \nabla^{m} \tilde{U} \rangle_{L^2} + \langle \nabla^{m} \mc L_{s}, \nabla^{m} \tilde{S} \rangle_{L^2} \\
&\quad \leq \int \left| \nabla(\chi_{2} (y + \bar U)) \right| \left( |\nabla^{m} \tilde{U}|^2 + |\nabla^{m} \tilde{S}|^2 \right) dy + \int |\div(\chi_2 \al \bar{S} )| | \nabla^{m} \tilde{U}| | \nabla^{m} \tilde{S}| dy + C_{J, m}' \|(\tilde{U},\tilde{S}) \|_{X}^{2}\\
&\quad \leq C_{J, m} \|(\tilde{U},\tilde{S})\|_{X}^{2}.
\end{align*}
\end{proof}
Now, we state the main Lemma of this section.

\begin{lemma}[Maximality of $\mc L$] \label{lemma:maximality} There exists some $\lambda$ (depending on $m$ and $J$) such that $\mc L - \lambda : X \to X$ is a surjective operator.
\end{lemma}
\begin{proof}

Let us define the operator $\mc L_N = P_N (\mc L_u, \mc L_s) P_N$ to be a two-sided truncation in frequency of $\mc L$, where $P_N$ is a projection to frequencies $\leq N$. Note that $P_N$ is defined from $H^m (\mathbb T_{4 C_0})$ to itself (it does not map $X$ to itself, because $P_N (f_u, f_s)$ does not need to satisfy the vanishing conditions that $(f_u, f_s) \in X$ satisfy). In particular, $P_N u \in C^\infty$. 

Now, from Lemma \ref{lemma:rough_energy_bound}, we see that
\begin{equation*}
\langle \mc L_N (\tilde{U},\tilde{S}), (\tilde{U},\tilde{S}) \rangle_{H^m} = \langle \mc L P_N (\tilde{U},\tilde{S}), P_N (\tilde{U},\tilde{S}) \rangle_{H^m} \leq C_{J, m} \| (\tilde{U},\tilde{S}) \|_{H^m}^2,
\end{equation*}
and we obtain that $\mc L_N - C_{J, m}$ is a dissipative operator.

Since the cokernel and image of $\mc L_N$ are both finite dimensional spaces, it is clear that
\begin{equation*}
\| \mc L_N (\tilde{U},\tilde{S}) \|_{H^m} \leq C'_{N,J,m} \| (\tilde{U},\tilde{S}) \|_{H^m}
\end{equation*}
for some constant $C'_{N,J,m}$. In particular, $\mc L_N - 2C_{J,m} - 2C_{N,J,m}'$ is surjective. Here $\mathcal L_{N}$ is a bounded operator, because $\| \frac{ \mathcal L_N}{2C_{J,m}+2C_{N,J,m}^{'}} \|_{H^m}< 1$, so we know $ \frac{\mc L_N}{2C_{J,m} + 2C_{N,J,m}'} -I$ is invertible. Using Lemma \ref{maximalitylemma} we conclude that $\mc L_N - C_{J, m}$ is maximally dissipative. 

This gives the bound
\begin{equation*}
\left\| ( \mc L_N - \lambda )^{-1} \right\|_X \lesssim \frac{1}{\lambda - C_{J, m}}
\end{equation*}
for $\lambda > C_{J, m}$. We take $\lambda = C_{J, m}+1$ and show that $\mc L - \lambda$ is surjective. From the equation above, we have $\| (\mc L_N - \lambda)^{-1} \|_{H^m} \leq 1$.

In order to solve the equation $(\mc L - \lambda  )f = F$, we define $f_N = (\mc L_N - \lambda)^{-1} F$, which are bounded in $X$ due to the bound above. Thus, since $X$ is a separable Hilbert space, there exists some subsequence converging weakly, $f_{N_i} \rightharpoonup f$, to some $f\in X$. Since we are in a bounded domain, the Rellich-Kondrachov compact embedding theorem gives us that the convergence $f_{N_i} \to f$ is strong in the topology of $H^{m-3}$ (with $f \in H^{m}$).

Recalling 
\begin{equation*}
(\mc L_N - \lambda )f_N = F,
\end{equation*}
and denoting $Y = H^{m-10}$, we obtain 
\begin{align*}
\left\| (\mc L - \lambda  )f - F \right\|_Y &\leq 
\underbrace{ \| P_N (\mc L - \lambda) P_N f - P_{N}F\|_Y }_0
+\|F - P_{N}F\|_Y+ \| P_N(\mc L - \lambda ) P_N f - (\mc L - \lambda) f\|_Y \\
&\leq  \|F - P_{N}F\|_Y+\lambda \| P_N f - f \|_Y + \| P_N \mc L (P_N f - f) \|_Y + \| (P_N - \text{Id} ) \mc L f \|_Y
\end{align*}
Let us take $N = N_i$ and $i \to \infty$. Since $X \subset Y$, the first two  summands clearly tend to zero. Moreover, $\|P_N f - f \|_{H^m} \to 0$, combined with the fact that $\| \mc L f\|_Y \les_{J, m} \| f \|_{H^m}$ yields that the third summand also tends to zero. The fourth summand tends to zero given that $\mc L f \in Y$. In summary,  
\begin{equation*}
\| (\mc L  - \lambda ) f - F \|_Y = 0
\end{equation*}
and we get
\begin{equation*}
\mc L f - \lambda f = F.
\end{equation*}
Then we are left to show $f = (\tilde{U},\tilde{S}) \in X$.  From the definition of $\chi_1, \chi_2$ before \eqref{eq:cutoffL}, we have that for $|y| \geq 5C_0/2$
\[
\mc L_{u}(\tilde{U},\tilde{S})=-J\tilde{U}, \qquad \mbox{ and } \qquad \mc L_{s}(\tilde{U},\tilde{S})=-J\tilde{S}.
\]
Therefore, for $|y| \geq 5C_0/2$
\[f= \left( \frac{F_U}{-J-\lambda}, \frac{F_S}{-J-\lambda} \right).
\]
Then since $F=(F_{U}, F_{S})\in (H^{m}_0(B(0,3C_0)))^2$, $(\tilde{U},\tilde{S})\in (H^{m})^2$ , we have $(\tilde{U},\tilde{S})\in (H^{m}_0(B(0,3C_0)))^2 $.

\end{proof}

We conclude the maximal dissipativity modulo compact perturbations. 

\begin{proposition}\label{prop:maxdismod} There exist $m_0$ and some $J_m$ such that for every $m\geq m_0$, and for every $J\geq J_m$, we have the following. For any $\delta_g \in (0, 1)$, we can express the cut-off linearized operator $\mc L$ as:
\begin{equation}\label{eq:maxdismod}
\mc L = \mc L_0 - \delta_g + \mc K,
\end{equation}
where $\mc L_0: X \to X$ is a maximally dissipative operator and $\mc K$ is a compact operator, both depending on $\delta_g$ (as well as $m$ and $J$).
\end{proposition}
\begin{proof} The proof is a consequence of: \begin{itemize}
\item $\langle \mc L (\tilde U, \tilde S), (\tilde U, \tilde S) \rangle_X \leq -\| (\tilde U, \tilde S) \|_X^2$ for every $(\tilde U, \tilde S) \in X_K$ (by \eqref{eq:diss_final}).
\item $\mc L - \lambda : X\to X$ is surjective for sufficiently large $\lambda$ (by Lemma \ref{lemma:maximality}).
\end{itemize}
The argument that combines those two facts to obtain the statement of the Proposition can be found in \cite[Corollary 7.7]{Buckmaster-CaoLabora-GomezSerrano:implosion-compressible}.
\end{proof}

\begin{lemma} \label{lemma:abstract_result} Let $m \geq m_0$, $J\geq J_m$, and $\delta_g \in (0, 1)$. Then we have: \begin{enumerate}
\item\label{item:spectrum} We denote by $\sigma (\mc L)$ the spectrum of $\mc L$. The set $\Lambda  = \sigma (\mc L) \cap \{ \lambda \in \mathbb{C} : \Re (\lambda) > -\delta_g/2  \}$ is finite and formed only by eigenvalues of $\mc L$. Moreover each $\lambda \in \Lambda$ has finite algebraic multiplicity. That is, if we let $\mu_{\lambda}$ to be the first natural such that ${\rm ker} (\mc L-\lambda  )^{\mu_{\lambda}} = {\rm ker} (\mc L- \lambda  )^{\mu_\lambda + 1}$, we have that the vector space
\begin{equation} \label{eq:spaceV}
V_{\rm{uns}} = \bigoplus_{\lambda \in \Lambda} {\rm ker} (\mc L - \lambda )^{\mu_{\lambda}}\,,
\end{equation}
is finite dimensional. 

\item \label{item:invariance} We denote by $\mc L^\ast$ the adjoint of $\mc L$. Let $\Lambda^\ast = \sigma (\mc L^\ast) \cap \{ \lambda \in \mathbb{C} : \Re (\lambda) >- \delta_g/2\}$. As before, we define
\begin{equation} \label{eq:spaceVast}
V_{\rm{sta}} = \left( \bigoplus_{\lambda \in \Lambda^\ast} {\rm ker} (\mc L^\ast - \lambda )^{\mu_{\lambda}^\ast}\right)^{\perp}.
\end{equation}
We have that both $V_{\rm{uns}}$ and $V_{\rm{sta}}$ are invariant under $\mc L$. We also have that $\Lambda^\ast = \overline{\Lambda}$ and $\mu_\lambda = \mu_{\overline{\lambda}}^\ast$. Moreover, we have the decomposition $X = V_{\rm{uns}} \oplus V_{\rm{sta}}$.

\item \label{item:stability_outwards} The linear transformation $\mc L|_{V_{\rm{uns}}}: V_{\rm{uns}} \rightarrow V_{\rm{uns}} $ obtained by restricting $\mc L$ to the finite dimensional space $V_{\rm{uns}}$ has all its eigenvalues with real part larger than $-\delta_g/2$. In particular, there is some basis such that we can express
\begin{equation*}
\mc L|_{V_{\rm{uns}}}  = 
\begin{bmatrix}
J_1 &  & &  \\
 & J_2 & &  \\
 & & \ddots & \\
 &  & & J_\ell
 \end{bmatrix},  
 \qquad \mbox{ where } \qquad J_i = 
 \begin{bmatrix}
\lambda_i & \frac{\delta_g}{10} & & \\
 & \lambda_i & \ddots & \\
 & & \ddots & \frac{\delta_g}{10} \\
 & & & \lambda_i
 \end{bmatrix},
\end{equation*}
where $\lambda_i$ are the eigenvalues of $\mc L|_{V_{\rm{uns}}}$. In that basis, we have that
\begin{equation} \label{eq:chisinau}
\Re \left(\bar{w}^T \cdot \mc L|_{V_{\rm{uns}}} \cdot w \right) \geq -\frac{6 \delta_g}{10} \| w \|^2, \qquad \forall w \in \mathbb{C}^N.
\end{equation}

Moreover, letting $T(t)$ be the semigroup generated by $\mc L$, for any  $v \in V_{\rm{sta}}$ we have
\begin{equation} \label{eq:prim}
\| T(t) v \|_X \lesssim e^{-\delta_g t/2} \| v \|_X.
\end{equation}
\end{enumerate}
\end{lemma}
\begin{proof}
The statement holds for any operator of the form \eqref{eq:maxdismod} in a Hilbert space (see \cite[Lemma 7.11]{Buckmaster-CaoLabora-GomezSerrano:implosion-compressible}, which is itself a slightly different formulation of \cite[Lemma 3.3]{Merle-Raphael-Rodnianski-Szeftel:implosion-ii}). The form \eqref{eq:maxdismod} is guaranteed by estimate \eqref{eq:diss_final}, since $m\geq m_0$, $J\geq J_m$, $\delta_g \in (0, 1)$.
\end{proof}

\subsection{Smoothness of $V_{\rm{uns}}$} \label{sec:smoothness}

The main result of this subsection is to show that all the elements of $V_{\rm{uns}}$ are smooth provided we choose again $m$ and $J$ sufficiently large (possibly larger than the restriction coming from Proposition \ref{prop:maxdismod}). Clearly, it suffices to show this for a basis. We will follow the choice of basis made in Lemma \ref{lemma:abstract_result} so that the elements of the basis are a union of chains $\{\psi_j\} = 
\{(\psi_{u,j}, \psi_{s,j})\}$ that satisfy:
\begin{align} \nonumber
&\psi_{0} = 0, \ \psi_{j}\neq 0 \text{ when }j\geq 1,\\
&\mc L \psi_{j+1} = \lambda \psi_{j+1} + \frac{\delta_g}{10} \psi_{j},\label{eq:GEE}
\end{align}
where $\lambda$ is the eigenvalue of $\psi_1$ (it will be fixed through all this subsection, so we omit subindices that would distinguish different eigenvalues). Let us also stress that $\lambda$ and $\psi$ may be complex-valued and $\{\psi_j\}$ must be a finite sequence. The operator $\mc L$ acts on complex-valued functions like $\psi$ just component-wise (if $\psi = \psi_r + i \psi_i$ with $\psi_r, \psi_i$ real, then the real and imaginary parts of $\mc L \psi$ are given by $\mc L \psi_r$ and $\mc L \psi_i$ respectively). We will also refer to \eqref{eq:GEE} as the generalized eigenfunction equation.

First of all, let us claim the following Lemma:
\begin{lemma}[Extension of generalized eigenfunctions]  \label{lemma:eigenextension} Let $m$ be sufficiently large and $J\geq J_{*}$ sufficiently large independent of $m$. Let $\{\psi_{j}\} = \{( \psi_{U,j}, \psi_{S,j})\}$ be in $H^{2m+100}$ defined on $B(0, C_0)$ so that it satisfies the generalized eigenfunction equation \eqref{eq:GEE} for $|y| \leq C_0$. Then, there exists a unique generalized eigenfunction $\{\tilde \psi_{j}\} \in X$ that agrees with $\{\psi_{j}\}$ on $B(0, C_0)$ and satisfies \eqref{eq:GEE} for all $|y| \leq 3C_0$. Moreover, such extension is constantly equal to zero for $|y| \in (5C_0/2, 3C_0)$.
\end{lemma}

Let us delay the proof of this Lemma to the end of the section and show our main result of this section assuming Lemma \ref{lemma:eigenextension}:

\begin{proposition} \label{prop:smooth} Let $\delta_g \in (0, 1)$. For sufficiently large $\bar{m}$, $\bar{J}$, there exist $m,J$ such that $m>\bar{m}$, $J>\bar{J}$ and  we have that the space $V_{\rm{uns}} \subset X$ defined in Lemma \ref{lemma:abstract_result} is formed by smooth functions.
\end{proposition}
\begin{proof} Given that the space $X$ depends on $m$ and the operator $\mc L$ depends on $J$, let us denote $X$ by $X^m$ and $\mc L$ by $\mc L_J$ for the rest of this proof.

First of all, we restrict ourselves to $m \geq m_0$, $J\geq \max\{J_m,J_{*}\}$, so that we satisfy the hypotheses of Proposition \ref{prop:maxdismod}, Lemma \ref{lemma:abstract_result} and Lemma \ref{lemma:eigenextension}. We take the definition of $V_{\rm{uns}}$ from Lemma \ref{lemma:abstract_result} for $\bar{J}_{m} = \max\{J_m,J_{*}\}$ and make the dependence on $m$ explicit, so that $V_{\rm{uns}} (m) \subset X^m$ refers to the space of unstable modes of $\mc L_{\bar{J}_m}: X^m \to X^m$ (if there are no unstable modes, $V_{\rm{uns}} (m) = \{ 0 \} $). In particular, Lemma \ref{lemma:abstract_result} gives us that $V_{\rm{uns}} (m)$ is finite dimensional. We moreover define $V_{\rm{uns}}^{C_0} (m)$ to be the vector space obtained by restricting the functions of $V_{\rm{uns}} (m)$ to $B(0, C_0)$. 

Let us show that if $m_0\leq m'\leq 2m'+100  \leq m$, then $V_{\rm{uns}}^{C_0}(m) \subset V_{\rm{uns}}^{C_0}(m')$. We take $\{\psi_{i}\}$ be a chain of eigenfunctions in $V_{\rm{uns}}(m)$ such that $\{\psi_i\}  \subset V_{\rm{uns}}(m)$, and it satisfies the generalized eigenfunction equation \eqref{eq:GEE}: 
\[
\mc L_{\bar J_m} \psi_{i+1} = \lambda \psi_{i+1} + (\delta_g/10) \psi_{i}.
\]
Given that the operator $\mc L_{J}$ is independent of $J$ on $B(0, C_0)$ (the only dependence on $J$ comes from $(1-\chi_1)J$, which vanishes on $B(0, C_0)$), we have that the restriction function $ \psi_{i}|_{B(0, C_0)}$ satisfies the generalized eigenfunction equation for $\mc L_{J_{m'}} $ on $B(0, C_0)$:
\[
\mc L_{\bar J_{m'}} \psi_{i+1} = \lambda \psi_{i+1} + (\delta_g/10) \psi_{i}.
\]
  Since it is clearly in $H^{2m'+100}$ (as $m \geq 2m'+100$), we are in the hypotheses of Lemma \ref{lemma:eigenextension}. Therefore, there exists a unique extension $\{\tilde \psi_{j}\} \subset X^{m'}$, which is a chain generalized eigenfunction of  $\mc L_{\bar{J}_{m'}} $ in and agrees with $\{\psi_{j}|_{B(0, C_0)}\}$ on $B(0, C_0)$. Therefore, $\{\psi_{j}|_{B(0, C_0)}\} \subset V_{\rm{uns}}^{C_0}(m')$. We conclude the desired property:
\begin{equation} \label{eq:Vunschain} V_{\rm{uns}}^{C_0}(m) \subset V_{\rm{uns}}^{C_0}(m'), \qquad \mbox{ for } m \geq 2m'+100.
\end{equation}

Let $\{m_{j}\}$ be an integer sequence such that $m_{j+1}\geq 2m_{j}+100$ with $m_{1}=2m_0$.  Since $\dim (V_{\rm{uns}}^{C_0}(m_{j}))$ is finite, integer, and non-increasing for $j\geq 1$ (due to \eqref{eq:Vunschain}); there must exist a value $j_{*}$ such that $\dim (V_{\rm{uns}}^{C_0}(m))$ is constant for every $m \in [m_j, \infty)$ with $j\geq j_{*}$. Then, since $V_{\rm{uns}}^{C_0}(m_{j+1}) \subset V_{\rm{uns}}^{C_0}(m_j)$ by \eqref{eq:Vunschain}, we deduce that $V_{\rm{uns}}^{C_0}(m_j)$ is the same vector space for $j\geq j_*$. 

Now, let us fix  $m_*=m_{j_*}$ and $J =\bar J_{m_{*}}$. We will show $V_{\rm{uns}}(m_{*})$ only consists of smooth functions. Let $\{\psi_{j}\}$ be a chain of generalized eigenfunctions in $V_{\rm{uns}}(m)$ and $\{\bar \psi_{j}\}$ be their restriction on $B(0,C_0)$.  As a consequence of the above, all $\bar \psi_{j}$ are smooth on $B(0, C_0)$.  Thus, by Lemma \ref{lemma:eigenextension} for any $m\geq m_*$, the $\bar \psi_j$ admit a unique $H^{m}$ extension to $B(0, 3C_0)$ satisfying \eqref{eq:GEE} for $\mc L_J$, which we call $\tilde \psi_{j}$ (note here that $m$ may be much larger than $J$, it is fundamental that in Lemma \ref{lemma:eigenextension} we do not require $J$ to be sufficiently large in terms of $m$). Now, both $\{\psi_{j}\}$ and $\{\tilde \psi_{j}\}$ are $H^{m}$ extensions of $\{\psi_{j}|_{B(0,C_0)}\}$ satisfying the generalized eigenvalue equation \eqref{eq:GEE} for $\mc L_J$. By the uniqueness in Lemma \ref{lemma:eigenextension}, they are equal. Then $\{\psi_j\}$ are in $H^{m}$ for any $m\geq m_{*}.$
\end{proof}

Thus, it only remains to show Lemma \ref{lemma:eigenextension}. 

\begin{proof}[Proof of Lemma \ref{lemma:eigenextension}]
 Let us denote $\psi_{j}=(\psi_{U,j}, \psi_{S,j}) = (U_j, S_j)$.

\textbf{Vanishing:} First, it is very easy to justify that 
\begin{equation}\label{vanishing}
\psi_{j}=0, \qquad \text{ for all } \,j
\end{equation}
 for $|y| \in \left( \frac52 C_0, 3C_0 \right)$ by induction. First, $\psi_0=0$ by definition \eqref{eq:GEE}. Then we  assume \eqref{vanishing} holds for $j\leq j_0$. Since $\chi_1 = \chi_2 = 0$ in that region due to \eqref{eq:cutoffL} we have that $\mathcal L = -J$ , we see that \eqref{eq:GEE} reads
\begin{equation*}
-J \psi_{U,j_0+1} = \lambda U_{j_0+1}, \qquad \mbox{ and } \qquad -J \psi_{S,j_0+1} = \lambda S_{j_0+1},
\end{equation*}
for $|y| \in \left( \frac52 C_0, 3 C_0 \right)$. Since $\Re(\lambda) > -\delta_g/2 \geq -1/2$ by construction, and we can assume $J\geq 1$, we conclude that $\psi = 0$ in the region $|y| \in \left( \frac52 C_0, 3 C_0 \right)$.

\textbf{Uniqueness:}

We consider two generalized eigenfunction chains agreeing on $B(0, C_0)$ (in particular, if they are not identically zero, they share the same $\lambda$). Letting $\{(U_j, S_j)\}$ be the difference of those eigenfunctions, we have $U_j|_{B(0,C_0)}=0$, $S_{j}|_{B(0,C_0)}=0$ and we want to show $(U_j,S_j)=0$ everywhere.

Now we proceed by induction. We  first assume that for $j\leq j_0$, $U_{j_0}=S_{j_0}=0$. Thus, for $j=j_0+1$, we have $\mc L (U_{j_0+1}, S_{j_0+1}) = \lambda (U_{j_0+1}, S_{j_0+1})$.

For the sake of simplicity, we will use $(U,S)$ to denote $(U_{j_0+1}, S_{j_0+1})$. Developing the equation $(\mc L_U, \mc L_S) (U,S) = \lambda (U, S) $,  we have that
\begin{align*}
\lambda U  &= -\chi_2 (r-1) U - \chi_2 (y + \bar U) \nabla U - \chi_2 \alpha \bar S \nabla S - \chi_2 U \nabla \bar U - \alpha \chi_2 S \nabla \bar S - J(1-\chi_1)U, \\
\lambda S &= -\chi_2 (r-1) S - \chi_2 (y + \bar U) \nabla S - \chi_2 \alpha \bar S \div (U) - \chi_2 S \div (\bar U ) - \alpha \chi_2 U \nabla \bar S - J(1-\chi_1)S.
\end{align*}

For some weight $\phi$ that will be defined later, we multiply the equations by $U^\ast \phi$ and $S^\ast \phi$ respectively, where $f^\ast$ denotes the complex conjugate of $f$. Then, we take the real part and integrate in some ball $B = B(0, \rho )$ with $\rho<\frac{5}{2}C_0$. We get that
\begin{align*}
\int_B \Re (\lambda ) |U|^2  \phi &= \frac12 \int_B \p_R (\chi_2 (R+\bar U_R))\phi |U|^2 + \frac12 \int_B \left( \frac{\p_R \phi}{\phi} (\chi_2 (R+\bar U_R)) \right) |U|^2 \phi+\int_B \frac{\chi_2 (R+\bar U_R)}{R} \phi |U|^2 \\
&\qquad
- \frac12 \int_{\p B} \chi_2(R+\bar U_R) \phi |U|^2
- \alpha \int_B \chi_2 \bar S \phi \Re (\nabla S \cdot U^\ast) \\
&\qquad - J\int (1-\chi_1) |U|^2\phi
+ \int_B \chi_2 (f_1 |U|^2 + f_2 |S|^2)\phi,
\end{align*}
and 
\begin{align*}
\int_B \Re (\lambda ) |S|^2    \phi 
&= \frac12 \int \p_R (\chi_2 (R+\bar U_R)) \phi |S|^2 + \frac12 \int_B \left( \frac{\p_R \phi}{\phi} (\chi_2 (R+\bar U_R)) \right) |S|^2 \phi+\int_B \frac{\chi_2 (R+\bar U_R)}{R} \phi |S|^2
\\
&\qquad - \frac12 \int_{\p B}\chi_2 (R+\bar U_R) \phi |S|^2
- \alpha \int_B \chi_2 \bar S \phi \Re (\div (U) S^\ast ) \\
&\qquad - J\int (1-\chi_1) |S|^2 \phi
+  \int_B \chi_2 (f_3 |U|^2 + f_4 |S|^2)\phi,
\end{align*}
where $f_i$ are radially symmetric functions with an $L^\infty$ norm bounded by some universal constant. Adding up both equations, noting that $$\Re (\nabla S \cdot U^\ast)  + \Re (\div (U) S^\ast)  = \Re ( U \cdot \nabla S^\ast ) + \Re ( \div (U ) S^\ast ) =   \Re\div (U S^\ast),$$ and integrating by parts, we see that
\begin{align}\nonumber
\Re (\lambda )\int_B ( |U|^2 + |S|^2 )\phi  &=  \frac12 \int_B \left( \frac{\p_R \phi}{\phi} (\chi_2 (R+\bar U_R)) \right) (|U|^2 + |S|^2) \phi + \alpha \int_B \chi_2 \bar S \frac{\p_R \phi}{\phi} \Re ( \frac{y}{R}\cdot( U S^\ast) )\phi \\
&\quad- J\int_B (1-\chi_1) (|U|^2 + |S|^2) \phi + O \left( \int_B (\chi_2 + |\chi_2'|  ) (|U|^2 + |S|^2) \phi \right) \nonumber \\
&\quad -\frac12 \int_{\p B} \chi_2 (R+\bar U_R)\phi (|U|^2 + |S|^2) - \alpha \int_{\p B} \chi_2 \bar S \phi \Re (  S^{*} \frac{y}{R}\cdot U ), \label{eq:edmundodantes}
\end{align}
where we used that $|US^\ast| \leq \frac12 (|U|^2 + |S|^2) $. Using it again, since $\partial_{R}\phi\leq 0$, we get 
\begin{align*}
&-\frac12 \int_B \chi_2 \frac{\p_R \phi}{\phi}\phi  (R + \bar U_R-\alpha |\bar S|)(|U|^2 + |S|^2) + J \int_B (1-\chi_1) \phi (|U|^2 + |S|^2) + \Re (\lambda )\int_B (|U|^2 + |S|^2 ) \phi \\
&\qquad \leq \frac12 \int_{\p B} \chi_2 (-R-\bar U_R+\alpha |\bar S|)\phi (|U|^2 + |S|^2) + O \left( \int_B (\chi_2 + |\chi_2'| ) (|U|^2 + |S|^2)\phi \right),
\end{align*}
where we used equation \eqref{eq:edmundodantes} to bound the last term. Now, recall that $U, S = 0$ for $R \leq C_0$ and we want to show they are zero in all $B(0, 3C_0)$. Note also that from Lemma \ref{lemma:integrated_repulsivity} and \eqref{eq:profiles_positive}, we have that 
$$R + \bar U_R (R) - \alpha |\bar S (R)| \geq \tilde \eta (R-1),$$ for some $\tilde \eta > 0$ and where we denote with $\bar U_R$ the radial component of the profile (which is the only nonzero component since the profile is radially symmetric). Since the contribution of the first integral is zero for $R \leq C_0$, we can lower bound $R + \bar U_R (R) - \alpha \bar S (R)$ by $\tilde \eta (R-1)$. Moreover, assuming $\rho > C_0$, we also get a sign for the boundary term in the right hand side. Thus, we obtain

\begin{align*}
\int_{B\backslash B(0,C_0)} \left( -\chi_2 \frac{\p_R \phi}{\phi}  \frac{\tilde \eta}{2} + J(1-\chi_1) 
 + \Re (\lambda )\right) (|U|^2 + |S|^2) \phi   \leq C \int_{B\backslash B(0,C_0)}  (|U|^2 + |S|^2)\phi,
\end{align*}
where $C$ is a sufficiently large constant. Finally, note that $\Re (\lambda ) > -\delta_g/2\geq -1/2$ from the definition of $V_{\rm{uns}}$ in Lemma \ref{lemma:abstract_result}. Take $\phi=e^{-\frac{2J}{\tilde{\eta}}R}$ to satisfy $\p_R \phi = -\phi \frac{2J}{\tilde \eta}$ in the region $|y| \in \left( C_0, \frac32 C_0 \right)$  outside. Since $(1-\chi_1)+\chi_2 \geq 1$ by definition, we conclude 
\begin{equation*}
\int_B (J-1/2) (|U|^2 + |S|^2)\phi \leq C  \int_B  (|U|^2 + |S|^2)\phi.
\end{equation*}
This yields a contradiction for $J$ sufficiently large, unless $U = S = 0$ everywhere. We conclude the proof of uniqueness.

\textbf{Existence:} 
We consider the equations 
\[
\mc L (U_j, S_j) = \lambda (U_j, S_j) + (\delta_g/10) (U_{j-1}, S_{j-1}).
\]
We have a solution on $B(0, C_0)$ and want to extend it to $B(0, 3C_0)$.

For the sake of simplicity, we first consider the following equation by letting $U,S$ denote $U_{j}, S_{j}$ and $U^{-},S^{-}$ denote $U_{j-1}, S_{j-1}$. Let us recall:
\begin{align*}
\lambda U + \frac{\delta_g}{10} U^- &= -\chi_2 (r-1) U - \chi_2 (R + \bar U_R) \p_R U - \chi_2 \alpha \bar S \nabla S - \chi_2 U \nabla \bar U - \alpha \chi_2 S \nabla \bar S - J(1-\chi_1)U, \\
\lambda S + \frac{\delta_g}{10} S^- &= -\chi_2 (r-1) S - \chi_2 (R + \bar U_R) \p_R S - \chi_2 \alpha \bar S \div U  - \chi_2 S \div \bar U - \alpha \chi_2 U \nabla \bar S - J(1-\chi_1)S, 
\end{align*}
Now, we express the complex scalar $S$ as $S = \sum_n S_n(R) e_n$, where $e_n$ are spherical harmonics. Moreover, any complex vector field can be expressed uniquely \cite{Hill:spherical-harmonics} as
\begin{equation*}
U = \sum_n \left( U_{R,n}(R) \hat R e_n + U_{\Psi, n}(R) R \nabla e_n + U_{\Phi, n}(R) y \wedge \nabla e_n  \right).
\end{equation*}
Let us note that
\begin{equation*}
\nabla S = \sum_n \left( \p_R S_n \hat R e_n + \frac{S_n}{R} R\nabla e_n  \right), \qquad 
\div U = \sum_n \left( \p_R U_{R, n}(R)  + \frac{2U_{R, n}(R)}{R} - \frac{U_{\Psi, n}(R)}{R} n_1(n_1+1) \right)e_n,
\end{equation*}
using that $R^2 \Delta e_n = -n_1 (n_1+1) e_n$. Note also that in any spherical coordinate system (that is, in any basis formed by $\hat R$ and vectors perpendicular to $\hat R$), we have $\p_R \hat R = 0$ and $\p_v \hat R = \frac{\p_v y}{R} = \frac{v}{R}$ for any $v$ perpendicular to $\hat R$. Let us recall that $\hat R$ is the unitary vector in the radial direction, as discussed in the notation.

Therefore
\begin{align*}
U\nabla \bar U &= U \nabla ( \bar U_R \hat R) =  \bar U_R' U_R \hat R + \bar U_R U\cdot \nabla\hat R = \sum_n \left( \bar U_R' U_{R, n} e_n \hat R + \bar U_R (U_{R, n} e_n \p_R + U_{\Psi, n} \p_{R \nabla e_n} + U_{\Phi, n} \p_{y \wedge \nabla e_n}) \hat R  \right) \\ 
&= \sum_n \left( \bar U_R' U_{R, n} e_n \hat R +   \frac{\bar U_R U_{\Psi, n}}{R}  R\nabla e_n + \frac{\bar U_R U_{\Phi, n}}{R} y  \wedge \nabla e_n  \right).
\end{align*}

We obtain the equations
\begin{align*}
\lambda U_{R, n} + \frac{\delta_g}{10} U^-_{R, n}  &= -\chi_2 (r-1) U_{R, n} - \chi_2 (R + \bar U_R) \p_R U_{R, n} - \chi_2 \alpha \bar S S_n' - \chi_2 \bar U_R' U_{R, n} - \alpha \chi_2 S_n \p_R \bar S - J(1-\chi_1)U_{R, n}, \\
\lambda U_{\Psi, n} + \frac{\delta_g}{10} U^-_{\Psi, n}  &= -\chi_2 (r-1) U_{\Psi, n} - \chi_2 (R + \bar U_R) \p_R U_{\Psi, n} - \chi_2 \alpha \bar S \frac{S_n(R)}{R} - \chi_2 \frac{\bar U_R U_{\Psi, n}}{R} - J(1-\chi_1)U_{\Psi, n}, \\
\lambda U_{\Phi, n} + \frac{\delta_g}{10} U^-_{\Phi, n} &= -\chi_2 (r-1) U_{\Phi, n} - \chi_2 (R + \bar U_R) \p_R U_{\Phi, n}  - \chi_2 \frac{\bar U_R U_{\Phi, n}}{R} - J(1-\chi_1)U_{\Phi, n}, \\
\lambda S_{n} + \frac{\delta_g}{10} S^-_{n} &= -\chi_2 (r-1) S_n - \chi_2 (R + \bar U_R) \p_R S_n - \chi_2 \alpha \bar S \left( U_{R, n}'(R)  + \frac{2U_{R, n}(R)}{R} - \frac{U_{\Psi, n}(R)}{R} n_1(n_1+1) \right) \\
&\quad - \chi_2 \alpha S_n \left( \p_R \bar U_R + \frac{2\bar U_R}{R} \right) - \alpha \chi_2 U_{R, n} \p_R \bar S - J(1-\chi_1)S_n, 
\end{align*}
so, as we can see, the equations diagonalize in the different modes (the evolution of the $n$-th spherical harmonic mode only depends on the other $n$-th spherical harmonic modes).

Defining $W = U_R + S$ and $Z = U_R - S$, we have
\begin{align} \begin{split}\label{eq:sys}
\chi_2 (R + \bar U_{R} + \alpha \bar S) \p_R W_{n} &= -\lambda W_{n} - \frac{\delta_g}{10} W^-_{n}  - J(1-\chi_1)W_{n} + \chi_2 \alpha \bar S \frac{U_{\Psi, n}(R)}{R} n_1 (n_1+1) + \chi_2 \mathcal E_{W, n}, \\
\chi_2 (R + \bar U_{R} - \alpha \bar S ) \p_R Z_n &= -\lambda Z_{ n} - \frac{\delta_g}{10} Z^-_{n}  - J (1-\chi_1) Z_n - \chi_2 \alpha \bar S \frac{U_{\Psi, n}(R)}{R} n_1 (n_1+1) + \chi_2 \mathcal E_{Z, n}, \\
\chi_2 (R + \bar U_{R}) \p_R U_{\Psi, n} &= -\lambda U_{\Psi, n} - \frac{\delta_g}{10} U^-_{\Psi, n} - J(1-\chi_1)U_{\Psi, n} - \chi_2 \left( (r-1) U_{\Psi, n} + \alpha \bar S \frac{W_n-Z_n}{2R} + \frac{\bar U_R U_{\Psi, n}}{R} \right), \\
\chi_2 (R + \bar U_{R}) \p_R U_{\Phi, n} &= -\lambda U_{\Phi, n} - \frac{\delta_g}{10} U^-_{\Phi, n}   - J(1-\chi_1)U_{\Phi, n}  - \chi_2 \left( (r-1) + \frac{\bar U_R}{R} \right)  U_{\Phi, n}, 
\end{split}\end{align}
where we defined the lower order terms
\begin{align*}
\mathcal E_{W, n} &= - \left( (r-1) + \alpha \p_R \bar S
 \right) W_n-\partial_{R}\bar{U}_{R}(U_{R,n}+\alpha S_{n})-\frac{2\alpha}{R}[\bar{S}U_{R,n}+\bar{U}_{R}S_n], \\  
 \mathcal E_{Z, n} &= -\left( (r-1) - \alpha \p_R \bar S
 \right) Z_n -\partial_{R}\bar{U}_{R}(U_{R,n}-\alpha S_{n})-\frac{2\alpha}{R}[\bar{S}U_{R,n}+\bar{U}_{R}S_n].
\end{align*}
The possibility of complex eigenfunctions is handled in the spherical mode decomposition, given that real functions correspond to the coefficient of the  mode $(n_1, n_2)$ being the conjugate of the  coefficient of the mode $(n_1, -n_2)$. 

Since $U^-, S^-$ are given in $B(0, 3C_0)$ by induction, we can solve \eqref{eq:sys} by standard ODE theory for $R\in [C_0, 3C_0)$ given that the terms multiplying the derivative do not vanish. 

With respect to $R > \frac52 C_0$ we obtain directly that $W_n = Z_n = U_{\Psi, n} = U_{\Phi, n} = 0$ from \eqref{eq:sys} (either in the diagonalizable or non-diagonalizable case).

Taking Taylor series, and using that $J$ is sufficiently large, one can also conclude smoothness for each $W_n,Z_n, U_{\phi,n},  U_{\psi,n} $ at $R = \frac{5}{2}C_0$ (see \cite[Lemma 7.5]{Buckmaster-CaoLabora-GomezSerrano:implosion-compressible}).

\begin{lemma}\label{eigenfunctionbound}
Consider the spherical energy
\begin{equation*}
E_{n,k}(R) =  ( R + \bar U_R + \alpha \bar S ) \frac{|\partial_{R}^{k}W_n|^2}{2} + ( R + \bar U_R - \alpha \bar S ) \frac{|\partial_{R}^{k}Z_n|^2}{2} + ( R + \bar U_R ) \left( |\partial_{R}^{k}U_{\Psi, n}|^2 + |\partial_{R}^{k}U_{\Phi, n}|^2 \right) (n_1\left( n_1+1\right)+1),
\end{equation*}
and 
\begin{equation*}
E_{n,k}^{-}(R) =  ( R + \bar U_R + \alpha \bar S ) \frac{|\partial_{R}^{k}W_n^{-}|^2}{2} + ( R + \bar U_R - \alpha \bar S ) \frac{|\partial_{R}^{k}Z_n^{-}|^2}{2} + ( R + \bar U_R ) \left( |\partial_{R}^{k}U_{\Psi, n}^{-}|^2 + |\partial_{R}^{k}U_{\Phi, n}^{-}|^2 \right) (n_1\left( n_1+1\right)+1).
\end{equation*}
Then we have for $C_0\leq R<\frac{5}{2}C_0$,
\begin{align}\label{modesmoothestimate}
E_{n,k}(R)\lesssim E_{n,k}(C_0)+\sup_{C_0\leq R<\frac{5}{2}C_0}[E_{n,k}^{-}(R)]+\sum_{k'=0}^{k-1}\sup_{C_0\leq R<\frac{5}{2}C_0}[E_{n,k'}(R)][\sqrt{n_1(n_1+1)}+1].
\end{align}
Here the implicit constant in $\lesssim$ can depend on $m,J$.
\end{lemma}
\begin{proof} 
In the following proof we use the notation:
\[
O(f(R))\Leftrightarrow \text{bounded by } C f(R)\text{ for some fixed constant $C$,}
\]
\[
O_{m,J}(f(R))\Leftrightarrow  \text{ bounded by } C_{m,J}f(R) \text{ for some constant depending on $m,J$}. 
\]
First, we show the bound for $k=0$. For $R\in [C_0, \frac{5}{2}C_0)$, since $\chi_2\neq 0$, we have
\begin{align*}
\p_R E_{n,0} &= - \frac{J(1-\chi_1) + \Re\lambda}{\chi_2}\left( |W_{n,0}|^2 + |Z_{n,0}|^2 + 2(n_1(n_1+1)+1) (|U_{\Psi, n,0}|^2 + |U_{\Phi, n,0}|^2 )\right)\\
&\quad +O( E_{n,0} ) + O\left( \frac{1}{\chi_2}\sqrt{E_{n,0} E_{n,0}^{-} }\right).
\end{align*}
because $\mathcal E_{W, n,0}, \mathcal E_{Z, n,0} = O(E_{n,0})$ and the terms $\alpha \bar S \frac{\Re (U_{\Psi, n,0}W^{\ast}_{n,0})}{2R}n_1(n_1+1)$ and $\alpha \bar S \frac{\Re(U_{\Psi, n,0}Z_{n,0}^{\ast})}{2R}n_1(n_1+1)$ cancel between the contribution coming from the equations of $W_{n,0}, Z_{n,0}$ and the contribution coming from the equation of $U_{\Psi, n,0}$. In fact, from \eqref{eq:sys}, we have
\begin{align*}
&\quad \Re( W_{n,0}^{\ast}\frac{U_{\Psi, n,0}}{R}n_1(n_1+1)\chi_2\alpha\bar{S}-Z_{n,0}^{\ast}\frac{U_{\Psi, n,0}}{R}n_1(n_1+1)\chi_2\alpha\bar{S}-n_1(n_1+1)U_{\Psi, n,0}^{\ast}(\frac{W_{n,0}-Z_{n,0}}{R})\chi_2\alpha\bar{S})\\
&=\chi_2\alpha\bar{S}n_1(n_1+1)\frac{1}{R}\Re(W_{n,0}^{\ast}U_{\Psi, n,0}-Z_{n,0}^{\ast}U_{\Psi, n,0}-U_{\Psi, n,0}^{\ast}(W_{n,0}-Z_{n,0}))\\
&=0.
\end{align*}                     
Here, $E_{n,0}^{-}$ corresponds to the energy of the $n$-th mode of the previous eigenfunction, in the non-diagonal case.

In particular, using that $\Re(\lambda)\geq \frac{-\delta_{g}}{2}$ as in Lemma \ref{lemma:abstract_result}, and the lower bounds on $R + \bar U_R \pm \alpha \bar S$ for $R \in [R_1, R_2]$, there exists an absolute constant $C$ such that
\begin{align*}
\quad \p_R E_{n,0} &\leq \frac{1}{\chi_2} C E_{n,0} - \frac{J(1-\chi_1)}{\chi_2} \frac{E_{n,0}}{100}
+ \frac{C}{\chi_2}\sqrt{E_{n,0} E_{n,0}^{-} } \\
&\leq \frac{2}{\chi_2} C E_{n,0} - \frac{J(1-\chi_1)}{\chi_2} \frac{E_{n,0}}{100}
+ \frac{C}{\chi_2}E_{n,0}^{-},
\end{align*}
where we apply Cauchy-Schwarz in the second inequality.

Now, note that $\chi_1 = 0$ in the region where $\chi_2<1$. Thus, the term $\frac{CE_{n,0}}{\chi_2}$ in the region $\chi_2<1$ can be absorbed assuming $J$ is sufficiently large. Then we have the estimate: $$\frac{1}{\chi_2}C\leq \frac{J}{400\chi_2}(1-\chi_1)+C. $$ Enlarging $C$ if needed, we obtain
\begin{equation*}
\p_R E_{n,0} \leq CE_{n,0} - \frac{J(1-\chi_1)}{\chi_2} \frac{E_{n,0}}{200} + \frac{C E_{n,0}^{-}}{\chi_2}.
\end{equation*}
Again, enlarging $C$ if needed
\begin{align*}
\p_R E_{n,0} &\leq C(E_{n,0}+E_{n,0}^{-}) - \frac{(1-\chi_1)}{\chi_2} \left( \frac{J E_{n,0}}{200} - CE_{n,0}^{-} \right) \\
&\leq C(E_{n,0}+E_{n,0}^{-}) - \frac{J(1-\chi_1)}{400\chi_2} \left( E_{n,0} - E_{n,0}^{-} \right) - \frac{J(1-\chi_1)}{400\chi_2} E_{n,0}
\end{align*}
taking $J \geq 400C$.
Then we have 
\begin{align}\label{es:eigenfunctionsmode}
\p_R (E_{n,0}e^{-2CR}) &\leq \left(C(E_{n,0}^{-}-E_{n,0}) - \frac{J(1-\chi_1)}{400\chi_2} \left( E_{n,0} - E_{n,0}^{-} \right) - \frac{J(1-\chi_1)}{400\chi_2} E_{n,0}\right)e^{-2CR}.
\end{align}
Therefore, we claim $E_{n,0}(R)e^{-2CR}\leq \sup_{R}(E_{n,0}^{-})+E_{n,0}(C_0)$. When $R=C_0$, the estimate is satisfied naturally. Moreover, from estimate \eqref{es:eigenfunctionsmode}, when $E_{n,0}(R)e^{-2CR}= \sup_{R}(E_{n,0}^{-})+E_{n,0}(C_0)$, we have $\frac{dE_{n,0}(R)e^{-2CR}}{dR}\leq 0$. We get 
\begin{align}\label{smootheigenfunction0}
E_{n,0}\lesssim e^{2CR}\left(\sup_{R}(E_{n,0}^{-})+E_{n,0}(C_0)\right)
\end{align}
and \eqref{modesmoothestimate} follows when $k=0$.

Now we estimate the higher derivatives with respect to $R$. From \eqref{eq:sys}, we have
\begin{align} \label{higherorderderivativeeigenfunctions} 
&\chi_2 (R + \bar U_R + \alpha \bar S) \p_R^{k+1} W_{n}+\sum_{j\leq k}C_{j,k}\p_{R}^{-j+1+k}(\chi_2 (R + \bar U_R + \alpha \bar S))\p_R^{j} W_{n} \\\nonumber
&= -\lambda \p_R^{k}W_{n} - \frac{\delta_g}{10} \p_R^{k}W^-_{n}  - J(1-\chi_1)\p_R^{k}W_{n}- \sum_{j\leq k-1}C_{j,k}\p_R^{k-j}(J(1-\chi_1))\p_R^{j}W_{n}\\\nonumber
&+ \chi_2 \alpha \bar S \frac{\p_R^{k}U_{\Psi, n}(R)}{R} n_1 (n_1+1)+ \sum_{j\leq k-1}C_{j,k}\p_R^{k-j} \left( \frac{\chi_2 \alpha \bar S}{R} \right) \p_R^{j}U_{\Psi, n}(R) n_1 (n_1+1) +\p_R^{k} (\chi_2 \mathcal E_{W, n}).
\end{align}
Here we use same $C_{j,k}$ but actually mean possibly different constants depending on $k,j $.
 Let $\delta_{m,J}$ sufficiently small to be fixed later. From the definition of $\chi_2$, there exists $\epsilon_{m,J}>0$  depending on $\delta_{m,J}$, such that for $R\in[C_0,\frac{5}{2}C_0-\delta_{m,J}]$, $|\chi_2|\geq \epsilon_{m,J}$. 
Then from \eqref{higherorderderivativeeigenfunctions}, when $R\in[C_0,\frac{5}{2}C_0-\delta_{m,J}]$, we have 
\begin{align*}
(R + \bar U_R + \alpha \bar S) \p_R^{k+1} W_{n}&=\sum_{k'=0}^{k-1}O_{m,J} \left( \sqrt{E_{n,k'}(R)} \right) \left( \sqrt{n_1(n_1+1)}+1 \right) +O_{m,J} \left( \sqrt{E_{n,k}(R)} \right)\\
&\quad+O_{m,J} \left( \sqrt{E_{n,k}^{-}(R)} \right) + \alpha \bar S \frac{\p_R^{k}U_{\Psi, n}(R)}{R} n_1 (n_1+1).
\end{align*}
Similarly, we also have the following estimates for other terms in the spherical energy:
\begin{align*}
(R + \bar U_R + \alpha \bar S) \p_R^{k+1} Z_{n}&= \sum_{k'=0}^{k-1}O_{m,J} \left( \sqrt{E_{n,k'}(R)} \right) \left( \sqrt{n_1(n_1+1)}+1 \right) + O_{m,J} \left( \sqrt{E_{n,k}(R)} \right)\\
&\quad+O_{m,J} \left( \sqrt{E_{n,k}^{-}(R)} \right)- \alpha \bar S \frac{\p_R^{k}U_{\Psi, n}(R)}{R} n_1 (n_1+1),
\end{align*}
\begin{align*}
(R + \bar U_R - \alpha \bar S) \p_R^{k+1} U_{\psi,n}&= \sum_{k'=0}^{k-1}O_{m,J} \left( \sqrt{E_{n,k'}(R)} \right) +O_{m,J} \left( \frac{\sqrt{E_{n,k}(R)}}{\sqrt{n_1(n_1+1)}+1)} \right)\\
&\quad+O_{m,J} \left( \frac{\sqrt{E_{n,k}^{-}(R)}}{(\sqrt{n_1(n_1+1)}+1)} \right)- \alpha \bar S \frac{\p_R^{k}(W_{n}(R)-Z_{n}(R))}{R},
\end{align*}
and
\begin{align*}
(R + \bar U_R ) \p_R^{k+1} U_{\psi,n}&= \sum_{k'=0}^{k-1}O_{m,J} \left( \sqrt{E_{n,k'}(R)} \right) +O_{m,J} \left( \frac{\sqrt{E_{n,k}(R)}}{(\sqrt{n_1(n_1+1)}+1)} \right)\\
&\quad+O_{m,J} \left( \frac{\sqrt{E_{n,k}^{-}(R)}}{(\sqrt{n_1(n_1+1)}+1)} \right).
\end{align*}
Then a similar cancellation as in the $k=0$ case happens, and we have the following estimate for the spherical energy:
\[
\partial_{R}E_{n,k}(R)= \sum_{k'=0}^{k-1}O_{m,J}(E_{n,k'}(R)) \left( \sqrt{n_1(n_1+1)}+1 \right) +O_{m,J}(E_{n,k}(R))+O_{m,J}(E_{n,k}^{-}(R)). 
\]
Therefore we get:
 \begin{align} \begin{split} \label{smootheigenfunction1}
 \sup_{R\in[C_0,\frac{5}{2}C_0-\delta_{m,J}]}E_{n,k}(R)&\lesssim |E_{n,k}(C_0)|+\sup_{R\in [C_0,\frac{5}{2}C_0-\delta_{m,J}]}E_{n,k}^{-}(R) \\ 
 &\qquad +\sum_{k'=0}^{k-1}\sup_{R\in [C_0,\frac{5}{2}C_0-\delta_{m,J}]}E_{n,k'}(R) \left( \sqrt{n_1 (n_1+1)}+1 \right).
 \end{split} \end{align}

Now if $\delta_{m,J}\leq \frac{1}{10}C_0$, for $\frac{5}{2}C_0\geq R\geq \frac{5}{2}C_0-\delta_{m,J}$, we have $\chi_1=0$. Then from \eqref{higherorderderivativeeigenfunctions}, the following estimate holds:
\begin{align*}
(R + \bar U_R + \alpha \bar S) \p_R^{k+1} W_{n}=&-\frac{\chi_2'}{\chi_2}k(R + \bar U_R + \alpha \bar S)(\p_R^{k}W_{n}(R))-\frac{J}{\chi_2}\partial_{R}^{k}W_{n}-\frac{\lambda}{\chi_{2}}\partial_{R}^{k}W_{n}(R)+O(\partial_{R}^{k}W_n(R))\\
&\qquad+\frac{1}{\chi_2}O_{m,J} \left( \sqrt{E_{n,k}^{-}(R)} \right) -\frac{1}{\chi_2}\sum_{k'=0}^{k-1}O_{m,J} \left( \sqrt{E_{n,k'}(R)} \right) \left( \sqrt{n_1(n_1+1)}+1 \right) \\
&\qquad+\alpha \bar S \frac{\p_R^{k}U_{\Psi, n}(R)}{R} n_1 (n_1+1).
\end{align*}
Similarly, we have the estimate for $\partial_{R}^{k}Z_n$:
\begin{align*}
(R + \bar U_R - \alpha \bar S) \p_R^{k+1} Z_{n}(R)
&=-\frac{\chi_2'}{\chi_2}k(R + \bar U_R - \alpha \bar S)\p_R^{k}Z_{n}(R)-\frac{J}{\chi_2}\partial_{R}^{k}Z_{n}(R)-\frac{\lambda}{\chi_{2}}\partial_{R}^{k}Z_{n}(R)+O(\partial_{R}^{k}Z_n(R))\\
&+\frac{1}{\chi_2} O_{m,J} \left( \sqrt{E_{n,k}^{-}(R)} \right) + \frac{1}{\chi_2}\sum_{k'=0}^{k-1} O_{m,J} \left( \sqrt{E_{n,k'}(R)} \right) \left( \sqrt{(n_1)(n_1+1)}+1 \right) \\
&-\alpha \bar S \frac{\p_R^{k}U_{\Psi, n}(R)}{R} n_1 (n_1+1).
\end{align*}

for $\partial_{R}^{k}U_{\psi,n}$:

\begin{align*}
(R + \bar U_R)\p_R^{k+1} U_{\psi,n}(R) &= 
-\frac{\chi_2'}{\chi_2}k(R + \bar U_R)\p_R^{k} U_{\psi,n}(R)-\frac{J}{\chi_2}\partial_{R}^{k}U_{\psi,n}(R) - \frac{\lambda}{\chi_{2}}\partial_{R}^{k}U_{\psi,n}(R) + O \left( \partial_{R}^{k}U_{\psi,n}(R) \right) \\
&+\frac{1}{\chi_2}\frac{O_{m,J} \left( \sqrt{E_{n,k}^{-}(R)} \right) }{\sqrt{n_1(n_1+1)}+1}+\frac{1}{\chi_2}\sum_{k'=0}^{k-1}O_{m,J} \left( \sqrt{E_{n,k'}(R)} \right) \\
&-\alpha \bar S \frac{(\p_R^{k}W_{n}(R)-\p_R^{k}Z_{n}(R))}{R}.
\end{align*}

and $\partial_{R}^{k}U_{\phi,n}:$
\begin{align*}
(R + \bar U_R)\p_R^{k+1} U_{\phi,n}&=-\frac{\chi_2'}{\chi_2}k(R + \bar U_R)\p_R^{k} U_{\phi,n}(R)-\frac{J}{\chi_2}\partial_{R}^{k}U_{\phi,n}(R)-\frac{\lambda}{\chi_{2}}\partial_{R}^{k}U_{\phi,n}(R)+O(\partial_{R}^{k}{U_{\phi,n}(R)})\\
&+\frac{1}{\chi_2}O_{m,J}\left(\frac{\sqrt{E_{n,k}^{-}(R)}}{\sqrt{n_1(n_1+1)}+1}\right)+\frac{1}{\chi_2}\sum_{k'=0}^{k-1}O_{m,J} \left( \sqrt{E_{n,k'}(R)} \right).
\end{align*}
Then for the energy $E_{n,k}(R)$, we have
\begin{align*}
\p_{R}E_{n,k}(R)&\leq -\frac{2\chi_2'}{\chi_2}k E_{n,k}(R)-\frac{J}{2\chi_2}E_{n,k}(R)\\
&+\frac{1}{\chi_2}O_{m,J}(E_{n,k}^{-}(R))+\frac{1}{\chi_2}\sum_{k'=0}^{k-1}O_{m,J}(E_{n,k'}(R)) \left( \sqrt{n_1(n_1+1)}+1 \right),
\end{align*}
where we also used that $\Re(\lambda) > -\delta_g/2 \geq -1/2$, $\chi_2\leq 1$ and $J$ sufficiently large.

Next we will use a similar method as in \cite[Lemma 7.5]{Buckmaster-CaoLabora-GomezSerrano:implosion-compressible} to bound the energy. In fact, let 
\begin{align*}
R_{m,J}=\frac{5}{2}C_0-\delta_{m,J}, \quad B_{m,J}=O_{m,J}(E_{n,k}^{-}(R))+\sum_{k'=0}^{k-1}O_{m,J}(E_{n,k'}(R) \left( \sqrt{n_1(n_1+1)}+1) \right).
\end{align*} 
    
Then 
$$\frac{d}{d\gamma}e^{\int_{R_{m,J}}^{\gamma}\frac{J}{2\chi_2(\beta)}}=e^{\int_{R_{m,J}}^{\gamma}\frac{J}{2\chi_2(\beta)}d\beta}\frac{J}{2\chi_2(\gamma)}.$$ 

We have 
\begin{align*}
E_{n,k}(R) &\leq e^{-\int_{R_{m,J}}^{R}\frac{J}{2\chi_2(\beta)}d\beta}(|E_{n,k}(R_m)|+m\int_{R_{m,J}}^{R}\frac{d}{d\gamma}e^{\int_{R_{m,J}}^{\gamma}\frac{J}{2\chi_2(\beta)}d\beta}(|B_{m,J}(\gamma)|+C_{m,J}|\chi_{2}'(\gamma)|E_{n,k}(\gamma))d\gamma)\\
& \leq |E_{n,k}(R_m)|+ 2m \sup_{\gamma \in [R_{m,J},\frac{5}{2}C_0]}|B_{m,J}(\gamma)|+C_{m,J}\sup_{\gamma \in [R_{m,J},\frac{5}{2}C_0]}(|\chi_2'(\gamma)||E_{n,k}(\gamma)|).
\end{align*}
 Since $\displaystyle \lim_{R\to\frac{5} {2}C_0}\chi'_2(R)=0$ from the definition, we choose $\delta_{m,J}$ sufficiently small to make 
 \[
 \sup_{\gamma \in {[R_{m,J},\frac{5}{2}C_0]}}C_{m,J}|\chi_2'(\gamma)|\leq \frac{1}{2}.
 \]
 Therefore, we have
 \begin{align} \begin{split}\label{smootheigenfunction2}
 \sup_{R\in[R_{m,J},\frac{5}{2}C_0]}E_{n,k}(R)&\lesssim |E_{n,k}(R_{m,J})|+\sup_{R\in [R_{m,J},\frac{5}{2}C_0]}E_{n,k}^{-}(R) \\ 
&\qquad +\sum_{k'=0}^{k-1}\sup_{R\in [R_{m,J},\frac{5}{2}C_0]}E_{n,k'}(R)(\sqrt{n_1 (n_1+1)}+1).
 \end{split} \end{align}
From \eqref{smootheigenfunction0}, \eqref{smootheigenfunction1} and \eqref{smootheigenfunction2}, we get \eqref{modesmoothestimate}.
\end{proof}
Now we change back $(U,S)$ to $(U_{j},S_{j})$, and $(U^{-},S^{-})$ to $(U_{j-1},S_{j-1})$. Let $E_{n,k}^{j}$ denote the corresponding $k$-th order spherical energy of $(U_{j},S_{j})$ . Then $E_{n,k}^{0}=0$ and from \eqref{modesmoothestimate}, we have
\begin{align}\label{estimate:eigenfunctions3}
E_{n,k}^{j}(R) \lesssim E_{n,k}^{j}(C_0) + \sup_{C_0\leq R<\frac{5}{2}C_0} E_{n,k}^{j-1}(R) 
 + \sum_{k'=0}^{k-1} \sup_{C_0 \leq R < \frac{5}{2} C_0} E_{n,k'}^{j}(R) \left( \sqrt{n_1(n_1+1)}+1 \right).
\end{align}
Hence we get that for all $j$, $0\leq k\leq m$,
\begin{align}\label{estimate:eigenfunctions2}
E_{n,k}^{j}(R)\lesssim \sum_{j'=1}^{j}\sum_{k'=0}^{k}E_{n,k'}^{j'}(C_0) \left( \sqrt{n_1(n_1+1)}+1 \right)^{k}.
\end{align}
First, for $k=0$ (from \eqref{smootheigenfunction0}) and $J=0$ ($E_{n,k}^{0}=0$), the estimate holds. Now we only need to show if for both $\tilde{j}\leq j-1$, $0\leq k\leq m$ and $\tilde{j}=j$, $\tilde{k}\leq k$, \eqref{estimate:eigenfunctions2} holds, then for $(j, k+1)$, \eqref{estimate:eigenfunctions2} also holds.
From the induction assumption and \eqref{estimate:eigenfunctions3} we have 
\begin{align*}
E_{n,k+1}^{j}(R)&\lesssim E_{n,k+1}^{j}(C_0)+\sum_{j'=1}^{j-1}\sum_{k'=0}^{k+1}E_{n,k'}^{j}(C_0)[\left( \sqrt{n_1(n_1+1)}+1 \right)^{k+1}\\
&\quad+\sum_{k'=0}^{k}\sum_{j'=1}^{j}\sum_{k''=0}^{k'}E_{n,k''}^{j'}(C_0)\left( \sqrt{n_1(n_1+1)}+1 \right)^{k+1}\\
&\lesssim \sum_{j'=1}^{j}\sum_{k'=0}^{k+1}E_{n,k'}^{j'}(C_0) \left( \sqrt{n_1(n_1+1)}+1\right)^{k+1}.
\end{align*}

Therefore, the regularity of $\{\psi_{j}\}$ in $\{C_0\leq R\leq 3 C_0\}$ is worse than in $\{\frac{1}{2}C_0\leq R\leq C_0\}$ by at most $(m+C_0)$ derivatives in the angular direction. Thus we are done.

%Now, we consider $R_\circ$ to be the last value of $R\in [R_1, R_2]$ such that $E_n \leq  E_n^{nd}$ (or $R_\circ = R_1$ if that never happens). For $R \in [R_\circ, R_2]$ we have that
%\begin{equation*}
%\p_R E_{n,0} \leq C(E_{n,0}+E_{n,0}^{nd})- \frac{J(1-%\chi_1)}{400\chi_2} E_{n,0}
%\end{equation*}
%Now, we can solve that ODE (either with $E_{n,0}(R_\circ) = E_{n,0}(R_1)$ or $E_{n,0}(R_\circ) = E_{n,0}^{-} (R_\circ)$), obtaining an upper bound for $E_{n,0}(R)$ which is smaller than $K \max \{ E_{n,0}(R_1), E_{n,0}^{-}(R_\circ) \}$, for some large constant $K$. Doing induction on the chain of non-diagonalizable eigenfunctions, and supposing the biggest chain has length $\ell$, we obtain that $E_{n,0}(R)$ is smaller than $K^\ell \max_j E_{n,0}^j (R_1)$ where $j$ runs over all the eigenfunctions with that eigenvalue

\end{proof}

\section{Nonlinear Stability}
\label{sec:nonlinear}

While the linear stability is treated with cut-off functions for the linearized operator, the nonlinear stability must be done for the full equation \eqref{nsequation1}. They are connected by the analysis done in subsection \ref{subsec:agreement}. After that, we perform the nonlinear stability analysis by bootstrap arguments with a decay assumption on the unstable mode. We formulate the bootstrap in Subsection \ref{subsec:bootstrap}, and defer the proof to Subsections \ref{lowerestimate} (low derivative bounds) and \ref{higherorderestimate} (high derivative bounds). We can get rid of the assumptions on the unstable modes by choosing initial data in a finite codimension manifold (see \eqref{initialdatatruncated}). This is done via a classical topological argument that we develop in Subsection \ref{subsec:topological}.

\subsection{Agreement of equations} \label{subsec:agreement}

In order to use our analysis for the linearized cut-off operator $\mc L = \chi_2 (\mc L_u^{e}, \mc L_s^{e}) - J(1-\chi_1)$, we consider an additional equation, which we call the \textit{truncated equation}:
\begin{align} 
\p_s (\mt U_t, \mt S_t) &= \mc L (\mt U_t, \mt S_t) + \chi_2 \mc F, \nonumber \\
(\mt U_t, \mt S_t)|_{s=s_0}&= (\mt U_{t, 0}, \mt S_{t, 0}) ,
\label{eq:truncated}\end{align}
where $\mc F$ is the forcing coming from the original (extended) equation \eqref{nsequation1} and $(\mt U_{t,0}, \mt S_{t,0})$ is the initial data for the extended equation which will be further specified later.

\begin{lemma}\label{truncatedestimate} Suppose that a classical solution $(\mt U, \mt S)$ to \eqref{nsequation1}--\eqref{nsequation2} exists up to time $s'$. Suppose that $(\mt U_t, \mt S_t)$ is a classical solution to \eqref{eq:truncated} until time $s'$ as well. If at $s=s_0$, $|y|\leq C_0$, we have $(\tilde{U}_{t},\tilde{S}_{t})$ agrees with $(\tilde{U},\tilde{S}),$ then, for $|y| \leq C_0$, $s \in [s_0, s')$, we also have that $(\mt U_t, \mt S_t)$ agrees with $(\mt U, \mt S)$.
\end{lemma}
\begin{proof} 

Let us show that they are equal by energy estimates. Let $(\delta_U, \delta_S) = (\mt U, \mt S) - (\mt U_t, \mt S_t)$. Given that $(\mt U_t, \mt S_t)$ satisfies equation \eqref{eq:truncated} (where we recall $\mc F$ is the forcing from the \textit{extended} equation), we have that
\begin{align}\begin{split} \label{eq:diff_truncated_extended}
\p_s (\delta_U, \delta_S) &= \mc L^{e} (\mt U, \mt S) - \chi_2 \mc L ^{e} (\mt U_t, \mt S_t) + (1-\chi_1) J (\mt U_t, \mt S_t) \\
&= \chi_2\mc L^{e} ( \delta_U, \delta_S) + (1-\chi_2) \mc L^{e}(\mt U, \mt S) + (1-\chi_1) J (\mt U_t, \mt S_t).
\end{split} \end{align}

We compute the evolution of the energy:
$$E(s) = \int_{B(0, C_0)}( |\delta_U(y, s)|^2 + \delta_S (y, s)^2 ) dy.$$
Clearly $E(s_0) = 0$ because of the hypothesis of the lemma. Moreover from \eqref{eq:diff_truncated_extended} we get that the evolution of $E$ is dictated by
\begin{align*}
\frac12 \frac{d}{ds} E
&= \int_{B(0, C_0)} - (y+\bar U)\frac12  \nabla \left( |\delta_U|^2 \right) - \alpha \bar S \nabla \delta_S \cdot \delta_U - (r-1)|\delta_U|^2 - \delta_U \cdot \nabla \bar U \cdot \delta_U - \alpha \delta_S \delta_U \cdot \nabla \bar Sdy \\
&\qquad + \int_{B(0, C_0)} -(y+\bar U) \frac12 \nabla \left( \delta_S^2 \right) - \alpha \bar S \delta_S \text{div}(\delta_U) - (r-1) \delta_S^2 - \delta_S \delta_U \cdot \nabla \bar S - \alpha \delta_S^2 \text{div} (\bar U)dy \\
&=
\int_{B(0, C_0)} \frac{3+\text{div}(\bar U)}{2} (|\delta_U|^2 + \delta_S^2) + \alpha \nabla \bar S \cdot (\delta_S\delta_U) - (r-1) |\delta_U|^2 - ((\delta_U \cdot \nabla) \bar U) \cdot \delta_U - \alpha \delta_S\delta_U \cdot \nabla \bar S \\
&\qquad- (r-1) \delta_S^2 - \delta_S \delta_U \cdot \nabla \bar S - \alpha \delta_S^2 \text{div} (\bar U)dy\\
&\qquad + \int_{\p B(0, C_0)} \frac{-R-\bar U_R}{2} (|\delta_U|^2 + \delta_S^2) - \alpha \bar S \delta_S (\delta_U\cdot \vec{n})d\sigma \\
&\leq C E  + \int_{\p B(0, C_0)} \frac{-R - \bar U_R + \alpha \bar S}{2} (|\delta_U|^2 + \delta_S^2) \leq CE.
\end{align*}
In the last inequality we used $-R - \bar U_R + \alpha \bar S < 0$ for $R =C_0$, due to Lemma \ref{lemma:integrated_repulsivity}.

In particular, since $\frac{d}{ds} E \leq 2CE$ and $E(s_0) = 0$, we conclude that $E(s) = 0$ for all $s \in [s_0, s')$, and therefore $(\tilde U, \tilde S)$ agrees with $(\tilde U_t, \tilde S_t)$ for $|y| \leq C_0$, $s\in [s_0, s')$. 
\end{proof}

\subsection{Bootstrap of bounds} \label{subsec:bootstrap}
In this subsection, we plan to show that given a small enough initial perturbation, if the unstable mode $P_{\rm{uns}}$ is controlled, then the perturbation $\tilde{U}$, $\tilde{S}$ is also controlled and has exponential decay as $s$ evolves. 

We will choose our parameters as follows: 
\begin{equation}\label{choiceofparameter}
\frac{1}{s_0}\ll \delta_{0}^{\frac{3}{2}}\ll \delta_1 \ll %\delta_g(\delta_0)^{\frac{5}{4}}
 \delta_0 \thicksim \frac{1}{R_0^{r-1}} \ll \frac{1}{E}\ll \epsilon^\sharp\ll \frac{1}{K}\ll\frac{1}{m} \ll \eta  \ll \epsilon =\frac{12}{25} \delta_g \ll \delta_{dis} = O(1),
\end{equation}
and
\begin{equation}\label{choiceofparameter2}
 \frac{1}{K}\ll\tilde{\eta} .
\end{equation} 
In the above chain of inequalities, we can think that any parameter is allowed to depend on all the parameters that it has to the right, with the exception of $\delta_0$ and $\delta_1$ which should satisfy $\delta_0^{3/2} \ll \delta_1 \ll \delta_0$ (we can take $\delta_1 = \delta_0^{5/4}$, for example). For the convenience of the reader, let us mention briefly where all these inequalities are used in our argument:

We control $\mc F_{dis}$ to be sufficiently small  via $\frac{1}{s_0}\ll \delta_0^{\frac{3}{2}}$, as well as in \eqref{eq:I4}.

In $\delta_0^{\frac{3}{2}}\ll \delta_1\ll\delta_0$, the second inequality is chosen to control the initial perturbation \eqref{initialdatacondition2}, which is sufficiently small compared to the bootstrap estimates in Proposition \ref{prop:bootstrap}. The first one is chosen to control the nonlinear terms $N_{u,s}$ as in Lemma \ref{forcingestimate1} and the lower bound of  $S$ Lemma \ref{lemma:Sbounds}.

$\delta_0\ll E$ is used to control the lower order Sobolev norm of $\tilde{U},\tilde{S}$ in Lemma \ref{boosstrapestimate2}, and Lemma \ref{lowerorderestimate}.

$\frac{1}{E}\ll \epsilon^\sharp$ is not crucial and we only need $\frac{1}{s_0}\ll\epsilon^\sharp$, which is also used to control $\mc F_{dis}$  in \eqref{eq:I4}.

$\epsilon^\sharp\ll \frac{1}{K}$ is used in the estimate of the $\mc F_{dis}$ in \eqref{epsicondition01} and \eqref{eq:grenada}. 

$\frac{1}{K}\ll \frac{1}{m}$ is used in Lemma \ref{forcingestimate2}. In order to bound the $X$ norm using our bootstrap assumptions \eqref{higherestimate3}, we need to perform the non-linear stability analysis at a higher regularity than the linear stability analysis. 

$\eta$ is defined in \eqref{weightdefi} and $\tilde{\eta}$ is defined in \eqref{eq:angular_repulsivity} and \eqref{eq:radial_repulsivity}. $\eta\ll\frac{1}{m}$ is not crucial and we only need $\frac{1}{k}\ll\eta,\tilde{\eta}$, which is used in the estimate \eqref{energyestimatewithweight}.

For $\eta\ll \delta_g$, we only use the weaker estimate $\eta \ll 1$. This estimate is used in \eqref{epsicondition01} and \eqref{eq:grenada}.

$\delta_g\ll \delta_{dis}$ is used in Lemmas \ref{boosstrapestimate1} and \ref{boosstrapestimate2} in order to control the term $\mc F_{dis}$ \eqref{nsequation1}. Because $\mc F_{dis}$  has decay not faster than $e^{-\delta_{dis}s}$, our bootstrap estimate can only have a lower decay.

We will define later $\epsilon$ via the relationship $\eps=\frac{12\delta_g}{25}$ in Proposition \ref{prop:bootstrap}. 

$R_0 \thicksim \delta_0^{\frac{-1}{r-1}}$  is defined in \eqref{choiceR0} and used in the choice of weight \eqref{weightdefi}.

For the initial data, let
\begin{align}\label{eq:tildeplusunstable}
\begin{cases}
\tilde{U}_0=\tilde{U}_0^{'}+\sum_{i=1}^{N}a_i,
\varphi_{i,u}\\
\tilde{S}_0=\tilde{S}_0^{'}+\sum_{i=1}^{N}a_i
\varphi_{i,s},
\end{cases}
\end{align}
where 
\begin{equation} \label{eq:tilde_is_stable}
(\chi_2 \tilde U_0', \chi_2 \tilde S_0' ) \in V_{\rm{sta}},
\end{equation}
and $(\varphi_{i,u},\varphi_{i,s})$ is an orthogonal and normal basis of $V_{\rm{uns}}$. ${a_i}$ is an $N$-dimensional vector to be fixed later. In this setting, the unstable component of the solution  will be determined by $\{a_i\}$. Moreover, we will pick $\tilde U_0', \tilde S_0'$ so that our initial data $(\tilde U_0, \tilde S_0)$ satisfies:
\begin{equation} \label{initialdatacondition2}
\|\tilde{U}_0\|_{L^{\infty}}, \|\tilde{S}_0\|_{L^{\infty}}\leq  \delta_1, \qquad  E_{K}(s_0)\leq\frac{E}{2}, \qquad \mbox{and} \qquad \tilde{S}_0+\hat{X}\bar{S}_0\geq \frac{\delta_1}{2} \left\langle\frac{R}{R_0} \right\rangle^{1-r},
\end{equation}
and the decay of the derivatives:
\begin{equation}\label{initialdatacondition}
|\nabla(\tilde{U}_0+\hat{X} \bar{U})|+|\nabla(\tilde{S}_0+\hat{X} \bar{S})|\lesssim \frac{1}{\langle y\rangle^r}.
\end{equation}
Recall that $\hat{X}$ is the cut-off function defined in \eqref{eq:periodic_perturbation}.

In order to make use of the property of the truncated linear operator, we also define the truncated solutions and initial data.
Recall that $(\tilde{U},\tilde{S})$ satisfy the following Navier-Stokes equations:
\begin{align}\begin{split} \label{navierstokesperturb1}
\p_s \tilde U &= \mc L_{u}^{e}(\tilde{U},\tilde{S})+\mc F_{u}(\tilde{U},\tilde{S}),\\
\p_s \tilde S &= \mc L_{s}^{e}(\tilde{U},\tilde{S})+\mc F_{s}(\tilde{U},\tilde{S}).
\end{split} \end{align}
with $\mc F_{u}=N_{u}+\mc F_{dis}+\mc E_{u}$, $\mc F_{s}=N_{s}+\mc E_{s}$. Here $\mc E_{u}$, $\mc E_{s }$, $N_{s}$, $N_{u}$, $\mc F_{dis}$ are defined in \eqref{nsequation1}, \eqref{nsequation2}, \eqref{Esdefinition}, \eqref{Eudefinition}. Now let $(\tilde{U}_{t},\tilde{S}_{t})$ be the solution of the truncated equation:
\begin{align} \begin{split} \label{navierstokesperturb1trun}
\p_s \tilde U_t &= \mc L_{u}(\tilde{U}_t,\tilde{S}_t)+\chi_2\mc F_{u}(\tilde{U},\tilde{S}),\\
\p_s \tilde S_t &= \mc L_{s}(\tilde{U}_t,\tilde{S}_t)+\chi_2\mc F_{s}(\tilde{U},\tilde{S}),
\end{split} \end{align}
with $\mc L = \chi_2 (\mc L_u^{e}, \mc L_s^{e}) - J(1-\chi_1)$ and the initial value
\begin{align}\label{initialdatatruncated}
\begin{cases}
\tilde{U}_{t,0}=\chi_2\tilde{U}_0^{'}+\sum_{i=1}^{N}a_i,
\varphi_{i,u}\\
\tilde{S}_{t,0}=\chi_2\tilde{S}_0^{'}+\sum_{i=1}^{N}a_i
\varphi_{i,s}.
\end{cases}
\end{align}
Here $(\varphi_{i,u},\varphi_{i,s})$ are a basis of the unstable space $V_{uns}$ as chosen in Lemma \ref{lemma:abstract_result}. From Lemma \ref{truncatedestimate}, since when $|y|\leq C_0$, $(\tilde{U}_{t,0},\tilde{S}_{t,0})=(\tilde{U}_{0},\tilde{S}_{0}),$ we have
$(\tilde{U}_{t},\tilde{S}_{t})=(\tilde{U},\tilde{S})$ for all time when $|y|\leq C_0$, as long as the solution exists.

%\begin{equation*}
%|\al\nabla\bar{S}|+|\nabla\bar{u}|+|\al\div\bar{U}|+|\nabla \bar{S}|\leq ,\text{for } \frac{6}{5}\leq |y|\leq C_0
%\end{equation*}

Before we state the main proposition of this section (Proposition \ref{prop:bootstrap}), we discuss the choice of $C_0$ and the weight $\phi$ that will be used in the energy estimates.

\begin{lemma}\label{C0choice}
We can choose $C_0$  sufficiently large such that the profile $(\bar{U},\bar{S})$ satisfies the following conditions for all $s\geq s_0$:  
\begin{equation}\label{profilecondition1}
|\al\hat{X}\bar{S}|, |\hat{X}\bar{U}|, |\al\nabla(\hat{X}\bar{S})|, |\nabla(\hat{X}\bar{U})|\leq\frac{1}{100C_1} \quad \text{ for } |y|\geq C_0,
\end{equation}
\begin{equation}\label{profilecondition4}
\|\nabla(\hat{X}\bar{U})\|_{L^{\infty}(|y|\geq C_0)}+\|\nabla(\hat{X}\bar{S})\|_{L^{\infty}(|y|\geq C_0)} \leq \frac{1}{C C_2}, \end{equation}
\begin{equation}\label{profilecondition2}
\|\nabla^{2}(\hat{X}\bar{U})\|_{L^{8}({|y|\geq C_0})}+\|\nabla^2(\hat{X}\bar{S})\|_{L^{8}({|y|\geq C_0})} \leq \frac{1}{C C_2},
\end{equation}
\begin{equation}\label{profilecondition3}
\sum_{j=3}^{5}(\|\hat{X}\bar{U}\|_{\dot{H}^{j}(|y|\geq C_0)}+\|\hat{X}\bar{S}\|_{\dot{H}^{j}(|y|\geq C_0)})\leq \frac{1}{C C_2},
\end{equation}
with $C_2 = 100$, $C_1$ determined via \begin{equation}\label{parameterchosen}
\frac{1}{C_2}=\frac{32}{r-1}\frac{1}{C_1}\left(\frac{1}{C_2}\right)^{\frac{1}{20}},
\end{equation}
and $C$ being some universal constant used in lemma \ref{boosstrapestimate2} which is independent of all parameters, r being the self-similar parameter as in \eqref{eq:SS_coordinates1}. Here the $L^{8}$ estimate \eqref{profilecondition2} will be used in the energy estimate of $\|\tilde{U}\|_{\dot{H}^{4}}+\|\tilde{S}\|_{\dot{H}^{4}}$ (Lemma \ref{boosstrapestimate2}) through a H\"older estimate.
\end{lemma}

\begin{proof}
Since $|\hat{X}|\leq 1$, $|\nabla\hat{X}|\leq 2e^{-s}$, \eqref{profilecondition1}, \eqref{profilecondition4} follows from the decay of $(\bar{S},\bar{U})$ \eqref{eq:profiles_decay}.

For \eqref{profilecondition2}, \eqref{profilecondition3}, the terms with no derivatives hitting on $\hat{X}$ also have the corresponding decay. Moreover, we have 
\[
\supp (\nabla\hat{X}) \subset B(0, e^s)\cap B^c(0,\frac{e^s}{2}), \  |\nabla^{j}\hat{X}|\lesssim e^{-js}.
\]
Then for $\nabla^{j}\hat{X}\nabla^{l}\tilde{S},$ with $j\geq 1$, $j+l\geq 2$, $q\geq 2$,  we get
\begin{align*}
\|\nabla^{j}\hat{X}\nabla^{l}\bar{S}\|_{L^{q}(\mathbb{R}^3)}^{q} &\lesssim e^{-jsq}\int_{\frac{1}{2}e^s}^{e^s}|y|^{-q(r-1)-ql+2}dy\\
&\lesssim e^{-jsq} |y|^{-q(r-1)-ql+3}\big|^{e^s}_{\frac{1}{2}e^s}\\
&\lesssim e^{-jsq} e^{(-q(r-1)-ql+3)s}\\
&\lesssim e^{-s(2q-3)}\\
&\lesssim e^{-s}.
\end{align*}
Then, since $s\geq s_0\gg 1$, those terms with at least one derivative on $\hat{X}$ are sufficiently small.
\end{proof}
We define the weight $\phi$ as follows: 
\begin{align}\label{weightdefi}
\phi(y)=\begin{cases}
1 & \text{ for } |y|\leq R_0,\\
\frac{|y|^{2(1-\eta)}}{2R_0^{2(1-\eta)}}, &\text{ for } |y|\geq 4R_0,
\end{cases}
\end{align}
with $R_0$ chosen such that  
\begin{equation}\label{Rochoice estimate}
\begin{split}
\bar{S}(y)\geq 2\delta_0,  &\text{ for all $|y|\leq R_0$,}\\
|\nabla \bar{S}(y)|,\ |\nabla \bar{U}(y)|\leq \frac{\delta_1}{2},\ \bar{S}(y)\leq C\delta_0, \bar{U}(y)\leq C\delta_0,&\text{ for all $|y|\geq R_0.$}
\end{split}
\end{equation}
The existence of $R_0$ follows from $\delta_0^{\frac{3}{2}}\ll \delta_1\ll\delta_0 \ll 1 $ and the decay of the profile \eqref{eq:profiles_decay}. We have:
\begin{align}\label{choiceR0}
\delta_0 \thicksim \frac{1}{R_0^{r-1}}.
\end{align}

Now we state our main proposition:
\begin{prop} \label{prop:bootstrap}
For the initial data chosen above, if the solution exists for $s\in [s_0,s_1]$ and the unstable mode satisfies $\|P_{\rm{uns}}(\tilde{U}_t,\tilde{S}_t)\|_{X} \leq \delta_1 e^{-\epsilon(s-s_0)}$, then 
\begin{align}\label{lowerestimate1}
|\tilde{U}|,|\tilde{S}| &\leq \delta_0 \frac{1}{C_2}e^{-\epsilon(s-s_0)}, \\ 
\label{lowerestimate2}
\|\tilde{U}\|_{\dot{H}^4(|y|\geq C_0)}, \|\tilde{S}\|_{\dot{H}^4(|y|\geq C_0)} &\leq \delta_0 e^{-\epsilon(s-s_0)}, \\
\label{higherestimate3}
E_{K}=\int_{\mathbb{R}^3}(|\nabla^{K}U|^2+|\nabla^{K}S|^2)\phi^{K}dy &\leq E.
\end{align}
where $\nabla^K S$ and $\nabla^K U_i$ are $K$-tensors containing all possible combinations of $K$ derivatives, and $\epsilon=\frac{12\delta_g}{25},$ $C_2=100$.

\end{prop}
\begin{proof}

This proposition consists of the lower bound decaying estimates \eqref{lowerestimate1}--\eqref{lowerestimate2} and the higher order estimate \eqref{higherestimate3}. The lower bound estimates \eqref{lowerestimate1}--\eqref{lowerestimate2} give the decay in the stability result. We note that $\delta_0 e^{-\eps (s-s_0)}$ is the typical size of the perturbation, which implies that any quadratic nonlinear terms are much smaller and enjoy a faster decay.  The higher bound estimate is needed to deal with the derivative loss in the nonlinear terms through a Duhamel formula.

The proof follows by a bootstrap argument: \eqref{lowerestimate1} and \eqref{lowerestimate2} will be shown in Subsection \ref{lowerestimate}. In particular, \eqref{lowerestimate1} will follow from the trajectory estimates developed in Lemmas \ref{linfitnityinside} and \ref{boosstrapestimate1}. Equation \eqref{lowerestimate2} will follow from energy estimates developed in \eqref{boosstrapestimate2}. In contrast to the radial case, \eqref{lowerestimate2} is needed to control the angular derivative term in the trajectory argument.

Equation \eqref{higherestimate3} will be shown in Subsection \ref{higherorderestimate}. The proof is similar to the non-radial dissipativity developed in the linear estimates, but we make use of the growing weight $\phi$ to control the behavior at infinity. Moreover, we need to develop additional decaying rate estimates on $S$ to control the dissipative term $\frac{\Delta U}{S^{1/\alpha}}$ as $|y| \to +\infty$.

\end{proof}

\subsection{Propagation of lower order estimates} \label{lowerestimate}
In this and the next subsection, we always assume \eqref{lowerestimate1}, \eqref{lowerestimate2}, \eqref{higherestimate3} are true and aim to show the better bounds: 

\[
LHS\leq \frac{1}{2}RHS .
\]
 In order to bound the last terms in $L^\infty$,
 We first show a few Lemmas regarding the decay of $S$ and $\nabla S$, $\nabla U$. We will use interpolation Lemmas \ref{lemma:GN_generaltorus}, \ref{lemma:GN_generalnoweighttorus} proved in the Appendix.

\begin{lemma}
We have 
\begin{equation}\label{perturbationenergy}
\tilde{E}_{K}:=\int_{e^s\mathbb{T}_{L}^{3}}(|\nabla^{K}\tilde{U}|^2+|\nabla^{K}\tilde{S}|^2)\phi^{K}dy\leq 2E.
\end{equation}
\end{lemma}
\begin{proof}
Since $(\tilde{U},\tilde{S})=(U,S)-(\hat{X}\bar{U},\hat{X}\bar{S}),$ and $K\ll E$, from \eqref{higherestimate3}, we only need to show
\begin{equation}\label{profilekestimate}
\int_{e^s\mathbb{T}_{L}^{3}}(|\nabla^{K}(\hat{X}\bar{U})|^2+|\nabla^{K}(\hat{X}\bar{S})|^2)\phi^{K}dy\lesssim_{K}  1.
\end{equation}
From the decay of the profile \eqref{eq:profiles_decay}, we have
\[
\left( |\nabla^{K} \bar{U} |^2+|\nabla^{K} \bar{S} |^2 \right) \phi^{K}\lesssim_{K} \langle R\rangle^{-2(r-1)-2K\eta}.
\]
From $\frac{1}{K}\ll \eta$ \eqref{choiceofparameter}, we claim 
\begin{equation}\label{profilekestimate1}
\int_{e^s\mathbb{T}_{L}^{3}} \left( |\nabla^{K}\bar{U}|^2+|\nabla^{K}\bar{S}|^2 \right) \phi^{K}dy\lesssim_{K}  1.
\end{equation}
From the definition of $\hat{X}$, we have
\begin{equation}\label{estimatehatX}
\supp \hat{X} \subset B \left( 0, e^s \right) ,\  \supp (1-\hat{X}) \subset B^c \left( 0,\frac{e^s}{2} \right), \ \supp \nabla\hat{X} \subset B \left( 0, e^s \right) \cap B^c \left( 0,\frac{e^s}{2} \right), 
\end{equation}
and
\begin{equation}\label{estimatehatX2}
\|\partial^{j}\hat{X}\|_{L^{\infty}}\lesssim e^{-js}.
\end{equation}
Then we have
\begin{align}\label{profilekestimate2}
&\quad\sum_{1\leq j\leq K}\int_{e^s\mathbb{T}_{L}^{3}}|\nabla^{K-j} \hat{X}|^2 \left( |\nabla^{j} \bar{U} |^2+|\nabla^{j} \bar{S} |^2 \right)\phi^{K}dy\\\nonumber
&\lesssim_{K} \sum_{1\leq j\leq K}e^{-2(K-j)s}\int_{\frac{1}{2}e^{s}}^{e^s}|y|^{2(1-\eta)K}|y|^{-2(r-1)-2j}|y|^{2}dy\\\nonumber
&\lesssim_{K} \sum_{1\leq j\leq K}e^{-2(K-j)s+(3+2(1-\eta)K-2(r-1)-2j)s}\\\nonumber
&\lesssim_{K} e^{(-2K\eta+3-2(r-1))s}\lesssim_{K}1.
\end{align}
The last inequality follows again from the choice of parameter $\frac{1}{K}\ll \eta$. \eqref{choiceofparameter}.
Then \eqref{profilekestimate} follows from \eqref{profilekestimate1} and \eqref{profilekestimate2}.
\end{proof}

\begin{lemma}[Bounds for $S$] \label{lemma:Sbounds} We have the lower bound   

\begin{equation*}
    S \gtrsim \delta_1 \left\langle \frac{R}{R_0} \right\rangle^{-(r-1)}.
\end{equation*}
On the region $|y| > R_0$, we have the upper bound
\begin{equation*}
S \les \delta_0 \max \left\{ \left\langle \frac{R}{R_0} \right\rangle^{-(r-1)}, e^{-(s-s_0)(r-1)} \right\}.
\end{equation*}
\end{lemma}
\begin{proof}
For $R \leq R_0$ we have $\bar S \geq 2 \delta_0$, so $S \geq 2 \delta_0 - \| \tilde S \|_{L^\infty} \geq\delta_0$ and the inequality is clear. Thus, let us assume $R \geq R_0$. We choose $(y_0, \bar{s})$ such that either $|y_0| = R_0$ or $\bar{s} = s_0$. We consider the trajectories starting at $(y_0, \bar{s})$.

Let us define $\omega_{y_0} (s) = e^{(r-1)(s-\bar{s})} S (y_0 e^{s-\bar{s}}, s)$. From \eqref{eq:US2}:
\begin{equation*}
\p_s \omega_{y_0} =  - e^{(r-1) (s-\bar{s})} U \cdot \nabla S - \alpha e^{(r-1) (s-\bar{s})} S \div (U).
\end{equation*}
Using the interpolation inequality \eqref{eq:GNresultinfty_simplifiedtorus} between $\| \phi^{K/2} \nabla^K \tilde S \|_{L^2}$ and $\| \tilde S \|_{L^\infty}$ yields $\| \phi^{1/2} \nabla \tilde S \|_{L^\infty}\leq \delta_0^{\frac{9}{10}}$. Moreover, for $R \geq R_0$
\begin{align*}
\phi^{1/2} | \nabla (\hat X \bar S) | \les  \left( \frac{R}{R_0} \right)^{1-\eta} \left( \nabla (\hat X ) \bar S + \hat X | \nabla S |\right)
\end{align*}
Now, note that $|\nabla (\hat X (|y| e^{-s}))| = -e^{-s} \hat X' (|y| e^{-s})$. Since $\hat X'$ is supported on $[1/2, 1]$, we see that $e^{s} \approx y$ on the support, and we obtain that $|\nabla \hat X| \les \frac{1}{R}$. Therefore, using \eqref{eq:profiles_decay} we get that for $R\geq R_0$:
\begin{align*}
\phi^{1/2} | \nabla (\hat X \bar S) | \les \frac{R}{R_0} \cdot \frac{1}{R^r} \leq \frac{1}{R_0} \approx \delta_0^{\frac{1}{r-1}} \leq \delta_0.
\end{align*}
Combining this with  $\| \phi^{1/2} \nabla \tilde S \|_{L^\infty}\leq \delta_0^{\frac{9}{10}}$, we obtain  $\| \phi^{1/2} \nabla S \|_{L^\infty}\leq \delta_0^{\frac{9}{10}}$. Analogously, it can be deduced that  $\| \phi^{1/2} \nabla U \|_{L^\infty (R\geq R_0)} \leq \delta_0^{9/10} $. Thus, we obtain
\begin{align*}
    &\left| e^{(r-1)(s-\bar{s})} U(y_0 e^{(s-\bar{s})}, s) \nabla S (y_0 e^{(s-\bar{s})}, s) \right| \les \eta^{-1} \left( \frac{ y_0 e^{(s-\bar{s})} }{R_0} \right)^{-1+\eta} e^{(r-1) (s-\bar{s})} \delta_0^{19/10} \\
    &\quad\les \eta^{-1} \delta_0^{19/10} \left\langle \frac{y_0}{R_0} \right\rangle^{-1+\eta} e^{(s-\bar{s})(\eta+r-2)}\les \delta_0^{9/5} \left\langle \frac{y_0}{R_0} \right\rangle^{1-r}e^{-(s-\bar{s})/10}.
\end{align*}
Similarly
\begin{equation*}
\left| \alpha \omega_{y_0} \div (U) \right| \les  \delta_0^{9/5} \left\langle \frac{y_0}{R_0} \right\rangle ^{1-r}e^{-(s-\bar{s})/10}.
\end{equation*}
Thus, we get
\begin{equation*}
    |\p_s \omega |\les  \delta_0^{9/5} \left\langle \frac{y_0}{R_0} \right\rangle^{1-r}e^{-(s-\bar{s})/10}.
\end{equation*}
Integrating we obtain 
\begin{equation} \label{eq:cadiz}
   \left|  \omega_{y_0}(s) - \omega_{y_0}(\bar{s}) \right| \les  \delta_0^{9/5} \left\langle \frac{y_0}{R_0} \right\rangle^{1-r},
\end{equation}
uniformly in $s$. Note also that $\delta_0\leq \omega_{y_0}(\bar{s}) \leq C \delta_0$  if $y_0 = R_0$, using that $ 2 \delta_0 \leq \bar S \leq C\delta_0 $. Thus, we get $\omega_{y_0}(s) \approx \delta_0$ if $y_0=R_0                                   $. That is,
\begin{equation} \label{eq:almeria}
S \left( y_0 e^{(s-\bar{s})} , s \right) \approx \delta_0 e^{-(r-1)(s-\bar{s})}.
\end{equation}
Undoing the change of variables, we note $\frac{y}{y_0} = e^{s-\bar{s}}$ and therefore
\begin{equation*}
S \left( y, s \right) \approx \delta_0 \left( \frac{y}{y_0} \right)^{-(r-1)}.
\end{equation*}
Since $y_0 = R_0$ this directly shows the lower bound and the upper bound.
For $s=s_0$, note that for $y\geq R_0$, from the definition of the initial data \eqref{initialdatacondition2}, we have
\[
\frac{\delta_1}{2} \left\langle \frac{y_0}{R_0} \right\rangle ^{1-r}\leq S(y_0,s_0) \leq C \delta_0.
\]
Thus, we get
\begin{align*}
 e^{(1-r)(s-s_0)}\delta_1 \left\langle\frac{y_0}{R_0} \right\rangle ^{1-r}\lesssim S(y,s)=e^{(1-r)(s-s_0)}w_{y_0}(s)\lesssim \delta_0e^{(1-r)(s-s_0)}.
\end{align*}
Then the upper bound directly follows. From $e^{(1-r)(s-s_0)}=(\frac{y}{y_0})^{1-r}$, we get the lower bound.

\end{proof}

\begin{lemma} \label{lemma:Sprimebounds} We have that
\begin{equation} \label{eq:voltaire1}
    |\nabla \mt S | + | \nabla \mt U | \les \delta_0 \left\langle \frac{R}{R_0} \right\rangle^{-r}
\end{equation}
and
\begin{equation} \label{eq:voltaire2}
    |\nabla S | + | \nabla  U | \les  \left\langle \frac{R}{R_0} \right\rangle^{-r}
\end{equation}
\end{lemma}
\begin{proof}
For $R\leq R_0$, \eqref{eq:voltaire1} is satisfied by the $L^{\infty}$ assumption \eqref{lowerestimate1}, the $H^{4}(|y|\geq C_0)$ assumption \eqref{lowerestimate2} and Lemma \ref{linfitnityinside} for the bound inside $B(0,C_0)$. Combining these with the decay of the profile (equation \eqref{eq:profiles_decay}), we get \eqref{eq:voltaire2}  given that $U =\hat{X} \bar U + \mt U$ and $S =\hat{X} \bar S + \mt S$.

For $R\geq R_0$, without loss of generality we will do the proof for $S$ and $\tilde{S}$. The proof for $U$ and $\tilde{U}$ will be analogous. 

Using again the energy estimate assumption \eqref{higherestimate3}, we get
\begin{equation} \label{eq:diderot}
    \int \left( | \nabla^K U |^2 + | \nabla^K S |^2 \right) \phi^K \les_K E.
\end{equation}
Now, we interpolate between $| U |, | S | \lesssim \delta_0$ (from the decay of the profile \eqref{Rochoice estimate} and the assumption of the perturbation \eqref{lowerestimate1} ) when $R\geq R_0$ and \eqref{eq:diderot}. Concretely, using equation \eqref{eq:GNresultinfty_simplified} for $m = K$, $j \in \{ 1, 2 \}$, $q = 2$, we obtain
\begin{equation*}
     | \nabla^j U | + | \nabla^j  S | \les_K \delta_0^{1 - \frac{j}{K-3/2}} E^{\frac{j}{K-3/2}} \phi^{- \frac{Kj}{2(K-3/2)} } + \delta_0 \langle x \rangle^{-j}.
\end{equation*}
Given that $j \in \{ 1, 2 \}$ and $\delta_0$ depends on $K, E$, we obtain that
\begin{equation} \label{eq:rough_decay2}
    | \nabla^2 U | + | \nabla^2 S | \les \delta_0^{9/10} \phi^{-1}, \quad \mbox{ and } \quad
    | \nabla U | + | \nabla S | \les \delta_0^{9/10} \phi^{-1/2}.
\end{equation}

We look at the equation for $\nabla S$. Taking a derivative in \eqref{eq:US2}, we have that
\begin{equation*}
    \p_s \p_i \mc  S = -r \p_i S - y \nabla \p_i S - \p_i (U \nabla S ) - \alpha \p_i ( S \div (U))
\end{equation*}
    For any $|y| \geq R_0$, we consider $y = y_0 e^{(s-\bar{s})}$ where either $\bar{s} = s_0$ or $|y_0| = R_0$. We define $\omega_{y_0}(s) = e^{r(s-\bar{s})} \p_i S( y_0 e^s, s )$, and we obtain the equation
    \begin{equation*}
       \p_s \omega = -e^{r(s-\bar{s})} \left( \p_i U \nabla S + U \nabla \p_i S + \alpha \p_i  S \div (U) + \alpha S \p_i \div (U) \right).
    \end{equation*}
    where everything is evaluated at $(y_0 e^s, s)$. Using the inequalities from \eqref{eq:rough_decay2}, we get
    \begin{equation*}
        | \p_s \omega_{y_0} | \les \delta_0^{9/5} \phi \left( y_0 e^{s-\bar{s}} \right)^{-1} e^{r(s-\bar{s})} \les \delta_0^{9/5}e^{(r-2+2\eta)(s-\bar{s})} \left\langle\frac{y_0}{R_0} \right\rangle^{-2+2\eta}.
    \end{equation*}
    Thus
    \begin{equation*} 
   \left|\partial_{i}S (y_0,s) - e^{-r(s-\bar{s})}\partial_{i}S (y_0,\bar{s})) \right|=e^{-r(s-\bar{s})}\left|  \omega_{y_0}(s) - \omega_{y_0}(\bar{s}) \right| \les \delta_0^{9/5}e^{(-2+2\eta)(s-\bar{s})} \left\langle \frac{y_0}{R_0} \right\rangle^{-r}.
\end{equation*}
    Therefore,  from $|\partial_{i}S(y_0,\bar{s})|\lesssim \delta_0$ when $|y_0|=R_0$ and the decay for the initial value $|\partial_{i}S(y_0,s_0)|\lesssim \delta_0 \left\langle\frac{|y_0|}{R_0} \right\rangle^{-r}$ (from the choice of $R_0$ \eqref{choiceR0} and the initial value \eqref{initialdatacondition}), we have when $R\geq R_0$, 
    \[
    |\nabla S|\lesssim \left\langle\frac{R}{R_0} \right\rangle^{-r}\delta_0. \]
    Then we get \eqref{eq:voltaire2}.
    
    For \eqref{eq:voltaire1}, we only need to show additionally the decay of the profile $|\nabla(\hat{X}\bar{S})|\lesssim \left\langle\frac{R}{R_0} \right\rangle^{-r}.$ Since $\nabla(\hat{X}\bar{S})=\nabla(\bar{S})\hat{X}+\nabla(\hat{X})\bar{S}$, the decay of the first term directly follows from the decay of $\bar{S}$. For the second term, we can combine the estimates $|\nabla \hat{X}|\lesssim e^{-s}$ and in $\supp \nabla{\hat{X}}, |y|\thicksim e^{-s}$ to get the decay.

\end{proof}
\begin{lemma} \label{lemma:brasilia} For every $1 \leq j \leq K-2$, $\epsilon^\sharp > 0$, we have that 
\begin{equation} \label{eq:brazil3}
    \frac{| \nabla^j S  | \phi^{j/2} }{S} + \frac{| \nabla^j U | \phi^{j/2}}{S} \les_{\delta_0, \epsilon^\sharp} \left\langle R \right\rangle^{-\frac{-r(j-1) + (1-\eta) \frac{5j}{2} + K\eta - 5/2 }{K-5/2} + \epsilon^\sharp}\lesssim_{\delta_0,\epsilon^\sharp} \left\langle R \right\rangle^{\epsilon^\sharp}.
\end{equation}
and
\begin{equation} \label{eq:brazil2}
 | \nabla^j U | + | \nabla^j  S | \les_{\delta_0,\epsilon^\sharp} \langle R \rangle^{-j(1-\eta) - (r-1)+\epsilon^\sharp}. 
\end{equation} 
For $j=K-1,$ we have
\begin{equation}\label{eq:brazilk-1}
\left\| \langle R \rangle^{K(1-\eta)\frac{K-2}{K-1}-\epsilon^\sharp}|\nabla^{K-1}S| \right\|_{L^{2+\frac{2}{K-2}}}+ \left\|\langle R \rangle^{K(1-\eta)\frac{K-2}{K-1}-\epsilon^\sharp}|\nabla^{K-1}U| \right\|_{L^{2+\frac{2}{K-2}}}\lesssim_{\delta_0,\epsilon^\sharp} 1.
\end{equation}
\end{lemma}
\begin{proof} 

In order to show \eqref{eq:brazil3}, we do interpolation between $\| \nabla S  \cdot \langle R\rangle ^{r} \|_{L^\infty}$ and $\| \phi^{K/2} \nabla^K S \|_{L^2}$. We consider Lemma \ref{lemma:GN_general} from the Appendix with $m=K-1$, $i = j-1$, $p=\infty$, $q = 2$, $\bar{r} = \infty$, so that $\theta = \frac{j-1}{K-5/2}$ (which is between $(j-1)/(K-1)$ and $1$ as Lemma \ref{lemma:GN_generaltorus} requires, since $j \leq K-2$). From \eqref{eq:GNresult}, using Lemma \ref{lemma:Sprimebounds} to bound $\| \nabla S  \langle R \rangle^{r} \|_{L^\infty}$, we obtain
\begin{equation*}
| \nabla^j U  | \les_{\delta_0} \phi^{-\frac{K\theta}{2}} \langle R \rangle^{-r(1-\theta ) +\epsilon^\sharp }+\langle R \rangle^{-j+1-r}
 \les_{\delta_0} \phi^{-\frac{K\theta}{2}} \langle R \rangle^{-r(1-\theta ) +\epsilon^\sharp }.
\end{equation*}
The last inequality follows from $1\ll K\eta.$
Thus, we have that
\begin{align}\label{decayestimatederi}
| \nabla^j U | &\les_{\delta_0} \phi^{-j/2} \langle R \rangle^{-(r-1)+\epsilon^\sharp}
\phi^{-\frac{K\theta}{2} + j/2} \cdot \langle R \rangle^{-r(1-\theta)+r-1}\\
&
=
\phi^{-j/2} \langle R \rangle^{-(r-1)+\epsilon^\sharp}
\phi^{\frac{-K(j-1) + j(K-5/2)}{2(K-5/2)}} \cdot \langle R \rangle^{r\frac{j-1}{K-5/2}-1}\\ \nonumber
&= \phi^{-j/2} \langle R \rangle^{-(r-1)+\epsilon^\sharp}
\phi^{\frac{K - \frac52 j}{2(K-5/2)}} \cdot \langle R \rangle^{\frac{rj-r - K + 5/2}{K-5/2}} \\ \nonumber
&\les
\phi^{-j/2} \langle R \rangle^{-(r-1)+\epsilon^\sharp}
 \cdot \langle R \rangle^{\frac{rj-r - K + 5/2 + (1-\eta) (K - \frac52 j )}{K-5/2}}.
\end{align}
This shows the first bound from \eqref{eq:brazil3} corresponding to $U$ and the one for $S$ is analogous. 
For the second bound in \eqref{eq:brazil3}, and \eqref{eq:brazil2}, we can use \eqref{decayestimatederi} and 
\begin{align*}
&rj-r - K + 5/2 + (1-\eta) (K - \frac52 j )= (j-1)(r-\frac{5}{2})+\eta(j\frac{5}{2}-K) \\
 &\quad  \left( \frac{5}{2}-r - \eta K   \right) + j\left( - \frac{5}{2} (1-\eta) + r \right) \leq 0,
\end{align*}
from the fact that $K\eta \gg 1$ and $r < \sqrt{3}$ \eqref{eq:rough_range_r}.

For \eqref{eq:brazilk-1}, we use the interpolation Lemma \ref{eq:GNresulttorus} between $\| \nabla S  \cdot \langle R\rangle ^{r} \|_{L^\infty}$ and $\| \phi^{K/2} \nabla^K S \|_{L^2}$ again. By taking $m=k-1$, $i=K-2$, $p=\infty$, $q=2$, $\theta=\frac{K-2}{K-1}$, we have
\begin{align*}
\left\| \langle R\rangle^{K(1-\eta)\frac{K-2}{K-1}+r\frac{1}{K-1}-\epsilon^\sharp}\nabla^{k-1}S \right\|_{L^{2+\frac{2}{K-2}}}\lesssim_{\delta_0,\epsilon^\sharp}1.
\end{align*}
Here \eqref{eq:GN_extracondtorus} is satisfied because $1\ll K\eta$.
Thus \eqref{eq:brazilk-1} follows from $r>1$.

\end{proof}

\begin{lemma}\label{forcingestimate1}
We have
\[
\chi_2 \mathcal E_{u}=0, \quad \chi_2 \mathcal E_{s}=0,
\]
\[
\|\mathcal E_{u}\|_{L^{\infty}}, \|\mathcal E_{u}\|_{\dot{H}^4(|y|\geq C_0)}\leq \delta_1 e^{-(s-s_0)}, 
\]
and
\[
\|\mathcal E_{s}\|_{L^{\infty}},  \|\mathcal E_{s}\|_{\dot{H}^4(|y|\geq C_0)}\leq \delta_1 e^{-(s-s_0)}.
\]
\end{lemma}
\begin{proof}
From the decay of the profile \eqref{eq:profiles_decay}, we have
\[
|\partial^{j}\bar{U}|, |\partial^{j}\bar{S}|\lesssim |y|^{1-r-j}.
\]
Then using the estimates for $\hat{X}$ \eqref{estimatehatX}, \eqref{estimatehatX2}, we have 
\[
\|(\hat{X}^2-\hat{X})\bar{U}\cdot\nabla\bar{U}\|_{L^{\infty}}\lesssim \||\hat{X}^2-\hat{X}||y|^{1-2r}\|_{L^{\infty}}\lesssim e^{(1-2r)s}.
\]
We also have 
\[
\|\hat{X}\bar{U}\cdot\nabla (\hat{X})\bar{U}\|_{L^{\infty}}\lesssim \||\hat{X}||\nabla (\hat{X})||y|^{2-2r}\|_{L^{\infty}}\lesssim e^{(1-2r)s}.
\]
Other terms in $\mathcal E_{u}$ and $\mathcal E_{s}$ (cf. \eqref{Eudefinition}) can be controlled in a similar way and we have
\[
\|\mathcal E_{u}\|_{L^{\infty}}+\|\mathcal E_{s}\|_{L^{\infty}}\lesssim e^{(1-2r)s}\leq \delta_1 e^{-(s-s_0)},
\]
where we used $r>1$ and $s_0\gg 1$.

By taking $4$ derivatives, we have 
\[
|\partial^4((\hat{X}^2-\hat{X})\bar{U}\cdot\nabla\bar{U})|\lesssim 
\sum_{i+j=4}\partial^{i} (\hat{X}^2-\hat{X})\partial^{j}(\bar{U}\cdot\nabla\bar{U})\lesssim e^{-is}|y|^{1-2r-j}.
\]
Then
\begin{align*}
\|\partial^4((\hat{X}^2-\hat{X})\bar{U}\cdot\nabla\bar{U})\|_{L^{2}(\mathbb{R}^3)}^2&\lesssim \sum_{i+j=4}e^{-2is}\int_{\frac{1}{2}e^s}^{e^s}|y|^{2(1-2r-j)}|y|^2dy\\
\qquad&\lesssim \sum_{i+j=4}e^{(-2i+5-4r-2j)s}\\
\qquad&\lesssim e^{(-4r-3)s}\lesssim e^{-2s}.
\end{align*}
Similarly, we have
\begin{align*}
&\|\partial^4(\hat{X}\bar{U}\cdot\nabla (\hat{X})\bar{U})\|_{L^{2}(\mathbb{R}^3)}^2\lesssim e^{-2s}.
%&\lesssim \sum_{i+j+k+l=4}e^{-2(i+j+1)s}\int_{\frac{1}{2}e^s}^{e^s}|y|^{2(2-2r-k-l)}|y|^2dy\\
%&\lesssim \sum_{i+j+k+l=4}e^{(-2(i+j+1)+2(2-2r-k-l)+3)s}\\
%&\lesssim e^{(-4r-3)s}\lesssim e^{-2s}.
\end{align*}
The rest of the terms can be bounded in the same way and we have
\[
\|\mathcal E_{u}\|_{\dot{H}^4(|y|\geq C_0)}+\|\mathcal E_{s}\|_{\dot{H}^4(|y|\geq C_0)}\lesssim e^{-s}\leq \delta_1 e^{-(s-s_0)}.
\]
\end{proof}
\begin{lemma}\label{forcingestimate2}
We have
\[
\|N_u\|_{L^{\infty}},\|\chi_2N_u\|_{X},\|N_s\|_{L^{\infty}},\|\chi_2N_s\|_{X}\leq \delta_1^{ \frac{6}{5}} e^{-\frac{3}{2}\epsilon(s-s_0)},
\]
and
\[
\|\mc F_{dis}\|_{L^{\infty}},\|\chi_2 \mc F_{dis}\|_{X} \leq \delta_1^{ \frac{6}{5}} e^{-\delta_{dis}\frac{s}{2}}.
\]
\end{lemma}
\begin{proof}
The proof is similar as \cite[Lemma 8.11]{Buckmaster-CaoLabora-GomezSerrano:implosion-compressible}. From the interpolation inequality  \eqref{eq:GNresultinfty_simplifiedtorus} in the Appendix, we have
for all $0\leq j\leq m$:
\[
\|\nabla^{j}U\|_{L^{\infty}}, \|\nabla^{j}S\|_{L^{\infty}}\lesssim_{\delta_1,E} 1.
\]
Moreover, when $|y|\leq 3C_0$, $S\geq \bar{S}-\delta_0 \gtrsim 1$.

Thus we have 
\begin{equation}
e^{-\delta_{dis}s}\left|\frac{\Delta U}{S^{\frac{1}{\alpha}}}\right|\lesssim_{\delta_1, E} e^{-\delta_{dis}s}\leq (\delta_1)^{\frac{6}{5}} e^{-\delta_{dis}\frac{s}{2}}.
\end{equation}
and
\begin{align}
e^{-\delta_{dis}s}\left\|\chi_2\left(\frac{\Delta U}{S^{\frac{1}{\alpha}}}\right)\right\|_{H^{m}}&\lesssim e^{-\delta_{dis}s}\left(\left\|\nabla^{m}\left(\frac{\Delta U}{S^{\frac{1}{\alpha}}}\right)\right\|_{L^{\infty}(B(0,3C_0))}+\left\|\left(\frac{\Delta U}{S^{\frac{1}{\alpha}}}\right)\right\|_{L^{\infty}(B(0,3C_0))}\right)\lesssim_{\delta_1, E} e^{-\delta_{dis}s} \nonumber \\
&\leq (\delta_1)^{\frac{6}{5}} e^{-\delta_{dis}\frac{s}{2}} .
\end{align}
For $N_u$ $N_s$, from inequality \eqref{eq:GNresultinfty_simplifiedtorus}:
\begin{align*}
\|\nabla^{j}(\tilde{U}\cdot(\nabla \tilde{U}))\|_{L^{\infty}}&\lesssim_{m} \sum_{l\leq j\leq m}\| \nabla^{l}\tilde{U} \|_{L^{\infty}}\|\nabla^{j-l+1}\tilde{U}\|_{L^{\infty}}\\
\qquad&\lesssim_{m} E^{\frac{1}{20}}
 \delta_0
^{\frac{39}{20}}e^{-\frac{39}{20} \epsilon(s-s_0)}\\\qquad&\lesssim\delta_0
^{\frac{19}{10}} e^{-\frac{39}{20} \epsilon(s-s_0)} .\end{align*}

Here we also use the bootstrap estimates \eqref{lowerestimate2}, \eqref{higherestimate3}. A similar estimate holds for $\tilde{U}\cdot (\nabla\tilde{S}),$ $\tilde{S} (\nabla\tilde{S})$ and $\tilde{S}\, \div \tilde{U}$:
\begin{align}
&\quad\|\nabla^{j}(\tilde{U}\cdot(\nabla \tilde{S}))\|_{L^{\infty}}+\|\nabla^{j}(\tilde{S}\cdot(\nabla \tilde{S}))\|_{L^{\infty}}+\|\nabla^{j}(\tilde{S}\cdot(\div\tilde{U}))\|_{L^{\infty}}\\\nonumber
&\lesssim_{m} \sum_{l\leq j\leq m}(\|\nabla^{j}\tilde{S}\|_{L^{\infty}}\|\nabla^{j-l+1}\tilde{U}\|_{L^{\infty}}+\|\nabla^{j}\tilde{S}\|_{L^{\infty}}\|\nabla^{j-l+1}\tilde{S}\|_{L^{\infty}})+\|\nabla^{j}\tilde{U}\|_{L^{\infty}}\|\nabla^{j-l+1}\tilde{S}\|_{L^{\infty}})\\\nonumber
\qquad
&\lesssim_{m} E^{\frac{1}{20}}  \delta_0
^{\frac{39}{20}}e^{-\frac{39}{20}\epsilon(s-s_0)} \\\nonumber
\qquad&\lesssim\delta_0
^{\frac{19}{10}}e^{-\frac{39}{20}\epsilon(s-s_0)} .
\end{align}

We also recall that $\delta_{0}^{\frac{3}{2}}\ll \delta_1 \ll \delta_0$ from \eqref{choiceofparameter}, leading to
\[
\|N_u\|_{L^{\infty}}, \|N_s\|_{L^{\infty}}\lesssim \delta_0
^{\frac{19}{10}} e^{-\frac{39}{20}\epsilon(s-s_0)} \leq \delta_1^{\frac{6}{5}} e^{-\frac{3}{2}\epsilon(s-s_0)}.
\]
Moreover, since $\supp \chi_2\in B(0,3C_0),$ we have
\[
\|\chi_2N_u\|_{H^{m}}\lesssim \sum_{j\leq m}\|\nabla^{j}(N_u)\|_{L^{\infty}}\lesssim\delta_0
^{\frac{19}{10}} e^{-\frac{39}{20}\epsilon(s-s_0)} \leq \delta_1^{\frac{6}{5}} e^{-\frac{3}{2}\epsilon(s-s_0)}.
\]
$\|\chi_2N_s\|_{H^{m}}$ can be bounded in the similar way.
\end{proof}
\begin{lemma}\label{linfitnityinside}
For $j\leq 4$, we have
\begin{equation*}
\|\nabla^{j}\tilde{U}\|_{L^{\infty}(B(0,C_0))}+\|\nabla^{j}\tilde{S}\|_{L^{\infty}(B(0,C_0))}\lesssim \frac{\delta_1}{\delta_g}e^{-\frac{\delta_g}{2}(s-s_0)}.
\end{equation*}
\end{lemma}
\begin{proof}
The proof is similar to \cite[Lemma 8.12]{Buckmaster-CaoLabora-GomezSerrano:implosion-compressible}. Since the truncated solution and the extended solution agree in $B(0,C_0).$ Then 
\[
\|\nabla^{j}\tilde{U}\|_{L^{\infty}(B(0,C_0))}=\|\nabla^{j}\tilde{U}_t\|_{L^{\infty}(B(0,C_0))}\lesssim \|\tilde{U}_t\|_{X}.
\]
Recall from \eqref{navierstokesperturb1trun}, we have
\[
\partial_{s}(\tilde{U}_t,\tilde{S}_t)=  \mathcal L (\tilde{U}_t,\tilde{S}_t)+\chi_2F(\tilde{U},\tilde{S}),
\]
 with $\mathcal L  = \chi_2 (\mathcal L_u^e, \mathcal L_s^e)$, $\chi_2F_{u}=\chi_2\mc F_{dis}+\chi_2N_{u}$, $
\chi_2F_{s}=\chi_2N_s.$ 
Using Lemma \ref{lemma:abstract_result}, we could project the equation onto the stable modes and have 
\begin{align*}
\partial_{s}P_{\rm{sta}}(\tilde{U}_t,\tilde{S}_t)=\mathcal{L}P_{\rm{sta}}(\tilde{U}_t,\tilde{S}_t)+P_{\rm{sta}}(\chi_2F(\tilde{U},\tilde{S})).
\end{align*}
Then from Duhamel's formula, we have 
\begin{align*}
P_{\rm{sta}}(\tilde{U}_t,\tilde{S}_t)(s)=T(s-s_0)P_{\rm{sta}}(\tilde{U}_t,\tilde{S}_t)(s_0)+\int_{s_0}^{s}T(s-\bar{s})(P_{\rm{sta}}(\chi_2F(\tilde{U},\tilde{S})))(\bar{s})d\bar{s},
\end{align*}
with $T$ the semigroup generated by $\mathcal{L}.$

From the fact that $\mathcal{L}$ generates a contraction semigroup in $X$, we have
\begin{align*}
&\quad\|P_{\rm{sta}}(\tilde{U}_t,\tilde{S}_t)\|_{X}\leq \delta_1 e^{-(s-s_0)\frac{\delta_g}{2}}+\int_{s_0}^{s}2\delta_1 e^{-(s-\bar{s})\frac{\delta_g}{2}}e^{-\frac{3}{2}\epsilon(\bar{s}-s_0)}d\bar{s}\lesssim \frac{\delta_1}{\delta_g}e^{-\frac{\delta_g}{2}(s-s_0)}.
\end{align*}
The unstable part of the solution is controlled directly from the assumption in Proposition \ref{prop:bootstrap}.
\end{proof}
Now we show the bootstrap when $|y|>C_0$.
\begin{lemma}\label{boosstrapestimate1}
Under the bootstrap assumptions \eqref{lowerestimate1}, \eqref{lowerestimate2}, \eqref{higherestimate3}  we have, for $s\geq s_0$, $|y|\geq C_0$: \begin{equation*}
|\tilde{U}|,|\tilde{S}| \leq {\delta_0}\frac{1}{2}\frac{1}{C_2}e^{-\epsilon(s-s_0)}.
\end{equation*}
\end{lemma}
\begin{proof}
From Gagliardo-Nirenberg's interpolation inequality \eqref{eq:GNresultnoweighttorus} in the Appendix, we have
\[
|\nabla\tilde{U}|, |\nabla \tilde{S}|\leq {\delta_0}e^{-\epsilon(s-s_0)}\left(\frac{1}{C_2}\right)^{\frac{1}{20}}.
\]

Then for $|y|\geq C_0$, from \eqref{profilecondition1}, since $|\hat{X}|\leq 1$, $|\nabla\hat{X}|\leq 10$, we have
\[
|\hat{X}\bar{U}\nabla\tilde{U}|+|\al\hat{X}\bar{S}\nabla\tilde{S}|+|\tilde{U}\cdot\nabla(\hat{X}\bar{U})|+|\al\tilde{S}\nabla(\hat{X}\bar{S})|\leq {\delta_0}e^{-\epsilon(s-s_0)}\frac{1}{C_1}\left(\frac{1}{C_2}\right)^{\frac{1}{20}},
\]
\[
|\hat{X}\bar{U}\nabla\tilde{S}|+|\al\hat{X}\bar{S}\div\tilde{U}|+|\div(\hat{X}\bar{U})\tilde{S}|+|\al\tilde{U}\nabla(\hat{X}\bar{S})|\leq {\delta_0}e^{-\epsilon(s-s_0)}\frac{1}{C_1}\left(\frac{1}{C_2}\right)^{\frac{1}{20}}.
\]
Moreover, by \eqref{nsequation1}, we have 

\begin{align*}
&\partial_{s}\tilde{U}=-(r-1)\tilde{U}-y\cdot{\nabla\tilde{U}}-((\hat{X}\bar{U})\cdot{\nabla\tilde{U}}+\al(\hat{X}\bar{S})\cdot{\nabla\tilde{S}}+\tilde{U}\cdot\nabla(\hat{X}\bar{U}) +\alpha \tilde{S} \nabla (\hat{X}\bar{S}))+\mathcal N_{u} +\mathcal E_u
+ \mc F_{dis}.
\end{align*}
\begin{align*}
&\partial_{s}\tilde{S}=-(r-1)\tilde{S}-y\cdot{\nabla\tilde{S}}-((\hat{X}\bar{U})\cdot{\nabla\tilde{S}}+\al(\hat{X}\bar{S})\div(\tilde{U})+\tilde{U}\cdot\nabla(\hat{X}\bar{S})+\alpha \tilde{S} \div(\hat{X}\bar{S}))+\mathcal N_{s}+\mathcal E_s.
\end{align*}
Then for $1\leq j\leq 3$, $s_0\leq \bar{s} \leq s$, using Lemmas \ref{forcingestimate1} and \ref{forcingestimate2}, we have 
\begin{equation*}
|\partial_{s}\tilde{U}_j(e^{s-\bar{s}}y_0)+(r-1)\tilde{U}_j(e^{s-\bar{s}}y_0)|\leq{\delta_0}e^{-\epsilon(s-s_0)}\frac{1}{C_1} \left( \frac{1}{C_2} \right)^{\frac{1}{20}}+\delta_1^{\frac{6}{5}}e^{-\frac{3}{2}\epsilon(s-s_0)}+\delta_1e^{-(s-s_0)}+ \delta_1^{ \frac{6}{5}} e^{-\delta_{dis}\frac{s}{2}},
\end{equation*}

and\begin{equation*}
|\partial_{s}\tilde{S}(e^{s-\bar{s}}y_0)+(r-1)\tilde{S}(e^{s-\bar{s}}y_0)|\leq{\delta_0}e^{-\epsilon(s-s_0)}\frac{1}{C_1} \left( \frac{1}{C_2} \right)^{\frac{1}{20}}+\delta_1^{\frac{6}{5}}e^{-\frac{3}{2}\epsilon(s-s_0)}+\delta_1e^{-(s-s_0)}.
\end{equation*}
From the choice of parameters \eqref{choiceofparameter}, we have

\begin{align*}
&\quad|\partial_{s}(\tilde{U}_j(e^{s-\bar{s}}y_0)e^{\epsilon(s-s_0)})+(r-1-\epsilon)\tilde{U}_j(e^{s-\bar{s}}y_0)e^{\epsilon(s-s_0)}| \\
&  = |e^{\epsilon(s-s_0)}(\p_s\tilde{U}_j(e^{s-\bar{s}}y_0)+(r-1)\tilde{U}_j(e^{s-\bar{s}}y_0)e^{\epsilon(s-s_0)})| \\
&
 \leq \delta_0 \frac{1}{C_1} \left( \frac{1}{C_2} \right)^{\frac{1}{20}}+\delta_1^{\frac{6}{5}}e^{-\frac{1}{2}\epsilon(s-s_0)}+\delta_1e^{-(1+\epsilon)(s-s_0)}    \\
&  \leq 2{\delta_0}\frac{1}{C_1} \left( \frac{1}{C_2} \right)^{\frac{1}{20}}. 
\end{align*}

Similarly, we have
\begin{equation*}
|\partial_{s}(\tilde{S}(e^{s-\bar{s}}y_0)e^{\epsilon(s-s_0)})+(r-1-\epsilon)(\tilde{S}(e^{s-\bar{s}}y_0)e^{\epsilon(s-s_0)})|\leq 2{\delta_0}\frac{1}{C_1} \left( \frac{1}{C_2} \right)^{\frac{1}{20}}.
\end{equation*}
Then we claim $(\tilde{U}(e^{s-\bar{s}}y_0)e^{\epsilon(s-s_0)}, \tilde{S}(e^{s-\bar{s}}y_0)e^{\epsilon(s-s_0)})$
never cross the domain 
\[
\left\{(x_1,x_2)\left||x_1|, |x_2|\leq \frac{8}{r-1}\delta_0\frac{1}{C_1} \left( \frac{1}{C_2} \right)^{\frac{1}{20}}\right\}\right.
\]
from inside. 

From initial data conditions \eqref{initialdatacondition} and Lemma \ref{linfitnityinside}, for $s=\bar{s}=s_0, |y|\geq C_0,$ $y\in e^{s_0}\T_{L}^3$ or $|y|=C_0$, $s=\bar{s}$, $y\in e^{\bar{s}}\mathbb{T}_{L}^3$, we have  \[|\tilde{U}_j(y_0)|e^{\epsilon(\bar{s}-s_0)},|\tilde{S}(y_0)|e^{\epsilon(\bar{s}-s_0)}\leq \frac{8}{r-1}\delta_0\frac{1}{C_1}\left( \frac{1}{C_2} \right)^{\frac{1}{20}}.
\]
Moreover, for any $|y|\geq C_0$, $s\geq s_0$, $y\in e^s\mathbb{T}_{L}^3$, there exists a trajectory $(y_0e^{s-\bar{s}},s)$ passing through $(y,s)$ that starts either from $s=\bar{s}=s_0, |y|\geq C_0,$ $y\in e^{s_0}\mathbb{T}_{L}^3$ or $|y|=C_0$, $s=\bar{s}$, $y\in e^{\bar{s}}\mathbb{T}_{L}^3$.

%\[|\tilde{U}_j(e^{s-\bar{s}}y_0)e^{\epsilon(s-s_0)}|,|\tilde{S}(e^{s-\bar{s}}y_0)e^{\epsilon(s-s_0)}|\leq\frac{4}{(r-1-\epsilon)-2\delta}{\delta_0}\frac{1}{C_1}(\frac{1}{C_2})^{\frac{1}{20}}\leq \frac{8}{r-1}\delta_0\frac{1}{C_1}(\frac{1}{C_2})^{\frac{1}{20}}.\]
Then the result follows from the chosen constant parameters \eqref{parameterchosen}: 
\[
\frac{32}{r-1}\frac{1}{C_1} \left( \frac{1}{C_2} \right)^{\frac{1}{20}}= \frac{1}{C_2}.
\]

\end{proof}
Now we control $\|\tilde{U}\|_{\dot{H}^4(|y|\geq C_0, y\in e^s\mathbb{T}_{L}^3}), $ and $\|\tilde{S}\|_{\dot{H}^4(|y|\geq C_0, y\in e^s\mathbb{T}_{L}^3)}$. 
\begin{lemma} \label{boosstrapestimate2}
Under bootstrap assumptions \eqref{lowerestimate1}, \eqref{lowerestimate2}, \eqref{higherestimate3}, we have \begin{equation*}
\|\tilde{U}\|_{\dot{H}^4(|y|\geq C_0, y\in e^s\mathbb{T}_L^3)}, \|\tilde{S}\|_{\dot{H}^4(|y|\geq C_0, y\in e^s\mathbb{T}_L^3)} \leq  \frac{\delta_0}{2C_2}e^{-\epsilon(s-s_0)}.
\end{equation*}
\end{lemma}
\begin{proof} For any $4$-th order derivative $\partial_{\beta}=\partial_{\beta_1,\beta_2,\beta_3,\beta_4}$,
we have
\begin{equation*}
\frac{d\partial_{\beta} \tilde{U}}{ds}=B_{u,0}(\tilde{U},\tilde{S})+B_{u,1}(\tilde{U},\tilde{S})+B_{u,2}(\tilde{U},\tilde{S})+B_{u,3}(\tilde{U},\tilde{S})+\partial_{\beta}(N_u)+\partial_{\beta}(\mc E_u)+\partial_{\beta}\mc F_{dis},
\end{equation*}
\begin{equation*}
\frac{d\partial_{\beta} \tilde{S}}{ds}=B_{s,0}(\tilde{U},\tilde{S})+B_{s,1}(\tilde{U},\tilde{S})+B_{s,2}(\tilde{U},\tilde{S})+B_{s,3}(\tilde{U},\tilde{S})+\partial_{\beta}(N_s)+\partial_{\beta}(\mc E_s).
\end{equation*}
Here $B_{u,0}$, $B_{s,0}$ contain linear terms with $5$ derivatives hitting on $\tilde{U}$, $\tilde{S}$. $B_{u,1}$, $B_{s,1}$ contain linear terms with $4$ derivatives hitting on $\tilde{U}$, $\tilde{S}$. $B_{u,2}$, $B_{s,2}$ contain linear terms with $3$ derivatives, and $B_{u,3}$, $B_{s,3}$ contains linear terms with less than or equal to 2 derivatives hitting on $\tilde{U}$, $\tilde{S}$.

 For the sake of simplicity in the following estimate we use $\int_{|y|\geq C_0}$ to denote the integration on $\{|y|\geq C_0\}\cap e^s\mathbb{T}_{L}^{3}.$ 
 since $y\cdot \vec{n}=e^s L$ on $\partial( e^s \mathbb{T}_{L}^3)$, we have
\begin{align}\label{highenergyestimate01}
I&=\frac{d}{ds} \left( \|\tilde{U}\|_{\dot{H}^4(|y|\geq C_0)}^{2}+\|\tilde{S}\|_{\dot{H}^4(|y|\geq C_0)}^{2} \right)\\\nonumber
&=\frac{d}{ds}\sum_{\beta\in [3]^{4}}\int_{|y|\geq C_0}|\p_{\beta}\tilde{U}|^2+(\p_{\beta}\tilde{S})^2dy\\\nonumber
&=\sum_{\beta\in [3]^{4}}\int_{|y|\geq C_0}\frac{d}{ds}(|\p_{\beta}\tilde{U}|^2+(\p_{\beta}\tilde{S})^2)dy+\sum_{\beta\in [3]^{4}}e^sL\int_{y\in \partial(e^s\mathbb{T}_L^3)}\sum_{j}|\partial_{\beta}\tilde{U}_j|^2+(\partial_{\beta}\tilde{S}
)^2d\sigma\\\nonumber
&=2\sum_{j=0}^{3} \left( \int_{|y|\geq C_0}\partial_{\beta}\tilde{U}\cdot B_{u,j}dy+\int_{|y|\geq C_0} \partial_{\beta}\tilde{S}B_{s,j}dy \right)\\\nonumber
&\quad+2 \left( \int_{|y|\geq C_0}\partial_{\beta}\tilde{U}\cdot \partial_{\beta}N_udy+\int_{|y|\geq C_0}\partial_{\beta}\tilde{S}\cdot \partial_{\beta}N_sdy \right) \\\nonumber
&\quad+2 \left( \int_{|y|\geq C_0}\partial_{\beta}\tilde{U}\cdot \partial_{\beta}\mc E_udy+\int_{|y|\geq C_0}\partial_{\beta}\tilde{S}\cdot \partial_{\beta}\mc E_sdy \right)\\\nonumber
&\quad+2 \left( \int_{|y|\geq C_0}\partial_{\beta}\tilde{U}\cdot \partial_{\beta}\mc F_{dis} dy \right)\\\nonumber
&\quad+\text{boundary terms}.
\end{align}
Now we start to bound each term separately. Recall that $(\tilde{U},\tilde{S})$ satisfy \eqref{nsequation1}, \eqref{nsequation2}, for $B_{u,0}$ and $B_{s,0}$, we have 
\begin{align*}
I_{0,\beta}&=\quad2\int_{|y|\geq C_0}\partial_{\beta}\tilde{U}\cdot B_{u,0}dy+2\int_{|y|\geq C_0} \partial_{\beta}\tilde{S}B_{s,0}dy\\
&=-2\int_{|y|\geq C_0}\partial_{\beta}\tilde{U}\cdot (y+\hat{X}\bar{U})\cdot \nabla \partial_{\beta}\tilde{U}dy-2\int_{|y|\geq C_0}\partial_{\beta}\tilde{U}\cdot  \nabla \partial_{\beta}\tilde{S}\al \hat{X}\bar{S}dy\\
&\quad-2\int_{|y|\geq C_0}\partial_{\beta}\tilde{S}\cdot (y+\hat{X}\bar{U})\cdot \nabla \partial_{\beta}\tilde{S}dy-2\int_{|y|\geq C_0}\partial_{\beta}\tilde{S}\div (\partial_{\beta}\tilde{U})\al \hat{X}\bar{S}dy\\
&\leq \int_{|y|\geq C_0}  \left( 3+|\hat{X}\bar{U}'(R)|+\frac{2\hat{X}|\bar{U}|}{R} \right) 
 |\partial_{\beta}\tilde{U}|^2dy+\int_{|y|\geq C_0}\left( 3+|\hat{X}\bar{U}'(R)|+\frac{2\hat{X}|\bar{U}|}{R} \right) |\partial_{\beta}\tilde{S}|^2dy\\
&\quad+2\int_{|y|\geq C_0}|\nabla(\al\hat{X}\bar{S})||\partial_{\beta}\tilde{S}||\partial_{\beta}\tilde{U}|dy\\
&\quad + \int_{|y|=C_0}(C_0+|\hat{X}\bar{U}|) |\partial_{\beta}\tilde{U}|^2d\sigma+\int_{|y|=C_0} (C_0+|\hat{X}\bar{U}|)|\partial_{\beta}\tilde{S}|^2d\sigma+\int_{|y|=C_0}\al |\hat{X}\bar{S}||\partial_{\beta}\tilde{U}||\partial_{\beta}\tilde{S}|d\sigma\\
&\quad-\int_{\partial (e^{s}\mathbb{T}_{L}^{3})} \left( (\partial_{\beta}\tilde{S})^2+\sum_{j}(\partial_{\beta}\tilde{U}_j)^2 \right) (y \cdot \vec{n})d\sigma.
\end{align*}
Using Lemma \ref{linfitnityinside}, we have
\begin{align*}
&\quad\int_{|y|=C_0}(C_0+|\hat{X}\bar{U}|)|\partial_{\beta}\tilde{U}|^2d\sigma+\int_{|y|=C_0}(C_0+|\hat{X}\bar{U}|)|\partial_{\beta}\tilde{S}|^2d\sigma+\int_{|y|=C_0}\al|\hat{X}\bar{S}||\partial_{\beta}\tilde{U}||\partial_{\beta}\tilde{S}|d\sigma\\
&\lesssim C_0^2(C_0+|\hat{X}\bar{U}|+\al |\hat{X}\bar{S}|) \left( \frac{\delta_1}{\delta_g} \right)^{2}e^{-\delta_g(s-s_0)}\\
&\leq \left( \frac{\delta_1}{\delta_g} \right)^{\frac{17}{10}}e^{-2\epsilon(s-s_0)}.
\end{align*}
Then 
\begin{align}\label{B0termestimate}
I_{0,\beta}&=2\int_{|y|\geq C_0}\partial_{\beta}\tilde{U}\cdot B_{u,0}dy+2\int_{|y|\geq C_0} \partial_{\beta}\tilde{S}B_{s,0}dy\\\nonumber
&\leq \int_{|y|\geq C_0} \left( 3+|\hat{X}\bar{U}|'(R)+\frac{2\hat{X}|\bar{U}|}{R}+|\nabla(\al\hat{X}\bar{S})| \right) \left( |\partial_{\beta}\tilde{U}|^2+|\partial_{\beta}\tilde{S}|^2 \right) dy+ \left( \frac{\delta_1}{\delta_g} \right)^{\frac{17}{10}}e^{-2\epsilon(s-s_0)}\\\nonumber
&\quad\underbrace{-\int_{\partial (e^{s}\mathbb{T}_{L}^{3})} \left( (\partial_{\beta}\tilde{S})^2+\sum_{j}(\partial_{\beta}\tilde{U}_j)^2 \right) (y \cdot \vec{n})d\sigma}_{\text{boundary term}}.
\end{align}
Notice that since $y \cdot \vec{n}=e^{s}L$ on $\partial(e^{s}\mathbb{T}_{L}^{3}),$ this boundary term in $I_{0,\beta}$ is same as the boundary term in \eqref{highenergyestimate01} with opposite sign.

For $B_{u,1}$, $B_{s,1}$ we have
\begin{align*}
B_{u,1}&=-(r-1)\partial_{\beta}\tilde{U}-\partial_{\beta}\tilde{U}\cdot\nabla(\hat{X}\bar{U})-\al\partial_{\beta}\tilde{S}\nabla(\hat{X}\bar{S})\\&
\quad-\sum_{\lambda=1}^{4}\sum_{i=1}^{3}\partial_{\beta_{\lambda}}(\hat{X}\bar{U}_i+y_i)\partial_{i}\partial_{\beta^{(\lambda)}} \tilde{U}-\al\sum_{\lambda=1}^{4}\sum_{i=1}^{3}\partial_{\beta_{\lambda}}(\hat{X}\bar{S})\partial_{i}\partial_{\beta^{(\lambda)}} \tilde{S}\\
&=-(r-1)\partial_{\beta}\tilde{U}-\partial_{\beta}\tilde{U}\cdot\nabla(\hat{X}\bar{U})-\al\partial_{\beta}\tilde{S}\nabla(\hat{X}\bar{S})\\&
\quad-\sum_{\lambda=1}^{4}\sum_{i=\beta_{\lambda}}\partial_{\beta_{\lambda}}(y_i)\partial_{i}\partial_{\beta^{(\lambda)}} \tilde{U}-\sum_{\lambda=1}^{4}\sum_{i=1 }^{3}\partial_{\beta_{\lambda}}(\hat{X}\bar{U}_i)\partial_{i}\partial_{\beta^{(\lambda)}} \tilde{U}-\al\sum_{\lambda=1}^{4}\sum_{i=1}^{3}\partial_{\beta_{\lambda}}(\hat{X}\bar{S})\partial_{i}\partial_{\beta^{(\lambda)}} \tilde{S}\\
&=-(r-1)\partial_{\beta}\tilde{U}-\partial_{\beta}\tilde{U}\cdot\nabla(\hat{X}\bar{U})-\al\partial_{\beta}\tilde{S}\nabla(\hat{X}\bar{S})\\&
\quad-4\partial_{\beta} \tilde{U}-\sum_{\lambda=1}^{4}\sum_{i=1}^{3}\partial_{\beta_{\lambda}}(\hat{X}\bar{U}_i)\partial_{i}\partial_{\beta^{(\lambda)}} \tilde{U}-\al\sum_{\lambda=1}^{4}\sum_{i=1}^{3}\partial_{\beta_{\lambda}}(\hat{X}\bar{S})\partial_{i}\partial_{\beta^{(\lambda)}} \tilde{S},
\end{align*}
and
\begin{align*}
B_{s,1}&=-(r-1)\partial_{\beta}\tilde{S}-\partial_{\beta}\tilde{U}\cdot \nabla (\hat{X}\bar{S})-\al\partial_{\beta}\tilde{S}\div(\hat{X}\bar{U})\\
&\quad-\sum_{\lambda=1}^{4}\sum_{i=1}^{3}\partial_{\beta_{\lambda}}(\hat{X}\bar{U}_i+y_i)\partial_{i}\partial_{\beta^{(\lambda)}} \tilde{S}-\al\sum_{\lambda=1}^{4}\partial_{\beta_{\lambda}}(\hat{X}\bar{S})\partial_{\beta^{(\lambda)}} \div(\tilde{U})\\
&=-(r-1)\partial_{\beta}\tilde{S}-\partial_{\beta}\tilde{U}\cdot \nabla (\hat{X}\bar{S})-\al\partial_{\beta}\tilde{S}\div(\hat{X}\bar{U})\\
&\quad-\sum_{\lambda=1}^{4}\sum_{i=\beta_{\lambda}}\partial_{\beta_{\lambda}}(y_i)\partial_{i}\partial_{\beta^{(\lambda)}} \tilde{S}-\sum_{\lambda=1}^{4}\sum_{i=1}^{3}\partial_{\beta_{\lambda}}(\hat{X}\bar{U}_i)\partial_{i}\partial_{\beta^{(\lambda)}} \tilde{S}-\al\sum_{\lambda=1}^{4}\partial_{\beta_{\lambda}}(\hat{X}\bar{S})\partial_{\beta^{(\lambda)}} \div(\tilde{U})\\
&=-(r-1)\partial_{\beta}\tilde{S}-\partial_{\beta}\tilde{U}\cdot \nabla (\hat{X}\bar{S})-\al\partial_{\beta}\tilde{S}\div(\hat{X}\bar{U})\\
&\quad-4\partial_{\beta} \tilde{S}-\sum_{\lambda=1}^{4}\sum_{i=1}^{3}\partial_{\beta_{\lambda}}(\hat{X}\bar{U}_i)\partial_{i}\partial_{\beta^{(\lambda)}} \tilde{S}-\al\sum_{\lambda=1}^{4}\partial_{\beta_{\lambda}}(\hat{X}\bar{S})\partial_{\beta^{(\lambda)}} \div(\tilde{U}).
\end{align*}
Then,
\begin{align}\label{B1termestimate}
I_{1,\beta}&=2\int_{|y|\geq C_0}\partial_{\beta}\tilde{U}\cdot B_{u,1}dy+2\int_{|y|\geq C_0}\partial_{\beta}\tilde{S}\cdot B_{s,1}dy\\\nonumber
&\leq -2(4+r-1)\int_{|y|\geq C_0}|\partial_{\beta}\tilde{U}|^2+(\partial_{\beta}\tilde{S})^2dy\\\nonumber
&\quad+C \left( \|\nabla(\hat{X}\bar{U})\|_{L^{\infty}(|y|\geq C_0)}+\|\nabla(\hat{X}\bar{S})\|_{L^{\infty}(|y|\geq C_0)} \right) \left(\|\tilde{U}\|_{\dot{H}^4(|y|\geq C_0)}^2+ \|\tilde{S}\|_{\dot{H}^4(|y|\geq C_0)}^2 \right).
\end{align}

For $B_{u,2}$ and $B_{s,2}$, from H\"older's inequality, we have
\begin{align}\label{B2termestimate}
&\|B_{s,2}\|_{L^2(|y|\geq C_0)}+\|B_{u,2}\|_{L^2(|y|\geq C_0)}\\\nonumber
&\quad\leq C \left( \|\nabla^{2}(\hat{X}\bar{U})\|_{L^{8}({|y|\geq C_0})}+\|\nabla^2(\hat{X}\bar{S})\|_{L^{8}({|y|\geq C_0})} \right) \left( \|\nabla^{3}\tilde{U}\|_{L^{\frac{8}{3}}(|y|\geq C_0)}+\|\nabla^3{\tilde{S}}\|_{L^{\frac{8}{3}}(|y|\geq C_0)} \right) \\ \nonumber
&\quad\leq C \left( \|\nabla^{2}(\hat{X}\bar{U})\|_{L^{8}({|y|\geq C_0})}+\|\nabla^2(\hat{X}\bar{S})\|_{L^{8}({|y|\geq C_0})} \right)  \delta_0 e^{-\epsilon(s-s_0)} \left( \frac{1}{C_2} \right)^{\frac{1}{20}}.
\end{align}
Here we also used the Gagliardo-Nirenberg interpolation Lemma \eqref{lemma:GN_generalnoweighttorus}, by taking $p=\infty$, $q=2$, $\theta=\frac{3}{4}$, $r=\frac{8}{3}$, $i=3$, $m=4$.
Thus
\begin{align*}
I_{2,\beta}&=2\int_{|y|\geq C_0}\partial_{\beta}\tilde{U}\cdot B_{u,2}dy+2\int_{|y|\geq C_0}\partial_{\beta}\tilde{S}\cdot B_{s,2}dy\\
&\leq C(\|\nabla^{2}(\hat{X}\bar{U})\|_{L^{8}({|y|\geq C_0})}+\|\nabla^2(\hat{X}\bar{S})\|_{L^{8}({|y|\geq C_0})})\left(\int_{|y|\geq C_0}|\partial_{\beta}\tilde{U}|^2+(\partial_{\beta}\tilde{S})^2dy\right)^{\frac{1}{2}}{\delta_0}e^{-\epsilon(s-s_0)} \left( \frac{1}{C_2} \right)^{\frac{1}{20}}.
\end{align*}
For $B_{u,3}$ and $B_{s,3}$, from the condition of $C_0$ \eqref{profilecondition3}, the bootstrap assumptions \eqref{lowerestimate1}, \eqref{lowerestimate2}  and interpolation inequality \eqref{eq:GNresultnoweighttorus} with $p=\infty,q=2,\theta=\frac{4}{5}, m=4$, we have
\begin{align*}
&\|B_{s,3}\|_{L^2(|y|\geq C_0)}+\|B_{u,3}\|_{L^2(|y|\geq C_0)} \\
 & \qquad\leq 
C\sum_{j=3}^{5} \left( \|\hat{X}\bar{U}\|_{\dot{H}^{j}(|y|\geq C_0)}+\|\hat{X}\bar{S}\|_{\dot{H}^{j}(|y|\geq C_0)} \right) \left( \|\tilde{U}\|_{W^{2,\infty}(|y|\geq C_0)}+\|\tilde{S}\|_{W^{2,\infty}(|y|\geq C_0)} \right) \\
& \qquad\leq C\sum_{j=3}^{5} \left( \|\hat{X}\bar{U}\|_{\dot{H}^{j}(|y|\geq C_0)}+\|\hat{X}\bar{S}\|_{\dot{H}^{j}(|y|\geq C_0)} \right) \delta_0 \left( \frac{1}{C_2} \right)^{\frac{1}{20}}e^{-\epsilon(s-s_0)}.
\end{align*}
Thus
\begin{align}\label{B3termestimate}
I_{3,\beta}&=2\int_{|y|\geq C_0}\partial_{\beta}\tilde{U}\cdot B_{u,3}dy+2\int_{|y|\geq C_0}\partial_{\beta}\tilde{S}\cdot B_{s,3}dy\\\nonumber
&\leq C\sum_{j=3}^{5} \left( \|\hat{X}\bar{U}\|_{\dot{H}^{j}(|y|\geq C_0)}+\|\hat{X}\bar{S}\|_{\dot{H}^{j}(|y|\geq C_0)} \right) {\delta_0} \left( \frac{1}{C_2} \right)^{\frac{1}{20}}e^{-\epsilon(s-s_0)}\left(\int_{|y|\geq C_0}|\partial_{\beta}\tilde{U}|^2+(\partial_{\beta}\tilde{S})^2dy\right)^{\frac{1}{2}}.
\end{align}
For the nonlinear terms, we use H\"older's inequality and get
\begin{align*}
&\|\partial_{\beta}N_u\|_{L^{2}(|y|\geq C_0)}+\|\partial_{\beta}N_s\|_{L^{2}(|y|\geq C_0)}\\
&\quad \leq C \left( \|\nabla ^{5}\tilde{U}\|_{{L^{2}(|y|\geq C_0)}}+\|\nabla ^{5}\tilde{S}\|_{{L^{2}(|y|\geq C_0)}} \right) \left( \|\tilde{U}\|_{{L^{\infty}(|y|\geq C_0)}}+\|\tilde{S}\|_{{L^{\infty}(|y|\geq C_0)}} \right) \\
&\qquad + C \left( \|\nabla ^{3}\tilde{U}\|_{{L^{\frac{8}{3}}(|y|\geq C_0)}}+\|\nabla ^{3}\tilde{S}\|_{{L^{\frac{8}{3}}(|y|\geq C_0)}} \right) \left( \|\nabla ^{2}\tilde{U}\|_{{L^{8}(|y|\geq C_0)}}+\|\nabla ^{2}\tilde{S}\|_{{L^{8}(|y|\geq C_0)}} \right)\\
&\qquad + C \left( \|\nabla ^{4}\tilde{U}\|_{{L^{2}(|y|\geq C_0)}}+\|\nabla ^{4}\tilde{S}\|_{{L^{2}(|y|\geq C_0)}} \right) \left( \|\nabla\tilde{U}\|_{{L^{\infty}(|y|\geq C_0)}}+\|\nabla \tilde{S}\|_{{L^{\infty}(|y|\geq C_0)}} \right).
\end{align*}
From the Gagliardo-Nirenberg interpolation Lemma \ref{lemma:GN_generalnoweighttorus}, and the bootstrap assumptions \eqref{lowerestimate2}, \eqref{higherestimate3}, 
we get
\begin{align*}
&\|\nabla ^{5}\tilde{U}\|_{{L^{2}(|y|\geq C_0)}}\lesssim \|\nabla ^{K}\tilde{U}\|_{{L^{2}(|y|\geq C_0)}}^{\theta_k}\|\nabla ^{4}\tilde{U}\|_{{L^{2}(|y|\geq C_0)}}^{1-\theta_k}+\|\nabla ^{4}\tilde{U}\|_{{L^{2}(|y|\geq C_0)}}\lesssim E^{\frac{1}{10}}\delta_0^{\frac{9}{10}}e^{-\epsilon(s-s_0)\frac{9}{10}}+\delta_0e^{-\epsilon(s-s_0)},\\
&\|\nabla ^{2}\tilde{U}\|_{L^{8}(|y|\geq C_0)}, \|\nabla ^{3}\tilde{U}\|_{{L^{\frac{8}{3}}(|y|\geq C_0)}}\lesssim \|\nabla ^{4}\tilde{U}\|_{{L^{2}(|y|\geq C_0)}}^{\frac{4}{3}}\|\tilde{U}\|_{{L^{\infty}(|y|\geq C_0)}}^{\frac{1}{4}}+\| \tilde{U}\|_{{L^{\infty}(|y|\geq C_0)}}\lesssim \delta_0e^{-\epsilon(s-s_0)}.
\end{align*}
Similar estimates hold for $\tilde{S}$. Then
\begin{align*}
&\|\partial_{\beta}N_u\|_{L^{2}(|y|\geq C_0)}+\|\partial_{\beta}N_s\|_{L^{2}(|y|\geq C_0)}\\
&\quad \leq C{\delta_0}^{\frac{19}{10}}e^{-\frac{19}{10}\epsilon(s-s_0)}E^{\frac{1}{10}}\\
&\quad \leq {\delta_0}^{\frac{3}{2}}e^{-\frac{3}{2}\epsilon(s-s_0)}.
\end{align*}
Here we used $\delta_0E\ll 1$ in the last step.
Then
\begin{align}\label{Ntermestimate}
I_{N,\beta}&=2\int_{|y|\geq C_0}\partial_{\beta}\tilde{U}\cdot \partial_{\beta}N_u dy+2\int_{|y|\geq C_0}\partial_{\beta}\tilde{S}\cdot \partial_{\beta}N_s dy\leq 2{\delta_0}^{\frac{3}{2}}e^{-\frac{3}{2}\epsilon(s-s_0)}\left(\int_{|y|\geq C_0}|\partial_{\beta}\tilde{U}|^2+(\partial_{\beta}\tilde{S})^2dy)\right)^{\frac{1}{2}}.
\end{align}
For $\p_{\beta}\left(\frac{\Delta U}{S^{\frac{1}{\alpha}}}\right)$, from Lemma \ref{lemma:brasilia} we have
\begin{align*}
    |\nabla^{4-j}\Delta U| &\lesssim_{\delta_0,\epsilon}S \phi^{-3+\frac{j}{2}}R^{\epsilon}.\\
    |\nabla^{j}\frac{1}{S^{\frac{1}{\alpha}}}| &\lesssim \sum_{l\leq j, k_1+k_2+...k_l=j} \frac{|\nabla^{k_1}S||\nabla^{k_2}S|...|\nabla^{k_l}S|}{S^{\frac{1}{\al}+l}}\\
    &\lesssim \sum_{l\leq j, k_1+k_2+...k_l=j} \phi^{-\frac{k_1+k_2+...k_{l}}{2}}\frac{1}{S^{\frac{1}{\alpha}}}R^{j\epsilon}\\
    &\lesssim_{\delta_0, R_0,\epsilon} \phi^{-\frac{j}{2}}\frac{1}{S^{\frac{1}{\alpha}}}R^{j\epsilon} . 
\end{align*}
Hence we have
\begin{align*}
\left| \partial_{\beta}\left(\frac{\Delta U}{S^{\frac{1}{\alpha}}}\right) \right| \lesssim &\sum_{j} \left| \nabla^{4-j}\Delta U \right| \cdot \left| \nabla^{j}\left(\frac{1}{S^{\frac{1}{\alpha}}} \right)\right| \lesssim_{\delta_1,\epsilon} \phi^{-3}S^{1-\frac{1}{\al}}R^{5\epsilon}\lesssim_{\delta_1,R_0,\epsilon} \left( \frac{1}{R} \right)^{6(1-\eta)-(r-1)(\frac{1}{\al}-1)}\\
\qquad &\lesssim_{\delta_0, R_0,\epsilon} \left( \frac{1}{R} \right)^{3},
\end{align*}
where we used $\eta\ll 1, (r-1)(\frac{1}{\alpha}-1)\leq (r-1)(\frac{1}{\alpha}+2)\leq 2$, and Lemma \ref{lemma:Sbounds}.

Then we have
\[
e^{-\delta_{dis}s} \left\| \partial_{\beta}\left(\frac{\Delta U}{S^{\frac{1}{\alpha}}}\right) \right\|_{L^2(|y|\geq C_0)}\lesssim_{\delta_0,R_0}e^{-\delta_{dis}s}\leq \delta_1 e^{-\delta_g(s-s_0)},
\]
and 
\begin{align*}
    I_{dis,\beta}=2\int_{|y|\geq C_0}\partial_{\beta}\tilde{U}\partial_{\beta} \left( e^{-\delta_{dis}s}\frac{\Delta U}{S^{\frac{1}{\alpha}}} \right) dy\leq 2\delta_1  e^{-\delta_g(s-s_0)}\left( \int_{|y|\geq C_0}(|\partial_{\beta}\tilde{U}|^2+|\partial_{\beta}\tilde{S}|^2)dy \right)^{\frac{1}{2}} .
\end{align*}
Additionally, from Lemma \ref{profilekestimate1}, we have 
\begin{align*}
I_{E,\beta}=4\int_{|y|\geq C_0}\partial_{\beta}\tilde{U}\cdot\partial_{\beta}\mathcal{E}_{u}+\partial_{\beta}\tilde{S}\partial_{\beta}\mathcal{E}_{s}dy\leq 2\delta_1  e^{-(s-s_0)} \left( \int_{|y|\geq C_0}\left( |\partial_{\beta}\tilde{U}|^2+|\partial_{\beta}\tilde{S}|^2 \right) dy \right)^{\frac{1}{2}}.
\end{align*}
In conclusion, from \eqref{highenergyestimate01} and estimates for each term, we have
\begin{align*}
I&\leq \left( -2(4+r-1)+3 \left( 1+ \|\nabla(\hat{X}\bar{U})+\|\frac{\hat{X}\bar{U}}{R}\|_{L^{\infty}(|y|\geq C_0)}+ \left\|\nabla(\hat{X}\bar{S}) \right\|_{L^{\infty}(|y|\geq C_0)} \right) \right) \sum_{\beta\in [3]^{4}}\int_{|y|\geq C_0}|\partial_{\beta}\tilde{U}|^2+(\partial_{\beta}\tilde{S})^2dy\\\nonumber
&\quad+C \left( \|\nabla(\hat{X}\bar{U})\|_{L^{\infty}(|y|\geq C_0)}+\|\nabla(\hat{X}\bar{S})\|_{L^{\infty}(|y|\geq C_0)} \right) 
\left( \|\tilde{U}\|^2_{\dot{H}^4(|y|\geq C_0)}+ \|\tilde{S}\|^2_{\dot{H}^4(|y|\geq C_0)} \right) + 3^{4} \left( \frac{\delta_1}{\delta_g} \right)^{\frac{17}{10}}e^{-2\epsilon(s-s_0)}\\
&\quad+\Bigg[C \left( \|\nabla^{2}(\hat{X}\bar{U})\|_{L^{8}({|y|\geq 
C_0})}+\|\nabla^2(\hat{X}\bar{S})\|_{L^{8}({|y|\geq C_0})} \right) \delta_0 e^{-\epsilon(s-s_0)} \left( \frac{1}{C_2} \right)^{\frac{1}{20}}\\
&\qquad+C\sum_{j=3}^{5} \left( \|\hat{X}\bar{U}\|_{\dot{H}^{j}(|y|\geq C_0)}+\|\hat{X}\bar{S}\|_{\dot{H}^{j}(|y|\geq C_0)} \right) \delta_0 \left( \frac{1}{C_2} \right)^{\frac{1}{20}}e^{-\epsilon(s-s_0)}+2{\delta_0}^{\frac{3}{2}}e^{-\frac{3}{2}\epsilon(s-s_0)}+2\delta_1  e^{-\delta_g(s-s_0)}\\
&\qquad+ 4\delta_1  e^{-(s-s_0)} \Bigg] \cdot\sum_{\beta\in[3]^{4}} \left( \int_{|y|\geq C_0}|\partial_{\beta}\tilde{U}|^2+(\partial_{\beta}\tilde{S})^2 dy \right)^{\frac{1}{2}}\\
%%%%%%%%%%%%%%%%%
&\leq
\left( \|\tilde{U}\|_{\dot{H}^4(|y|\geq C_0)}^{2}+\|\tilde{S}\|_{\dot{H}^4(|y|\geq C_0)}^{2} \right)\\\nonumber
&\qquad\times \left( -5-2(r-1)+C \left( \|\nabla(\hat{X}\bar{U})\|_{L^{\infty}(|y|\geq C_0)}+\|\nabla(\hat{X}\bar{S})\|_{L^{\infty}(|y|\geq C_0)}+\|\frac{\hat{X}\bar{U}}{R}\|_{L^{\infty}(|y|\geq C_0)} \right) \right) \\
&\quad+C \left( \|\tilde{U}\|_{\dot{H}^4(|y|\geq C_0)}^{2}+\|\tilde{S}\|_{\dot{H}^4(|y|\geq C_0)}^{2} \right)^{\frac{1}{2}} \Bigg[(\|\nabla^{2}(\hat{X}\bar{U})\|_{L^{8}({|y|\geq C_0})}+\|\nabla^2(\hat{X}\bar{S})\|_{L^{8}({|y|\geq C_0})}) \delta_0 e^{-\epsilon(s-s_0)}\left( \frac{1}{C_2} \right)^{\frac{1}{20}}\\
&\qquad+\sum_{j=3}^{5} \left( \|\hat{X}\bar{U}\|_{\dot{H}^{j}(|y|\geq C_0)}+\|\hat{X}\bar{S}\|_{\dot{H}^{j}(|y|\geq C_0)} \right) \delta_0 \left( \frac{1}{C_2} \right)^{\frac{1}{20}}e^{-\epsilon(s-s_0)}
+2{\delta_0}^{\frac{3}{2}}e^{-\frac{3}{2}\epsilon(s-s_0)}+4\delta_1  e^{-\delta_g(s-s_0)} \Bigg]\\
&\quad+3^{4} \left( \frac{\delta_1}{\delta_g} \right)^{\frac{17}{10}}e^{-2\epsilon(s-s_0)}.
\end{align*}
From the profile conditions \eqref{profilecondition4}, \eqref{profilecondition2}, \eqref{profilecondition3}, we have
\begin{align*}
&\quad\frac{d}{ds} \left( \|\tilde{U}\|_{\dot{H}^4(|y|\geq C_0)}^{2}+\|\tilde{S}\|_{\dot{H}^4(|y|\geq C_0)}^{2} \right) \leq - \left( \|\tilde{U}\|_{\dot{H}^4(|y|\geq C_0)}^{2}+\|\tilde{S}\|_{\dot{H}^4(|y|\geq C_0)}^{2} \right) \\ 
&\qquad +\frac{1}{20} \left( \frac{1}{C_2} \right)^{\frac{1}{20}}\delta_0
e^{-\epsilon(s-s_0)} \left( \|\tilde{U}\|_{\dot{H}^4(|y|\geq C_0)}^{2}+\|\tilde{S}\|_{\dot{H}^4(|y|\geq C_0)}^{2} \right)^{\frac{1}{2}}+81 \left( \frac{\delta_1}{\delta_g} \right)^{\frac{17}{10}} e^{-2\epsilon(s-s_0)}.
\end{align*}
Let $A(s)= \left( \|\tilde{U}\|_{\dot{H}^4(|y|\geq C_0)}^{2}+\|\tilde{S}\|_{\dot{H}^4(|y|\geq C_0)}^{2} \right) e^{2\epsilon(s-s_0)}.$ Then we get:
\begin{align*}
\frac{dA(s)}{ds}\leq -\frac{1}{2}A(s)+\frac{1}{20} \left( \frac{1}{C_2} \right)^{\frac{1}{20}}\delta_0
\sqrt{A(s)}+81 \left( \frac{\delta_1}{\delta_g} \right)^{\frac{17}{10}}.
\end{align*}
We then claim 
\begin{equation}\label{Aenergyestimate}
A(s)< \left( \frac{1}{5} \left( \frac{1}{C_2} \right)^{\frac{1}{20}}\delta_0 \right)^2+4 \cdot 81 \left( \frac{\delta_1}{\delta_g} \right)^{\frac{17}{10}}\leq \frac{1}{4} \delta_0^2,
\end{equation}
since \eqref{Aenergyestimate} holds for $s_0=0$ and if there exists $s^{*}$ such that the equality holds, we have $\frac{dA(s^{*})}{ds}<0$.
\end{proof}

\subsection{Energy estimate for higher order bounds} \label{higherorderestimate} 

In this section, we show how to close the bootstrap assumption for the energy. That is, assuming we are under the hypotheses of Proposition \ref{prop:bootstrap}, we will show that $E_{K} \leq \frac12 E$, see \eqref{higherestimate3} for a definition of $E_K$.

First of all, we let $[3]^K$ to be the set of all possible $K$-th order multiindices formed with $\{ 1, 2, 3 \}$. Thus, from definition \eqref{higherestimate3}, we have that
\begin{equation}
    E_{K} = \sum_{\beta \in [3]^K}  \int_{e^s\mathbb T_{L}^3} 
    \left( | \p_\beta U |^2 + | \p_\beta S |^2 \right) \phi^K dy,
\end{equation}
where $\p_\beta U$ is the vector $(\p_\beta U_1, \p_\beta U_2, \p_\beta U_3)$.

Taking a $\p_\beta$ derivative in \eqref{eq:US2} (with viscosity) and using that $\p_\beta (y \nabla f) = K \p_\beta f + y \nabla \p_\beta f$, we have that: 
\begin{align} \label{eq:US3} \begin{split}
(\p_s + r -1 + K )\p_\beta U_i + y\cdot \nabla \p_\beta U_i + \p_\beta (U \cdot \nabla U_i) + \alpha \p_\beta (S \p_i S) &= \frac{r^{1+\frac{1}{\alpha}}}{\alpha^{1/\alpha}  }
e^{-\delta_{\rm dis} s}  \p_\beta \frac{\Delta U_i}{ S^{1/\alpha }}\,, \\
(\p_s + r - 1 + K)\p_\beta S + y \cdot \nabla \p_\beta S + \alpha \p_\beta (S \div U) + \p_\beta (U \cdot \nabla S) &= 0.
\end{split} \end{align}
Now, multiplying each equation of \eqref{eq:US3} by $\phi^K \p_\beta S$ and $\phi^K \p_\beta  U$ respectively, and integrating over $e^s\mathbb T_{L}^3$, we obtain the following equality, (we will use the notation $\int$ but mean integration over $e^s\mathbb T_{L}^3$):
\begin{align} \begin{split} 
\left( \frac{1}{2}\p_s + r -1 + K \right) E_{K} & + \underbrace{ \sum_{\beta \in [3]^K} \int \phi^K y \cdot \left( \p_\beta S \nabla \p_\beta S + \sum_i \p_\beta U_i \nabla \p_\beta U_i \right) }_{\mc I_1} \\
&\quad+ \underbrace{ \sum_{\beta \in [3]^K}\int \phi^K \left( \sum_i \p_\beta U_i  \p_\beta (U \cdot \nabla U_i) + \p_\beta S  \p_\beta (U \cdot \nabla S) \right) }_{\mc I_2}\\
&\quad + \underbrace{ \sum_{\beta \in [3]^K} \int \phi^K  \alpha \left( \sum_i \p_\beta U_i \p_\beta (S \p_i S) + \p_\beta S   \p_\beta (S \div U)\right) }_{\mc I_3} \\
&\quad \underbrace{-\frac{1}{2}e^{s}L\sum_{\beta \in [3]^K}\int_{\partial (e^s\mathbb{T}_{L}^{3})}(|\partial_{\beta}U|^2+|\partial_{\beta}S|^2)}_{\text{boundary term}}\\
&= \underbrace{ \sum_{\beta \in [3]^K} C_{\rm{dis}}
e^{-\delta_{\rm dis} s}  \int \phi^K \p_\beta U \cdot \p_\beta \frac{\Delta U_i}{ S^{1/\alpha }} }_{\mc I_4}. \label{eq:EE1}
\end{split} \end{align}

\subsubsection{Main energy estimate: terms $\mc I_1$, $\mc I_2$ and $\mc I_3$}

First of all, integrating by parts, we see that
\begin{align} \label{estimateI1}
    \mc I_1 &= \sum_{\beta \in [3]^K}  \int \phi^K \frac{y}{2} \left( \nabla | \p_\beta S |^2 + \sum_i \nabla | \p_\beta U_i |^2 \right) dy \\
    &= \frac{-1}{2} \sum_{\beta \in [3]^K} \int \frac{\div (\phi^K y)}{\phi^K} \left( |\p_\beta U |^2 + |\p_\beta S |^2 \right) \phi^K dy \notag+\underbrace{\frac{1}{2}e^{s}L\sum_{\beta \in [3]^K}\int_{\partial (e^s\mathbb{T}_{L}^{3})}(|\partial_{\beta}U|^2+|\partial_{\beta}S|^2)d\sigma}_{\text{boundary term}}\\\nonumber
    &= \frac{-K}{2} \sum_{\beta \in [3]^K} \int \frac{y \cdot \nabla \phi}{\phi}  \left( |\p_\beta U |^2 + |\p_\beta S |^2 \right) \phi^K dy + O ( E )+\underbrace{\frac{1}{2}e^{s}L\sum_{\beta \in [3]^K}\int_{\partial (e^s\mathbb{T}_{L}^{3})}(|\partial_{\beta}U|^2+|\partial_{\beta}S|^2)d\sigma.}_{\text{boundary term}}
\end{align}
The boundary term cancels the same term in \eqref{eq:EE1}.
Next, we prove the following lemma:
\begin{lemma}\label{lowerorderestimate} Let $\beta_j$ denote the $j-$th index of the multiindex $\beta$, and analogously, let $\beta^{(j)}$ be the $(K-1)$-th order multiindex generated by erasing $\beta_j$. More specifically, we have 
\[
\partial_{\beta^{(j)}}=\partial_{\beta_1}\partial_{\beta_2}...\partial_{\beta_{j-1}}\partial_{\beta_{j+1}}...\partial_{\beta_K}.
\] Under the bootstrap assumptions \eqref{lowerestimate1}, \eqref{lowerestimate2}, \eqref{higherestimate3}, we have
\begin{align} \begin{split}
\Bigg\| 
\p_\beta (U \nabla U_i ) &- U \nabla \p_\beta U_i - \sum_{j=1}^K \sum_{k=1}^3 \left( \left( \p_R (\hat{X}\bar U_R) - \frac{\hat{X}\bar U_R}{R} \right) \frac{y_{\beta_j} y_k}{R^2} + \delta_{k, \beta_j} \frac{\hat{X}\bar U_R}{R} \right) \p_k \p_{\beta^{(j)}} U_i
\Bigg\|_{L^2_{\phi^{K}}} \\
&= O\left( E^{1/2-1/(4K)} + \|\phi^{K/2} \p_\beta U  \|_{L^2} \right), \label{eq:mainder1}\end{split}
\end{align}
\begin{align} \begin{split}
\Bigg\| 
\p_\beta (U \nabla S ) &- U \nabla \p_\beta S - \sum_{j=1}^K \sum_{k=1}^3 \left( \left( \p_R (\hat{X}\bar U_R) - \frac{\hat{X}\bar U_R}{R} \right) \frac{y_{\beta_j} y_k}{R^2} + \delta_{k, \beta_j} \frac{\hat{X}\bar U_R}{R} \right) \p_k \p_{\beta^{(j)}} S
\Bigg\|_{L^2_{\phi^{K}}}  \\
&= O\left( E^{1/2-1/(4K)} + \|  \phi^{K/2} \p_\beta S \|_{L^2} \right), \label{eq:mainder2}  \end{split}
\end{align}
\begin{align}
\left\| 
\p_\beta (S \p_i S ) - S \nabla \p_\beta S - \sum_{j=1}^K  \frac{y_{\beta_j}}{R} \p_R (\hat{X}\bar S) \p_i \p_{\beta^{(j)}} S
\right\|_{L^2_{\phi^{K}}} &= O\left( E^{1/2-1/(4K)} + \|  \phi^{K/2}\p_\beta S \|_{L^2} \right), \label{eq:mainder3} \\
\left\| 
\p_\beta (S \div (U) ) - S  \p_\beta \div (U) - \sum_{j=1}^K \sum_{i=1}^3 \frac{y_{\beta_j}}{R} \p_R (\hat{X}\bar S) \p_i \p_{\beta^{(j)}} U_i
\right\|_{L^2_{\phi^{K}}} &= O\left( E^{1/2-1/(4K)} + \|\phi^{K/2} \p_\beta U  \|_{L^2} \right). \label{eq:mainder4} 
\end{align}
\end{lemma}
\begin{proof}
Since all those bounds are analogous, let us just show \eqref{eq:mainder1}. We develop $\p_\beta (U \nabla U_i)$ by the Leibniz rule and we observe that:
\begin{equation} \label{eq:uruguay1}
\left| \p_\beta(U \nabla U_i) - \p_\beta U \nabla U_i - U \nabla \p_\beta U_i - \sum_{j=1}^K \p_{\beta_j} U \nabla \p_{\beta^{(j)}} U_i  \right| \les_K \sum_{\ell, \ell' \geq 2, \ell + \ell' = K+1} | \nabla^\ell U | \cdot | \nabla^{\ell'} U |,
\end{equation}
where $\les_K$ indicates that the implicit constant is allowed to depend on $K$. Now, using \eqref{eq:GNresult} from Lemma \ref{lemma:GN_general} for $\psi = 1$, $p=\infty$, $q=2$, $m=K$, $\bar{r} = \frac{2(K+1)}{\ell}$, $\theta = \frac{\ell - 3/\bar{r}}{K - 3/2}$ we have that
\begin{align} \begin{split} \label{eq:mendel}
\| \langle R \rangle^{-1/10} \phi^{\frac{K\theta}{2}} \nabla^\ell \tilde U \|_{L^{{\bar{r}}} } &\les \delta_0^{1-\theta} E^{\frac{\theta}{2}} \\ 
\| \langle R \rangle^{-1/10} \phi^{\frac{K\theta}{2}} \nabla^\ell (\hat X \bar U_R) \|_{L^{{\bar{r}}} } &\les_K 1.
\end{split} \end{align}
Thus, we obtain that for $\ell+\ell' = K+1$, $\ell, \ell' \geq 2$, using \eqref{perturbationenergy}:
\begin{align}  \label{eq:trinidad1}
    \left\| |\nabla^\ell \tilde U | | \nabla^{\ell'} \tilde U| \phi^{K/2} \right\|_{L^2} &= \left\| |\nabla^\ell \tilde U | | \nabla^{\ell'} \tilde U| \phi^{\frac{K}{2}\cdot \frac{\ell - 3/\bar{r}}{K-3/2}}\phi^{\frac{K}{2}\cdot \frac{\ell' - 3/r'}{K-3/2}} \phi^{-\frac12 \cdot \frac{K}{K-3/2}} \right\|_{L^2} \notag \\
    &\leq \left\| \langle R \rangle^{-1/10} \phi^{\frac{K}{2}\cdot \frac{\ell - 3/\bar{r}}{K-3/2}} \nabla^\ell \tilde U \right\|_{L^{\frac{2(K+1)}{\ell}}}
    \left\| \langle R \rangle^{-1/10} \phi^{\frac{K}{2}\cdot \frac{\ell' - 3/r'}{K-3/2}} \nabla^{\ell'} \tilde U \right\|_{L^{\frac{2(K+1)}{\ell'}}} \notag \\
    & \les \delta_0^{\frac{K-5/2}{K-3/2}}E^{\frac{K-\frac{1}{2}}{2(K-3/2)}}.
\end{align}
In the first inequality we used that $\ell+\ell' = K+1$ together with $\frac{K(\ell-3/\bar{r})}{K-3/2} + \frac{K(\ell' - 3/r')}{K-3/2} = K\frac{K-1/2}{K-3/2}$ for $\bar{r}= \frac{2(K+1)}{\ell}$, $r' = \frac{2(K+1)}{\ell'}$. \eqref{eq:GN_extracond} is satisfied because $\eta\gg \frac{1}{K}.$ In the second one we used \eqref{eq:mendel} with $\theta + \theta' = \frac{K-1/2}{K-3/2}$.
Similar to \eqref{eq:trinidad1}, using again \eqref{perturbationenergy} and \eqref{eq:mendel}, we obtain that
\begin{align}  \label{eq:trinidad2}
    \left\| |\nabla^\ell \tilde U | | \nabla^{\ell'}  (\hat{X}\bar U_R) | \phi^{K/2} \right\|_{L^2} &\leq
\left\| \langle R \rangle^{-1/10} \phi^{\frac{K}{2}\cdot \frac{\ell - 3/\bar{r}}{K-3/2}} \nabla^\ell \tilde U \right\|_{L^{\frac{2(K+1)}{\ell}}}
    \left\| \langle R \rangle^{-1/10} \phi^{\frac{K}{2}\cdot \frac{\ell' - 3/r'}{K-3/2}} \nabla^{\ell'} (\hat X \bar U_R) \right\|_{L^{\frac{2(K+1)}{\ell'}}} \notag \\
     &\les_K \delta_0^{\frac{K-2}{(K-3/2)(K+1)}}E^{\frac12 \left(1- \frac{K-2}{(K-3/2)(K+1)} \right)}
     \les 1,
\end{align}
and
\begin{align}  \label{eq:trinidad3}
    \left\| |\nabla^\ell  (\hat{X}\bar U_R)  | | \nabla^{\ell'} (\hat{X}\bar U_R) | \phi^{K/2} \right\|_{L^2} \leq\left\| \phi^{\frac{K}{2}\cdot \frac{(\ell - 3/\bar{r})}{K-3/2}} \nabla^\ell   (\hat{X}\bar U_R)  \right\|_{L^{\frac{2(K+1)}{\ell}}}
    \left\| \phi^{\frac{K}{2}\cdot \frac{(\ell' - 3/\bar{r})}{K-3/2}} \nabla^{\ell'} (\hat{X}\bar U_R) \right\|_{L^{\frac{2(K+1)}{\ell'}}}\les_K 1.
\end{align}
Combining \eqref{eq:trinidad1}--\eqref{eq:trinidad3} (and using that $\delta_0$ is sufficiently small depending on $K$, $E$), we obtain
\begin{equation}\label{eq:uruguay2}
\left\| | \nabla^\ell U | \cdot |\nabla^{\ell'} U | \phi^{K/2} \right\|_{L^2} 
\les_K 1.
\end{equation}

From \eqref{eq:uruguay1}--\eqref{eq:uruguay2}, using that $E$ is sufficiently large depending on $K$, we deduce
\begin{equation} \label{eq:uruguay3}
\left\| \p_\beta(U \nabla U_i) - \p_\beta U \nabla U_i - U \nabla \p_\beta U_i - \sum_{j=1}^K \p_{\beta_j} U \nabla \p_{\beta^{(j)}} U_i  \right\|_{L^2_{\phi^{K}}} \les E^{1/4}.
\end{equation}

Using the Gagliardo-Nirenberg inequality between $\| \tilde U \|_{L^\infty} \leq \delta_0$ and $\| \tilde U \|_{\dot{H}^{m}} \leq E$ (and using that $\delta_0$ is sufficiently small depending on $E$), we obtain that 
\begin{equation} \label{eq:nablaUtilde}
    \| \nabla \tilde U \|_{L^\infty} \les \delta_0^{9/10}.
\end{equation}
We deduce that
\begin{equation} \label{eq:nablaUtotal}
    \| \nabla U \|_{L^\infty} \leq \| \nabla \tilde U \|_{L^\infty} + \| \nabla (  \hat{X}\bar U_R)  \|_{L^\infty} \les 1.
\end{equation}
Therefore
\begin{equation} \label{eq:uruguay4}
\| \p_\beta U \nabla U_i \phi^K \|_{L^2} \les \| \nabla U_i \|_{L^\infty} \| \p_\beta U \|_{L^2_{\phi^{K}}} \les \| \p_\beta U \|_{L^2_{\phi^{K}}}.
\end{equation}

Thus, from \eqref{eq:uruguay3} and \eqref{eq:uruguay4}, we conclude
\begin{equation} \label{eq:uruguay5}
\left\| \p_\beta(U \nabla U_i) - U \nabla \p_\beta U_i - \sum_{j=1}^K \p_{\beta_j} U \nabla \p_{\beta^{(j)}} U_i  \right\|_{L^2_{\phi^{K}}} \les E^{1/4} + \| \p_\beta U \phi^{K/2} \|_{L^2}.
\end{equation}
Using \eqref{eq:nablaUtilde}, we also have that
\begin{equation*}
    \left\| \sum_{j=1}^K \p_{\beta_j} \tilde U \nabla \p_{\beta^{(j)}} U_i  \right\|_{L^2_{\phi^{K}}} \les \delta_{0}^{9/10} E^{\frac{1}{2}} \les \delta_{0}^{4/5}.
\end{equation*}
Using this into \eqref{eq:uruguay5}, we deduce
\begin{equation}\label{eq:uruguay6}
\left\| \p_\beta(U \nabla U_i) - U \nabla \p_\beta U_i - \sum_{j=1}^K \p_{\beta_j}  (\hat{X}\bar U_R)  \nabla \p_{\beta^{(j)}} U_i  \right\|_{L^2_{\phi^{K}}} \les E^{1/2-1/(4K)} + \| \p_\beta U \phi^{K/2} \|_{L^2}.
\end{equation}
Finally, we compute explicitly the last term of the norm. Using that $\bar U$ is radially symmetric, we have
\begin{equation*}
    \p_{\beta_j} (\hat{X}\bar U_k) = \p_{\beta_j} \left( \frac{y_k}{R} \hat{X}\bar U_R \right) = \left( \p_R (\hat{X}\bar U_R) - \frac{\hat{X}\bar U_R}{R} \right) \frac{y_k y_{\beta_j}}{R^2} + \delta_{\beta_j, k} \frac{\hat{X} \bar U_R}{R},
\end{equation*}
and plugging this into \eqref{eq:uruguay6}, we conclude the proof.
\end{proof}

We will use this Lemma to estimate $\mc I_2$ and $\mc I_3$. Note that once $j$ is fixed, it is equivalent to say that $\beta$ runs over $[3]^K$ than to say that $\beta^{(j)}$ runs over $[3]^{K-1}$ and $\beta_j$ over $[3] = \{1, 2, 3\}$. For fixed $j$ we will denote $\tilde \beta = \beta^{(j)}$ and $m = \beta_j$. Using \eqref{eq:mainder1}, we see that
\begin{align}
\sum_{\beta \in [3]^K} &\int \phi^K \sum_i \p_\beta U_i \p_\beta (U \cdot \nabla U_i) = \sum_{\beta \in [3]^K } \left( O(E^{1/2-1/(4K)} ) + O(\| \phi^{K/2} \p_\beta U \|_{L^2}) \right) \| \p_\beta U \phi^{K/2} \|_{L^2} \notag \\
&\quad + \sum_{\beta \in [3]^K} \int \phi^K \sum_i \Bigg[ U \nabla \p_\beta U_i \p_\beta U_i + \p_\beta U_i \sum_{j=1}^K \sum_{k=1}^3 \left( \left( \p_R (\hat{X}\bar U_R) - \frac{\hat{X}\bar U_R}{R} \right) \frac{y_{\beta_j} y_k}{R^2} + \delta_{k, \beta_j} \frac{\hat{X}\bar U_R}{R}  \right) \p_k \p_{\beta^{(j)}} U_i \Bigg] \notag \\
&=
O(3^K E^{1-1/(4K)} ) + O\left( \sum_{\beta \in [3]^K } \| \phi^{K/2} \p_\beta U \|_{L^2}^2 \right) - \sum_{\beta \in [3]^K} \int \frac{\div (\phi^K U)}{2 \phi^K } |\p_\beta U |^2 \phi^K \notag \\ 
&\quad + \sum_{j=1}^K \sum_{\tilde \beta \in [3]^{K-1}} \sum_{i, k, m = 1}^3 \int \phi^K \left( \left( \p_R (\hat{X}\bar U_R) - \frac{\hat{X}\bar U_R}{R} \right) \frac{y_k y_m}{R^2} + \delta_{k, m} \frac{\hat{X}\bar U_R}{R} \right)
\p_m \p_{\tilde \beta} U_i \p_k \p_{\tilde \beta} U_i \notag \\
&=
O(E ) - \sum_{\beta \in [3]^K} \int \frac{K \nabla \phi \cdot U+ \phi \div (U)}{2\phi}  |\p_\beta U |^2 \phi^K \notag \\ 
&\quad +K \sum_{\tilde \beta \in [3]^{K-1}} \int \phi^K \left( \p_R(\hat{X}\bar U_R) - \frac{\hat{X}\bar U_R}{R} \right) \sum_{i=1}^3 | \p_R \p_{\tilde \beta} U_i |^2
+ K \sum_{\tilde \beta \in [3]^{K-1}} \sum_{i, k = 1}^3 \int \phi^K \frac{\hat{X}\bar U_R}{R} |\p_k \p_{\tilde \beta} U_i|^2 \notag \\
&\geq
O(E ) + K \sum_{\tilde \beta \in [3]^{K-1}} \int \left( 
\frac{- \nabla \phi \cdot \bar U}{2\phi } | \nabla \p_{\tilde \beta} U|^2  
+ \p_R (\hat X \bar U_R) | \p_R \p_{\tilde \beta} U|^2 + 
\frac{\hat{X}\bar U_R}{R} |\nabla_{\theta} \p_{\tilde \beta} U |^2 \right) \phi^K.   \label{eq:paraguay1}
\end{align}
The first equality is a combination of \eqref{eq:mainder1} with Cauchy-Schwarz. After the second equality, the first two big-$O$ terms arise from the geometric-quadratic inequality, the third term arises from integration by parts, and the fourth term is just obtained by rearranging indices (we extract the $j$-th index in $\beta$ and define $\tilde \beta = \beta^{(j)}$ together with $m = \beta_j$). In the third equality we used that $\sum_{k, m} \frac{y_m y_k}{R^2} \p_{k} f \p_{m} f = |\p_R f|^2$. In the last inequality we used that $\|\div (U)\|_{L^\infty} \les 1$ (from \eqref{eq:nablaUtotal}) and $\|\frac{\nabla \phi}{2\phi}\tilde{U}\|_{L^{\infty}}\lesssim \delta_0$. Recall that $\nabla_{\theta}$ is the gradient in spherical coordinates without the radial derivative. 

An analogous computation using \eqref{eq:mainder2} yields
\begin{align}
&\sum_{\beta \in [3]^K} \int \phi^K \p_\beta S \p_\beta (U \cdot \nabla S) \notag \\
&\qquad  \label{eq:paraguay2} \geq
K \sum_{\tilde \beta \in [3]^{K-1}} \int \left( \frac{-\nabla \phi \cdot \hat{X}\bar U_R}{2\phi} | \nabla \p_{\tilde \beta} S |^2 
+ \p_R (\hat X \bar U_R) |\p_R \p_{\tilde \beta} U|^2
+ \frac{\hat X \bar U_R}{R} | \nabla_{\theta} \p_{\tilde \beta} S|^2\right)  \phi^K  + O(E).
\end{align}

Adding \eqref{eq:paraguay1} and \eqref{eq:paraguay2}, we see that
\begin{align} 
-\mc I_2  \leq O(E) &+ K \int \left( \frac{\nabla \phi \cdot  \hat{X}\bar U_R}{2\phi} - \p_R (\hat X \bar U_R) \right) \phi^K \left( | \p_R \nabla^{K-1} U |^2 + | \p_R \nabla^{K-1} S|^2 \right) \notag \\
&+ K \int \left( \frac{\nabla \phi \cdot  \hat{X}\bar U_R}{2\phi} - \frac{\hat X \bar U_R}{R} \right) \phi^K \left( | \nabla_\theta \nabla^{K-1} U |^2 + | \nabla_\theta \nabla^{K-1} S|^2 \right). \label{eq:paraguay3}
\end{align}

We focus on $\mc I_3$. Using \eqref{eq:mainder3}, we start computing 
\begin{align}
\sum_{\beta \in [3]^K} &\sum_{i=1}^3 \int \phi^K \p_\beta (S \p_i S) \p_\beta U_i = 
\sum_{\beta \in [3]^K} O \left( (E^{1/2-1/(4K)}  + \| \phi^{K/2} \p_\beta S \|_{L^2} ) \| \phi^{K/2} \p_\beta U \|_{L^2} \right)  \notag \\
&\quad + \sum_{\beta \in [3]^K}  \int \phi^K S \nabla \p_\beta S \cdot \p_\beta U  + \sum_{i=1}^3 \sum_{\beta \in [3]^K } \int \phi^K \p_\beta U_i \sum_{j=1}^K \frac{y_{\beta_j}}{R} \p_R (\hat{X}\bar S) \p_{\beta^{(j)}} S\notag  \\
&= 
O \left( 3^K E^{1-1/(4K)} + \sum_{\beta \in [3]^K } (\| 
 \phi^{K/2} \p_\beta U \|_{L^2}^2 + \| \phi^{K/2} \p_\beta S \|_{L^2}^2 ) \right) + \sum_{\beta \in [3]^K } \sum_{i=1}^3 \int \phi^K S \p_i \p_\beta S \p_\beta U_i  \notag \\
 &\quad + 
 \sum_{j=1}^K \sum_{\beta^{(j)} \in [3]^{K-1}} \sum_{\beta_j =1}^3 \sum_{i = 1}^3 \int \phi^K \p_R (\hat{X}\bar S)  \frac{y_{\beta_j}}{R} \p_{\beta_j} \p_{\beta^{(j)}} U_i \p_i \p_{\beta^{(j)} } S  \notag \\
&= 
O \left( 3^K E^{1-1/(4K)} + \sum_{\beta \in [3]^K } (\| 
 \phi^{K/2} \p_\beta U \|_{L^2}^2 + \| \phi^{K/2} \p_\beta S \|_{L^2}^2 ) \right) + \sum_{\beta \in [3]^K } \sum_{i=1}^3 \int \phi^K S \p_i \p_\beta S \p_\beta U_i  \notag \\
 &\quad + 
 \sum_{j=1}^K \sum_{\tilde \beta \in [3]^{K-1}} \sum_{i, m = 1}^3 \int \phi^K \p_R (\hat{X}\bar S)  \frac{y_m}{R} \p_m \p_{\tilde \beta} U_i \p_i \p_{\tilde \beta } S  \notag \\
 &=
 O(E) + \sum_{\beta \in [3]^K} \sum_{i=1}^3 \int \phi^K S \p_i \p_\beta S \p_\beta U_i + K \sum_{\tilde \beta \in [3]^{K-1}} \sum_i \int \phi^K \p_R( \hat{X}\bar S) \p_R \p_{\tilde \beta} U_i  \p_i \p_{\tilde \beta} S \notag \\
 &\geq O(E) + \sum_{\beta \in [3]^K} \sum_{i=1}^3 \int \phi^K S \p_i \p_\beta S \p_\beta U_i - K \sum_{\tilde \beta \in [3]^{K-1}} \sum_{i=1}^3 \int \phi^K |\p_R(\hat{X} \bar S) | \left( \frac{| \nabla \p_{\tilde \beta} U_i |^2}{2} +  \frac{|\p_i \p_{\tilde \beta} S|^2}{2} \right) \notag \\
 &= O(E) + \sum_{\beta \in [3]^K} \sum_{i=1}^3 \int \phi^K S \p_i \p_\beta S \p_\beta U_i - \frac{K}{2} \sum_{\beta \in [3]^{K}}\int \phi^K |\p_R (\hat{X}\bar S)| \left( |  \p_{\beta} U |^2 +  | \p_{ \beta} S|^2\right). \label{eq:bolivia1}
\end{align}
An analogous computation using \eqref{eq:mainder4} yields:
\begin{align} \begin{split}\label{eq:bolivia2}
\sum_{\beta \in [3]^K} \sum_{i=1}^3 \int \phi^K \p_\beta (S \p_i U_i ) \p_\beta S &\geq
O(E) + \sum_{\beta \in [3]^K} \sum_{i=1}^3 \int \phi^K S \p_\beta S \p_i  \p_\beta U_i  \\
&\quad - \frac{K}{2} \int \phi^K |\p_R(\hat{X}\bar S)| ( | \p_\beta U |^2 + | \p_\beta S |^2 ).
\end{split} \end{align}

Adding up \eqref{eq:bolivia1} and \eqref{eq:bolivia2}, and integrating by parts, we get
\begin{align} \label{eq:bolivia3}
-\mc I_3 &\leq O(E) - \alpha \sum_{\beta \in [3]^K} \sum_{i=1}^3 \int \phi^K \frac{\p_i (\phi^K S)}{\phi^K} \p_\beta S \p_\beta U_i + K\alpha \sum_{\beta \in [3]^K} \int \phi^K |\p_R (\hat{X}\bar S) | ( |\p_\beta U |^2 + | \p_\beta S |^2 )\\\nonumber
&\leq O(E) - \alpha \sum_{\beta \in [3]^K} \sum_{i=1}^3 \int \phi^K \frac{\p_i (\phi^K \hat{X}\bar S)}{\phi^K} \p_\beta S \p_\beta U_i + K\alpha \sum_{\beta \in [3]^K} \int \phi^K |\p_R (\hat{X}\bar S) | ( |\p_\beta U |^2 + | \p_\beta S |^2 ), 
\end{align}
where we recall $\alpha = \frac{\ga - 1}{2}$ and we also use $\|\frac{\partial_{i}(\phi^{K})\tilde{S}}{\phi^{K}}\|_{L^{\infty}}\lesssim K\delta_0\lesssim 1.$
Since $\| \nabla \bar S \|_{L^\infty} \les 1$, we have
\begin{align*}
 \left| \sum_{\beta \in [3]^K} \sum_{i=1}^3 \int \phi^K \frac{\p_i (\phi^K \hat{X}\bar S)}{\phi^K} \p_\beta S \p_\beta U_i \right|  &\leq
    \sum_{\beta \in [3]^K} \sum_{i=1}^3 \int \phi^K \left( | \p_i (\hat{X}\bar S) | \frac{| \p_\beta U_i|^2 + | \p_\beta S |^2}{2} + K \frac{\p_i \phi \hat{X}\bar S}{\phi} \p_\beta S \p_\beta U_i \right)  \\
    &\leq O(E) + K \sum_{\beta \in [3]^K} \int \phi^K \frac{| \nabla \phi | \hat{X}\bar S}{\phi} |\p_\beta S| \cdot |\p_\beta U | \\
    &\leq O(E) + \frac{K}{2} \sum_{\beta \in [3]^K} \int \phi^K \frac{| \nabla \phi | \hat{X}\bar S}{\phi} \left( |\p_\beta S|^2 +  |\p_\beta U |^2 \right),
\end{align*}
and using this into \eqref{eq:bolivia3}, we get
\begin{equation} \label{eq:bolivia4}
    -\mc I_3 \leq O(E) + K\alpha  \int \phi^K \left( \frac{| \nabla \phi | \hat{X}\bar S}{2\phi} + | \p_R (\hat{X}\bar S) |\right)\left( |\nabla^K S |^2 + | \nabla^K U |^2  \right) 
\end{equation}

Finally, adding \eqref{estimateI1}, \eqref{eq:paraguay3} and \eqref{eq:bolivia4}, we obtain
\begin{align} \begin{split} \label{eq:argentina}
    -\mc I_1-\mc I_2-\mc I_3 &\leq O(E) + K \int \left(\p_R(\hat{X} \bar U_R) + \alpha |\p_R(\hat{X} \bar S) | \right) \phi^K ( |\p_R \nabla^{K-1} U|^2 + | \p_R \nabla^{K-1} S |^2 ) \\
    & \qquad+ K \int  \left( \frac{\hat{X}\bar U_R}{R} + \alpha |\p_R(\hat{X} \bar S) | \right) \phi^K ( |\nabla_\theta \nabla^{K-1} U|^2 + | \nabla_\theta \nabla^{K-1} S |^2 ) \\
    & \qquad + K  \int \frac{\nabla \phi \cdot (y + \hat{X}\bar U) + | \nabla \phi | \hat{X}\bar S}{2\phi} \phi^K ( |\nabla^K U|^2 + | \nabla^K S |^2 )\\
    &\qquad-\underbrace{\frac{1}{2}e^{s}L\sum_{\beta \in [3]^K}\int_{\partial (e^s\mathbb{T}_{L}^{3})}(|\partial_{\beta}U|^2+|\partial_{\beta}S|^2)d\sigma.}_{\text{boundary term}}.
\end{split} \end{align}

Now, we claim the following modified repulsivity properties: 
\begin{align} \label{eq:argentina2}
  \p_R (\hat{X}\bar U_R) + \alpha |\p_R (\hat{X}\bar S) |+ \frac{\nabla \phi \cdot (y+\hat{X}\bar U) + | \nabla \phi | \hat{X}\bar S}{2\phi} \leq 1 - \frac{\min \{ \tilde \eta, \eta \}}{2}, \\
    \frac{\hat{X}\bar U_R}{R} + \alpha |\p_R (\hat{X}\bar S) |+ \frac{\nabla \phi \cdot (y+\hat{X}\bar U) + | \nabla \phi | \hat{X}\bar S}{2\phi} \leq 1 - \frac{\min \{ \tilde \eta , \eta \} }{2}. \label{eq:argentina5halfs}
\end{align}
In the region $|y| \leq R_0$, we have $\phi = 1$, so the last fraction on the LHS is zero. Also, $|\p_R \hat X| \les e^{-s_0}$. Therefore, by taking $\tilde \eta$ sufficiently small \eqref{eq:argentina2} and \eqref{eq:argentina5halfs} follow from the standard repulsivity properties \eqref{eq:radial_repulsivity}--\eqref{eq:angular_repulsivity}.

In the region $|y| \geq R_0$, since $R_0$ is sufficiently large depending on $ \eta$ \eqref{Rochoice estimate} and the profiles decay (equation \eqref{eq:profiles_decay}), we have that:
\begin{equation*}
-\p_R \bar U_R + \alpha |\p_R \bar S |+ \frac{\nabla \phi \cdot (y+\bar U) + | \nabla \phi | \bar S}{2\phi}  \leq \frac{\eta}{4} + \frac{R\p_R \phi}{2\phi}  \leq \frac{\eta}{4} + 1-\eta < 1-\frac{\eta}{2},
\end{equation*}
and this concludes the proof of \eqref{eq:argentina2}. We can conclude \eqref{eq:argentina5halfs} analogously just interchanging $\p_R \bar U$ by $\bar U/R$. Finally, combining \eqref{eq:argentina} with \eqref{eq:argentina2}--\eqref{eq:argentina5halfs}, we get that
\begin{align} 
-\mc I_1 - \mc I_2 - \mc I_3 &\leq O(E) + K\left( 1 - \frac{\min \{ \eta , \tilde \eta \} }{2} \right) \sum_{\beta \in [3]^K} \int \phi^K ( |\p_\beta U |^2 + | \p_\beta S |^2 )\\\nonumber
& \qquad-\underbrace{\frac{1}{2}e^{s}L\sum_{\beta \in [3]^K}\int_{\partial (e^s\mathbb{T}_{L}^{3})}(|\partial_{\beta}U|^2+|\partial_{\beta}S|^2)d\sigma.}_{\text{boundary term}}
\end{align}
and plugging this into \eqref{eq:EE1}, we get 
\begin{equation} \label{eq:EE2}
    \p_s E_{K} \leq -\frac{K\min \{ \eta, \tilde \eta \} }{2} E_{K} + O(E)+\mc I_{4}.
\end{equation}

\subsubsection{Dissipation term}
 Starting from the expression of $\mc I_4$ in \eqref{eq:EE1}, we commute $\p_\beta$ with $S^{-1/\alpha}$ and integrate by parts:
 \begin{align} \label{eq:Dmain}
     \mc I_4 = \frac{r^{1 + \frac{1}{\alpha}} }{\alpha^{1/\alpha}} e^{-\delta_{\rm{dis}} s } \sum_{i=1}^{3} \sum_{\beta \in [3]^K} 
     \left( \underbrace{ \int \phi^K \p_\beta U_i \left[ \p_\beta, \frac{1}{S^{1/\alpha}} \right] \Delta U_i }_{\mc D_{1, \beta, i}}
     - \underbrace{ \int \frac{\phi^K}{S^{1/\alpha}} |\nabla \p_\beta U_i|^2 }_{\mc D_{2, \beta, i}}
      -\underbrace{ \int \nabla \left( \frac{\phi^K}{S^{1/\alpha}} \right) \nabla \p_\beta U_i \p_\beta U_i }_{\mc D_{3, \beta, i}} \right).
 \end{align}
Clearly $\mc D_{2, \beta, i } > 0$, so that term already has the correct sign. We will bound $\mc D_{1, \beta, i}$ and $\mc D_{3, \beta, i}$. Let us start with the latter: 
\begin{align}
    | \mc D_{3, \beta, i} | &=\left| \int \left( \frac{K \phi^{K-1} \nabla \phi}{S^{1/\alpha}} - \frac{1}{\alpha} \cdot \frac{\phi^K \nabla S}{S^{1+1/\alpha}} \right) \nabla \p_\beta U_i \p_\beta U_i \notag \right|\\
    &\leq D_{2, \beta, i}^{1/2} \left[ K \left( \int \frac{\phi^{K-2} |\nabla \phi |^2}{S^{1/\alpha} } | \p_\beta U_i |^2 \right)^{1/2} + \frac{1}{\alpha} \left( \int \phi^K \frac{|\nabla S|^2}{S^{2+1/\alpha}} | \p_\beta U_i |^2 \right)^{1/2} \right] \notag \\
    &\leq D_{2, \beta, i}^{1/2} E^{1/2} \left( K \left\| \frac{| \nabla \phi |^2}{\phi^2 S^{1/\alpha } }\right\|_{L^\infty} + \frac{1}{\alpha} \left\| \frac{| \nabla S |^2}{S^{2+1/\alpha}} \right\|_{L^\infty} \right). \label{eq:suriname1}
\end{align} 

Using Lemmas \ref{lemma:Sbounds} and \ref{lemma:Sprimebounds}, we see that
\begin{align} \label{eq:suriname2}
    \frac{| \nabla \phi |^2}{\phi^2 S^{1/\alpha} } \les_{\delta_0}  \frac{1}{\langle R \rangle^2 \cdot \langle R \rangle^{-(r-1)/\alpha}} \leq 1, \qquad \mbox{ and } \qquad 
    \frac{| \nabla S |^2}{S^{2+1/\alpha}} \les_{\delta_0}  \frac{\langle R \rangle^{-2r}}{\langle R \rangle^{-(2+1/\alpha)(r-1)}} \leq 1,
\end{align}
where in the last inequality of both equations we used $2>\frac{r-1}{\alpha}$ (due to equation \eqref{eq:range_r}). 

Plugging \eqref{eq:suriname2} into \eqref{eq:suriname1}, and using that $\delta_0$ is sufficiently small depending on $E, K$, we obtain
\begin{equation} \label{eq:suriname3}
    |\mc D_{3, \beta, i}| \les_{\delta_0} \mc D_{2, \beta, i}^{1/2} 
\end{equation}

We move to study $\mc D_{1, \beta, i}$ in \eqref{eq:Dmain}. Note that
\begin{equation*}
    \left| \left[ \p_\beta, \frac{1}{S^{1/\alpha}} \right] f - \sum_{j=1}^K \p_{\beta_j} \frac{1}{S^{1/\alpha}} \p_{\beta^{(j)}}f \right| \les 3^K\max_{\substack{ |\beta| + | \beta ' | = K \\ | \beta | \leq K-2}} |\p_{\beta'} \left( \frac{1}{S^{1/\alpha}} \right) \p_\beta f|
\end{equation*}
just by identifying the terms of order $K-1$ in the commutator. Using this into the expression of $\mc D_{1,\beta, i}$, we obtain
\begin{equation} \label{eq:chile}
   \mc D_{1, \beta, i} \les \underbrace{ \sum_{j=1}^K \int \phi^K \p_\beta U_i \p_{\beta^{(j)}} \Delta U_i \frac{\p_{\beta_j} S}{S^{1+1/\alpha}} }_{\mc D_{1, \beta, i}' }  + 3^K E^{1/2} \underbrace{ \max_{\substack{ |\beta| + | \beta ' | = K \\ | \beta | \leq K-2}} \left\| \p_{\beta'} \left( \frac{1}{S^{1/\alpha}} \right) \p_\beta \Delta U_i \phi^K \right\|_{L^2}. }_{\mc D_{1, \beta, i}'' }
\end{equation}

Defining 
\begin{equation*}
\mc D_2 =   \sum_{\beta \in [3]^K} \sum_{i=1} \mc D_{2, \beta, i} = \int | \nabla^{K+1}  U |^2 \frac{\phi^K}{S^{1/\alpha} }, 
\end{equation*}
we can bound $\mc D_{1, \beta, i}'$ as
\begin{align}
\mc D_{1, \beta, i}' &\les K \int \phi^K | \p_\beta U | \cdot |\nabla^{K+1} U | \frac{| \nabla S |}{S^{1+1/\alpha}} \les_{K} \mc D_2^{1/2} \left( \int \phi^K |\p_\beta U|^2 \frac{| \nabla S |^2}{S^{2+1/\alpha}} \right)^{1/2} \notag \\
&\les_{K}
\mc D_2^{1/2} E^{1/2} \left\| \frac{| \nabla S |^2}{S^{2+1/\alpha}} \right\|_{L^\infty} \les_{\delta_0} \mc D_2^{1/2}. \label{eq:kamchatka}
\end{align}
where in the last inequality we used \eqref{eq:suriname2} and the fact that $\delta_0$ is allowed to depend on $E$.

Using Lemma \ref{lemma:brasilia}, we proceed to bound $\mc D_{1, \beta, i}''$. Developing the derivative $\p_{\beta ' } \frac{1}{S^{1/\alpha}}$, we have
\begin{equation}  \label{eq:stockholm}
\mc D_{1, \beta, i}'' = \max_{\substack{ |\beta| + | \beta ' | = K \\ | \beta | \leq K-2}} \left\| \p_{\beta'} \left( \frac{1}{S^{1/\alpha}} \right) \p_\beta \Delta U_i \phi^{K/2} \right\|_{L^2} 
\les_K \max_{\substack{ |\beta^0| + | \beta^1 |  + \ldots + | \beta^\ell |= K \\ | \beta^0 | \leq K-2}}
\left\| \frac{1}{S^{1/\alpha}} \p_{\beta^0} \Delta U_i \frac{\p_{\beta^1} S}{S} \frac{\p_{\beta^2} S}{S} \ldots \frac{\p_{\beta^\ell} S}{S}  \phi^{K/2} \right\|_{L^2}
\end{equation}
We consider different cases.

\textbf{Generic case. $| \beta^0 | \leq K-4$ and $| \beta^i | \leq K-2$ for all $i \geq 1$.} In this case, we always have less than $K-2$ derivatives falling on each term. Let us define
\begin{equation*}
F(j) = -\frac{-r(j-1) + (1-\eta) \frac{5j}{2} + K\eta - 5/2 }{K-5/2} + \epsilon^\sharp ,
\end{equation*}
so a direct application of inequality \eqref{eq:brazil3} and Lemma \ref{lemma:Sbounds} yields
\begin{align*}
&\left\| \frac{\p_{\beta^0} \Delta U_i }{S^{1/\alpha}} \frac{\p_{\beta^1} S}{S} \frac{\p_{\beta^2} S}{S} \ldots \frac{\p_{\beta^\ell} S}{S}  \phi^{K/2} \right\|_{L^2} \\
&\qquad\leq \left\| \frac{\phi^{\frac{| \beta^0 |+ 2}{2} }\p_{\beta^0} \Delta U_i}{S  \langle R \rangle^{F(|\beta^0|+2 )}} \right\|_{L^\infty}
\left\| \frac{\phi^{ \frac{| \beta^1 |}{2} }\p_{\beta^1} S}{S \langle R \rangle^{F(|\beta^1| )}} \right\|_{L^\infty} \ldots
\left\| \frac{\phi^{\frac{| \beta^\ell |}{2} }\p_{\beta^\ell} S}{S \langle R \rangle^{F(|\beta^\ell| )}} \right\|_{L^\infty} 
\left\| \frac{\langle R \rangle^{F(|\beta^0|+2)} \langle R \rangle^{ \sum_{i=1}^\ell F(| \beta^i |)}}{\phi S^{1/\alpha - 1}  } \right\|_{L^2}  \\
&\qquad\les_{\delta_0}
 \left\| \frac{\langle R \rangle^{ -\frac{-r(K+1-\ell) + (1-\eta) \frac{5(K+2)}{2} + K\eta (\ell+1) - 5(\ell+1) /2 }{K-5/2}+ \epsilon^\sharp (\ell + 1)}   }{\phi \langle R \rangle^{-(r-1)(1/\alpha - 1)}} \right\|_{L^2}.
\end{align*}
Let us look at the exponent in the numerator of the expression above and define
\begin{align*}
\Xi &= -\left( -r(K+1-\ell) + (1-\eta) \frac{5(K+2)}{2} + K\eta (\ell+1) - 5 (\ell+1)/2 \right) \\
&= 
 rK + r -r\ell - \frac{5K}{2}  + \left( \frac{3}{2} - \ell \right)K\eta + 5 \ell/2  -5/2+5\eta.
\end{align*}
If $\ell \geq 2$, since $K$ is sufficiently large with respect to $\ell$, we have that $(3/2 - \ell)K\eta < -100\ell$ and we obtain
\begin{equation}
\Xi \leq \left( r - \frac52 \right)K.
\end{equation}
If $\ell = 1$, we get that
\begin{equation}
\Xi =  \left( r - \frac52 \right) K + \frac12 K\eta + 5\eta \leq \left( r - \frac52 + \eta \right)K,
\end{equation}
using that $K\eta \gg \eta$. Thus, we have
\begin{align}\label{epsicondition01}
\left\| \frac{\p_{\beta^0} \Delta U_i }{S^{1/\alpha}} \frac{\p_{\beta^1} S}{S} \frac{\p_{\beta^2} S}{S} \ldots \frac{\p_{\beta^\ell} S}{S}  \phi^{K/2} \right\|_{L^2} 
&\les_{\delta_0}
 \left\| \frac{\langle R \rangle^{(r-5/2+\eta + \epsilon^\sharp(K+1)) } }{ \langle R \rangle^{2(1-\eta)-(r-1)(1/\alpha - 1)}} \right\|_{L^2} \\\nonumber
 &= \left( \int \langle R \rangle^{-9+2r+2(r-1)(1/\alpha-1)+6\eta + 2(K+1)\epsilon^\sharp}dR \right)^{1/2}.
\end{align}
Finally, we use that $\eta$ is sufficiently small since $\frac{r-1}{\alpha} \leq 2$ (from equation \eqref{eq:range_r}), $\epsilon^\sharp$ is sufficiently small depending on $K$, $2r < 4$ and $\langle R \rangle^{-3+\epsilon^\sharp}$ is integrable. We obtain 
\begin{equation*}
\left\| \frac{\p_{\beta^0} \Delta U_i }{S^{1/\alpha}} \frac{\p_{\beta^1} S}{S} \frac{\p_{\beta^2} S}{S} \ldots \frac{\p_{\beta^\ell} S}{S}  \phi^{K/2} \right\|_{L^2} \les_{\delta_0} 1
\end{equation*}

\textbf{Case where some $| \beta^i | \in \{ K-1, K \}$ for $i \geq 1$.} Without loss of generality assume $i=1$. If $| \beta^1 | = K$, using the energy bound
\begin{equation*}
    \left\| \phi^{K/2} \p_{\beta^1} S \frac{\Delta U_i}{S^{1+1/\alpha}}  \right\|_{L^2} \les_{\delta_0} \left\| \frac{| \nabla^2 U |}{S^{1+1/\alpha}} \right\|_{L^\infty}
\end{equation*}
and if $| \beta^1 | = K-1$, using \eqref{eq:brazilk-1} the two cases are:
\begin{align*}
    \left\| \phi^{K/2} \p_{\beta^1} S \frac{\p_j S \Delta U_i}{S^{2+1/\alpha}}  \right\|_{L^2} &\les_{\delta_0,\epsilon^\sharp} \|\langle R\rangle^{K(1-\eta)\frac{K-2}{K-1}-\epsilon^\sharp}\partial_{\beta^{1}}S \|_{L^{2+\frac{2}{K-2}}}\left\|  \frac{ | \nabla S |\langle R\rangle}{S} \right\|_{L^\infty} \left\| \frac{\langle R\rangle^{K(1-\eta)\frac{1}{K-1}+\epsilon^\sharp-1} |\nabla^2 U|}{S^{1+1/\alpha}} \right\|_{L^{2(K-1)}},  \\
    \left\| \phi^{K/2} \p_{\beta^1} S \frac{\p_j \Delta U_i}{S^{1+1/\alpha}}  \right\|_{L^2} &\les_{\delta_0,\epsilon^\sharp} \|\langle R\rangle^{K(1-\eta)\frac{K-2}{K-1}-\epsilon^\sharp}\partial_{\beta^{1}}S \|_{L^{2+\frac{2}{K-2}}}  \left\| \frac{ \langle R\rangle^{K(1-\eta)\frac{1}{K-1}+\epsilon^\sharp} | \nabla^3 U | }{S^{1+1/\alpha}} \right\|_{L^{2(K-1)}}.
\end{align*}
On the one hand, combining Lemma \ref{lemma:Sbounds} with Lemma \ref{lemma:Sprimebounds}, we see that $\left\|  \frac{ | \nabla S |\langle R \rangle}{S } \right\|_{L^\infty} \les_{\delta_0} 1$. 
 On the other hand, using \eqref{eq:brazil2} for $j = 2, 3$ together with Lemma \ref{lemma:Sbounds}, we get
\begin{align} \label{eq:grenada}
&\left\| \frac{\langle R\rangle^{K(1-\eta)\frac{1}{K-1}+\epsilon^\sharp-1}| \nabla^2 U |}{S^{1+1/\alpha}} \right\|_{L^{2(K-1)}} + \left\| \frac{\langle R\rangle^{K(1-\eta)\frac{1}{K-1}+\epsilon^\sharp} | \nabla^3 U | }{S^{1+1/\alpha}} \right\|_{L^{2(K-1)}}\\\nonumber
& \quad \les_{\delta_0,\epsilon^\sharp}
\left\| \frac{R ^{-r-2+2\eta+\epsilon^\sharp+K(1-\eta)\frac{1}{K-1}+2\epsilon^\sharp}}{\langle R \rangle^{-(r-1)(1+1/\alpha)}} \right\|_{L^{2(K-1)}} + \left\| \frac{ R ^{-r-2+3\eta+\epsilon^\sharp+K(1-\eta)\frac{1}{K-1}+2\epsilon^\sharp}}{\langle R \rangle^{ -(r-1)(1+1/\alpha)}} \right\|_{L^{2(K-1)}} \les_{\delta_0,\epsilon^\sharp} 1,
\end{align}
where in the last inequality we used $(r-1)\left( 2 + \frac{1}{\alpha}  \right) < 2$ (equivalent to $r - 1 < 1-1/\gamma$ from equation \eqref{eq:rough_range_r}) and $r>1$.

\textbf{Case where $| \beta^0  | = K-2$.}
We can bound 
 \begin{align*}
\left\| \p_{\beta^0} \Delta U   \phi^{K/2} \nabla^2 \left( \frac{1}{S^{1/\alpha}} \right) \right\|_{L^2} 
&\leq \| \phi^{K/2} \nabla^K U \|_{L^2} \left( \left\| \frac{\nabla^2 S}{S^{1+1/\alpha}}\right\|_{L^\infty} + \left\| \frac{|\nabla S|^2}{S^{2+1/\alpha}}\right\|_{L^\infty} \right) \\
&\leq E \left( \left\| \frac{\nabla^2 S}{S^{1+1/\alpha}}\right\|_{L^\infty} + \left\| \frac{\langle R\rangle^{-2r}}{\langle R \rangle^{-(r-1)(2+1/\alpha)}}\right\|_{L^\infty} \right) \les_{\delta_0} 1
\end{align*}
where in the last inequality we used again that $2 > (r-1)(2+\frac{1}{\alpha})$ (from equation \eqref{eq:range_r}. We also bounded $\left\| \frac{\nabla^2 S}{S^{1+1/\alpha}}\right\|_{L^\infty}$ in the same way as $\left\| \frac{\nabla^2 U}{S^{1+1/\alpha}}\right\|_{L^\infty}$ in the last case.

\textbf{Case where $| \beta^0  | = K-3$.}
Using \eqref{eq:brazilk-1}, together with $S \gtrsim_{\delta_0} \langle R \rangle^{-r+1}$ (Lemma \ref{lemma:Sbounds}) and $| \nabla S | \les_{\delta_0} \langle R \rangle^{-r}$ (Lemma \ref{lemma:Sprimebounds}), we have that
\begin{align*}
&\quad\left\| \p_{\beta^0} \Delta U   \phi^{K/2} \nabla^3 \left( \frac{1}{S^{1/\alpha}} \right) \right\|_{L^2}  \\
&\leq
\left\| \langle R\rangle^{K(1-\eta)\frac{K-2}{K-1}-\epsilon^\sharp} \nabla^{K-1} U \right\|_{L^{2+\frac{2}{K-2}}}\\
&\quad\cdot\left( \left\| \frac{\langle R\rangle^{K(1-\eta)\frac{1}{K-1}+\epsilon^\sharp}\nabla^{3} S}{S^{1+1/\alpha}}\right\|_{L^{2(K-1)}} + \left\| \frac{\langle R\rangle^{K(1-\eta)\frac{1}{K-1}+\epsilon^\sharp}|\nabla^{2} S| | \nabla S |}{S^{2+1/\alpha}}\right\|_{L^{2(K-1)}} + \left\| \frac{\langle R\rangle^{K(1-\eta)\frac{1}{K-1}+\epsilon^\sharp}|\nabla S|^3}{S^{3+1/\alpha}}\right\|_{L^{2(K-1)}} \right) \\
&\les_{\delta_0,\epsilon^\sharp} 
 \left( \left\| \frac{\langle R\rangle^{K(1-\eta)\frac{1}{K-1}+\epsilon^\sharp}\nabla^{3} S}{S^{1+1/\alpha}}\right\|_{L^{2(K-1)}} + \left\| \frac{\langle R\rangle^{K(1-\eta)\frac{1}{K-1}+\epsilon^\sharp}|\nabla^{2} S| }{S^{1+1/\alpha}}\right\|_{L^{2(K-1)}} + \left\| \frac{\langle R\rangle^{K(1-\eta)\frac{1}{K-1}+\epsilon^\sharp-1}\langle R \rangle^{-2r} }{ \langle R \rangle^{-(r-1)(2+1/\alpha)}}\right\|_{L^{2(K-1)}} \right).
\end{align*}
The same reasoning applied in \eqref{eq:grenada} for $U$ applies here for $S$, and using again that $2 > (r-1)(2+1/\alpha)$ (from equation \eqref{eq:range_r}), we conclude that
\begin{equation*}
\left\| \p_{\beta^0}\Delta U   \phi^{K/2} \nabla^3 \left( \frac{1}{S^{1/\alpha}} \right) \right\|_{L^2} \les_{\delta_0} 1
\end{equation*}

Combining all our different cases and using \eqref{eq:stockholm}, we obtain
\begin{equation} \label{eq:albacete}
\mc D_{1, \beta, i}'' \les_{\delta_0} 1.
\end{equation}

Using \eqref{eq:suriname3}, \eqref{eq:chile}, \eqref{eq:kamchatka}, and \eqref{eq:albacete}, we obtain that
\begin{equation} \label{eq:calamares}
\mc D_{1, \beta, i} + |\mc D_{3, \beta, i}| \les_{\delta_0} \mc D_2^{1/2} + 1.
\end{equation}
Recall that our objective is to bound
\begin{equation*}
\mc I_4 = C_{\rm{dis}} e^{-\delta_{\rm{dis}}s} \sum_{i=1}^3 \sum_{\beta \in [3]^{K}} \left( \mc D_{1, \beta, i} - \mc D_{2, \beta, i} - \mc D_{3, \beta, i} \right),
\end{equation*}
and that $\mc D_2 = \sum_{\beta \in [3]^K} \sum_{i=1}^3 \mc D_{2, \beta, i}$. Using \eqref{eq:calamares} and that $\delta_0$ is allowed to depend on $r, \alpha, K$, we have that
\begin{equation*}
 \mc I_4 + C_{\rm{dis}} e^{-\delta_{\rm{dis}}s} \mc D_2  \leq C_{\delta_0} e^{-\delta_{\rm{dis}}s} \left( \mc D_2^{1/2} + 1 \right), 
\end{equation*}
where $C_{\delta_0}$ is sufficiently large depending on $\delta_0$ (and therefore on $r, \alpha, K$). Using Cauchy-Schwarz, we get
\begin{equation*}
 \mc I_4 + C_{\rm{dis}} e^{-\delta_{\rm{dis}}s} \mc D_2  \leq C_{\delta_0} e^{-\delta_{\rm{dis}}s} + e^{-\delta_{\rm{dis}}s} \left( \mc D_2\frac{C_{\rm{dis}}}{2} + C_{\delta_0}^2 \frac{2 }{C_{\rm{dis}}} \right), 
\end{equation*}
and therefore
\begin{equation*}
\mc I_4 \leq \frac{r^{1+\frac{1}{\alpha}}}{\alpha^{\frac{1}{\alpha}}} e^{-\delta_{\rm{dis}}s} \left( - \mc D_2 + \frac12 \mc D_2 \right) + e^{-\delta_{\rm{dis}}s} \left( C_{\delta_0} + C_{\delta_0}^2 \frac{2 }{C_{\rm{dis}}} \right) \les_{\delta_0} e^{-\delta_{\rm{dis}}s},
\end{equation*}
where we used that $\mc D_2 > 0$. Since $s_0$ is sufficiently large depending on $\delta_0$, we conclude that
\begin{equation} \label{eq:I4}
\mc I_4 \leq e^{-\delta_{\rm{dis}} (s-s_0)}.
\end{equation}

Plugging this into \eqref{eq:EE2}, we have
\begin{equation} \label{energyestimatewithweight}
    \p_s E_{K} \leq -\frac{K\min \{ \eta , \tilde \eta \} }{2} E_{K} + O(E) + e^{-\delta_{\rm{dis}}(s-s_0)}.
\end{equation}
Given that $K$ is sufficiently large depending on $\eta, \tilde \eta$, when $E_k\geq \frac{1}{2}E$, we have $\partial_{s}E_{k}<0$ and we close the bootstrap assumption \eqref{higherestimate3}.

\subsection{Topological argument section for the unstable modes} \label{subsec:topological}

In this section we control the unstable modes.

Recall that $(\varphi_{i,u},\varphi_{i,s})$, $1\leq i\leq N $ is a normalized basis of $V_{\rm{uns}}$. For any $(\tilde{U}_{t}(\cdot, s),\tilde{S}_{t}(\cdot,s))$, we denote 
\[
k(s)=P_{\rm{uns}}(\tilde{U}_{t}(\dot,s),\tilde{S}_{t}(\cdot,s))=\sum_{i=1}^{N}k_i(s)(\varphi_{i,u},\varphi_{i,s}).
\]
Due to the initial data condition \eqref{initialdatatruncated}, together with ($\chi_2 \tilde U_0', \chi_2 \tilde S_0' ) \in V_{\rm{sta}}$ (by \eqref{eq:tilde_is_stable}), we have that 
\begin{equation*}
k_i(0)=a_i.
\end{equation*} 

From Lemma \ref{prop:maxdissmooth}, there exists a metric $B$ such that 
\[
\langle \mathcal{L}k(s),k(s) \rangle_{B}\geq \frac{-6\delta_g}{10}\langle k(s),k(s) \rangle_{B}.
\]
Since the unstable space $V_{\rm{uns}}$ only has finite dimension, $X\subset H^m (B(0, 3C_0))$ we have
\begin{equation}\label{normequivalence}
\|k(s)\|_{X}\lesssim_{m}\|k(s)\|_{B}\lesssim_{m}\|k(s)\|_{X}.
\end{equation}
where we denote with $\| a \|_X$ the norm $\left\| \sum a_i \psi_i \right\|_X$. Here $m$ shows up because of \eqref{spacexnorm}.

We will always restrict ourselves to values of $a_i$ such that 
\begin{equation*}
\left\| a \right\|_B \leq \delta_1^{11/10}.
\end{equation*}
In particular, from \eqref{normequivalence}, we have $\| a \|_X \les_m \delta_1^{11/10}$. Recalling that $\psi_i$ are compactly supported, that $\delta_1$ is sufficiently small depending on $m, E$ and that $\| \psi_i \|_X = 1$ (normalized eigenfunctions), it is clear that the required initial conditions for $\tilde U_0, \tilde S_0$ \eqref{initialdatacondition2}--\eqref{initialdatacondition} are satisfied for any such choice of $a_i$ as long as 
\begin{align} \begin{split} \label{eq:conditions_for_tilde}
\|\tilde{U}_0' \|_{L^{\infty}}, \|\tilde{S}_0' \|_{L^{\infty}}\leq  \delta_1, \qquad \| \tilde U_0' \|_{H^K}, \| \tilde S_0' \|_{H^K} \leq\frac{E}{4}, 
\qquad \tilde{S}_0'+\hat{X}\bar{S}_0\geq \frac{\delta_1}{2} \left\langle \frac{R}{R_0} \right\rangle^{1-r}, \\
|\nabla(\tilde{U}_0'+\hat{X} \bar{U})|+|\nabla(\tilde{S}_0'+\hat{X} \bar{S})|\lesssim \frac{1}{\langle |y| \rangle^r}.
\end{split} \end{align}

Let
\[
\mc R(s)=\{k(s)| \|k(s)\|_{X}\leq \delta_1 e^{-\frac{4}{3}\epsilon(s-s_0)}\},
\]
\[
\tilde{\mc R}(s)=\{k(s)| \|k(s)\|_{B}\leq \delta_1^{\frac{11}{10}}e^{-\frac{4}{3}\epsilon(s-s_0)}\}.
\]
From \eqref{normequivalence} and the choice of parameter \eqref{choiceofparameter}, we have
\[
\tilde{\mc R}(s)\Subset \mc R(s).
\]

We will always choose $\{ a_i \} = k(s_0)\in \tilde{\mc R}(s_0)$. From Proposition \ref{prop:bootstrap}, and Lemmas \ref{forcingestimate1}, \ref{forcingestimate2}, when $k(s)$ does not leave $\mc R(s)$, the unstable mode is controlled and we have the estimate for the forcing: 
\begin{equation}\label{forcingest}
\|\chi_2 \mc F\|_{X}\lesssim \delta_1^{\frac{6}{5}}e^{-\frac{3}{2}\epsilon(s-s_0)}.
\end{equation}
Next, we show the following outgoing property:
\begin{proposition}
Suppose $k(s)\in \tilde{\mc R}(s)$ at times $s\in[s_0,s_1]$ and at time $s_1$, we have $k(s_1)\in \partial{\tilde{\mc R}(s_1)},$ that is 
\[
\|k(s_1)\|_{B}=\delta_1^{\frac{11}{10}}e^{-\frac{4}{3}\epsilon(s_1-s_0)}.
\]
Then, we have that $k(s)$ does not belong to $\tilde{\mc R}(s)$ for s close enough to $s_1$ from  above. 
\end{proposition}
\begin{proof}
Recall that $(\tilde{U}_{t},\tilde{S}_{t})$ satisfy the equation \eqref{navierstokesperturb1trun}:
\begin{align*}
\p_s \tilde U_t &= \mc L_{u}(\tilde{U}_t,\tilde{S}_t)+\chi_2\mc F_{u}(\tilde{U},\tilde{S}),\\\nonumber
\p_s \tilde S_t &= \mc L_{s}(\tilde{U}_t,\tilde{S}_t)+\chi_2\mc F_{s}(\tilde{U},\tilde{S}),
\end{align*}
with initial value \eqref{initialdatatruncated}:
\begin{align}
\begin{cases}
\tilde{U}_{t,0}=\chi_2\tilde{U}_0^{'}+\sum_{i=1}^{N}a_i
\varphi_{i,u}\\
\tilde{S}_{t,0}=\chi_2\tilde{S}_0^{'}+\sum_{i=1}^{N}a_i
\varphi_{i,s}.
\end{cases}
\end{align}
Since $V_{\rm{uns}}$ is invariant with respect to the linear operator $\mc L$, we have
\begin{align*}
\begin{cases}
\p_s k(s) &= \mc L k(s) + P_{\rm{uns}}(\chi_2 \mc F), \\
k(s_0) &= (\sum_{i=1}^{N}a_i
\varphi_{i,u},\sum_{i=1}^{N}a_i
\varphi_{i,s}),
\end{cases}
\end{align*}
From the forcing estimate \eqref{forcingest}, we have
\begin{align*}
&\|P_{\rm{uns}}(\chi_2 \mc F)\|_{B}\lesssim_{m}\|P_{\rm{uns}}(\chi_2 \mc F)\|_{X}\lesssim_{m} \|\chi_2 \mc F\|_{X}\lesssim_{m}\delta_{1}^{\frac{6}{5}}e^{-\frac{3}{2}\epsilon(s_1-s_0)}.
\end{align*}
Then 
\begin{align*}
\left\langle \frac{d}{ds}k(s),k(s) \right\rangle_{B}\bigg|_{s=s_1}\geq -\frac{6\delta_g}{10}\langle k(s_1),k(s_1) \rangle_{B}-C_m  \delta_{1}^{\frac{6}{5}}e^{-\frac{3}{2}\epsilon(s_1-s_0)}\|k(s_1)\|_{B}.
\end{align*}
Then when $\|k(s_1)\|_{B}=\delta_1^{\frac{11}{10}}e^{-\frac{4}{3}\epsilon(s-s_0)}$, we have
\begin{align*}
\left\langle \frac{d}{ds}k(s),k(s) \right\rangle_{B}\bigg|_{s=s_1}&\geq -\frac{6\delta_g}{10}\langle k(s_1),k(s_1) \rangle_{B}-C_m\delta_{1}^{\frac{6}{5}-\frac{11}{10}}e^{-\epsilon(\frac{3}{2}-\frac{4}{3})(s_1-s_0)} \langle k(s_1),k(s_1) \rangle_{B}.\\
&\geq \frac{-31}{50}\delta_g \langle k(s_1),k(s_1) \rangle_{B}.
\end{align*}
Since $\frac{31}{50}\delta_g< \frac{4}{3}\epsilon=\frac{16}{25}\delta_g,$ 
 when $s=s_1$ we have
\[
\frac{d}{ds}(\langle k(s),k(s)\rangle_{B}e^{\frac{8}{3}\epsilon(s-s_0)})\geq \frac{-31}{25}\delta_g\langle k(s),k(s)\rangle e^{\frac{8}{3}\epsilon(s-s_0)}+\frac{8}{3} \epsilon\langle k(s),k(s) \rangle_{B}e^{\frac{8}{3}\epsilon(s-s_0)}>0.
\]

Hence $k(s)$ exits $\tilde{\mc R}(s)$ at $s=s_1$.
\end{proof}

Next, we show the existence of at least one choice of $\{a_i\}$ such that $k(s)$ is in $\mc R(s)$ for all $s$. We borrow the result from \cite[Proposition 8.15]{Buckmaster-CaoLabora-GomezSerrano:implosion-compressible} (the same proof works here).
\begin{proposition} \label{prop:ai}
There exist specific initial conditions $\{a_i\} \in \tilde{\mc R} (s_0)$ such that $ k(s_0)=\{a_i\}$ and $k(s)\in \mc R(s)$ for all $s>s_0$.
\end{proposition}

Thus, for any choice of $\tilde U_0'$, $\tilde S_0'$ satisfying our initial data conditions \eqref{initialdatacondition2}--\eqref{initialdatacondition}, we can take $\{a_i\}$ as in Proposition \ref{prop:ai}. Then, recalling equation \eqref{eq:tildeplusunstable}, the initial data for the perturbation is given by: 
\begin{equation*}
\tilde U_0 = \tilde U_0' + \sum_{i=1}^N a_i \psi_{i, u}, \qquad \mbox{ and } \qquad \tilde S_0 = \tilde S_0' + \sum_{i=1}^N a_i \psi_{i, s}.
\end{equation*}
For such initial data we are under the hypothesis of Proposition \ref{prop:bootstrap} for all $s_1 \geq s_0$. Therefore, from \eqref{lowerestimate} we obtain the exponential decay bounds
\begin{equation*}
| \tilde U |, | \tilde S| \leq \frac{\delta_0}{C_2} e^{-\eps (s-s_0)}.
\end{equation*}

Let us also recall from \eqref{eq:periodic_perturbation} that $U = \tilde U (y, s) + X(ye^{-s}) \bar U (y)$ and $S = \tilde S (y, s) + X(ye^{-s}) \bar S (y)$. Using equations \eqref{eq:SS_coordinates1}--\eqref{eq:SS_coordinates2} to change back from self-similar variables to original physical variables, we obtain that 
\begin{align*}
u(x, t) &= \frac{X(x)}{r}\cdot \frac{1}{(T-t)^{1-\frac{1}{r}}} \bar U \left( \frac{x}{(T-t)^{1/r}} \right) + \frac{1}{(T-t)^{1-\frac{1}{r}-\frac{\eps}{r}}} U_{\rm{err}} \left( \frac{x}{(T-t)^{1/r}}, t \right), \\
\sigma (x, t) &= \frac{X(x)}{r}\cdot \frac{1}{(T-t)^{1-\frac{1}{r}}} \bar S \left( \frac{x}{(T-t)^{1/r}} \right) + \frac{1}{(T-t)^{1-\frac{1}{r}-\frac{\eps}{r}}} S_{\rm{err}} \left( \frac{x}{(T-t)^{1/r}}, t \right),
\end{align*}
where
\begin{equation*}
| U_{\rm{err}} |, | S_{\rm{err}} | \leq \frac{r \delta_0}{C_2}, \qquad \forall x \in \mathbb T_L^3, \; \mbox{ and } \; \forall t \in [0, T)
\end{equation*}
Finally, recalling that $\sigma = \frac{1}{\alpha} \rho^\alpha$, we obtain Theorem \ref{th:periodic}.

\begin{remark} \label{rem:codimension}
The only conditions we have imposed on $(\tilde U_0', \tilde S_0')$ are \eqref{eq:conditions_for_tilde} and \eqref{eq:tilde_is_stable}. Conditions \eqref{eq:conditions_for_tilde} are open, so we can consider some Banach space $\mc B$ including all the norms considered in \eqref{eq:conditions_for_tilde} so that \eqref{eq:conditions_for_tilde} is satisfied in a sufficiently small ball of $\mc B$ centered at the origin that we denote by $B_{\mc B}$. 

On the other hand, the restriction from \eqref{eq:tilde_is_stable}, that is, $P_{\rm{uns}} (\tilde U_0', \tilde S_0') = 0$, is a finite dimensional closed restriction. Since $V_{\rm{uns}}$ is the range of $P_{\rm{uns}}$, we know that there is always a representative of $(\tilde U_0', \tilde S_0')$ in the quotient space $\mc B / V_{\rm{uns}}$ that satisfies \eqref{eq:tilde_is_stable}.

Moreover, note that from the choice of initial data
\begin{equation*}
\left( \tilde U_0, \tilde S_0 \right) = (\tilde U_0', \tilde S_0') + \sum_{i=1}^N a_i \psi_i,
\end{equation*}
that $(\tilde U_0, \tilde S_0)$ and $(\tilde U_0', \tilde S_0')$ belong to the same equivalence class of $\mc B / V_{\rm{uns}}$. 

Therefore, the set of $(\tilde U_0, \tilde S_0)\in \mc B$ for which we show singularity formation covers an open set of the quotient space $\mc B / V_{\rm{uns}}$. Since $V_{\rm{uns}}$ is finite dimensional, we refer to the set where blow-up occurs as a finite codimension set of initial data.
\end{remark}

\section{Appendix}

\begin{lemma}[Angular repulsivity] \label{lemma:angular_repulsivity} We have that the profiles $(\bar U, \bar S)$ from \cite{Buckmaster-CaoLabora-GomezSerrano:implosion-compressible} or \cite{Merle-Raphael-Rodnianski-Szeftel:implosion-i} satisfy:
\begin{equation}
1 + \frac{\bar U_R}{R} - \alpha | \p_R \bar S | > \eta,
\end{equation}
for some $\eta > 0$ sufficiently small that is allowed to depend on $r$ or $\gamma$.
\end{lemma}

\begin{proof}
Before showing the Lemma, let us introduce some notation and properties of the known self-similar profiles (these properties apply to both profiles from \cite{Merle-Raphael-Rodnianski-Szeftel:implosion-i} and \cite{Buckmaster-CaoLabora-GomezSerrano:implosion-compressible}). 

We will adopt the notation of \cite{Buckmaster-CaoLabora-GomezSerrano:implosion-compressible} and denote $U = \frac{\bar U}{R}$, $S = \frac{\bar S}{R}$. We will also regularly use $W, Z$ coordinates, which are just given by $W = U + S$ and $Z = U-S$. We will also reparametrize by $\xi = \log(R)$. The reason to introduce that extra factor of $\frac{1}{R}$ and the $\xi$ reparametrization is that now the ODEs satisfied by the self-similar profiles in these new coordinates are autonomous:
\begin{equation*}
\p_\xi W (\xi) = \frac{N_W (W, Z)}{D_W (W, Z)}, \qquad \mbox{ and } \quad \p_\xi Z = \frac{N_Z (W, Z)}{D_Z (W, Z)}
\end{equation*}
where $D_W, D_Z$ are first-degree polynomials given by
\begin{equation}
D_W(W, Z) = 1 + \frac{1 + \gamma}{4} W + \frac{3-\gamma}{4}Z, \qquad \mbox{ and }\qquad D_Z(W,Z) = 1 + \frac{3-\gamma}{4} W + \frac{1+\gamma}{4} Z,
\end{equation}
and $N_W, N_Z$ are second degree polynomials given by
\begin{align*}
N_W(W, Z) &= -rW +\frac{\ga-1}{4} Z^2 + \frac{\ga-3}{4} WZ-\frac{\ga W^2}{2}, \\
N_Z(W, Z) &= -r Z +\frac{\ga-1}{4} W^2+ \frac{\ga-3}{4} W Z-\frac{\ga Z^2}{2}.
\end{align*}

Moreover, the family of profiles considered pass through the point $P_s$ which is the rightmost of the two solutions to the second degree equation $N_Z(W, Z) = D_Z(W, Z) = 0$. We define $\bar P_s$ to be the other solution to the system $N_Z(W, Z) = D_Z(W, Z) = 0$. It is also assumed that the profiles pass through $P_s$ at $R=1$ (otherwise, they can be rescaled to do so). The formulas for $P_s$ and $\bar P_s$ are given by
\begin{align*}
P_s = (W_0, Z_0) &= \Big( \frac{\ga^2 r+(\ga+1) \mathcal{R}_1 -3 \ga^2-2 \ga r+10 \ga-3 r-3}{4 (\ga-1)^2}, \\
& \qquad \frac{\ga^2 r+(\ga -3) \mathcal{R}_1 -3 \ga^2-6 \ga r+6 \ga+9 r-7}{4 (\ga-1)^2} \Big)\,, \\ 
\bar P_s = (\bar W_0, \bar Z_0) &= \Big( \frac{\ga^2 r-(\ga+1) \mathcal{R}_1-3 \ga^2-2 \ga r+10 \ga-3 r-3}{4 (\ga-1)^2}, \\
  & \qquad \frac{\ga^2 r+(3-\ga) \mathcal{R}_1 -3 \ga^2-6 \ga r+6 \ga+9 r-7}{4 (\ga-1)^2} \Big)\,,
\end{align*}
where
\begin{equation} \label{eq:R1}
\mathcal{R}_1 = \sqrt{\ga^2 (r-3)^2-2 \ga (3r^2-6r+7)+(9r^2-14r+9)}\,.
\end{equation}

One can now apply the formula $\p_\xi W = \frac{N_W}{D_W}$ at point $P_s$ and get the formula for $W_1$, the first-order Taylor coefficient of the solution at $P_s$. The equation $\p_\xi Z = \frac{N_Z}{D_Z}$ it is however singular at $P_s$ (since both $N_Z$ and $D_Z$ vanish), but taking a derivative one can see that $Z_1 \nabla D_Z (P_s) \cdot (W_1, Z_1) = \nabla N_Z(P_s) \cdot (W_1, Z_1)$. Solving that second-degree equation one gets two possible solutions for $Z_1$ and all the profiles in the family we are considering correspond to the following choice:
\begin{align*}
W_1 &= \frac{\ga \left(-3 \left(\mathcal{R}_1+6\right)-3 \ga (r-3)+2 r\right)+\mathcal{R}_1+5 r+5}{4 (\ga-1)^2} \,,\\
Z_1 &=  \frac{-\left(3 \gamma ^3-7 \gamma ^2+\gamma +11\right) r+\gamma  (\gamma  (9 \gamma -3 \mc R_1-25)+10 \mc R_1-4(\gamma-1) \mc R_2+27)-3 \mc R_1+4 (\gamma-1) \mc R_2-3}{4 (\gamma -1)^2 (\gamma +1)}\,,
\end{align*}
where
\begin{align*}
\mathcal{R}_2 &=\frac{1}{\gamma-1}\bigg(\gamma  ((76-27 \gamma ) \gamma -71)-\left((3 \gamma -5) ((\gamma -5) \gamma +2) r^2\right)+(\gamma  (\gamma  (18 \gamma -52)+50)-8) 
   r\\
   &\qquad+\mc R_1 (9 (\gamma -2) \gamma +((2-3 \gamma ) \gamma +5) r+5)+18\bigg)^{\frac12}\,.
\end{align*}

Another point of interest of the phase portrait is the intersection $N_Z = 0$ and $N_W = 0$. Since both are second-degree polynomials there are at most four solutions to $N_Z = N_W = 0$. Two of them ($(0, 0)$ and $(-r, -r)$) lie on the diagonal $W = Z$, and of the other two only one of them lies in the positive density halfplane $W>Z$. That one is given by:
\begin{equation*}
P_\star = \left( \frac{2 \left(\sqrt{3}-1\right) r}{3 \ga-1}, -\frac{2 \left(1+\sqrt{3}\right) r}{3 \ga-1} \right).
\end{equation*}

Moreover, the profiles considered satisfy that:
\begin{equation} \label{eq:DWDZ}
D_W > 0, D_Z < 0, \quad \forall \xi < 0, \qquad \mbox{ and } \qquad D_W > 0, D_Z > 0 \quad \forall \xi > 0.
\end{equation} 
We also sometimes call these regions 'left of the phase portrait' (region where $D_W > 0, D_Z < 0$, corresponding to $ R<1$) and 'right of the phase portrait' (region where $D_W, D_Z > 0$, corresponding to $R>1$).

For the convenience of the reader, we refer to Figure \ref{fig:phase_portrait}, where have plotted the phase portrait that the profile $(W, Z)$ satisfies, together with all the definitions that we have introduced.

\begin{figure}[h]
\centering
\includegraphics[width=0.5\textwidth]{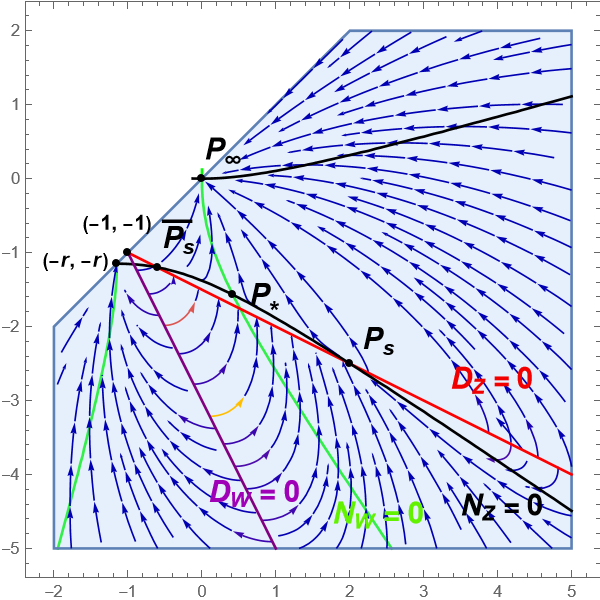}
\caption{Phase portrait of the profile for $\gamma = 5/3$, $r=1.15$}
\label{fig:phase_portrait}
\end{figure}

Lastly, let us recall from \eqref{eq:range_r} that for the profiles considered (including all the profiles of \cite{Merle-Raphael-Rodnianski-Szeftel:implosion-i} and \cite{Buckmaster-CaoLabora-GomezSerrano:implosion-compressible}), we have $1<r<r^\ast$, where 
\begin{equation} \label{eq:defr}
r^\ast(\gamma) = \begin{cases}
r^\ast_{<5/3} = 1 + \frac{2}{\left( 1 + \sqrt{ \frac{2}{\ga - 1} }\right)^2} \quad \mbox{ for } 1< \ga < 5/3, \\ 
r_{\geq 5/3}^\ast = \frac{3\ga - 1}{2 + \sqrt{3} (\ga - 1)} \quad \mbox{ for } \ga \geq 5/3.
\end{cases}
\end{equation}
We have defined here the quantities $r^\ast_{<5/3}$ and $r_{\geq 5/3}^\ast$ to be given by their respective expressions, so they are defined for all $\gamma > 1$ (however, we just have $r^\ast = r_{<5/3}^\ast$ for $\gamma \leq 5/3$ and $r^\ast = r^\ast_{\geq 5/3}$ for $\gamma \geq 5/3$). Let us also recall that \eqref{eq:defr} implies the following inequality (that can be found in \eqref{eq:rough_range_r}):
\begin{equation} \label{eq:danglars}
1 < r < 2 - \frac{1}{\gamma} = \frac{2\gamma - 1}{\gamma}.
\end{equation}
Inequality \eqref{eq:danglars} is looser than \eqref{eq:defr}, but it will be very useful, since it is simpler than \eqref{eq:defr} and it will suffice in many cases where we need to upper bound $r$.

\textbf{Part I. $R<1$: Right of the phase portrait.}

Let us start showing Lemma \ref{lemma:angular_repulsivity} for $\xi < 0$, that is, $R < 1$. In this region, using the radial repulsivity \eqref{eq:radial_repulsivity}, it suffices to show
\begin{equation}
\frac{\bar U}{R} - \p_R \bar U > 0.
\end{equation}
In our $ W, Z$ variables, this corresponds to showing
\begin{equation*}
\frac{N_W}{D_W} + \frac{N_Z}{D_Z} < 0,
\end{equation*}
or equivalently
\begin{equation} \label{eq:part1}
N_W D_Z + N_Z D_W > 0.
\end{equation}

We consider the branch of $N_W D_Z + N_Z D_W = 0$ that lies in our region $D_W > 0, D_Z < 0$, which actually connects $P_0$ to $P_s$ (recall $P_0$ is a point at infinity). Using variables $ U = \frac{W+Z}{2}, S = \frac{W-Z}{2}$ we have
\begin{equation*}
N_W D_Z + N_Z D_W = \frac{\gamma - 1}{2} S^2 (3 (\gamma - 1) U + 2 r - 2) - 2 U (U+1) (r+U)
\end{equation*}
and the curve solving $N_W D_Z + N_Z D_W = 0$ inside the region $D_W > 0, D_Z < 0$ can be parametrized as
\begin{equation*}
b_S = \frac{2}{\gamma - 1} \sqrt{ b_U+1} \sqrt{r+ b_U} \sqrt{ \frac{ b_U }{ 3  b_U+2(r-1)/(\gamma - 1)} }, \quad \mbox{ and } \quad b_U \in \left(  U(P_s),  \frac{-2(r-1)}{3(\gamma - 1)} \right).
\end{equation*}
where $b_U$ goes from $ U (P_s)$ to $\frac{-2(r-1)}{3(\gamma - 1)}$ (at $P_0$) and $(b_U, b_S)$ describes our curve in $(U, S)$ coordinates. We have that $-1 < U(P_s) < \frac{-2(r-1)}{3(\gamma - 1)} < 0$ by Lemma \ref{lemma:auxiliary}. Thus, it can be seen that all three square roots have positive quantities inside (the first two due to $b_U > -1$, the fraction because both $b_U$ and $\frac{2(r-1)}{3(\gamma - 1)} + b_U$ are negative). We define $b$ to be the barrier defined by $b_U, b_S$ in $W, Z$ coordinates.

Now the approach is the following. We show \eqref{eq:part1} in small neighbourhoods of $P_s$ and $P_0$ using a Taylor expansion around those points.  Then, if \eqref{eq:part1} was to fail at some intermediate point $\xi \in (-\infty, 0)$, it would need to cross $b$ in both directions. However, we will also show that the vector field $(N_W/D_W, N_Z/D_W)$ along $b$ points always to the same side of $b$, making that impossible. \\

\textbf{Part I.1. Constant sign along $b$} \\

We start by showing that that the vector field $(N_W/D_W, N_Z/ D_Z )$ along $b$ points always to the same side of $b$. Equivalently, we can show that $(N_W D_Z, N_Z D_W)$ along $b$ always points to the same side of $b$. Note that $N_W D_Z = -N_Z D_W$ along $b$ (since it solves $N_W D_Z + N_Z D_W = 0$) and that the vector field  $(N_W D_Z, N_Z D_W)$ does not vanish through $b$ except at the endpoint $P_s$. Thus, it suffices to show that $(-1, 1)$ along $b = (b_W, b_Z)$ always points to the same side of $b = (b_W, b_Z)$.

Given that $b$ is defined implicitly by $N_W D_Z + N_Z D_W = 0$, a normal vector at any given point is given by $\nabla ( N_W D_Z + N_Z D_W )$. Thus, we need to show that $(-1, 1) \cdot \nabla ( N_W D_Z + N_Z D_W )$ has constant sign along $b$. We have
\begin{equation*}
(-1, 1) \cdot \nabla ( N_W D_Z + N_Z D_W ) = -\frac{3(\gamma - 1)^2}{4}  ( W-  Z) \left(   W + Z+ \frac{4 (r-1)}{3 (\gamma - 1)} \right).
\end{equation*}
Since $\gamma > 1$ and $ W - Z = 2 \frac{ S}{R}> 0$, we just need to show that $ W +  Z+ \frac{4 (r-1)}{3 (\gamma - 1)}$ has a constant sign along $b$. Recalling that $U = \frac{W+Z}{2}$, this is equivalent to show that $b_U + \frac{2 (r-1)}{3 (\gamma - 1)}$ has constant sign. Since $b_U$ ranges from $ U(P_s)$ to $\frac{-2(r-1)}{3(\gamma - 1)}$ this is true. \\

\textbf{Part I.2 Taylor analysis at $P_s$ and $P_0$.} \\

We start by checking \eqref{eq:part1} at $P_s$. We need to show $W_1 + Z_1 < 0$. The inequality $W_1 + Z_1 < 0$ is deferred to Lemma \ref{lemma:auxiliary}.

With respect to the analysis at $P_0$, from \cite[Proposition 2.5]{Buckmaster-CaoLabora-GomezSerrano:implosion-compressible}, we have that the solution near $R=0$ has the form:
\begin{equation*}
W = \frac{w_0}{R} + \sum_{i=1}^\infty w_i R^{i-1}, \qquad \mbox { and } \qquad Z = \frac{-w_0}{R} + \sum_{i=1}^\infty w_i  (-R)^{i-1},  
\end{equation*}
where $w_1 = \frac{-2(r-1)}{3(\gamma - 1)}$ and $w_2, w_3, \ldots$ can be expressed in terms of $w_0$ (changing $w_0$ corresponds to a dilation of the profile, so it leaves $w_1$ unchanged but modifies the rest of the coefficients).

Plugging this into expression \eqref{eq:part1} and using $w_1 = \frac{-2(r-1)}{3(\gamma - 1)}$, we get
\begin{equation}
\frac{N_W}{D_W} + \frac{N_Z}{D_Z} = \left( -\frac{16 (r-1) (3 \gamma -2 r-1) ((3 \gamma -5) r+2)}{27 w_0^2 (\gamma-1)^5}-6 w_3\right) R^2 + O(R^4).
\end{equation}

Developing the equation $\p_\xi W = \frac{N_W}{D_W}$, we get
\begin{equation}
w_3 = -\frac{8 (r-1) (3 \gamma-2 r-1) ((3 \gamma-5) r+2)}{135 w_0^2 (\gamma-1)^5},
\end{equation}
and therefore
\begin{equation}
\frac{N_W}{D_W} + \frac{N_Z}{D_Z} = -\frac{32 (r-1) (3 \gamma -2 r-1) ((3 \gamma-5) r+2)}{135 w_0^2 (\gamma-1)^5} R^2 + O(R^4).
\end{equation}
Since $\gamma > 1$ and $\gamma > r > 1$ (from \eqref{eq:rough_range_r}), we just need to show that $(3 \gamma-5) r+2 > 0$. If $\gamma > 5/3$ it is obvious, otherwise using also \eqref{eq:rough_range_r}, we get
\begin{equation} \label{eq:villefort}
(3 \gamma-5) r+2 > (3\gamma - 5) \left(2 - \frac{1}{\gamma} \right) +2 = 6 \gamma - 3 - 10 + \frac{5}{\gamma} + 2 = \frac{(6\gamma - 5)(\gamma - 1)}{\gamma} > 0.
\end{equation}

\textbf{Part II: $R>1$. Left of the phase portrait}\\

Let us recall that from the proof of \cite[Proposition 3.1]{Buckmaster-CaoLabora-GomezSerrano:implosion-compressible} we know that for $\xi > 0$, the solution lies to the left of $P_s$. In particular, it is inside the triangle delimited by $D_Z = 0$, $W=Z$ and $W=W_0$. 

\textbf{Part II.1. Showing that the solution lies in $N_W < 0$.}
Let us show moreover that the solution has to lie in the region where $N_W < 0$. It is easy to check $N_W = 0$ describes a hyperbola in $(W, Z)$ coordinates (it has discriminant $\frac{9-14\gamma+9\gamma^2}{16} > 0$). One can solve that second-degree equation and get that the right branch of the hyperbola is parametrized in terms of $Z$ as
\begin{equation*}
p_W(Z) = \frac{\sqrt{9 \gamma^2 Z^2-8 \gamma r Z-14 \gamma Z^2+16 r^2+24 r Z+9 Z^2}+\gamma Z-4 r-3 Z}{4 \gamma}.
\end{equation*}
It satisfies the following properties: \begin{itemize}
\item $N_W(W, Z) < 0$ for every $W>p_W(Z)$. That is, $N_W < 0$ for points to the right of the branch.
\item The minimum of $p_W(Z)$ is zero and it is achieved at $Z = 0$. That is, the leftmost point of the branch is $(0, 0)$.
\item $p_W(Z)$ is decreasing for $Z < 0$ and increasing for $Z > 0$.
\item $(p_W(Z), Z)$ lies on the plane $S > 0$ (that is $p_W(Z) - Z > 0$) whenever $Z < 0$.
\end{itemize}

Let $P_i = (W_i, Z_i)$ the point of intersection of $N_W = 0$ with $D_Z = 0$. Our objective is to show that the solution cannot traverse the branch $(p_W(Z), Z)$ for $Z \in (Z_i, 0)$. Since the solution remains to the right of $D_Z = 0$ and $W=Z$, this will also show that the solution is to the right of $(p_W(Z), Z)$ for $Z < Z_i$ or $Z > 0$ respectively. This situation can be visualized in Figure \ref{fig:part21}.

\begin{figure}[h]
\centering
\includegraphics[width=0.5\textwidth]{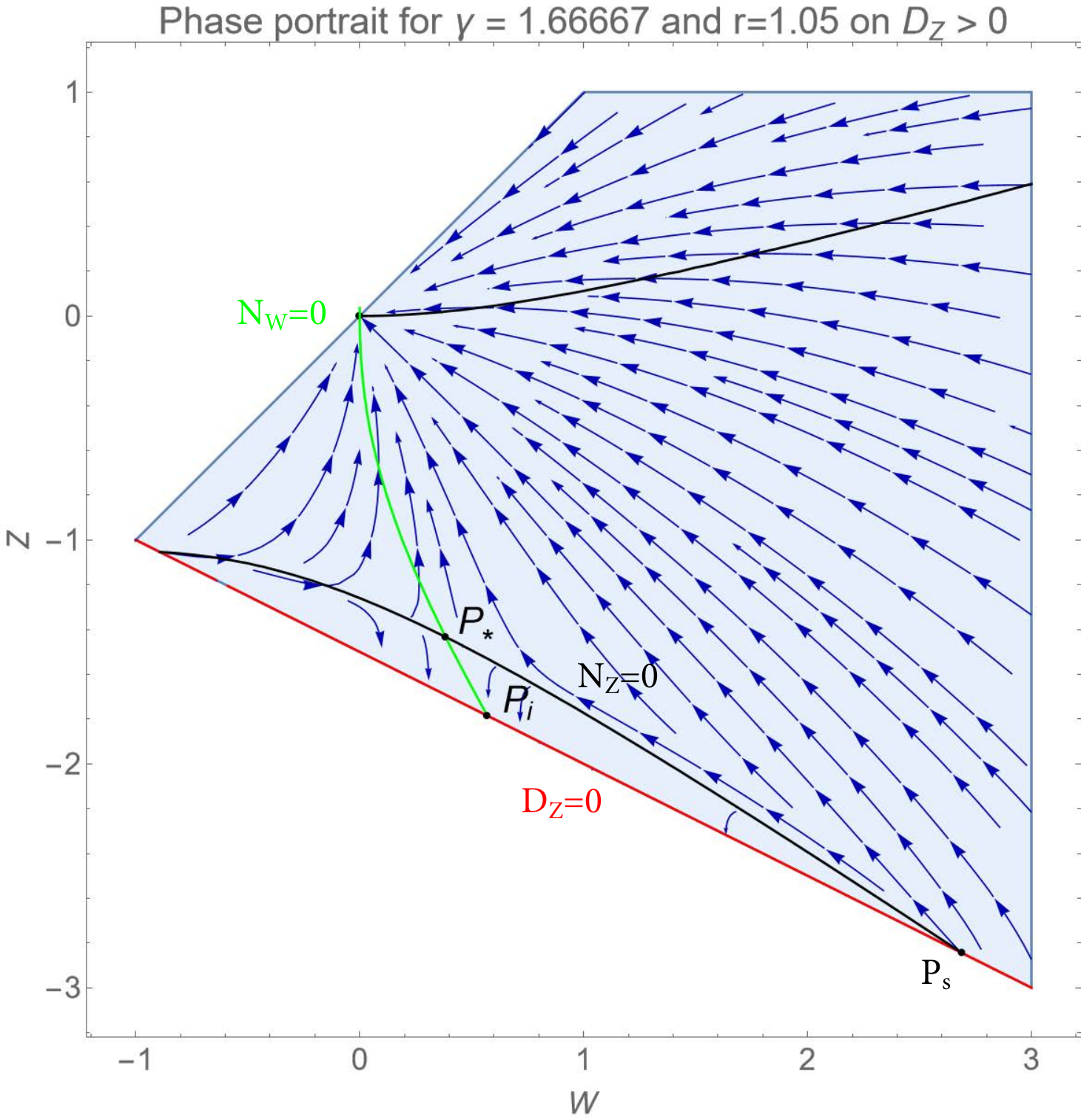}
\caption{Phase portrait to the left of $P_s$ for $\gamma = 5/3$ and $r = 1.05$. In green, the curve $(p_W(Z), Z)$ (that is, $N_W = 0$). The black curve represents $N_Z = 0$ and the red line represents $D_Z = 0$.}
\label{fig:part21}
\end{figure}

Note that the vector field $\left(\frac{N_W}{D_W}, \frac{N_Z}{D_Z}\right)$ over $N_W = 0$ is only vertical, and its sign is dictated by $N_Z$ (recall $D_Z > 0$). If $N_Z > 0$, then the field points upwards, and since $p_W(Z)$ is decreasing (third property) this means that solutions cannot go from the right of $(p_W(Z), Z)$ towards its left part. 

The only solution to $N_W = N_Z = 0$ in the open halfplane $S > 0$ is given by
\begin{equation}
(W_\ast, Z_\ast) = \left( \frac{2 \left(\sqrt{3}-1\right) r}{3 \gamma-1},-\frac{2 \left(1+\sqrt{3}\right) r}{3 \gamma-1}\right).
\end{equation}
The sign of $N_Z(p_W(Z), Z)$ is positive for $Z \in (Z_\ast, 0)$. It is clear that the sign is constant since $p_W(Z)$ does not pass through any other solution of $N_W = N_Z = 0$ from $P_\star$ to $(0, 0)$ (recall the other two solutions are $-P_\star$ and $(-r, -r)$). Moreover, that sign is positive because a Taylor analysis shows $N_Z (p_W(Z),Z) = -rZ + O(Z^2)$, showing that the sign is positive for some $Z < 0$.

Since $N_Z(p_W(Z)) > 0$ for $Z \in (Z_\ast, 0)$, we know that the solution cannot traverse $(p_W(Z), Z)$ from right to left whenever $Z \in (Z_\ast, 0)$. It is possible that $Z_\ast \leq Z_i$ (that is, if $(W_\ast, Z_\ast)$ lies on $D_Z \leq 0$), and in that case, we are done. Otherwise, we have that $Z_i < Z_\ast$ (as in Figure \ref{fig:part21}), and we still need to treat the region $Z \in (Z_i, Z_\ast)$, which is inside the eye-shaped area bounded by $D_Z = 0$ and $N_Z = 0$ (red line and black arc in Figure \ref{fig:part21}).

Let us suppose that indeed the solution traverses $(p_W(Z), Z)$ for some $Z \in (Z_i, Z_\ast)$. Let us call $\Delta$ the (curved) triangle determined by $N_Z < 0, N_W > 0, D_Z > 0$ (in Figure \ref{fig:part21}, this is the triangle located to the left of $P_\star$ and $P_i$, determined by the green, red and black lines). After traversing $(p_W(Z), Z)$, the solution lies inside $\Delta$. We will prove a contradiction by showing that the solution cannot exit through any of the sides of $\Delta$. It cannot exit through the side $N_W = 0$, because $\frac{N_Z}{D_Z} < 0$ and therefore the field points in the inward direction to $\Delta$. It cannot exit through $D_Z = 0$, because we know our solution satisfies $D_Z > 0$ from \eqref{eq:DWDZ}. Moreover, in $\Delta$, $W$ is increasing (since $N_W/D_W > 0$), so the solution inside $\Delta$ is confined to values of $W$ bigger than the value of $W$ when entering $\Delta$. This makes it impossible to exit through $N_Z = 0$, since that side is located to the left of $N_W = 0$. This concludes the proof that the solution cannot traverse $(p_W(Z), Z)$ and therefore, $N_W < 0$.

Once equipped with the knowledge that the trajectory lies on $N_W < 0$, we divide the proof in two parts, showing both
\begin{equation} 
1 + \frac{\bar U_R}{R} + \alpha \p_R \bar S > 0, \quad \mbox{ and } \quad
1 + \frac{\bar U_R}{R} - \alpha \p_R \bar S > 0.
\end{equation}
We express both conditions in one equation as 
\begin{equation} \label{eq:part2_1}
1 + \frac{\bar U_R}{R} \pm \alpha \p_R \bar S > 0,
\end{equation}
meaning that the inequality should hold for both choices of the sign. Going to $(W, Z)$ variables, equation \eqref{eq:part2_1} can be written as
\begin{equation*}
\left( 1 + \frac{ W +  Z}{2}
\pm \alpha \frac{W - Z}{2} \right) \pm \frac{\alpha}{2} \left( \frac{N_W}{D_W} - \frac{N_Z}{D_Z} \right) > 0
\end{equation*}
Given that the left parenthesis is exactly equal to $D_W$ or $D_Z$ (depending on the sign), we just need to show
\begin{equation*}
D_W + \frac{\alpha}{2} \frac{N_W}{D_W} - \frac{\alpha}{2} \frac{N_Z}{D_Z} > 0 \qquad \mbox{ and } \qquad D_Z + \frac{\alpha}{2} \frac{N_Z}{D_Z} - \frac{\alpha}{2} \frac{N_W}{D_W} > 0.
\end{equation*}
Since $D_W, D_Z > 0$, we can multiply by $D_W D_Z > 0$, and it suffices to show
\begin{equation} \label{part22}
\Xi_1 := D_W^2 D_Z + \frac{\alpha}{2} N_W D_Z - \frac{\alpha}{2} N_Z D_W > 0;
\end{equation}
and
\begin{equation} \label{part12}
 D_Z^2 D_W + \frac{\alpha}{2} N_Z D_W - \frac{\alpha}{2} N_W D_Z > 0.
\end{equation}

\textbf{Part II.2 Showing \eqref{part22}}

First, we express $\Xi_1$ in $(U, S)$ coordinates, obtaining
\begin{equation*}
\Xi_1 = (U+1)^3 + \frac{S}{4} (\gamma-1)  \left(r ((\gamma-3) U-2)-2 (\gamma-1) U^2+(5-3 \gamma) U+2\right)-\frac{1}{4} (\gamma-1)^2 (U+1) S^2.
\end{equation*}

Thus, $\Xi_1$ is a second-degree polynomial in terms of $S$. Moreover, it is clear that the discriminant is positive, since the independent term and the second-order term have opposite signs. Let $S_-(U), S_+(U)$ be the two real roots of the second-degree polynomial $\Xi_1$. Since $U = \frac{W+Z}{2} = \frac{D_Z + D_W}{2} - 1 > -1$, the principal coefficient of $\Xi_1$ (as a polynomial in $S$) is negative and therefore $\Xi_1$ is positive whenever $S_- < S < S_+$. Since the independent term is positive, $S_- < 0$, and we just need to check $S < S_+(U)$ due to the constraint $S > 0$.

Let us recall that the solution lies in the triangle $D_Z > 0$, $W < W_0$, $W-Z > 0$ (red triangle in Figure \ref{fig:PartII2}). We will show that the solution satisfies the extra condition $U > U(\bar P_s)$ (this corresponds to being to the upper-right part of the green segment in Figure \ref{fig:PartII2}). We defer the proof of $U > U(\bar P_s)$ to the end of Part II.2 and let us for now assume that the solution lies in the quadrilateral determined by $D_Z > 0$, $W < W_0$, $W-Z > 0$, $U > U(\bar P_s)$. Let us denote that quadrilateral by $Q$ (note that the inequalities are strict, so $Q$ is an open set). We will show that the boundary of $Q$ satisfies $\Xi_1 \geq 0$ (this corresponds to being in the upper-left side of the brown curve in Figure \ref{fig:PartII2}). This means that $S \leq S_+(U)$ for every point of $\partial Q$, implying that $S < S_+(U)$ on $Q$, and concluding the proof of \eqref{part22} (assuming the condition $U > U(\bar P_s)$ that we will show later). 

Thus, we proceed to show that $\Xi_1 \geq 0$ on $\partial Q$. Let us note that we will be always working on the half-plane $S > 0$, so we will alternate between the equivalent conditions $S \leq S_+(U)$ and $\Xi_1 \geq 0$, depending on which one is more convenient for each specific side of our quadrilateral.

\begin{figure}[h]
\centering
\includegraphics[width=0.5\textwidth]{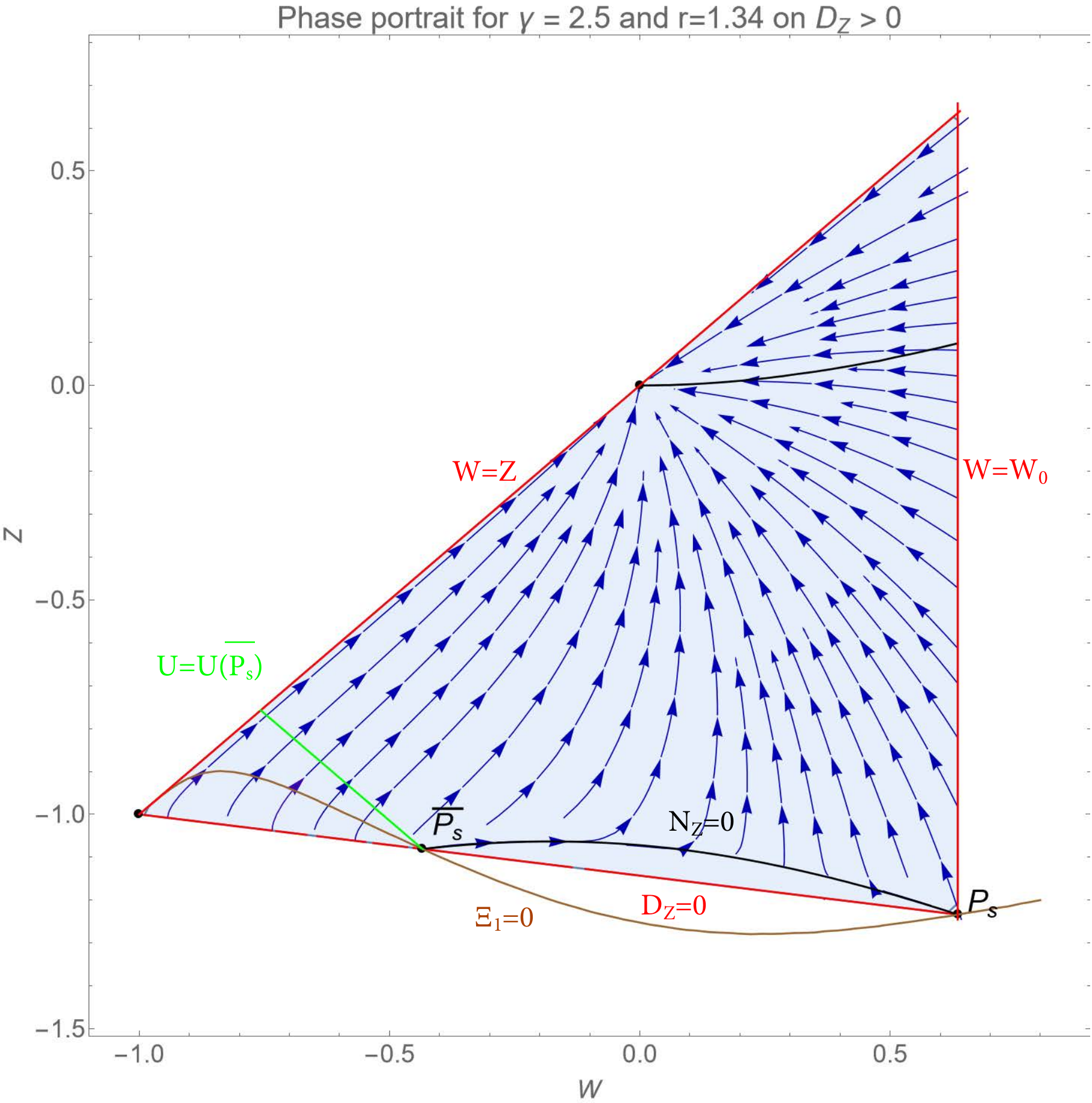}
\caption{Phase portrait of the profile for $\gamma = 5/3$, $r=1.34$. The red triangle corresponds to the triangle where  $D_Z > 0$, $W < W_0$, $W-Z > 0$. The black curve corresponds to $N_Z = 0$ and the brown curve to $\Xi_1 = 0$. The green line represents $U = U(\bar P_s)$. The quadrilateral $Q$ corresponds to the area of the picture that is located to the upper-right part of the green segment. Thus, it has three red sides and one green side, and $\bar P_s$, $P_s$ are two of its vertices. }
\label{fig:PartII2}
\end{figure}

Clearly, the side $W=Z$ satisfies $S < S_+(U)$, since $S_+(U) > 0$ and $S = 0$ on that line. Let us now proceed with the side $W = W_0$ (right side of the red triangle in Figure \ref{fig:PartII2}). Since $\Xi_1$ is a third-degree polynomial in $Z$ vanishing at $P_s$ ($N_Z(P_s) = D_Z(P_s) = 0$) it suffices to show that $\p_Z^i \Xi_1 > 0$ for $i \in \{ 1, 2, 3 \}$. That would show $\Xi_1$ is positive on all the vertical halfline of $W=W_0$ above $P_s$, including the side of our region. We have that
\begin{align*}
\p_Z \Xi_1 |_{P_s} &= \frac{1}{4} (\gamma (r-1)+r (2 Z_0-1)+2 W_0+5) > 0, \\
\p_Z^2 \Xi_1 |_{P_s} &= \frac{1}{8} \left(\gamma^2 (-r+3 Z_0+2)+4 \gamma r-3 r+12 W_0+9 Z_0+22\right) > 0,\\ 
\p_Z^3 \Xi_1 |_{P_s} &= \frac{3}{16} \left(\gamma^2-2 \gamma+5\right) = \frac{3}{4} + \frac{3}{16} (\gamma - 1)^2 > 0.
\end{align*}
The third inequality is trivial while the first two are shown in Lemma \ref{lemma:auxiliary}.

We proceed with the side $D_Z = 0$ of $Q$. If $D_Z = 0$ we have $\Xi_1 = \frac{-\alpha}{2} N_Z D_W$, so the sign is determined by $N_Z$. Let us recall $N_Z = 0$ intersects $D_Z = 0$ twice, at $\bar P_s$ (to the left) and $P_s$ to the right. We need to show that $N_Z \leq 0$ along $D_Z = 0$ between the points $\bar P_s$ and $P_s$ (cf Figure \ref{fig:PartII2}). This is true because $N_Z$ along $D_Z = 0$ (parametrized by $W$) is a second degree polynomial with quadratic coefficient $\frac{2(\gamma - 1)^2}{(\gamma - 1)^2}$. Thus, the polynomial is negative between the roots (that is, between $\bar P_s$ and $P_s$) and positive outside. We have therefore showed that $\Xi_1 \geq 0$ (equivalently $S < S_+(U)$) the side $D_Z = 0$ of our quadrilateral, that goes from $\bar P_s$ to $P_s$.

Lastly, we need to treat the side $U = U(\bar P_s)$ of our quadrilateral region $Q$. We know that $S = S_+(U)$ at $\bar P_s$, since $\Xi_1 = 0$ there ($D_Z = N_Z = 0$). Therefore, we have that $S < S_+(U)$ for any other point along the line $U = U(\bar P_s)$ with smaller value of $S$. This includes the side $U = U(\bar P_s)$ of our quadrilateral. We have shown that $\Xi_1 \geq 0$ on $\partial Q$. From the discussion above, using $\Xi_1 \geq 0$ on $\partial Q$ we conclude \eqref{part22} assuming that the solution stays in the region $U > U(\bar P_s)$ (and therefore, stays inside $Q$). 

Let us finally show that $U > U(\bar P_s)$. We construct the barrier $b(t) = ( \bar W_0 - t, \bar Z_0 + t)$. In order for the solution to remain in the upper-right part of the barrier, the condition needed is:
\begin{equation*}
\Xi_3 = N_W(b(t)) D_Z(b(t)) +  N_Z(b(t)) D_W(b(t)) > 0,
\end{equation*}
where we are using the vector field $(N_W D_Z, N_Z D_W)$, which has the same direction and sign as $\left(N_W/D_W, N_Z/D_W \right)$, given that $D_W, D_Z > 0$. A computation shows that
\begin{align*}
\frac{\Xi_3}{t} & = -\frac{(r-1) \left(3 \gamma^2 (r-3)+\gamma (-14 r-3 \mc R_1+22)+15 r+5 \mc R_1-17\right)}{4 (\gamma -1)}\\
&-\frac{1}{8} t ((\gamma-1) (-3 \gamma (r-3)+r+3 \mc R_1-7)).
\end{align*}
Using Lemma \ref{lemma:auxiliary}, we prove that the independent term of $\frac{\Xi_3}{t}$ is positive, thus showing $\Xi_3 > 0$ for small $t$ (equivalently, $U > U(\bar P_s)$ close to $(\bar W_0, \bar Z_0)$). If we also show that $\Xi_3$ is positive at the intersection of $b(t)$ with $W = Z$, we will be done, since $\Xi_3/t$ is an affine function of $t$ (and therefore positive between any two positive points). Let us denote by $(\kappa, \kappa)$ the intersection of $b(t)$ with $W = Z$. Using the expression $N_W D_Z + N_Z D_W$ on a point of the form $(\kappa, \kappa)$, we obtain
\begin{equation*}
\Xi_3 = -2\kappa (\kappa +1 )(\kappa + r).
\end{equation*}
In particular, since $r>1$, we have that $\Xi_3 > 0$ for $\kappa \in (-1, 0)$. Thus, we are done by showing that
\begin{equation*}
-1 < \frac{\bar W_0 + \bar Z_0}{2} < 0.
\end{equation*}
We refer to Lemma \ref{lemma:auxiliary} for this inequality. \\

\textbf{Part II.3. Showing \eqref{part12}.} \\ 
Since $D_Z, D_W > 0$, and $\alpha > 0$, we can reduce \eqref{part12} to show the stronger inequality
\begin{equation*}
\Xi_2 := N_Z D_W - N_W D_Z > 0.
\end{equation*}
Expressing $\Xi_2$ in $(U, S)$ variables, we obtain
\begin{equation} \label{eq:morcef}
\frac{\Xi_2}{S} = r (2-(\gamma-3) U)+U (2 \gamma U+3 \gamma-1) -\frac{1}{2} (\gamma-1)^2 S^2
\end{equation}
In particular, in the halfplane $S > 0$ we have that $\Xi_2 > 0$ whenever $$S < S_+(U) = \frac{\sqrt{ 2 \left( r (2-(\gamma-3) U)+U (2 \gamma U+3 \gamma-1)\right)}}{\gamma - 1}.$$

Note that this definition of $S_+(U)$ is different from the one we gave in Part II.2. Each definition of $S_+(U)$ only applies to the part where is it defined. Note that $S_+(U)$ may be complex, since the radical is not always real. If $S_+(U)$ is complex for some $U$, then there is no possible $S$ such that $\Xi_2 \geq 0$ (one example for this is at point $(W, Z) = (-0.5, -0.8)$ of Figure \ref{fig:PartII3_1} since the brown hyperbola $S = S_+(U)$ does not intersect the diagonal $U = \frac{-0.5-0.8}{2}$). In particular, whenever we say that $S \leq S_+(U)$ it should be interpreted that both $S_+(U)$ is real and $S \leq S_+(U)$.

\begin{figure}[h]
\centering
\includegraphics[width=0.5\textwidth]{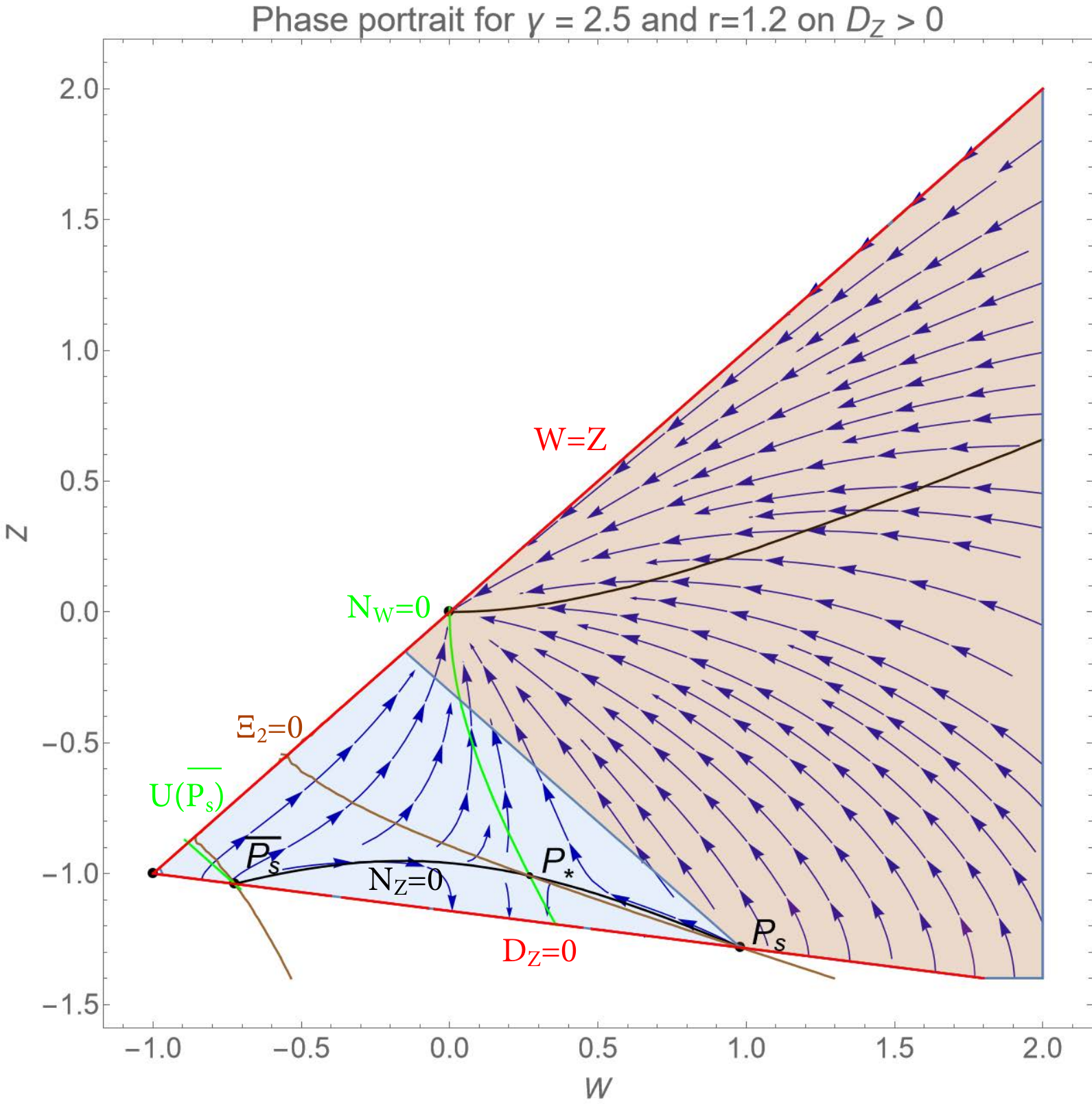}
\caption{Phase portrait of the profile for $\gamma = 5/2$, $r=1.20$. We are plotting the region corresponding to $W - Z > 0$ and $D_Z > 0$, with the boundary in red. The orange region corresponds to $U > U(P_s)$ and the blue to $U < U(P_s)$. The black curve corresponds to $N_Z = 0$. The brown curve corresponds to $\Xi_2 = 0$, and we want to show the solution stays in the region with $S$ below the brown curve (above the right branch of the brown curve). The green curves represent $N_W = 0$ and $U = U(\bar P_s)$.}
\label{fig:PartII3_1}
\end{figure}

Let us start showing that $\Xi_2 \geq 0$ for $U \geq U(P_s)$ (orange region in Figure \ref{fig:PartII3_1}). On the one hand, writing the condition $D_Z > 0$ in $(U, S)$ coordinates, we obtain that $S < \frac{2(1+U)}{\gamma - 1}$. On the other hand, we have that $\Xi_2$ is positive along $D_Z = 0$ for $U > U(P_s)$. This comes from the observation that $\Xi_2 = -N_Z D_W$ along $D_Z = 0$, the fact that $D_W > 0$ along $D_Z = 0$ and the observation from Part II.2 that, along $D_Z = 0$, $N_Z$ is positive between the two roots ($\bar P_s$ and $P_s$) and $N_Z < 0$ outside. Since $\Xi_2 \geq 0$ along $D_Z = 0$ for $U > U(P_s)$, we obtain that $S_+(U)$ is real and bigger than $\frac{2(1+U)}{\gamma - 1}$ (which is the value that $S$ takes at $D_Z = 0$). Since the solution satisfies $S < \frac{2(1+U)}{\gamma - 1}$, we are done with this region.

Now, we just need to show that $\Xi_2 > 0$ whenever the solution is in the region $U < U(P_s)$ (this corresponds to the blue region of Figure \ref{fig:PartII3_1}). The second-degree equation $\Xi_2 = 0$ describes a conic that clearly passes through $P_s$, and we will denote by $b(t)$ the conic branch of $\Xi_2 = 0$ that starts from $P_s$ on the side $D_Z > 0$ (with initial value $b(0) = P_s$). Since $\Xi_2 = 0$ is a conic we can determine $b(t)$ for all $t \geq 0$ (with an arbitrary parametrization). Since $N_W(P_s) < 0$ and $D_W(P_s) > 0$, from the expression of $\Xi_2$, $b(t)$ has to lie in the region $N_Z < 0$ for $t \geq 0$ sufficiently small. Thus, $b(t)$ starts in the eye-shaped region with vertices  $P_s, \bar P_s$ delimited by $N_Z < 0$, $D_Z > 0$ (in Figure \ref{fig:PartII3_1}, that region is delimited by the black curve and red line, and $b(t)$ is the brown curve starting at $P_s$). We denote the eye-shaped region $N_Z < 0$, $D_Z > 0$ by $E$. Moreover, $E$ is contained on the halfplane $W > Z$ since the hyperbola $N_Z$ only intersects $W = Z$ twice: at $(0, 0)$ and $(-r, -r)$. The intersection at $(0, 0)$ corresponds to the upper branch (which is not the one bounding $E$) and the intersection $(-r, -r)$ is located in $D_Z < 0$ (thus, it is not on the boundary of $E$. Moreover, $b(t)$ has to exit $E$ at some point (if $b(t)$ is unbounded that is obvious, if it is an ellipse, it needs to reach $P_s$ from the side $D_Z < 0$). There are two possibilities regarding how $b(t)$ exits $E$:
\begin{itemize}
\item Case A. $b(t)$ exits $E$ through some point with with $D_Z > 0$. This implies that the point of exit satisfies $N_Z = 0$. Combining that with $\Xi_2 = 0$, we get $N_W D_Z = 0$. Since $D_Z > 0$, we also have $N_W = 0$. There is only one solution to $N_W = N_Z = 0$ on the open halfplane $W > Z$, and it is given by $P_\star$. Thus, the point of exit is $P_\star$. This possibility is depicted in Figure \ref{fig:PartII3_1}.
\item Case B. $b(t)$ exits $E$ through $D_Z = 0$. From $\Xi_2 = 0$, and $D_Z = 0$, we obtain that $N_Z D_W = 0$ at that point of exit. Since $D_Z > D_W$ on our halfplane $W > Z$, it is not possible that $D_W$ vanishes, and therefore $N_Z = 0$. There are two points with $N_Z = D_Z = 0$, namely $P_s$ and $\bar P_s$. The point of exit cannot be $P_s$, since, even if $b(t)$ was a periodic conic (an ellipse) it would arrive back to $P_s$ from the halfplane $D_Z < 0$ (since we started from $D_Z > 0$). Thus, the point of exit is $\bar P_s$. This case is depicted in Figure \ref{fig:PartII3_2}.
\end{itemize}
Moreover, we can characterize cases A and B as follows. Since $E$ is contained in $W > Z$ and $N_Z = N_W = 0$ only has one simple solution ($P_\star$) on that halfplane. If $D_Z (P_\star) > 0$ we have that $P_\star \in \partial E$, and since $N_W = 0$ cuts the region $E$, it has to exit through the side $D_Z = 0$ (there are no more solutions to $N_W = N_Z = 0$). Thus, if $D_Z (P_\star) > 0$, it has to divide $E$ in two, so that $P_s$ and $\bar P_s$ lie on different parts. Thus, case B is not possible, since before arriving to $\bar P_s$, $b(t)$ would need to cross $N_W = 0$ (and since $\Xi_2 = 0$ we obtain that $N_Z = 0$ as well, meaning that $b(t)$ reaches $P_\star$ before $\bar P_s$). Reciprocally if $D_Z(P_\star) < 0$, we have that $P_\star \notin \partial E$, so it must happen that we are in case B. If $D_Z(P_\star) = 0$ we have $\bar P_s = P_\star$ and we can take any case. 

\begin{figure}[h]
\centering
\begin{tikzpicture}[scale=1.5]
  % Define coordinates
  \coordinate (Origin) at (0,0);
  \coordinate (UpRight) at (3.2,3.2);    % Slope 1 line extends up and right
  \coordinate (DownRight) at (5,-2.5); % Slope -1/2 line extends down and right
  
  % Points on the lower line (slope -1/2)
  \coordinate (PsBar) at (2,-1);     % First point 
  \coordinate (Ps) at (4,-2);        % Second point
  
  % Add P_infty on the top red line
  \coordinate (Pinfty) at (2,2);    % P_infty on the top red line
  
  % Find intersection of green line with slope -1 from PsBar with upper red line
  \coordinate (IntersectionPoint) at (0.5,0.5);
  
  % Draw the red corner lines
  \draw[thick, red] (Origin) -- (UpRight);   % Upper-right part
  \draw[thick, red] (Origin) -- (DownRight); % Down-right part
  
  % Label the red lines
  \node[red, above] at (2.5,2.72) {$W = Z$};
  \node[red, below] at (3,-1.7) {$D_Z = 0$};
  
  % Add black dots for the special points
  \fill (Ps) circle (2pt);
  \fill (PsBar) circle (2pt);
  \fill (Pinfty) circle (2pt);
  
  % Add labels
  \node[below right] at (PsBar) {$\overline{P}_s$};
  \node[below right] at (Ps.south east) {$P_s$};
  \node[above right] at (1.9,2.1) {$P_{\infty}$};
  
  % Draw a parabola between the two points on the lower line
  \draw[thick, black] (PsBar) .. controls (2.5,0) and (3, 0) .. (Ps);
  
  % Draw the green line with slope -1 from PsBar to the intersection with the upper red line
  \draw[thick, green] (PsBar) -- (IntersectionPoint);
  
  % Label the diagonal green line
  \node[green, above left] at (1.6,0) {$U = U(\overline{P}_s)$};
  
  % Draw the nearly vertical green line (NW = 0) with more curvature in opposite direction
  \draw[thick, green] (Pinfty) .. controls (1.5,0) and (1.5,-2) .. (1.9,-3);
  
  % Label the vertical green curve
  \node[green, right] at (1.7,0) {$N_W = 0$};
  
  % Add P_star near the green curve below Psbar
  \coordinate (Pstar) at (1.67,-2);
  \fill (Pstar) circle (2pt);
  \node[left] at (Pstar) {$P_{\star}$};
  
  % Brown curve passing through Pstar, PsBar, and Ps with sharp turn at PsBar
  \draw[thick, brown] (Pstar) .. controls (1.7,-1.6) and (1.9,-1.2) .. 
                     (PsBar) .. controls (2.2,-0.6) and (3,-0.5) .. (Ps);
                     
  % Label the brown curve
  \node[brown] at (3,-1.2) {$b(t)$};
\end{tikzpicture}
\caption{This figure is analogous to Figure \ref{fig:PartII3_1}, but for Case B. The Figure is schematic and in the real plot $D_Z = 0$, $N_Z = 0$ and $b(t)$ would be too close to distinguish from each other. Recall that $E$ is the region between the black and red lines, and observe that in this case the brown curve $b(t)$ exits $E$ through $\bar P_s$. }
\label{fig:PartII3_2}
\end{figure}

Let us show that in any of the cases above the solution will not traverse $b(t)$ while $b(t) \in E$. The solution starts above $b(t)$ due to the Taylor analysis performed in Part I.2 (the term $D_W D_Z^2$ that we omitted is irrelevant since it is second order at $P_s$). We will start showing that the field $(N_W D_Z, N_Z D_W)$ (which is parallel and has the same orientation as $(N_W/D_W, N_Z/D_Z)$ points in the lower-down direction of $b(t)$ while $b(t)\in E$. Given the implicit definition of $b(t)$ as $\Xi_2 = 0$, a normal vector of $\Xi_2$ pointing down-left is given by $\nabla (N_W D_Z - N_Z D_W)$. We have that
\begin{equation*}
\left( -1, -1 \right) \cdot \nabla (N_W D_Z - N_Z D_W) = \frac{1}{2} (W-Z) (\gamma (-r+2 W+2 Z+3)+3 r-1) = S \left(-1+3r +\gamma (3-r+4U) \right).
\end{equation*}
We show that this quantity is positive. Since the parenthesis is increasing with $U$ it suffices to show that it is positive for the left limit of the interval of $U$ considered. We show that $\left(-1+3r +\gamma (3-r+4U) \right) > 0$ for $U = U(P_\star)$ if $D_Z (P_\star ) \geq 0$ and for $U = U(\bar P_s)$ if $D_Z (P_\star) < 0$ in Lemma \ref{lemma:auxiliary}. This concludes that in both of our cases, on the arc $\Xi_2 = 0$, the field points always in the lower-left direction. Now, let us divide our argument with respect to the aforementioned two cases.

In case B, we have that $b(t)$ and $D_Z = 0$ form an inner subregion inside $E$. If the solution traverses $b(t)$ while in $E$, it would enter that region. This region cannot be exited, since the field of $b(t)$ points to the down-left direction and we know that the solution stays in $D_Z > 0$. This contradicts the fact that the solution reaches $P_\infty = (0, 0)$. So the solution could not have traversed $b(t)$ in the first place.

In case A, we consider the triangular-shaped region $T$ inside $E$ bounded by $N_W = 0$, $D_Z = 0$ and $b(t)$ (the three sides have colours red, green and brown in Figure \ref{fig:PartII3_1}, and two of its vertices are $P_s$ and $P_\star$). If the solution crosses $b(t)$ while inside $E$, it necessarily enters $T$. We know that the solution satisfies $N_W > 0$ (from Part II.1) and $D_Z > 0$. Thus, the solution can only exit through $b(t)$. But this is impossible since we showed that while $b(t) \in E$, the vector field points to the down-left direction (thus, inwards $T$). Thus, the solution cannot exit $T$. This contradicts that the solution converges to $P_\infty = (0, 0) \notin T$ and we deduce that the solution did not traverse $b(t)$ in the first place.

We have proven that the solution does not traverse $b(t)$ while in $E$. Let us now conclude the proof that the solution satisfies $\Xi_2 > 0$ while in the region $U < U(P_s)$. 

If we are in case A, we consider $b(t)$ until reaching $P_\star$ and then we connect with the branch of $N_W = 0$ that joins $P_\star$ with $(0, 0)$ (green curve in Figure \ref{fig:PartII3_1}. We close the region with the condition $U < U(P_s)$ (recall $U(P_s) < 0$ from Lemma \ref{lemma:auxiliary}). While in $U < U(P_s)$, the smooth solution has to stay in this region (we just showed it cannot traverse $b(t)$, and we showed $N_W < 0$ in Part II.1). Moreover, we have that $S \leq S_+(U)$ on the boundary (and in particular, $S_+(U)$ is real). Along the side $U = U(P_s)$ we showed $S \leq S_+(U)$ when treating the case $U \geq U(P_s)$. We clearly have $S = S_+(U)$ along $b(t)$ (since solves $\Xi_2 = 0$). We have that $\Xi_2 = N_Z D_W$ along $N_W = 0$ and we recall from Part II.1 that $N_Z D_W$ is nonnegative along the branch of $N_W = 0$ that joint $P_\star$ with $(0, 0)$. Equivalently, $S_+(U)$ is real and $S \leq S_+(U)$. Since we have that $S \leq S_+(U)$ at the boundary of our region, we conclude $S < S_+(U)$ in the interior, and therefore the solution satisfies $\Xi_2 > 0$.

If we are in case B, we apply the same procedure but when constructing our region, we substitute the branch $N_W = 0$ of case A by the line $U = U(\bar P_s)$ (green line in Figure \ref{fig:PartII3_2}). We also have that the solution cannot traverse this line due to the argument given in Part II.2. It is also clear that $S \leq S_+(U)$ on this side, since $U$ is constant and $S$ is decreasing starting from $\bar P_s$ (at which $S = S_+(U)$). The other sides of the boundary also satisfy $S \leq S_+(U)$ by the same arguments as in case A, so we conclude that $S < S_+(U)$ in the interior of the region. Since the solution lives in this region while $U < U(P_s)$, we are done.
\end{proof}

\begin{lemma} \label{lemma:auxiliary} For all $\gamma > 1$ and $1 < r < r^\ast (\gamma )$ we have the following inequalities: \begin{itemize}
\item $W_1 + Z_1 < 0$.
\item $\frac{1}{4} (\gamma (r-1)+r (2 Z_0-1)+2 W_0+5) > 0$.
\item  $\frac{1}{8} \left(\gamma^2 (-r+3 Z_0+2)+4 \gamma r-3 r+12 W_0+9 Z_0+22\right) > 0$.
\item $3 \gamma^2 (r-3)+\gamma (-14 r-3 \mc R_1+22)+15 r+5 \mc R_1-17$ is negative.
\item $-1  < U(P_s) < \frac{-2(r-1)}{3(\gamma - 1)} < 0$.
\item $-1 < U(\bar P_s) < 0$.
\item $U = U(P_\star)$ satisfies $-1 + 3 r +\ga (3-r + 4 U) > 0$. Moreover, this is also satisfied for $U = U(\bar P_s)$ in the case where $D_Z (P_\star) < 0$. 
\end{itemize}
\end{lemma}
\begin{proof}
\textbf{Item 1} \\
We start showing the inequality $W_1 + Z_1 < 0$. It will be useful to express this quantity in terms of 
$$D_{Z, 1} = \nabla D_Z (W_1, Z_1) = \frac{3-\gamma}{4}W_1 + \frac{1+\gamma}{4}Z_1.$$
We have
\begin{equation*}
W_1 + Z_1 = \frac{4}{\gamma + 1} D_{Z, 1} + \left( 1 - \frac{(3-\gamma)}{\gamma + 1}\right) W_1 = \frac{2(\gamma - 1)}{\gamma + 1} \left( \frac{2 D_{Z, 1}}{\gamma - 1}+W_1\right),
\end{equation*}
so we just need to show that the right parenthesis is negative. Using the explicit formulas for $D_{Z, 1}, W_1$ from \cite{Buckmaster-CaoLabora-GomezSerrano:implosion-compressible} in terms of $\gamma, r, \mc R_2$, we obtain
\begin{equation*}
\left( \frac{2 D_{Z, 1}}{\gamma - 1}+W_1\right) = \frac{-3 \ga^2 (r-3)-\ga (3 \mc R_1+8)+3 r+\mc R_1-1 -2(\ga -1) \mc R_2}{4 (\ga-1)^2}.
\end{equation*}
The denominator is clearly positive, so we just need to show that the numerator is negative. Since the term $-2(\gamma - 1)\mc R_2$ is negative, it suffices to show that this term dominates. That is, it suffices to show that
\begin{align*}
4(\ga -1)^2 \mc R_2^2 &- \left( -3 \ga^2 (r-3)-\ga (3 \mc R_1+8)+3 r+\mc R_1-1 \right)^2 \\ 
&= -2(\ga + 1) \Big( \underbrace{ \left(9 \ga^3 (r-3)^2-33 \ga^2 r^2+114 \ga^2 r-189 \ga^2+43 \ga r^2-82 \ga r+147 \ga-11 r^2+6  r-31\right) }_A \\
&\qquad+ \mc R_1 \underbrace{ \left(9 \ga^2 r-27 \ga^2-6 \ga r+42 \ga-7 r-11\right) }_B \Big)
\end{align*}
is positive. Let us denote by $\Xi = A + B \mc R_1$ the big parenthesis in the right hand side. We need to show that $\Xi$ is negative. 

In order to study $\Xi$, we multiply it by its algebraic conjugate $A - B\mc R_1$ and obtain
\begin{align} \begin{split} \label{eq:calvacanti}
\Xi (A - B\mc R_1) &= -64 (r-1) (\gamma - 1) \Bigg( \left(\frac{3 \ga-1}{\sqrt{3} (\ga-1)+2}-r\right) ((3 \ga-5) r+2) \\ 
 & \qquad \left((3 (\ga-2) \ga-1) r+(3 \ga-1)
   \left(\sqrt{3} \ga-\sqrt{3}+2\right)  \right)\Bigg).
\end{split} \end{align}
The terms $r-1$ and $\gamma - 1$ are clearly positive. The term $\left(\frac{3 \ga-1}{\sqrt{3} (\ga-1)+2}-r\right)$ is equal to $(r^\ast_{\geq 5/3} - r)$ from the definition of $r^\ast_{\geq 5/3}$ (equation \eqref{eq:defr}). The term $(3\gamma - 5)r+2$ is positive due to \eqref{eq:villefort}. Let us argue the term in the last parenthesis is also positive. Since it is linear in $r$, we just need to check that it is positive at $r=1$ and $r = r^\ast_{<5/3}(\gamma) \geq r^\ast(\gamma)$ (we have that $r_{<5/3}^\ast (\gamma) = r^\ast (\gamma)$ for $\gamma \leq 5/3$ but we also have that $r_{<5/3}^\ast (\gamma) > r^\ast (\gamma )$ for $\gamma > 5/3$). We have that  
\begin{equation*}
(3 (\ga-2) \ga-1) r+(3 \ga-1)
   \left(\sqrt{3} \ga-\sqrt{3}+2\right)  \Big|_{r=1} = (\ga-1) \left(3 \sqrt{3} \ga+3 \ga-\sqrt{3}+3\right) > 0,
\end{equation*}
and letting $\ell = \sqrt{\gamma - 1}$:
\begin{align*}
&(3 (\ga-2) \ga-1) r+(3 \ga-1)
   \left(\sqrt{3} \ga-\sqrt{3}+2\right)  \Big|_{r=r_{< 5/3}^\ast} \\
   &= \quad \frac{\ell^2 \left(3 \sqrt{3} \ell^4+9 \ell^4+6 \sqrt{6} \ell^3+6 \sqrt{2} \ell^3+8 \sqrt{3} \ell^2+12
   \ell^2+4 \sqrt{6} \ell+12 \sqrt{2} \ell+4 \sqrt{3}+4\right)}{\left(\ell+\sqrt{2}\right)^2},
\end{align*}
which is clearly positive for $\ell > 0$ (all the coefficients are positive). 

Going back to \eqref{eq:calvacanti}, we deduce that $\Xi (A - B\mc R_1)$ has the same sign as $-\left(r_{\geq 5/3}^\ast - r \right)$, where we recall $r_{\geq 5/3}^\ast, r^\ast_{<5/3}, r^\ast$ were defined in \eqref{eq:defr}. In particular, we know that $\Xi$ does not vanish for $r\in (1, r^\ast)$, $\gamma \geq 5/3$ and that $\Xi$ at most vanishes once for $r\in (1, r^\ast)$, $\gamma < 5/3$ (and if it does, the root is a single root). Recall that we want to show that $\Xi < 0$.

We have that
\begin{align*}
\Xi \Big|_{r=1} = 0, \qquad \mbox{ and } \qquad
\frac{\p}{\p r}\Xi \Big|_{r=1} = -16(\gamma - 1).
\end{align*}
This shows that $\Xi < 0$ for $\gamma \geq 5/3$, since it is negative on some right neighbourhood of $r=1$ and we know it does not change sign since $\Xi (A-B\mc R_2)$ does not vanish.

With respect to the case $\gamma < 5/3$, in that case we have
\begin{align*}
\Xi \Big|_{r=r^\ast_{<5/3}} = \frac{16 \left(\sqrt{2} \sqrt{\frac{1}{\ga-1}}+3\right) (\ga-1) (3 \ga-5)}{\left(\sqrt{2}
   \sqrt{\frac{1}{\ga-1}}+1\right)^4},
\end{align*}
which is clearly negative since $3\gamma - 5 < 0$. In this case, we know that $\Xi$ vanishes at most once with a single root in $r\in (1, r^\ast)$. However, since $\Xi$ is negative for a right neighbourhood of $r=1$ and for $r=r^\ast$, $\Xi$ can only vanish in between if it has a double root or if has more than one root. Thus, $\Xi$ remains negative in all the interval $(1, r^\ast)$. \\

\textbf{Item 2} \\
Using the expressions for $W_0$ and $Z_0$ we have
\begin{align*}
&\gamma (r-1) + r (2Z_0 -1) + 2W_0 + 5 \\
&= 
\frac{ \overbrace{ 2 \ga^3 (r-1)+\ga^2 ((r-8) r+11)-6 \ga r^2+10 \ga r-12 \ga+9 r^2-12 r+7 }^A
+ \mc R_1 \overbrace{ (\ga r+\ga-3 r+1) }^B}{2 (\ga-1)^2}.
\end{align*}
We will conclude our proof showing that both $A$ and $B$ are positive. Using that $r < \frac{2\ga - 1}{\ga} $ from \eqref{eq:danglars}, we have
\begin{equation*}
 B = \ga r+\ga-3 r+1  < 3\ga - 3\frac{2\ga - 1}{\ga} = 3 \frac{\ga^2 - 2\ga + 1}{\ga} = 3\frac{(\ga-1)^2}{\ga} > 0.
\end{equation*}
With respect to $A$, we have that
\begin{equation*}
A =4 (\ga-1)^2 +  2 (\ga-3) (\ga+1) (\ga-1) (r-1)+(\ga-3)^2 (r-1)^2,
\end{equation*}
which clearly shows $A > 0$ whenever $\ga \geq 3$ (since all terms are nonnegative and $4(\ga - 1)^2 > 0$). In the case $\ga < 3$, note that we have a second-degree polynomial of $(r-1)$ with positive principal coefficient. Thus, it is $A$ is a convex hyperbola with respect to $r$, so suffices to show that $A$ is positive and has negative derivative (with respect to $r$) at $r = r^\ast_{<5/3}$, given that $r^\ast \leq r^\ast_{<5/3}$. 

Denoting $\ell = \sqrt{\ga - 1}$ (which is in $(0, \sqrt 2)$ for $\gamma \in (1, 3)$), we have:
\begin{align*}
A \Big|_{r = r_{<5/3}^\ast} &= \frac{4 \ell^6 \left(\ell^2+2\right) \left(\ell^2+2 \sqrt{2} \ell+2\right)}{\left(\ell+\sqrt{2}\right)^4} > 0, \\
\frac{\p}{\p r} A \Big|_{r = r_{<5/3}^\ast} &= \frac{2 \ell^3 \left(\ell^2-2\right) \left(\ell^3+2 \sqrt{2} \ell^2+6 \ell+4
   \sqrt{2}\right)}{\left(\ell+\sqrt{2}\right)^2} < 0.
\end{align*}

\textbf{Item 3} \\

We have to show that $\gamma^2 (-r + 3 Z_0 + 2) + 4\gamma r - 3 r + 12 W_0 + 9 Z_0 + 22 > 0$. We start by reexpressing this quantity as
\begin{equation*}
\frac{\overbrace{11 + 13\ga -\ga^3+\ga^2+\left(-\ga^3+5 \ga^2+5 \ga-33\right) r }^A +\overbrace{ 3 \left(\ga^2-2 \ga+5\right) }^B\mc R_1}{4 (\ga-1)}.
\end{equation*}
Thus, we need to show that the numerator is positive. It is clear that $B \mc R_1 > 0$ (since $\gamma^2 - 2\gamma + 1 = (\ga - 1)^2 \geq 0$). We will show that $A > 0$ for $\gamma < 4$ and that $B^2 \mc R_1^2 - A^2 > 0$ for $\gamma \geq 4$. This would conclude the proof.

Let us start with $A > 0$ for $\gamma < 4$. Since $A$ is a first-degree polynomial in $r$, it suffices to show that $A > 0$ for both $r=1$ and $r=r^\ast$. We have that
\begin{align*}
A \Big|_{r=1} &= 2(\ga - 1)(11 + 2\ga - \ga^2 ) , \\
A \Big|_{r=r_{<5/3}^\ast} &= \frac{4 (\ga-1) (11 + 2\ga - \ga^2 )}{\sqrt{2} \sqrt{\frac{1}{\ga-1}}+1}, \\ 
A \Big|_{ r = r_{\geq 5/3}^\ast} &= \frac{(\ga-1) \left( (3\ga-5)+\sqrt{3} (\ga + 1) \right)(11 + 2\ga - \ga^2 ) }{\sqrt{3}(\ga-1)+2}.
\end{align*}
We observe that all three expressions are positive because $(11 + 2\ga - \ga^2) = 12 - (\ga - 1)^2$, so it is positive for $\gamma \in (1, 4]$. In the last expression note also that $3\gamma - 5 \geq 0$ because this expression is only applied for $\gamma \geq 5/3$. 

Finally, we need to show $B^2 \mc R_1^2 - A^2 > 0$ for $\gamma > 4$. Computing $B^2 \mc R_1^2 - A^2$ we obtain
\begin{align*}
C:=\frac{B^2 \mc R_1^2 - A^2}{8(\ga-1)^2} &= \Big(
(\ga-3)^2 \left(\ga^2-2 \ga+13\right) (r-1)^2 \\
&\qquad + \left(-5\ga^4+12 \ga^3-30 \ga^2-52 \ga-69\right) (r-1) 
+4 \left(\ga^2-2 \ga+13\right) (\ga-1)^2\Big).
\end{align*}
Since $C$ is a second-degree polynomial in $(r-1)$ with positive dominant coefficient, it is enough to show that $C > 0$ at $r = r^\ast$ and $\frac{\p C}{\p r} < 0$ at $r = r^\ast$. Recall $r^\ast = r^\ast_{\geq 5/3}$ since we are in the case $\ga > 4$. We have that
\begin{align*}
C \Big|_{r=r^\ast} &=  -\frac{\ga-1}{\left(\sqrt{3}   \ga-\sqrt{3}+2\right)^2} \Big( 3 \left(7 \sqrt{3}-13\right) \ga^5+\left(249-131 \sqrt{3}\right) \ga^4+\left(466 \sqrt{3}-982\right)   \ga^3 \\ 
&\qquad -6 \left(177 \sqrt{3}-379\right) \ga^2+\left(1673 \sqrt{3}-3027\right) \ga-1255 \sqrt{3}+2389\Big) \\ 
& =-\frac{\ga-1}{\left(\sqrt{3}   \ga-\sqrt{3}+2\right)^2}\Big(\left(21 \sqrt{3}-39\right) (\ga-4)^5+\left(289 \sqrt{3}-531\right) (\ga-4)^4 \\
&+\left(1730 \sqrt{3}-3238\right) (\ga-4)^3+6 \left(899 \sqrt{3}-1761\right) (\ga-4)^2+\left(8889 \sqrt{3}-18147\right) (\ga-4)\\
&+99 \left(63 \sqrt{3}-125\right)\Big) \\
\frac{\p C}{\p r} \Big|_{r= r^\ast} &=  \frac{1}{{\sqrt{3} \ga-\sqrt{3}+2}} \Big( \left(6-7 \sqrt{3}\right) \ga^5+\left(35 \sqrt{3}-64\right) \ga^4-6 \left(21 \sqrt{3}-46\right) \ga^3 \\ 
&\qquad +14 \left(17   \sqrt{3}-60\right) \ga^2+\left(1174-443 \sqrt{3}\right) \ga+303 \sqrt{3}-840 \Big) \\
&= \frac{1}{{\sqrt{3} \ga-\sqrt{3}+2}} \Big( \left(6-7 \sqrt{3}\right) (\ga-4)^5+\left(56-105 \sqrt{3}\right) (\ga-4)^4 \\
&+\left(212-686 \sqrt{3}\right) (\ga-4)^3+42 \left(4-57 \sqrt{3}\right) (\ga-4)^2-3 \left(1529 \sqrt{3}+334\right) (\ga-4)\\
&-9 \left(437 \sqrt{3}+240\right) \Big).
\end{align*}
Thus, the proof follows from the two fifth-degree polynomials of $\gamma$ above being negative for $\gamma \geq 4$ due to all coefficients being negative.\\

\textbf{Item 4} \\
Let us define $$A = 3 \ga^2 (r-3)-14 \ga r+22 \ga+15 r-17,$$ which includes all terms except the ones with $\mc R_1$. We need to show that $A + (-3\ga+5)\mc R_1 < 0$. 

We start showing that $A < 0$. Since $A$ is an affine function of $r$ and $1 < r < r^\ast < \frac{2\ga-1}{\ga}$ (the last inequality is due to \eqref{eq:danglars}), it suffices to show that $A<0$ for $r=1$ and $r = \frac{2\ga-1}{\ga}$. We have that
\begin{equation*}
A \Big|_{r=1} = -2(\ga-1) (3\ga-1) < 0, \qquad \mbox{ and } \qquad A\Big|_{r=(2\ga-1)/\ga} = -3\frac{(\ga-1)^2(\ga +5)}{\ga} < 0.
\end{equation*}

Therefore, we have that $A < 0$. This immediately shows that $A + (-3\ga+5)\mc R_1 < 0$ for $\ga \geq 5/3$. With respect to the case $\ga < 5/3$, it suffices to show $A^2 - (-3\ga+5)^2\mc R_1^2 > 0$. We have that
\begin{equation*}
A^2 - (-3\ga+5)^2\mc R_1^2
= 32 (\ga-1)^2 ((3 \ga-5) r+2)
\end{equation*}
and we conclude by recalling that $(3 \ga-5) r+2$ is positive due to \eqref{eq:villefort} and the discussion immediately before that equation. \\

\textbf{Item 5} \\
With respect to $- \frac{2(r-1)}{3(\ga - 1)} - U(P_s)$ being positive, we have that 
\begin{equation} \label{eq:ulanbator}
- \frac{2(r-1)}{3(\ga - 1)} - U(P_s) = \frac{9\ga - 7 + (1-3 \ga )r -3 \mc R_1}{12 (\ga-1)}
\end{equation}
Since clearly $\ga - 1 > 0$ it suffices to show that the numerator is positive. Now, let us recall from \eqref{eq:danglars} that $r^\ast < \frac{2\ga-1}{\ga}$. Using that (and $1-3\ga < 0$) we have 
\begin{equation*}
9\ga - 7 + (1-3 \ga )r > 9\ga - 7 + \frac{(2\ga-1) (1-3\ga )}{\ga} = \frac{3\ga^2 -2\ga   - 1}{\ga} > 0.
\end{equation*}
Thus, in order to show that the numerator of \eqref{eq:ulanbator} is positive, we just need to show that $(9\ga - 7 + (1-3 \ga )r)^2$ is bigger than $(3\mc R_1)^2$. Indeed, we have
\begin{equation*}
( 9\ga - 7 + (1-3 \ga )r )^2 - 9 \mc R_1^2 = 16 (r-1) ((3 \ga-5) r+2) > 0,
\end{equation*}
where the factor $(3\ga - 5)r + 2$ is positive because of \eqref{eq:villefort} and the discussion immediately before that equation.

The other part of this item is to show that $U(P_s) + 1$ is positive. We have that
\begin{equation*}
U(P_s) + 1 = \frac{\ga r+\ga-3r+1+\mc R_1 }{4 (\ga-1)}.
\end{equation*}
Since $\mc R_1$ and $\ga-1$ are positive, we will be done if we show that $\ga r+\ga-3r+1 > 0$. Since this is an affine function of $r$, for $r\in (1, r^\ast)$, it will take it values between the value at $r=1$ and the value at $r = \frac{2\ga-1}{\ga}$ (recall that from \eqref{eq:danglars} this is an upper bound of $r^\ast$). When $r=1$ we clearly have $\ga r+\ga-3r+1 = 2(\ga - 1)>0$. For $r = \frac{2\ga-1}{\ga}$, we have
\begin{equation} \label{eq:seoul}
\left( \ga r+\ga-3r+1 \right) \Big|_{r = (2\ga-1)/\ga } = 3 \frac{(\ga-1)^2}{\ga}
\end{equation}

\textbf{Item 6} \\
Let us show that $U(\bar P_s) \in (-1, 0)$. We have that
\begin{equation*}
U(\bar P_s)= \frac{\ga (r-3)-3 r+5-\mc R_1}{4 (\ga-1)}
\end{equation*}
In order to show that $U(\bar P_s)$ is negative, it suffices to show that $\ga r - 3\ga - 3 r + 5 < 0$. Since this function is affine in $r$, it suffices to check so at $r=1$ and $r=\frac{2\ga-1}{\ga}$ (which is an upper bound of $r^\ast$). At $r=1$ we have $\ga r - 3\ga - 3 r + 5 = -2\ga+2$ and at $r=\frac{2\ga-1}{\ga}$ we have 
\begin{equation*}
\left( \ga r - 3\ga - 3 r + 5  \right) \Big|_{r = (2\ga-1)/\ga} = \frac{-\ga^2 -2\ga +3}{\ga} < 0.
\end{equation*}

With respect to $U(\bar P_s) > -1$, we have that
\begin{equation*}
U(\bar P_s) + 1 = \frac{\ga r+\ga-3 r+1-\mc R_1}{4 (\ga-1)}
\end{equation*}
Let us recall that in \eqref{eq:seoul} and the paragraph before, we showed $\ga r+\ga-3 r+1 > 0$. Therefore, in order to show that $U(\bar P_s) + 1$ is positive, it suffices to show that $(\ga r+\ga-3 r+1)^2 - \mc R_1^2$ is positive. We have
\begin{equation*}
(\ga r+\ga-3 r+1)^2 - \mc R_1^2
= 8 (\ga-1)^2 (r-1) > 0
\end{equation*}

\textbf{Item 7} \\
Let us start by showing the inequality for $U(P_\star)$. We have that
\begin{equation*}
-1 + 3 r +\ga (3-r + 4 U(P_\star)) = \frac{  -3 \ga^2 r+9 \ga^2+2 \ga r-6 \ga-3 r+1 }{3   \ga-1}
\end{equation*}
Let us call $A$ to the numerator of the fraction above. It suffices to show that $A > 0$. Since $A$ is an affine function of $r$, it suffices to show that $A > 0$ at $r=1$ and $r = \frac{2\ga-1}{\ga}$ (since $r^\ast < \frac{2\ga-1}{\ga}$ by \eqref{eq:danglars}). We have that
\begin{equation*}
A \Big|_{r=1} = 2(\ga-1)(3\ga+1) > 0, \qquad \mbox{ and }\qquad A\Big|_{r=(2\ga-1)/\ga} = \frac{(\ga-1) (3\ga^2 + 4\ga - 3)}{\ga} > 0.
\end{equation*}

Now, let us show the inequality for $U(\bar P_s)$. We have that
\begin{equation*}
-1 +3r + \ga (3-r+ 4 U (\bar P_s)) =
\frac{\ga r+\ga-3 r+1 -\ga \mc R_1}{\ga-1}
\end{equation*}
From \eqref{eq:seoul} and the discussion immediately before, we have that $\ga r+\ga-3 r+1 > 0$. Therefore, we just need to show that $(\ga r+\ga-3 r+1)^2 - \ga^2 \mc R_1^2$ is positive. We have that
\begin{equation*}
B = (\ga r+\ga-3 r+1)^2 - \ga^2 \mc R_1^2
= - (\ga-1) \left(  \ga^3 (r-3)^2+\ga^2 \left(-5 r^2+6
   r-5\right)+\ga \left(3 r^2-10 r+3\right)+(1-3 r)^2 \right)
\end{equation*}
Extracting the factor $\gamma - 1$ and taking a derivative with respect to $r$, we obtain
\begin{equation*}
\frac{\p }{\p r} \left( \frac{B}{\ga-1} \right) = 10 \ga+6 + 6 \ga^3-6 \ga^2+\left(-2 \ga^3+10 \ga^2-6
   \ga-18\right) r
\end{equation*}
Let us note that this is an affine function of $r$. Moreover it is positive at $r=1$ and $r = \frac{2\ga-1}{\ga} > r^\ast$ since
\begin{align*}
\frac{\p }{\p r}\Big|_{r=1} \left( \frac{B}{\ga-1} \right) &= 4 (\ga-1)(3 + 2\ga + \ga^2), \\
\frac{\p }{\p r}\Big|_{r=(2\ga - 1)/\ga} \left( \frac{B}{\ga-1} \right) &= \frac{2(\ga-1) (-9+3\ga + 9\ga^2 + \ga^3)}{\ga}.
\end{align*}
Therefore, we have that $\frac{\p }{\p r} \left( \frac{B}{\ga-1} \right)$ is positive for all $r \in (1, r^\ast )$ and we obtain that $B$ is increasing in that interval.

We recall that we only need to show $B > 0$ for the values of $\ga, r$ such that $D_Z (P_\star) < 0$. We have that
\begin{equation*}
D_Z (P_\star ) = \frac{-1+3\ga - r(2 + \sqrt{3} (\ga-1 ) )}{3\ga - 1},
\end{equation*}
which is clearly decreasing with $r$. Therefore, solving the equation $D_Z (P_\star ) = 0$, we obtain that the condition $D_Z(P_\star) < 0$ is equivalent to 
\begin{equation} \label{eq:rthres}
r >  \frac{3\ga - 1}{2 + \sqrt{3} (\ga-1) } = r_{\geq 5/3}^\ast .
\end{equation}
That is, the threshold value for $r$ at which $D_Z(P_\star)$ becomes negative is given by the same formula that describes $r^\ast$ for $\gamma \geq 5/3$. In particular, for $\gamma \geq 5/3$, there is nothing to show, since we will always have $D_Z(P_\star ) > 0$ for $r \in (1, r^\ast )$. However, for $\gamma \in (1, 5/3)$, we need to show that $B > 0$ on $(r_{\geq 5/3}^\ast (\gamma ), r_{<5/3}^\ast (\gamma ) )$. Since we have shown that $B$ is increasing with $r$, it suffices to do so at the left endpoint of the interval. We have that
\begin{equation*}
\frac{B}{\gamma - 1} \Big|_{r = r_{\geq 5/3}^\ast} = \frac{2 (\ga-1)^2 \left(9 \left(\sqrt{3}-2\right) \ga^3-21
   \left(\sqrt{3}-2\right) \ga^2+\left(7 \sqrt{3}-2\right) \ga+5
   \sqrt{3}-14\right)}{\left(\sqrt{3} \ga-\sqrt{3}+2\right)^2}
\end{equation*}
so it suffices to show that the third-degree polynomial $$\left(9 \left(\sqrt{3}-2\right) \ga^3-21   \left(\sqrt{3}-2\right) \ga^2+\left(7 \sqrt{3}-2\right) \ga+5  \sqrt{3}-14\right)$$ is positive for $\gamma \in (1, 5/3)$.

We can rewrite it as
\begin{align*}
& 16 + 12 (-1 + \sqrt{3}) \left(\ga - \frac53\right) + 24 (-2 + \sqrt{3}) \left(\ga - \frac53\right)^2 + 
 9 (-2 + \sqrt{3}) \left(\ga - \frac53\right)^3 \\
 &> 
16 + 12 (-1 + \sqrt{3}) \left(\ga - \frac53\right) + 24 (-2 + \sqrt{3}) \left(\ga - \frac53\right)^2
\end{align*}
for $\ga \in \left(1,\frac53\right)$, and the conclusion follows from the inequality $$16 > |12 (-1 + \sqrt{3})|  + |24 (-2 + \sqrt{3})| = 36 - 12\sqrt{3}.$$

% Previous proof for the third-degree polynomial being positive

%Its principal coefficient ($9(\sqrt{3}-2)$) is negative, and its three roots are given by:
%\begin{align*}
%\gamma_1 &= \frac{1}{3} \left(-1-2 \sqrt{3}\right) = -1.48\ldots, \qquad 
%\gamma_2 = \frac{1}{3} \left(4+\sqrt{3}-\sqrt{16+2 \sqrt{3}}\right) = 0.44\ldots, \\
%\gamma_3 &= \frac{1}{3} \left(4+\sqrt{3}+\sqrt{16+2 \sqrt{3}}\right) = 3.38\ldots
%\end{align*}
%Since there are no double roots and the polynomial is negative for sufficiently large $\gamma$ (the principal coefficient is negative), the polynomial has to be positive between $\gamma_2 = 0.44\ldots$ and $\gamma_3 = 3.38\ldots$. In particular, it is positive for $\gamma \in (1, 5/3)$.

\end{proof}
\begin{lemma} \label{lemma:integrated_repulsivity} We have that the profiles $(\bar U, \bar S)$ satisfy
\begin{equation}
R + \bar U_R - \alpha \bar S > (R-1)\eta
\end{equation}
for $R > 1$.
\end{lemma}
\begin{proof} From \cite{Merle-Raphael-Rodnianski-Szeftel:implosion-i}, \cite{Buckmaster-CaoLabora-GomezSerrano:implosion-compressible} we have that the profiles constructed satisfy $R+\bar U_R - \alpha \bar S = 0$ for $R=1$ (i.e. $D_Z(P_s) = 0$). %(in the jargon of \cite{Buckmaster-CaoLabora-GomezSerrano:implosion-compressible}: $D_Z(P_s) = 0$).
Therefore, using the fundamental theorem of calculus and \eqref{eq:radial_repulsivity}, we have
\begin{align*}
R+\bar U_R - \alpha \bar S = \int_1^R \left( \tilde R + \bar U_R' (\tilde R)  - \alpha \bar S' (\tilde R) \right) d\tilde R > \int_1^R \eta d\tilde R.
\end{align*}
\end{proof}

\begin{lemma} \label{lemma:GN_general} Let $f:\mathbb R^3 \to \mathbb R^3$. Let $I_{-1} = [0, 1]$ and $I_j = [2^{j}, 2^{j+1}]$. Let $\phi (x), \psi(x) \geq 0$ be weights such that there exist values $\phi_j, \psi_j$ and a constant $C$ satisfying that $\phi (x) \in [\phi_j/C, C \phi_j]$ and $\psi(x) \in [\psi_j/C, C\psi_j]$ for every $x$ with $|x| \in I_j$. Moreover, let us assume the weights are approximately $1$ at the origin, that is $\phi_{-1} = \psi_{-1} = \phi_0 = \psi_0 = 1$. Let $1 \leq i \leq m$. Assume that the parameters $p, q, \bar{r}, \theta$ satisfy:
\begin{equation} \label{eq:GN_conditions}
    \frac{1}{\bar{r}} = \frac{i}{3} + \theta \left( \frac{1}{q} - \frac{m}{3} \right) + \frac{1-\theta}{p}, \qquad \mbox{ and } \qquad \frac{i}{m} \leq \theta < 1,
\end{equation}
Letting $\bar{r} = \infty$, we have
\begin{equation} \label{eq:GNresultinfty}
    \left| \nabla^i f \right| \les \| \psi^m f \|_{L^p}^{1-\theta} \| \phi^m \nabla^m f \|_{L^q}^{\theta} \psi^{-m(1-\theta)} \phi^{-m\theta}  + \left\| \psi^m f \right\|_{L^p} \cdot  \langle x \rangle^{3\theta (1/q-1/p)-m\theta} \psi^{-m}.
\end{equation}
If moreover $p = \infty, q = 2$ and $\psi = 1$, we obtain
\begin{equation} \label{eq:GNresultinfty_simplified}
    \left| \nabla^i f \right| \les \| f \|_{L^\infty}^{1-\frac{i}{m-3/2}} \| \phi^m \nabla^m f \|_{L^q}^{\frac{i}{m-3/2}}  \phi^{-m\frac{i}{m-3/2}}  + \left\| f \right\|_{L^\infty} \cdot  \langle x \rangle^{-i}.
\end{equation}
For $\bar{r} \in [1,\infty)$, any $\eps > 0$, and under the extra assumption:
\begin{equation}\label{eq:GN_extracond}
    \left( \frac{\phi(x)}{\langle x \rangle \psi(x) }\right)^{m\theta} \langle x \rangle^{3\theta (1/q-1/p)} \les 1
\end{equation}
we have the weighted Gagliardo-Nirenberg inequalities:
\begin{equation} \label{eq:GNresult}
    \| \langle x \rangle^{-\eps} \psi^{m(1-\theta)} \phi^{m\theta} \nabla^i f \|_{L^{\bar{r}}} \les \| \psi^m f \|_{L^p}^{1-\theta} \| \phi^m \nabla^m f \|_{L^q}^{\theta} + \| \psi^m f \|_{L^p}
\end{equation}
The implicit constants in \eqref{eq:GNresultinfty} and \eqref{eq:GNresult} may depend on the parameters $p, q, i, m, \theta, C$, (as well as $r$, $\eps >0$ for \eqref{eq:GNresult}) but they are independent of $f$ and $\psi$, $\phi$. 
\end{lemma}
\begin{proof} We show first \eqref{eq:GNresult}. We divide $\mathbb R^3$ dyadically, defining $A_{-1} = B(0, 1)$ and $A_j = B(0, 2^{j+1}) \setminus B(0, 2^{j})$ for $j \geq 0$. It is clear that $\mathbb R^3 = \cup_{j \geq -1} A_i$. Now, note that
\begin{equation} \label{eq:rousseau}
    \| \langle x \rangle^\eps \psi^{m-i} \phi^i \nabla^i f \|_{L^{\bar{r}}} \leq \sum_{j=-1}^\infty 2^{-j\eps} \| \mathbbm{1}_{A_j} \psi^{m-i} \phi^i \nabla^i f \|_{L^{\bar{r}}} \les_\eps \max_{-1 \leq j <\infty } \| \mathbbm{1}_{A_j} \psi^{m-i} \phi^i \nabla^i f \|_{L^{\bar{r}}}
\end{equation}
so we reduce to bound by the right hand side of \eqref{eq:GNresult}, \eqref{eq:GNresultinfty} each of the dyadic pieces. 

Now, if the maximum on \eqref{eq:rousseau} is achieved for $j \in \{-1, 0 \}$, we are done, because the Gagliardo-Nirenberg inequality for bounded domains ensures that
\begin{align}
    \| \mathbbm{1}_{A_{-1}} \nabla^i f \|_{L^{\bar{r}}} &\les \| \mathbbm{1}_{A_{-1}} \nabla^m f \|_{L^q}^{\theta} \| \mathbbm{1}_{A_{-1}}  f \|_{L^p}^{1-\theta} + \| \mathbbm{1}_{A_0}  f \|_{L^p} \label{eq:Am1} \\
    \| \mathbbm{1}_{A_0} \nabla^i f \|_{L^{\bar{r}}} &\les \| \mathbbm{1}_{A_0} \nabla^m f \|_{L^q}^{\theta} \| \mathbbm{1}_{A_0}  f \|_{L^p}^{1-\theta} + \| \mathbbm{1}_{A_0}  f \|_{L^p}, \label{eq:A0}
\end{align}
and we have that $\phi_{j}, \psi_{j} = 1$ for $j \in \{-1, 0 \}$ (so the weight's range is inside $[1/C, C]$ on those regions). Thus, we can assume that the maximum is achieved for $j \geq 1$. Now, for $j \geq 1$, consider the function $f_j(x) = f(2^{j} x)$. The change of variables formula yields
\begin{equation}
    \| \mathbbm{1}_{A_0} \nabla^{i'} f_j \|_{L^{p'}} = 2^{ji'} 2^{-3j/p'} \| \mathbbm{1}_{A_j} \nabla^{i'} f \|_{L^{p'}}
\end{equation}
Plugging this into \eqref{eq:A0}, we obtain
\begin{equation*} 
    2^{3j(i/3-1/\bar{r})} \| \mathbbm{1}_{A_j} \nabla^i f \|_{L^{\bar{r}}} \les  2^{3j(m/3-1/q)(1-\theta)}\| \mathbbm{1}_{A_j} \nabla^m f \|_{L^q}^{\theta} 2^{(-3j/p)\theta} \| \mathbbm{1}_{A_j} f \|_{L^p}^{1-\theta} + 2^{-3j/p} \| \mathbbm{1}_{A_j}  f \|_{L^p}.
\end{equation*}
Using condition \eqref{eq:GN_conditions}, this just reads
\begin{equation} \label{eq:Aj}
    \| \mathbbm{1}_{A_j} \nabla^i f \|_{L^{\bar{r}}} \les  \| \mathbbm{1}_{A_j} \nabla^m f \|_{L^q}^{\theta}  \| \mathbbm{1}_{A_j}f \|_{L^p}^{1-\theta} + 2^{3j(-i/3+1/\bar{r}-1/p)} \| \mathbbm{1}_{A_j}  f \|_{L^p},
\end{equation}
for $j \geq 0$.

Now, introducing the weights, we see that
\begin{align}
    \| \mathbbm{1}_{A_j} \psi^{m(1-\theta)} \phi^{m\theta} \nabla^i f \|_{L^{\bar{r}}} &\les  \psi_j^{m(1-\theta)} \phi_j^{m\theta} \| \mathbbm{1}_{A_j} \nabla^m f \|_{L^q}^{\theta}  \| \mathbbm{1}_{A_j}  f \|_{L^p}^{1-\theta} + 2^{3j (-i/3+1/\bar{r}-1/p)} \psi_j^{-m\theta} \phi_j^{m\theta} \| \psi^m \mathbbm{1}_{A_j}  f \|_{L^p} \notag \\
    &\les  
    \| \mathbbm{1}_{A_j} \psi^m \nabla^m f \|_{L^q}^{\theta}  \| \phi \mathbbm{1}_{A_j}  f \|_{L^p}^{1-\theta} + 2^{j (-i+3/\bar{r}-3/p)} \left( \frac{\phi_j}{\psi_j} \right)^{m\theta} \| \psi^m \mathbbm{1}_{A_j} f \|_{L^p}  \label{eq:marat}
\end{align}
where we recall that our weights are within $[\phi_j/C, \phi_j C]$ and $[\psi_j/C, \psi_j C]$, respectively, for $x \in A_j$. Recall also that $\les$ is allowed to depend on $C, i, m, p, q, \bar{r}, \eps$ (in particular, it can absorb constants $C^m$). 

Now, note from \eqref{eq:GN_conditions}
\begin{equation} \label{eq:robespierre}
    -i + \frac{3}{\bar{r}} - \frac{3}{p} = - m \theta + \theta \left(\frac{3}{q} - \frac{3}{p} \right).
\end{equation}
Using that together with the fact that $\phi, \psi$ are within a constant factor of $\phi_j, \psi_j$ for $x \in A_j$, we obtain that 
\begin{equation*}
2^{j (-i+3/\bar{r}-3/p)} \left( \frac{\phi_j}{\psi_j} \right)^{m\theta} = \left( 2^j \right)^{\theta (3/q - 3/p)} \left( \frac{ \phi_j}{\langle x \rangle\psi_j} \right)^{m\theta} \les 
\langle x \rangle^{\theta (3/q - 3/p)} \left( \frac{ \phi (x)}{\langle x \rangle \psi (x)} \right)^{m\theta} \les 1.
\end{equation*}
where in the last inequality we used \eqref{eq:GN_extracond}. Plugging this inequality into \eqref{eq:marat}, we conclude \eqref{eq:GNresult}.

Now, let us show \eqref{eq:GNresultinfty}. If $x\in A_{-1}$ we are done by equation \eqref{eq:Am1} taking $\bar{r} = \infty$. Thus, let us suppose that $x \in A_j$ for some $j\geq 0$. Using the same dilation argument on the Gagliardo-Nirenberg inequality for bounded domains as in the $\bar{r}<\infty$ case, we get \eqref{eq:marat} in a completely analogous way. Therefore:
\begin{align*}
    |\nabla^{i} f(x)| \psi^{m(1-\theta)}(x) \phi^{m\theta}(x) | &\leq 
    \| \mathbbm{1}_{A_j} \psi^m \nabla^m f \|_{L^q}^{\theta}  \| \phi^m \mathbbm{1}_{A_j}  f \|_{L^p}^{1-\theta} \\
    &\quad + 2^{j (-i+3/\bar{r}-3/p)} \left( \frac{\phi_j}{\psi_j} \right)^{m\theta} \| \psi^m \mathbbm{1}_{A_j} f \|_{L^p}
\end{align*}
Lastly, since $x\in A_j$, we have
\begin{equation*}
    2^{j (-i+3/\bar{r}-3/p)} \left( \frac{\phi_j}{\psi_j} \right)^{m\theta} \les \langle x \rangle^{-i+3/\bar{r}-3/p} \left( \frac{\phi (x)}{\psi (x)} \right)^{m\theta} = \langle x \rangle^{3\theta (1/q-1/p)} \left( \frac{\phi (x)}{\langle x \rangle \psi (x)} \right)^{m\theta} 
\end{equation*}
using again \eqref{eq:robespierre}. Thus, we conclude \eqref{eq:GNresultinfty}.
\end{proof}

\begin{lemma} \label{lemma:GN_generaltorus} Let $f:{\mathbb T}_{L}^3 \to \mathbb R^3$, with $L\geq 1$. Let $I_{-1} = [0, 1]$ and $I_j = [2^{j}, 2^{j+1}]$. Let $\phi (x), \psi(x) \geq 0$ be weights such that there exist values $\phi_j, \psi_j$ and a constant $C$ satisfying that $\phi (x) \in [\phi_j/C, C \phi_j]$ and $\psi(x) \in [\psi_j/C, C\psi_j]$ for every $x$ with $|x| \in I_j$. Moreover, let us assume the weights are approximately $1$ at the origin, that is $\phi_{-1} = \psi_{-1} = \phi_0 = \psi_0 = 1$. Let $1 \leq i \leq m$. Assume that the parameters $p, q, r, \theta$ satisfy:
\begin{equation} \label{eq:GN_conditionstorus}
    \frac{1}{\bar{r}} = \frac{i}{3} + \theta \left( \frac{1}{q} - \frac{m}{3} \right) + \frac{1-\theta}{p}, \qquad \mbox{ and } \qquad \frac{i}{m} \leq \theta < 1,
\end{equation}
Letting $\bar{r} = \infty$, we have
\begin{equation} \label{eq:GNresultinftytorus}
    \left| \nabla^i f \right| \les \| \psi^m f \|_{L^p({\mathbb T}_{L}^3)}^{1-\theta} \| \phi^m \nabla^m f \|_{L^q({\mathbb T}_{L}^3)}^{\theta} \psi^{-m(1-\theta)} \phi^{-m\theta}  + \left\| \psi^m f \right\|_{L^p({\mathbb T}_{L}^3)} \cdot  \langle x \rangle^{3\theta (1/q-1/p)-m\theta} \psi^{-m}.
\end{equation}
If moreover $p = \infty, q = 2$ and $\psi = 1$, we obtain
\begin{equation} \label{eq:GNresultinfty_simplifiedtorus}
    \left| \nabla^i f \right| \les \| f \|_{L^\infty({\mathbb T}_{L}^3)}^{1-\frac{i}{m-3/2}} \| \phi^m \nabla^m f \|_{L^q({\mathbb T}_{L}^3)}^{\frac{i}{m-3/2}}  \phi^{-m\frac{i}{m-3/2}}  + \left\| f \right\|_{L^\infty({\mathbb T}_{L}^3)} \cdot  \langle x \rangle^{-i}.
\end{equation}
For $\bar{r} \in [1,\infty)$, any $\eps > 0$, and under the extra assumption:
\begin{equation}\label{eq:GN_extracondtorus}
    \left( \frac{\phi(x)}{\langle x \rangle \psi(x) }\right)^{m\theta} \langle x \rangle^{3\theta (1/q-1/p)} \les 1
\end{equation}
we have the weighted Gagliardo-Nirenberg inequalities:
\begin{equation} \label{eq:GNresulttorus}
    \| \langle x \rangle^{-\eps} \psi^{m(1-\theta)} \phi^{m\theta} \nabla^i f \|_{L^{\bar{r}}({\mathbb T}_{L}^3)} \les \| \psi^m f \|_{L^p({\mathbb T}_{L}^3)}^{1-\theta} \| \phi^m \nabla^m f \|_{L^q({\mathbb T}_{L}^3)}^{\theta} + \| \psi^m f \|_{L^p({\mathbb T}_{L}^3)}
\end{equation}
The implicit constants in \eqref{eq:GNresultinftytorus} and \eqref{eq:GNresult} may depend on the parameters $p, q, i, m, \theta, C$, (as well as $\bar{r}$, $\eps >0$ for \eqref{eq:GNresult}) but they are independent of $f$,$L$, and $\psi$, $\phi$. 
\end{lemma}
\begin{proof}
We first divide dyadically $\mathbb{T}_{L}^{3}$.
Let $\delta=\log_{2}(L)-[\log_{2}(L)]$, where $[x]$ is the largest integer not greater than $x$. From the def, $\delta \in[1,2]$ and  $\log_{2}(\frac{L}{\delta})$ is an integer. Define $\tilde{A}_{-1}=\mathbb{T}_{\delta}^{3},$ $\tilde{A}_{j}=\mathbb{T}_{\delta 2^{j+1}}^{3}\backslash \mathbb{T}_{\delta 2^{j}}^{3}$ for $j\geq 0$. Then we have $\mathbb{T}_{L}^{3}=\cup_{j\leq \log_{2}(\frac{L}{\delta})-1}\tilde{A}_{j}.$

Now we show for any $j$, $x\in \tilde{A}_{j}$, there exists $\bar{\phi}_{j}$, $\bar{\psi}_{j},$ such that  $\phi (x) \in [\bar{\phi}_{j}/C^{3}, C^{3} \bar{\phi_j}]$ and $\psi(x) \in [\bar{\phi}_{j}/C^{3}, C^{3}\bar{\psi_j}].$ For any $x\in \tilde{A_j},$ we have $\delta 2^{j}\leq |x|\leq \sqrt{2}\delta 2^{j+1}\leq 2^{j+3}.$ Let $\tilde{\psi_{j}}=\max\{\psi_{j},\psi_{j+1},\psi_{j+2}\},\hat{\psi_{j}}=\min\{\psi_{j},\psi_{j+1},\psi_{j+2}\}$. Then we get $\psi(x)\in [\frac{\hat{\psi_{j}}}{C},C\tilde{\psi_{j}}].$ Moreover, from the condition of $\psi$, we have $\psi_{j+1}\leq \psi_{j} C^2 \leq \psi_{j+1} C^4.$ Then $\tilde{\psi}_{j}\leq \hat{\psi}_{j} C^6.$ Therefore $[\hat{\psi}_{j}/C, \tilde{\psi_j}]\subset [\sqrt{\hat{\phi_{j}}\tilde{\phi_{j}}}\frac{1}{C^{3}},\sqrt{\hat{\phi_{j}}\tilde{\phi_{j}}}C^3].$

Then we could use the similar proof as in Lemma \ref{lemma:GN_generaltorus} by estimate the dyadic pieces on $\tilde{A}_j$ instead of $A_{j}$. 
\end{proof}
\begin{lemma} \label{lemma:GN_generalnoweighttorus} Let $f:\mathbb{T}_{L}^3 \to \mathbb R^3$ be a periodic  function on $H^{m}(\mathbb{T}_{L}^3)$. $C_0\geq 1$, $L\geq 10C_0$. Assume that the parameters $p, q, r, \theta$ satisfy:
\begin{equation} \label{eq:GN_conditions02torus_periodic}
    \frac{1}{\bar{r}} = \frac{i}{3} + \theta \left( \frac{1}{q} - \frac{m}{3} \right) + \frac{1-\theta}{p},  \qquad \frac{i}{m} \leq \theta < 1, \qquad \mbox{ and }  \qquad \frac{m}{3}\geq \frac{1}{q}-\frac{1}{p},
\end{equation}
we have the Gagliardo-Nirenberg inequalities:
\begin{equation} \label{eq:GNresultnoweighttorus}
    \|   \nabla^i f \|_{L^{\bar{r}}(|y|\geq C_0, y\in e^s\mathbb{T}_{L}^3)} \les \|  f \|_{L^p(|y|\geq C_0, y\in e^s\mathbb{T}_{L}^3)}^{1-\theta} \|  \nabla^m f \|_{L^q(|y|\geq C_0, y\in e^s\mathbb{T}_{L}^3)}^{\theta} + \| f \|_{L^p(|y|\geq C_0, y\in e^s\mathbb{T}_{L}^3)}.
\end{equation}
The implicit constants in \eqref{eq:GNresultnoweighttorus} may depend on the parameters $p, q, i, m, \theta, \bar{r}$, but they are independent of $f$ and $L, C_0$. 
\begin{proof}

We first show for any $k$, $D_{k}=\{y|C_0\leq |y|\leq C_0(\frac{4}{3})^{k}\}$, we have
\begin{align}\label{GNinequalityDk}
    \|   \nabla^i f \|_{L^{\bar{r}}(D_k)} \les \|  f \|_{L^p(D_k)}^{1-\theta} \|  \nabla^m f \|_{L^q(D_k)}^{\theta} + \| f \|_{L^p(D_k)},
\end{align}
with the implicit constant independent of $C_0$, $k$.

Let $A$ be the annulus where $1\leq |x| \leq \frac{4}{3}$. Then $D_k=\cup_{j=1}^{k}\lambda_j A$ with $\lambda_j=C_0(\frac{4}{3})^{j-1}.$

Moreover, from the Gagliardo-Nirenberg inequality in bounded domains, we have 
\begin{align*}
\|\nabla^{i}(f(\lambda_j x))\|_{L^{\bar{r}}(A)}\les \|\nabla^{m}(f(\lambda_j x))\|_{L^{q}(A)}^{\theta}\|f(\lambda_j x)\|_{L^{p}(A)}^{1-\theta}+\|f(\lambda_j x)\|_{L^{p}(A)}.
\end{align*}
This is equivalent to 
\begin{align*}
\lambda_{j}^{i-\frac{3}{\bar{r}zz}}\|\nabla^{i}(f(x))\|_{L^{\bar{r}}(\lambda_j A)}\les  \lambda_j^{(m-\frac{3}{q})\theta-\frac{3}{p}(1-\theta)}\|\nabla^{n}(f(x))\|_{L^{q}(A)}^{\theta}\|f( x)\|_{L^{p}(\lambda_j A)}^{1-\theta}+\lambda_j^{-\frac{3}{p}}\|f(x)\|_{L^{p}(\lambda_j A)}.
\end{align*}
Since $i-\frac{3}{\bar{r}}=(m-\frac{3}{q})\theta-\frac{3}{p}(1-\theta)$, we have
\begin{align*}
\|\nabla^{i}(f(x))\|_{L^{\bar{r}}(D_k)}^{\bar{r}}=\sum_{j=1}^{k}(\|\nabla^{i}(f(x))\|_{L^{\bar{r}}(\lambda_j A)})^{\bar{r}}&\les \sum_{j=1}^{k} \|\nabla^{m}(f(x))\|_{L^{q}(\lambda_j A)}^{\bar{r}\theta}\|f( x)\|_{L^{p}(\lambda_j A)}^{\bar{r}(1-\theta)}+\sum_{j=1}^{k} \lambda_j^{(-\frac{3}{p}+\frac{3}{r}-i)r}\|f(x)\|_{L^{p}(\lambda_j A)}^{\bar{r}}\\
&\les  \sum_{j=1}^{k} \|\nabla^{m}(f(x))\|_{L^{q}(\lambda_j A)}^{\bar{r}\theta}\|f( x)\|_{L^{p}(\lambda_j A)}^{\bar{r}(1-\theta)}+\sum_{j=1}^{k} (\lambda_j^{(-\frac{3}{p}+\frac{3}{r}-i)r})\|f(x)\|_{L^{p}(D_k)}^{\bar{r}}\\
&\les  \sum_{j=1}^{k} \|\nabla^{m}(f(x))\|_{L^{q}(\lambda_j A)}^{q*\frac{\bar{r}\theta}{q}}\|f( x)\|_{L^{p}(\lambda_j A)}^{p*\frac{\bar{r}(1-\theta)}{p}}+\|f(x)\|_{L^{p}(D_k)}^{\bar{r}},
\end{align*}
where we use $-\frac{3}{p}+\frac{3}{\bar{r}}-i<0$ in the last inequality.

Now we claim for $l_1,l_2>0, l_1+l_2\geq 1$, $a_j,b_j\geq 0$,  
\begin{align}\label{holderine}
\sum_{j}a_j^{l_1}b_j^{l_2}\leq (\sum_{j}a_j)^{l_1}(\sum_{j}b_j)^{l_2}.
\end{align}
In fact, when $l_1\geq \frac{1}{2}$, $l_2\geq \frac{1}{2}$, from Holder's inequality, we have 
\begin{align*}
&\sum_{j}a_j^{l_1}b_j^{l_2}\leq \sum_{j}a_j^{\frac{1}{2}}b_j^{\frac{1}{2}}(\sup_{j}{a_j})^{l_1-\frac{1}{2}}(\sup_{j}{b_j})^{l_2-\frac{1}{2}}\\
&\leq (\sum_j(a_j))^{\frac{1}{2}}(\sum_j(b_j))^{\frac{1}{2}}(\sum_{j}{a_j})^{l_1-\frac{1}{2}}(\sum_{j}{b_j})^{l_2-\frac{1}{2}}\\
&\leq(\sum_{j}a_j)^{l_1}(\sum_{j}b_j)^{l_2}.
\end{align*}
When $l_1\leq \frac{1}{2}, $ we have
\begin{align*}
&\sum_{j}a_j^{l_1}b_j^{l_2}\leq \sum_{j}a_j^{l_1}b_j^{1-l_1}(\sup_{j}{b_j})^{l_2+l_1-1}\\
&\leq (\sum_ja_j)^{l_1}(\sum_jb_j)^{1-l_1}(\sum_{j}{b_j})^{l_2+l_1-1}\\
&\leq(\sum_{j}a_j)^{l_1}(\sum_{j}b_j)^{l_2}.
\end{align*}
Then from \eqref{holderine}, since $\frac{\bar{r}\theta}{q}+\frac{\bar{r}(1-\theta)}{p}=1-\frac{i\bar{r}}{3}+\theta\frac{m\bar{r}}{3}\geq 1$, we have
\begin{align*}
&\|\nabla^{i}(f(x))\|_{L^{\bar{r}}(D_k)}^{\bar{r}}=\sum_{j=1}^{k}(\|\nabla^{i}(f(x))\|_{L^{\bar{r}}(\lambda_j A)})^{\bar{r}}\\
&\les  (\sum_{j=1}^{k} \|\nabla^{m}(f(x))\|_{L^{q}(\lambda_j A)}^{q})^{\frac{\bar{r}\theta}{q}}(\sum_{j=1}^{k}\|f( x)\|_{L^{p}(\lambda_j A)}^{p})^{\frac{\bar{r}(1-\theta)}{p}}+\|f(x)\|_{L^{p}(D_k)}^{\bar{r}},\\
&\les\|\nabla^{m}(f(x))\|_{L^{q}(D_k)}^{\bar{r}\theta}\|f(x)\|_{L^{p}(D_k)}^{\bar{r}(1-\theta)}+\|f(x)\|_{L^{p}(D_k)}^{\bar{r}}.
\end{align*}
Then \eqref{GNinequalityDk} follows.

For the inequality over $e^{s}\mathbb{T}_{L}^3$, let us denote $F = [-L e^s, L e^s]^3 \setminus B ( C_0 )$ corresponding to the points $y$ of $e^s \mathbb T_L^3$ (in the cube centered at the origin) with $|y| \geq C_0$. We fix the minimum $k\in \mathbb N$ for which $F \subset D_k$ (that is, the minimum $k$ for which $C_0 \left( \frac43 \right)^k \geq \sqrt{3} L e^s$ ).  %Moreover, for any $x\in \partial {D_{k-1}}$, we have
%\begin{align*}
%&\text{dist}(x,\partial {D_{k}})\leq \frac{1}{3}C_{0}(\frac{4}{3})^{k-1}\leq \frac{1}{3}L\leq \frac{1}{2}(L-C_0).
%\end{align*}
Since $f$ is an periodic function on $e^s\mathbb{T}_{L}^3$, there is an natural periodic extension $E(f)$ to $D_k$ satisfying
\[
\|\nabla^{i}f\|_{L^{l}(F)}\leq
\|\nabla^{i}E(f)\|_{L^{l}(D_k)}\leq 3^3 \|\nabla^{i}f\|_{L^{l}(F)}.
\]
for all $i,l$. This allows us to conclude \eqref{eq:GNresultnoweighttorus} from \eqref{GNinequalityDk} 
\end{proof}
\end{lemma}

\begin{lemma} \label{lemma:GN_generalnoweightwholespace} Let $f:\mathbb R^3 \to \mathbb R^3$ be in $H^{m}(\mathbb R^3 )$. $C_0\geq 1$, $L\geq 10C_0$. Assume that the parameters $p, q, r, \theta$ satisfy:
\begin{equation} \label{eq:GN_conditions02torus}
    \frac{1}{\bar{r}} = \frac{i}{3} + \theta \left( \frac{1}{q} - \frac{m}{3} \right) + \frac{1-\theta}{p},  \qquad \frac{i}{m} \leq \theta < 1, \qquad \mbox{ and }  \qquad \frac{m}{3}\geq \frac{1}{q}-\frac{1}{p},
\end{equation}
we have the Gagliardo-Nirenberg inequalities:
\begin{equation} \label{eq:GNresultnoweightwholespace}
    \|   \nabla^i f \|_{L^{\bar{r}}(|y|\geq C_0)} \les \|  f \|_{L^p(|y|\geq C_0)}^{1-\theta} \|  \nabla^m f \|_{L^q(|y|\geq C_0)}^{\theta} + \| f \|_{L^p(|y|\geq C_0)}.\end{equation}
The implicit constants in \eqref{eq:GNresultnoweighttorus} may depend on the parameters $p, q, i, m, \theta, \bar{r}$, but they are independent of $f$ and $L, C_0$.
\end{lemma}
\begin{proof}
 We could write 
$D=\{y||y|\geq C_0\}=\cup_{j=1}^{\infty}\lambda_j A$, where $A$, $\lambda_j$ are same as in Lemma \ref{lemma:GN_generalnoweighttorus}. Then the similar proof holds.
\end{proof}

\printbibliography

\begin{tabular}{l}
  \textbf{Gonzalo Cao-Labora} \\
  {Courant Institute of Mathematical Sciences} \\
  {New York University} \\
  {251 Mercer Street, 619} \\
  {New York, NY 10012, USA} \\
  {Email: gc2703@nyu.edu} \\ \\
  \textbf{Javier G\'omez-Serrano}\\
  {Department of Mathematics} \\
  {Brown University} \\
  {314 Kassar House, 151 Thayer St.} \\
  {Providence, RI 02912, USA} \\
  {Email: javier\_gomez\_serrano@brown.edu} \\ \\
  \textbf{Jia Shi} \\
  {Departament of Mathematics} \\
  {Massachusetts Institute of Technology} \\
  {182 Memorial Drive, 2-157} \\
  {Cambridge, MA 02139, USA} \\
  {Email: jiashi@mit.edu} \\ \\
  \textbf{Gigliola Staffilani} \\
  {Departament of Mathematics} \\
  {Massachusetts Institute of Technology} \\
  {182 Memorial Drive, 2-251} \\
  {Cambridge, MA 02139, USA} \\
  {Email: gigliola@math.mit.edu} \\
\end{tabular}

\end{document}